\documentclass[reqno]{amsart}
\pdfoutput=1

\usepackage[mathscr]{euscript}
\usepackage{amssymb}
\usepackage{graphicx}
\usepackage[shortlabels]{enumitem}
\usepackage{mathtools}
\usepackage{booktabs}
\usepackage[dvipsnames]{xcolor}
\usepackage{tikz}
\usepackage{tikz-cd}
\usepackage{tikz-3dplot}
\usetikzlibrary{decorations.pathreplacing}
\usetikzlibrary{decorations.markings}
\usetikzlibrary{positioning,cd,arrows, patterns}
\usetikzlibrary{arrows.meta}
\usepackage{pgfplots}
\usepackage{subfigure}
\usepackage{ragged2e}
\usepackage{a4wide}
\usepackage{comment}
\usepackage{hyperref}
\hypersetup{colorlinks}
\usepackage{amsbsy}
\usepackage{xargs}
\usepackage{multirow}
\usepackage{array}
\usepackage{float}
\usepackage{colonequals}
\hypersetup{
colorlinks=true,       
linkcolor=blue,          
citecolor=red,        
filecolor=violet,      
urlcolor=violet       
}
\usepackage{caption} 
\makeatletter
\def\paragraph{\@startsection{paragraph}{4}%
  \z@\z@{-\fontdimen2\font}%
  {\normalfont\bfseries}}
\makeatother

\theoremstyle{plain}
\newtheorem{theorem}{Theorem}[section]
\newtheorem{lemma}[theorem]{Lemma}
\newtheorem{proposition}[theorem]{Proposition}
\newtheorem{corollary}[theorem]{Corollary}

\theoremstyle{remark}
\newtheorem{remark}[theorem]{Remark}
\newtheorem{question}[theorem]{Question}

\theoremstyle{definition}
\newtheorem{definition}[theorem]{Definition}
\newtheorem{assumption}[theorem]{Assumption}
\newtheorem{example}[theorem]{Example}

\DeclareMathOperator{\mmon}{\mu mon}
\DeclareMathOperator{\cl}{Cl}
\DeclareMathOperator{\Sh}{\mathscr{S}\mathsf{h}}
\DeclareMathOperator{\Loc}{\mathscr{L}\mathsf{oc}}
\DeclareMathOperator{\Spec}{Spec}

\newcommand\bbA{\mathbb{A}}
\newcommand\R{\mathbb{R}}
\newcommand\bbC{\mathbb{C}}
\newcommand\D{\mathbb{D}}
\newcommand\Z{\mathbb{Z}}
\newcommand\N{\mathbb{N}}
\newcommand\bS{\mathbb{S}}
\newcommand\F{\mathbb{F}}

\newcommand\bfa{\mathbf{a}}
\newcommand\bfb{\mathbf{b}}
\newcommand\bfc{\mathbf{c}}
\newcommand\bfd{\mathbf{d}}
\newcommand\bfe{\mathbf{e}}
\newcommand\bfs{\mathbf{s}}
\newcommand\bfx{\mathbf{x}}
\newcommand\bfy{\mathbf{y}}
\newcommand\bfz{\mathbf{z}}

\newcommand\cA{\mathcal{A}}
\newcommand\cC{\mathcal{C}}
\newcommand\cF{\mathcal{F}}
\newcommand\cL{\mathcal{L}}
\newcommand\cM{\mathcal{M}}
\newcommand\cX{\mathcal{X}}

\newcommand\sfI{\mathsf{I}}
\newcommand\sfT{\mathsf{T}}
\newcommand\sfY{\mathsf{Y}}

\newcommand{\clusterfont}{\mathcal}
\newcommand{\dynkinfont}{\mathsf}
\newcommand{\ngraphfont}{\mathscr}

\newcommand{\quiver}{\clusterfont{Q}}
\newcommand{\qbasis}{\clusterfont{\tilde{B}}}
\newcommand{\qbasispr}{\clusterfont{B}}
\newcommand{\qcoxeter}{\mutation_{\quiver}}
\DeclareMathOperator{\exchange}{Ex}
\DeclareMathOperator{\pr}{pr}
\newcommand{\qbpr}{\clusterfont{B}}
\newcommand{\exchangesub}[2]{\exchange({#1},{#2})}
\DeclareMathAlphabet{\mathpzc}{OT1}{pzc}{m}{it}
\newcommand{\Atori}{\mathpzc{A}}
\newcommand{\Xtori}{\mathpzc{X}}

\newcommand{\ngraph}{\ngraphfont{G}}
\newcommand{\nbasis}{\ngraphfont{B}}
\newcommand{\ncoxeter}{\mutation_\ngraph}
\newcommand{\coxeterpadding}{\ngraphfont{C}}
\newcommand{\nannulus}{\ngraphfont{A}}
\newcommand{\nbasistilde}{\tilde{\nbasis}}

\newcommand{\dynA}{\dynkinfont{A}}
\newcommand{\dynB}{\dynkinfont{B}}
\newcommand{\dynC}{\dynkinfont{C}}
\newcommand{\dynD}{\dynkinfont{D}}
\newcommand{\dynE}{\dynkinfont{E}}
\newcommand{\dynF}{\dynkinfont{F}}
\newcommand{\dynG}{\dynkinfont{G}}
\newcommand{\dynX}{\dynkinfont{Z}}
\newcommand{\dynY}{{\dynkinfont{Z}^G}}

\newcommand{\dynADE}{\dynkinfont{ADE}}
\newcommand{\dynBCFG}{\dynkinfont{BCFG}}

\newcommand{\exdynA}{\widetilde{\dynA}}
\newcommand{\exdynB}{\widetilde{\dynB}}
\newcommand{\exdynC}{\widetilde{\dynC}}
\newcommand{\exdynD}{\widetilde{\dynD}}
\newcommand{\exdynE}{\widetilde{\dynE}}
\newcommand{\exdynF}{\widetilde{\dynF}}
\newcommand{\exdynG}{\widetilde{\dynG}}
\newcommand{\exdynX}{\widetilde{\dynX}}

\newcommand{\facet}{\mathscr{F}}
\newcommand{\seed}{\Sigma}
\newcommand{\initialseed}{\Sigma_{t_0}}
\newcommand{\flags}{\mathcal{F}_\legendrian}
\newcommand{\mutation}{\mu}
\DeclareMathOperator{\Fan}{Fan}
\DeclareMathOperator{\re}{re}

\newcommand{\annulus}{\mathbb{A}}
\newcommand{\disk}{\mathbb{D}}
\newcommand{\sphere}{\mathbb{S}}

\newcommand{\legendrian}{\lambda}
\newcommand{\Legendrian}{\Lambda}

\newcommand{\Ngraphs}{\mathscr{N}\mathsf{graphs}}

\newcommand{\boundary}{\partial}
\newcommand{\cycle}{\gamma}

\newcommand{\field}{\mathbb{F}}

\newcommand{\Roots}{\Phi}
\newcommand{\SRoots}{\Pi}
\newcommand{\alposRoots}{\Roots_{\ge -1}}

\newcommand{\Move}[1]{{\rm{(#1)}}}

\newcommand{\wavefront}{\Gamma}

\newcommand{\boundellipse}[3]
{(#1) ellipse (#2 and #3)}

\newcommand{\1}{\sigma_1}
\newcommand{\2}{\sigma_2}
\newcommand{\3}{\sigma_3}
\newcommand{\4}{\sigma_4}
\newcommand{\5}{\sigma_{1,3}}

\newcommand{\reducedto}{\succ}

\newcommand{\brick}{\mathsf{brick}}
\newcommand{\degen}{\mathsf{degen}}

\newcommand{\yjnote}[1]{\marginpar{\RaggedRight\hsize1.7cm\tiny{\color{cyan}\raggedright #1}}}
\newcommand{\ejnote}[1]{\marginpar{\RaggedRight\hsize1.7cm\tiny{\color{magenta}\raggedright #1}}}
\newcommand{\bhnote}[1]{\marginpar{\RaggedRight\hsize1.7cm\tiny{\color{blue}\raggedright #1}}}

\colorlet{cyclecolor1}{orange}
\colorlet{cyclecolor2}{green}
\colorlet{cyclecolor3}{black}
\colorlet{cyclecolor4}{violet}
\def\cyclecolornamefirst{orange}
\def\cyclecolornamesecond{green}

\numberwithin{equation}{section}

\tikzset{ynode/.style = {circle, fill=cyclecolor1, inner sep = 2pt, opacity = 0.5}}
\tikzset{gnode/.style = {circle, fill=cyclecolor2, inner sep = 2pt, opacity = 0.5}}
\tikzset{bnode/.style = {circle, fill=cyclecolor3, inner sep = 2pt, opacity = 0.5}}
\tikzset{vnode/.style = {circle, fill=cyclecolor4, inner sep = 2pt, opacity = 0.5}}
\tikzstyle{Dnode}=[draw, circle, inner sep = 0.07cm]
\tikzstyle{Fnode}=[draw, rectangle, inner sep = 0.1cm]
\tikzstyle{double line} = [
decoration={
    markings,
    mark=at position 0.55 with {\arrow[line width = 0.5pt,scale=1]{angle 90}}},
	double distance = 1.5pt, 
	double=\pgfkeysvalueof{/tikz/commutative diagrams/background color},
	postaction={decorate}
]

\tikzstyle{triple line} = [
decoration={
    markings,
    mark=at position 0.55 with {\arrow[line width = 0.5pt,scale=1]{angle 90}}},
	double distance = 2pt, 
	double=\pgfkeysvalueof{/tikz/commutative diagrams/background color},
	postaction={decorate}
]

\tikzset{Dble/.style args={#1 and #2}{%
decoration={%
markings,
mark=at position 0 with {%
    \coordinate (Axx) at (0,.5\pgflinewidth);%
    \coordinate (Bxx) at (0,-.5\pgflinewidth);},%
mark=at position 1 with {%
    \coordinate (Ayy) at (0,.5\pgflinewidth);%
    \coordinate (Byy) at (0,-.5\pgflinewidth);%
    \draw [#1] (Axx)--(Ayy) ;%
    \draw [#2] (Bxx)--(Byy) ;}%
},%
postaction={decorate}}
}

\tikzset{
  on each segment/.style={
    decorate,
    decoration={
      show path construction,
      moveto code={},
      lineto code={
        \path [#1]
        (\tikzinputsegmentfirst) -- (\tikzinputsegmentlast);
      },
      curveto code={
        \path [#1] (\tikzinputsegmentfirst)
        .. controls
        (\tikzinputsegmentsupporta) and (\tikzinputsegmentsupportb)
        ..
        (\tikzinputsegmentlast);
      },
      closepath code={
        \path [#1]
        (\tikzinputsegmentfirst) -- (\tikzinputsegmentlast);
      },
    },
  },
  mid arrow/.style={postaction={decorate,decoration={
        markings,
        mark=at position .5 with {\arrow[#1]{stealth}}
      }}},
} 

\newcommand*\circled[1]{\tikz[baseline=(char.base)]{
\node[shape=circle,draw,inner sep=0.8pt] (char) {#1};}}

\makeatletter 
\tikzset{curlybrace/.style={rounded corners=2pt,line cap=round}}%
\pgfkeys{
/curlybrace/.cd,%
tip angle/.code     =  \def\cb@angle{#1},
/curlybrace/.unknown/.code ={\let\searchname=\pgfkeyscurrentname
                          \pgfkeysalso{\searchname/.try=#1,
                          /tikz/\searchname/.retry=#1}}}  
\def\curlybrace{\pgfutil@ifnextchar[{\curly@brace}{\curly@brace[]}}%

\def\curly@brace[#1]#2#3#4{%
\pgfkeys{/curlybrace/.cd,
tip angle = 0.75}%
\pgfqkeys{/curlybrace}{#1}%
\ifnum 1>#4 \def\cbrd{0.05} \else \def\cbrd{0.075} \fi
\draw[/curlybrace/.cd,curlybrace,#1]  (#2:#4-\cbrd) -- (#2:#4) arc (#2:{(#2+#3)/2-\cb@angle}:#4) --({(#2+#3)/2}:#4+\cbrd) coordinate (curlybracetipn);
\draw[/curlybrace/.cd,curlybrace,#1] ({(#2+#3)/2}:#4+\cbrd) -- ({(#2+#3)/2+\cb@angle}:#4) arc ({(#2+#3)/2+\cb@angle} :#3:#4) --(#3:#4-\cbrd);
}
\makeatother

\title{Lagrangian fillings for Legendrian links of finite or affine Dynkin type}

\author{Byung Hee An}
\email{anbyhee@knu.ac.kr}
\address{Department of Mathematics Education, Kyungpook National University, Republic of Korea}

\author{Youngjin Bae}
\email{yjbae@inu.ac.kr}
\address{Department of Mathematics, Incheon National University, Republic of Korea}

\author{Eunjeong Lee}
\email{eunjeong.lee@chungbuk.ac.kr}
\address{Center for Geometry and Physics, Institute for Basic Science (IBS), Pohang 37673, Republic of Korea}
\curraddr{Department of Mathematics, 
	Chungbuk National University,
	Cheongju 28644, Republic of Korea}


\begin{document}

\begin{abstract}
We prove that there are at least as many exact embedded Lagrangian fillings 
as seeds for Legendrian links of finite type $\dynADE$ or affine type $\exdynD \exdynE$. We also provide 
as many Lagrangian fillings with rotational symmetry as seeds of type $\dynB$, $\dynG_2$, $\exdynG_2$, $\exdynB$, or $\exdynC_2$, and with conjugation symmetry as seeds of type $\dynF_4$, $\dynC$, $\dynE_6^{(2)}$, $\exdynF_4$, or $\dynA_5^{(2)}$.
These families are the first known Legendrian links with (infinitely many) exact Lagrangian fillings (with symmetry) that exhaust all seeds in the corresponding cluster structures beyond type $\dynA \dynD$.
Furthermore, we show that the $N$-graph realization of (twice of) Coxeter mutation of type $\exdynD\exdynE$ corresponds to a Legendrian loop of the corresponding Legendrian links. Especially, the loop of type $\exdynD$ coincides with the one considered by Casals and Ng.
\end{abstract}

\maketitle

\tableofcontents

\section{Introduction}

\subsection{Background}
Legendrian knots are central objects in the study of 3-dimensional
contact manifolds. Classification of Legendrian knots is important in its
own right and also plays a prominent role in classifying 4-dimensional
Weinstein manifolds.

Classical Legendrian knot invariants are Thurston--Bennequin number and rotation
number~\cite{Gei2008} which distinguish the pair of Legendrian knots with the
same knot type. There are non-classical invariants including the Legendrian
contact algebra via the method of Floer theory~\cite{EGH2000, Che2002}, and
the space of constructible sheaves using microlocal
analysis~\cite{GKS2012,STZ2017}. These non-classical invariants distinguish
the Chekanov pair, a pair of Legendrian knots of type $m5_2$ having the same
classical invariants.

Recently, the study of exact Lagrangian fillings for Legendrian links has
been extremely plentiful. In the context of Legendrian contact algebra, an exact
Lagrangian filling gives an augmentation through the functorial view 
point~\cite{EHK2016}. There are several level of equivalence between 
augmentations
and the constructible sheaves for Legendrian links from counting to
categorical equivalence~\cite{NRSSZ2015}. 
Using these idea of
augmentations and constructible sheaves, people construct infinitely many fillings for certain Legendrian links~\cite{CG2020, GSW2020b,
CZ2020}. Here is the summarized list of methods of constructing Lagrangian
fillings for Legendrian links:
\begin{enumerate}
\item Decomposable Lagrangian fillings via pinching sequences and Legendrian loops \cite{EHK2016, Kal2006, CN2021}.
\item Alternating Legendrians and its conjugate Lagrangian fillings
\cite{STWZ2019}. 
\item Legendrian weaves via $N$-graphs and Legendrian mutations \cite{TZ2018, CZ2020}. 
\item Donaldson--Thomas transformation on augmentation varieties
\cite{SW2019, GSW2020a, GSW2020b}.
\end{enumerate}

Cluster algebras, introduced by Fomin and Zelevinsky~\cite{FZ1_2002}, play a crucial role in
the above constructions and applications. More precisely, the space of
augmentations and the moduli of constructible sheaves of microlocal rank one adapted to
Legendrian links admit structures of cluster pattern and $Y$-pattern, respectively~\cite{STWZ2019, SW2019, GSW2020a}. 
Note that a $Y$-seed of cluster algebra consists of a quiver whose vertices are
decorated with variables, called \emph{coefficients}. An involutory operation at each vertex,
called \emph{mutation}, generates all seeds of the $Y$-pattern.
The main point is to identify the
mutation in the $Y$-pattern and an operation in the space of Lagrangian
fillings. This geometric operation is deeply related to the Lagrangian surgery~\cite{Pol1991} and
the wall-crossing phenomenon~\cite{Aur2007}.

Indeed, a Legendrian torus link of type $(2,n)$ admits as many exact
Lagrangian fillings as Catalan number up to exact Lagrangian isotopy \cite{Pan2017, STWZ2019,
TZ2018}. Interestingly enough, the Catalan number is the number of seeds in a
cluster pattern of Dynkin type $\dynA_{n-1}$. 
There are also Legendrian links corresponding to finite Dynkin type $\dynD\dynE$, and affine Dynkin type $\exdynD\exdynE$ \cite{GSW2020b}.
A conjecture by Casals \cite[Conjecture~5.1]{Cas2020} says that the number of
distinct exact embedded Lagrangian fillings (up to exact Lagrangian isotopy) for
Legendrian links of type $\dynADE$ is exactly the same as the number of seeds
of the corresponding cluster algebras.

Furthermore, it is also conjectured by Casals \cite[Conjecture~5.4]{Cas2020} that for Legendrian links of type $\dynA_{2n-1}, \dynD_{n+1}, \dynE_6$ and $\dynD_4$, Lagrangian fillings having certain $\Z/2\Z$ or $\Z/3\Z$-symmetry form the cluster patterns of type $\dynB_n, \dynC_n, \dynF_4$ and $\dynG_2$, which are Dynkin diagrams obtained by \emph{folding} as explained in~\cite{FZ_Ysystem03}.

\subsection{The results}
\subsubsection{Lagrangian fillings for Legendrians of type $\dynADE$ or $\exdynD\exdynE$}
Our main result is that there are at least as many Lagrangian fillings for
Legendrian links of finite type as seeds in the corresponding cluster structures. We deal with $N$-graphs introduced by Casals and Zaslow \cite{CZ2020} to
construct the Lagrangian fillings. An $N$-graph $\ngraph$ on $\disk^2$ gives
a Legendrian surface $\Legendrian(\ngraph)$ in $J^1\disk^2$ while the
boundary $\boundary \ngraph$ on $\sphere^1$ induces a Legendrian link~$\legendrian(\boundary \ngraph)$. Then projection of $\Legendrian(\ngraph)$
along the Reeb direction becomes a Lagrangian filling of~$\legendrian(\boundary \ngraph)$.

As mentioned above, we interpret an $N$-graph as a $Y$-seed in the corresponding
$Y$-pattern. A one-cycle in the Legendrian surface
$\Legendrian(\ngraph)$ corresponds to a vertex of the quiver, and a signed
intersection between one-cycles gives an arrow between corresponding
vertices. From constructible sheaves adapted to $\Legendrian(\ngraph)$, one
can assign a monodromy to each one-cycle which becomes the coefficient at each vertex.

There is an operation so called a \emph{Legendrian mutation}
$\mutation_\cycle$ on an $N$-graph $\ngraph$ along one-cycle $[\cycle]\in
H_1(\Legendrian(\ngraph))$ which is the counterpart of the mutation on the
$Y$-pattern, see Proposition~\ref{proposition:equivariance of mutations}. The
delicate and challenging part is that we do not know whether Legendrian
mutations are always possible or not. Simply put, this is because the
mutation in cluster side is algebraic, whereas the Legendrian mutation is
rather geometric.

The main idea of our construction is to consider $N$-graphs $\ngraph(a,b,c)$ and $\ngraph(\exdynD_n)$ bounding Legendrian links $\legendrian(a,b,c)$ and $\legendrian(\exdynD_n)$, respectively.

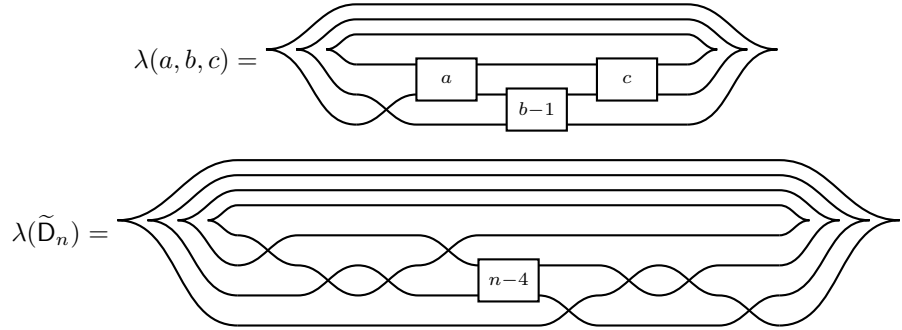
\begin{align*}
\legendrian(a,b,c)=

\end{align*}
Note that the above Legendrians $\legendrian(a,b,c)$ and $\legendrian(\exdynD_n)$ can be obtained by ($-1$)-closure of the following braids, respectively,
\begin{align*}
\beta(a,b,c)&=\sigma_2\sigma_1^{a+1}\sigma_2\sigma_1^{b+1}\sigma_2\sigma_1^{c+1},&
\beta(\exdynD_n)&=\left(\sigma_2\sigma_1^3\sigma_2\sigma_1^3\sigma_2\sigma_1^k\sigma_3\right)\cdot\left(\sigma_2\sigma_1^3\sigma_2\sigma_1^3\sigma_2\sigma_1^\ell\sigma_3\right),
\end{align*}
where $k=\lfloor \frac{n-3}2\rfloor$ and $\ell=\lfloor \frac{n-4}2\rfloor$, see Section~\ref{sec:N-graph of finite or affine type}.
Those braids provide boundary data of the following $N$-graphs which represent exact Lagrangian fillings of corresponding Legendrian links:

\begin{figure}[ht]
\subfigure[$(\ngraph(a,b,c),\nbasis(a,b,c))$\label{N-graph(a,b,c)}]{
$

$}
\caption{Pairs of $N$-graphs and tuples of cycles}
\label{fig:N-graphs of (a,b,c) and Dn}
\end{figure}

\noindent Here, the \colorbox{cyclecolor1!50!}{\cyclecolornamefirst}- and \colorbox{cyclecolor2!50!}{\cyclecolornamesecond}-shaded edges indicate a tuple of one-cycles $\nbasis$ in the corresponding Legendrian surface.
See~\S\ref{sec:1-cycles in Legendrian weaves} for the detail.

The Legendrians $\legendrian(a,b,c), \legendrian(\exdynD_n)$ are the rainbow closure of \emph{positive braids}. 
By the work of Shen--Weng \cite{SW2019}, it is direct to check that 
the corresponding cluster structure of Legendrian $\legendrian(\dynX)$ is 
indeed of type $\dynX$ for $\dynX\in\{\dynA,\dynD,\dynE,\exdynD,\exdynE\}$. More precisely, the coordinate 
ring of the moduli space $\cM_1(\legendrian(\dynX))$ of microlocal rank one 
sheaves in $\Sh^\bullet_{\legendrian(\dynX)}(\R^2)$ admits the aforementioned 
$Y$-pattern structure.

The (candidate) Legendrians of type $\exdynA$ are not the rainbow 
closure of positive braids, in general. 
Indeed, Casals--Ng~\cite{CN2021} considered a Legendrian link of type 
$\exdynA_{1,1}$ which is not the rainbow closure of a  positive braid. 
So we can not directly apply the subsequent argument to Legendrians of type 
$\exdynA$.

To prove the realizability of each $Y$-seed in the corresponding $Y$-pattern, we use an induction argument on the rank of the type $\dynX$. 
More precisely, for each $Y$-pattern, we consider the \emph{exchange graph}, whose vertices are the $Y$-seeds and whose edges connect the vertices related by a single mutation. 
It has been known that the exchange graph of a $Y$-pattern is determined by the Dynkin type $\dynX$ of the $Y$-pattern when $\dynX$ is finite or affine (cf. Propositions~\ref{thm_exchange_graph_Dynkin} and~\ref{prop_Y-pattern_exchange_graph}). Because of this, we denote by $\exchange(\Roots(\dynX))$ the exchange graph of a $Y$-pattern of type $\dynX$.
Here, $\Roots(\dynX)$ is the root system of type $\dynX$. Note that when $\dynX$ is of finite type, the exchange graph $\exchange(\Roots(\dynX))$ becomes the one-skeleton of a polytope, called the (\emph{generalized}) \emph{associahedron} (see Figures~\ref{fig_asso_A3_intro} and~\ref{fig_asso_D4}).
\begin{figure}[ht]
\tdplotsetmaincoords{110}{-30}

\caption{The type $\dynA_3$ associahedron}\label{fig_asso_A3_intro}
\end{figure}

A (fixed) sequence of mutations corresponding to a chosen Coxeter element provides an action on the exchange graph. We call this specific sequence of mutations a \emph{Coxeter mutation} $\mutation_{\quiver}$. The orbit of the initial seed is called \emph{bipartite belt}. The green dots in Figure~\ref{fig_asso_A3_intro} present the elements of the bipartite belt.
We notice that the facets meeting at the initial seed correspond to the exchange graphs $\exchange(\Roots(\dynX\setminus \{i\}))$. In Figure~\ref{fig_asso_A3_intro}, there are two pentagons and one square intersecting a green dot. Indeed, a pentagon is the type $\dynA_2$ generalized associahedron; a square is the type $\dynA_1 \times \dynA_1$ generalized associahedron. Moreover, by applying the Coxeter mutation on these facets iteratively, one can obtain all facets in the associahedron. 
Even though we do not have a polytope model for the exchange graph of affine type, similar properties hold, that is, one can reach any $Y$-seed in the exchange graph from the initial seed by taking Coxeter mutations and then applying a certain sequence of mutations omitting at least one vertex. 

The following good properties of the above pairs $(\ngraph(a,b,c),\nbasis(a,b,c))$ and $(\ngraph(\exdynD_{n}),\nbasis(\exdynD_{n}))$ play a crucial role in interpreting the Coxeter mutation $\qcoxeter$ in terms of $N$-graphs:
\begin{enumerate}
\item The geometric and algebraic intersection numbers between chosen one-cycles coincide. 
\item The corresponding quivers $\quiver(a,b,c)$, $\quiver(\exdynD_n)$ are bipartite, see~\S\ref{sec:N-graphs and seeds} for the details. 
\end{enumerate}
The property (2) naturally splits $\nbasis$ into two subsets $\nbasis_+$ and $\nbasis_-$.
In Figure~\ref{fig:N-graphs of (a,b,c) and Dn}, they consist of \colorbox{cyclecolor1!50!}{\cyclecolornamefirst}- and \colorbox{cyclecolor2!50!}{\cyclecolornamesecond}-shaded edges, respectively. 
Then the property (1) enables us to perform the \emph{Legendrian Coxeter mutation}, which is the $N$-graph realization of the Coxeter mutation defiend by the sequence of Legendrian mutations:
\[
\ncoxeter=\prod_{\cycle \in \nbasis_+} \mutation_{\cycle}\cdot\prod_{\cycle\in \nbasis_-} \mutation_{\cycle}.
\]
Then the resulting $N$-graphs $\ncoxeter(\ngraph(a,b,c),\nbasis(a,b,c))$ and $\ncoxeter(\ngraph(\exdynD_n),\nbasis(\exdynD_n))$ become the
$N$-graphs shown in Figure~\ref{figure:intro_Legendrian Coxeter mutation} up to a sequence of Move~\Move{II} in Figure~\ref{fig:move1-6}.
\begin{figure}[ht]
\subfigure[$\ncoxeter(\ngraph(a,b,c),\nbasis(a,b,c))$\label{ncoxeter_n(a,b,c)}]{
$

$}
\caption{After applying Legendrian Coxeter mutation on the initial pair}
\label{figure:intro_Legendrian Coxeter mutation}
\end{figure}

Removing the gray-shaded annulus region, $(\ngraph(\exdynD_n),\nbasis(\exdynD_n))$ and $\ncoxeter(\ngraph(\exdynD_n),\nbasis(\exdynD_n))$ are identical, and the only difference between
$(\ngraph(a,b,c),\nbasis(a,b,c))$ and $\ncoxeter(\ngraph(a,b,c),\nbasis(a,b,c))$ is the reverse of the color. Note that the intersection pattern between one-cycles and the
Legendrian mutability are preserved under the action of the Legendrian Coxeter mutation
$\ncoxeter$.
By the induction argument on the rank of root system, we conclude that there
in no (geometric) obstruction to realize each seed via the $N$-graph.

Note that the $N$-graphs $\ngraph(a,b,c)$ and $\ngraph(\exdynD_{n})$ include Lagrangian fillings of Legendrian links of type $\dynX\in\{\dynA,\dynD,\dynE,\exdynD,\exdynE\}$, see Table~\ref{table:short notations}.
In particular, $\ngraph(a,b,c)$ is of type $\dynADE$ or $\exdynD\exdynE$ if and only if $\frac 1a+\frac 1b+\frac 1c>1$ or $\frac 1a+\frac 1b+\frac 1c=1$, respectively.

This guarantees that there are at least as many Lagrangian fillings as seeds for $\legendrian(\dynX)$ for $\dynX\in\{\dynA,\dynD,\dynE,\exdynD,\exdynE\}$.

\begin{theorem}[Theorem~\ref{theorem:seed many fillings}]\label{thm_intro_1}
Let $\legendrian$ be a Legendrian knot or link of type~$\dynADE$ or type $\exdynD\exdynE$.
Then it admits at least as many distinct exact embedded Lagrangian fillings up to exact Lagrangian isotopy (rel boundary) as the number of seeds in the seed pattern of the same type.
See Table~\ref{table_seeds_and_cluster_variables} for the number of seeds of finite type. 
\end{theorem}

There are several ways of constructing exact embedded Lagrangian fillings as mentioned above.
Especially in $\dynD_4$ case, there are 34 distinct Lagrangian fillings constructed by the method of the alternating Legendrians \cite{BFFH2018,STWZ2019},
while the above $N$-graphs give 50 distinct Lagrangian fillings which is the number of seeds of the corresponding cluster pattern.
Most recently, for Legendrian links of type $\dynD_n$, Hughes \cite{Hughes2021} makes use of $3$-graphs together with 1-cycles to show that every sequence of quiver mutations can be realized by Legendrian weave mutations. Compared with our strategy using structural results of the cluster pattern, he studies $3$-graph moves arise from quivers of type $\dynD_n$ in a more direct and concrete way. As a corollary, he also obtained at least as many Lagrangian fillings as seeds in the cluster algebra of type $\dynD_n$.

There are many results showing the existence of (infinitely many) distinct Lagrangian fillings for Legendrian links, see \cite{EHK2016,Pan2017,TZ2018,STWZ2019,CG2020,CZ2020,GSW2020b,CN2021}. To the best of authors' knowledge, Theorem~\ref{thm_intro_1} is the first results of (infinitely many) Lagrangian fillings of Legendrian links which exhaust all seeds in the corresponding cluster pattern beyond type $\dynA\dynD$.

The gray-shaded annular $N$-graphs in the above figure can be seen as exact Lagrangian cobordisms. 
In particular, the annular $N$-graph $\coxeterpadding({\exdynD}_{4})$ for $\ncoxeter(\ngraph(\exdynD_{4}),\nbasis(\exdynD_{4}))$ corresponds to the cobordism from the Legendrian $\legendrian(\exdynD_4)$ to itself which defines a \emph{Legendrian loop} $\vartheta(\exdynD_4)$. 
See Figure~\ref{figure:Legendrian Coxeter padding and legendrian loop} for the case of $\dynD_n$ for general $n\geq 4$.
Note that this coincides with the Legendrian loop described in \cite[Figure~2]{CN2021} up to Reidemeister moves. For type~$\exdynE$, the twice of 
Legendrian Coxeter mutation on the pair $(\ngraph(a,b,c),\nbasis(a,b,c))$ gives
the Legendrian loop~$\vartheta(a,b,c)$ of $\legendrian(a,b,c)$ as shown in Figure~\ref{fig:legendrian loop of E_intro}.
The Legendrian loop $\vartheta(a,b,c)$ can be interpreted as the move of the half twist $\Delta_3$ along the three-strand braid band, whereas the Legendrian loop $\vartheta(\exdynD_n)$ is essentially the move of the half twist $\Delta_2$ along the two-strand braid band as depicted in Figure~\ref{fig:legendrian loop of D_intro}.

\begin{figure}[ht]
\subfigure[$\coxeterpadding({\exdynD}_{n})$]{

}
\caption{Legendrian Coxeter padding $\coxeterpadding({\exdynD}_{n})$ and the corresponding Legendrian loop $\vartheta_0({\exdynD}_{n})$}
\label{figure:Legendrian Coxeter padding and legendrian loop}
\end{figure}

\begin{theorem}[Theorem~\ref{thm:legendrian loop}]\label{theorem:legendrian loop}
The Legendrian Coxeter mutation $\mutation_\ngraph$ on 
$(\ngraph(\exdynD),\nbasis(\exdynD))$ and twice of Legendrian mutation 
$\mutation_\ngraph^{2}$ on $(\ngraph(\exdynE),\nbasis(\exdynE))$ induce 
Legendrian loops $\vartheta(\exdynD)$ and $\vartheta(\exdynE)$ in Figures~\ref{fig:legendrian loop of E_intro} and \ref{fig:legendrian loop of D_intro}, respectively. 
In particular, the order of the Legendrian loops are infinite as elements of the fundamental group of the space of Legendrians isotopic to $\lambda(\exdynD)$ and $\lambda(\exdynE)$, respectively.
\end{theorem}

Note that the above idea of Coxeter mutation also works for $(\ngraph(a,b,c),\nbasis(a,b,c))$ with $\frac{1}{a}+\frac{1}{b}+\frac{1}{c} < 1$. Indeed the operation $\qcoxeter$ is of infinite order and so is $\ncoxeter$, hence Legendrian weaves 
\[\Legendrian(\ncoxeter^r(\ngraph(a,b,c),\nbasis(a,b,c)))\]
produce infinitely many distinct Lagrangian fillings.
The quiver $\quiver(a,b,c)$ is also bipartite and one can perform 
the Legendrian Coxeter mutation $\mutation_{\ngraph}$ on the $N$-graph $\ngraph(a,b,c)$ by stacking the gray-shaded annulus like as before. Therefore, there is no obstruction to
realize seeds obtained by mutations $\ncoxeter^r$ via the $N$-graphs.
Since the order of the Legendrian Coxeter mutation is infinite (see Lemma~\ref{lemma:order of coxeter mutation}), we obtain infinitely many $N$-graphs
and hence infinitely many exact embedded Lagrangian fillings for the Legendrian
link $\legendrian(a,b,c)$ with $\frac{1}{a}+\frac{1}{b}+\frac{1}{c} < 1$.

\begin{theorem}[Theorem~\ref{theorem:infinite fillings}]\label{thm_intro_infinite_fillings}
For each $a,b,c\ge 1$, the Legendrian knot or link $\legendrian(a,b,c)$ has 
infinitely many distinct Lagrangian fillings if
\[
\frac1a+\frac1b+\frac1c < 1.
\]
\end{theorem}

Gao--Shen--Weng \cite{GSW2020b} already proved the existence of infinitely many Lagrangian fillings for much general type of positive braid Legendrian links. Their main idea is to use the aperiodicity of \emph{Donaldson--Thomas transformation}(DT) on cluster varieties. An interesting observation is that the corresponding action of DT on the bipartite quivers in the $Y$-pattern coincides with the Coxeter mutation. See \cite[Theorem 2.6]{GSW2020b} and its proof. Accordingly, Theorem~\ref{thm_intro_infinite_fillings} can be interpreted as an $N$-graph analogue of the aperiodicity of DT.

\subsubsection{Lagrangian fillings for Legendrians of type $\dynBCFG$ or standard affine type with symmetry}
Now we move to cluster structure of type $\dynBCFG$ and standard affine types with certain symmetry.
They are obtained by the folding procedure from type $\dynADE$ or $\exdynD\exdynE$, see
Table~\ref{table:foldings}. 

In order to interpret those symmetries into Legendrians links and surfaces, we need to introduce corresponding actions on symplectic- and contact manifolds.
Consider two actions on $\sphere^3\times \R_u$, the rotation $R_{\theta_0}$ and conjugation  $\eta$ as follows:
\begin{align*}
R_{\theta_0}(z_1, z_2,u)&=(z_1\cos\theta_0 -z_2\sin\theta_0,z_1\sin\theta_0+z_2\cos\theta_0,u);\\
\eta(z_1,z_2,u)&= (\bar z_1,\bar z_2 ,u).
\end{align*}
Here $\sphere^3$ is the unit sphere in $\bbC^2$ with coordinates $z_1=r_1 e^{i\theta_1}, z_2=r_2 e^{i\theta_2}$ with $r_1^2 + r_2^2=1$.
Note that $\eta$ is an anti-symplectic involution which naturally gives $\Z/2\Z$-action on the symplectic manifold.
Under certain coordinate changes, the restrictions of $R_{\theta_0}$ and $\eta$ on $J^1\sphere^1$ become
\begin{align*}
R_{\theta_0}|_{J^1\sphere^1}(\theta,p_{\theta},z)&=(\theta+\theta_0,p_{\theta},z);\\
\eta|_{J^1\sphere^1}(\theta,p_{\theta},z)&=(\theta,-p_{\theta},-z).
\end{align*}
In turn, the rotation $R_{\theta_0}$ acts on the $N$-graph $\ngraph(\dynX)$ by rotating the disk $\disk^2$, and $\eta$ acts by flipping the $z$-coordinate.

Any $Y$-pattern of non-simply-laced finite or affine type can be obtained by folding 
a $Y$-pattern of type $\dynADE$ or $\exdynA \exdynD \exdynE$. In other words, 
those $Y$-pattern of non-simply-laced type can be seen as sub-patterns of $\dynADE$- or $\exdynA \exdynD \exdynE$-types 
consisting of $Y$-seeds with certain symmetries of finite group $G$ action. We call 
such $Y$-seeds or $N$-graphs \emph{$G$-admissible}, and the mutation in the folded 
cluster structure is a sequence of mutations respecting the $G$-orbits. 
We say that a $Y$-seed (or an $N$-graph) is \emph{globally foldable} if it is 
$G$-admissible and its arbitrary mutations along $G$-orbits are again 
$G$-admissible.

Figure~\ref{figure:N-graph with rotational symmetry} illustrates the $N$-graphs with rotational symmetry and the corresponding $Y$-patterns of folding. Indeed, they are $\ngraph(1,n,n)$, $\ngraph(2,2,2)$, $\ngraph(3,3,3)$, $\ngraph(\exdynD_{2n})$, $\ngraph(\exdynD_4)$ which admits $\Z/2\Z$-, $\Z/3\Z$-, $\Z/3\Z$-, $\Z/2\Z$-, $\Z/2\Z$-action, respectively.

\begin{figure}[ht]
\[
\begin{tikzcd}[column sep=0.5pc, row sep=small]

\end{tikzcd}
\]

\[
\begin{tikzcd}[column sep=0.5pc, row sep=small]
%
\end{tikzcd}
\]
\caption{Examples of $N$-graphs with rotational symmetry}
\label{figure:N-graph with rotational symmetry}
\end{figure}

In order to present conjugation invariant $N$-graphs, we need to adopt a degenerate version of $N$-graphs which allows overlapping edges and cycles as in Figure~\ref{figure:N-graph with conjugation symmetry}.
They are equivalent to $\ngraph(\exdynD_{n+1})$, $\ngraph(\exdynD_4)$, $\ngraph(2,3,3)$, $\ngraph(3,3,3)$, and $\ngraph(2,4,4)$ up to $\partial$-Legendrian isotopy see Definition~\ref{def:boundary Legendrian isotopic}, respectively.

\begin{figure}[ht]
\[
\begin{tikzcd}[column sep=0.5pc, row sep=small]

\end{tikzcd}
\]
\caption{Examples of $N$-graphs with conjugation symmetry}
\label{figure:N-graph with conjugation symmetry}
\end{figure}

\begin{theorem}[Theorem~\ref{thm:folding of N-graphs}]\label{Thm:folding of N-graphs}
The following holds:
\begin{enumerate}
\item The Legendrian $\lambda(\dynA_{2n-1})$ has at least $\binom{2n}{n}$ distinct Lagrangian fillings up to exact Lagrangian isotopy (rel boundary) which are invariant under the $\pi$-rotation and  admit the $Y$-pattern of type $\dynB_n$.
\item The Legendrian $\lambda(\dynD_{4})$ has at least $8$ distinct Lagrangian fillings up to exact Lagrangian isotopy (rel boundary) which are invariant under the $2\pi/3$-rotation and admit the $Y$-pattern of type $\dynG_2$.
\item The Legendrian $\lambda(\exdynE_{6})$ has  infinitely many distinct Lagrangian fillings up to exact Lagrangian isotopy (rel boundary) which are invariant under the $2\pi/3$-rotation and admit the $Y$-pattern of type $\exdynG_2$.
\item The Legendrian $\lambda(\exdynD_{2n})$ with $n\ge 3$ has infinitely many distinct Lagrangian fillings up to exact Lagrangian isotopy (rel boundary) which are invariant under the $\pi$-rotation and admit the $Y$-pattern of type $\exdynB_n$.
\item The Legendrian $\lambda(\exdynD_4)$ has infinitely many distinct Lagrangian fillings up to exact Lagrangian isotopy (rel boundary) which are invariant under the $\pi$-rotation and admit the $Y$-pattern of type $\exdynC_2$.
\item The Legendrian $\tilde\lambda(\dynE_{6})$ has at least $105$ distinct Lagrangian fillings up to exact Lagrangian isotopy (rel boundary) which are invariant under the antisymplectic involution and admit the $Y$-pattern of type $\dynF_4$.
\item The Legendrian $\tilde\lambda(\dynD_{n+1})$ has  at least $\binom{2n}{n}$ Lagrangian fillings up to exact Lagrangian isotopy (rel boundary) which are invariant under the antisymplectic involution and admit the $Y$-pattern of type $\dynC_n$.
\item The Legendrian $\tilde\lambda(\exdynE_{6})$ has infinitely many distinct Lagrangian fillings up to exact Lagrangian isotopy (rel boundary) which are invariant under the antisymplectic involution and admit the $Y$-pattern of type $\dynE_6^{(2)}$.
\item The Legendrian $\tilde\lambda(\exdynE_{7})$ has infinitely many distinct Lagrangian fillings up to exact Lagrangian isotopy (rel boundary) which are invariant under the antisymplectic involution and admit the $Y$-pattern of type $\exdynF_4$.
\item The Legendrian $\tilde\lambda(\exdynD_4)$ has infinitely many distinct Lagrangian fillings up to exact Lagrangian isotopy (rel boundary) which are invariant under the antisymplectic involution and admit the $Y$-pattern of type $\dynA_5^{(2)}$.
\end{enumerate}
\end{theorem}

The study of Lagrangian fillings with symmetry, again to the best of authors' knowledge, is started from \cite{Cas2020}. We clarify the actions on the symplectic and contact manifold, together with the induced actions on Lagrangian fillings and Legendrian links. The items (1),(2),(6),(7) in Theorem~\ref{Thm:folding of N-graphs} answer that the half of the conjecture \cite[Conjecture 5.4]{Cas2020}, i.e. the surjectivity from fillings to seeds, is true, and furthermore we extend our results to certain non-simply-laced affine types.

\subsection{Organization of the paper}
The rest of the paper is divided into six sections including appendices. 
We review, in Section~\ref{sec:cluster algebras}, some basics on finite and affine cluster algebra. Especially we focus on structural results about the combinatorics of exchange graphs using Coxeter mutations.

In Section~\ref{sec:N-graph}, we recall how $N$-graphs and their moves encode Legendrian surfaces and the Legendrian isotopies. We also introduce degenerate $N$-graphs which will be used to construct Lagrangian fillings having conjugation symmetry. 
After that we review the assignment of $Y$-seeds in the cluster structure from $N$-graphs together with certain flag moduli. We also discuss the Legendrian mutation on (degenerate) $N$-graphs.

In Section~\ref{sec:N-graph of finite or affine type}, we investigate Legendrian links and $N$-graphs of type $\dynADE$ or $\exdynD \exdynE$. 
We discuss $N$-graph realization of the Coxeter mutation and prove Theorem~\ref{theorem:legendrian loop} on the relationship between Coxeter mutations and Legendrian loops.
By combining the structural results in the seed pattern of cluster algebra and $N$-graph realization of the Coxeter mutation, we construct as many Lagrangian fillings as seeds for Legendrian links of type $\dynADE$ or $\exdynD \exdynE$, and hence prove Theorem~\ref{thm_intro_1}.

In Section~\ref{section:folding}, we discuss rotation and conjugation actions on $N$-graphs and invariant $N$-graphs. We also prove Theorem~\ref{Thm:folding of N-graphs}.

In Appendix~\ref{section:invariance and admissibility}, 
we argue that $G$-invariance of type $\dynADE$ implies $G$-admissibility. 
Finally, in Appendix~\ref{sec:supplementary pictorial proofs}, we collect several equivalences between different presentation of $N$-graphs.

If some readers are familiar with the notion of cluster algebra and $N$-graph, then one may skip Section~\ref{sec:cluster algebras} and Section~\ref{sec:N-graph}, respectively, and start from Section~\ref{sec:N-graph of finite or affine type}.

\subsection*{Acknowledgement}
B. An and Y. Bae were supported by the National Research Foundation of Korea (NRF) grants funded by the Korea government(MSIT) (RS-2023-00208405) and (No. 2020R1A2C1A0100320), respectively.
E. Lee was supported by the Institute for Basic Science (IBS-R003-D1).

\section{Cluster algebras}\label{sec:cluster algebras}

Cluster algebras, introduced by Fomin and Zelevinsky~\cite{FZ1_2002}, are
commutative algebras with specific generators, called \emph{cluster
	variables}, defined recursively.
In this section, we recall basic notions in the theory of cluster algebras.
For more details, we refer the reader to~\cite{FZ1_2002, FZ2_2003, FZ4_2007}.

Throughout this section, we fix $m, n \in \Z_{>0}$ such that $n \leq m$, and
we let $\field$ be the rational function field with $m$ independent
variables over $\bbC$.

\subsection{Basics on cluster algebras}

\subsubsection{Cluster algebras}
We first recall the definition of seeds and $Y$-seeds from \cite{FZ1_2002, FZ2_2003, FZ4_2007}.
\begin{definition}[{cf. \cite{FZ1_2002, FZ2_2003, FZ4_2007}}]\label{definition:seeds}
A seed and $Y$-seed are defined as follows.
\begin{enumerate}
\item	A \emph{seed} $(\bfx, \qbasis)$ is a pair of 
	\begin{itemize}
		\item a tuple $\bfx = (x_1,\dots,x_m)$ of algebraically
		independent generators of $\field$, that is, 
		$\field = \bbC(x_1,\dots,x_m)$;
		\item an $m \times n$ integer matrix $\qbasis = (b_{i,j})_{i,j}$ such
		that the \emph{principal part} $\qbasispr \colonequals
		(b_{i,j})_{1\leq i,j\leq n}$ is skew-symmetrizable, that is, there
		exist positive integers $d_1,\dots,d_n$ such that
\[		
\textrm{diag}(d_1,\dots,d_n) \cdot \qbasispr 
\]
		is a
		skew-symmetric matrix.
	\end{itemize}
	We refer to $\bfx$ as the \emph{cluster} of a seed $(\bfx, \qbasis)$, to elements $x_1,\dots,x_m$ as \emph{cluster variables}, and to
	$\qbasis$ as the \emph{exchange matrix}. Moreover, we call $x_1,\dots,x_n$
	\emph{unfrozen} (or, \emph{mutable}) variables and $x_{n+1},\dots,x_m$
	\emph{frozen} variables.
\item A \emph{$Y$-seed} $(\bfy,\qbasispr)$ is a pair of an $n$-tuple $\bfy=(y_1,\dots,y_n)$ of elements in $\field$ and an $n\times n$ skew-symmetrizable matrix $\qbasispr$. 
We call $\bfy$ the \emph{coefficient tuple} of a $Y$-seed $(\bfy, \qbasispr)$ and call $y_1,\dots,y_n$ \emph{coefficients}.
\end{enumerate}
\end{definition}
We say that two seeds $(\bfx, \qbasis)$ and $(\bfx', \qbasis')$ are \textit{equivalent}, denoted by $(\bfx, \qbasis)\sim(\bfx', \qbasis')$ if there exists a permutation $\sigma$ on $[m]$ fixing $n+1,\dots, m$ such that
\[
x_i' = x_{\sigma(i)}\quad \text{ and } \quad b_{i,j}' = b_{\sigma(i),\sigma(j)} \quad \text{ for }1\le i\le m, 1\le j\le n,
\]
where $\bfx = (x_{1},\dots,x_{m})$, $\bfx' = (x_1',\dots,x_m')$, $\qbasis = (b_{i,j})$, and $\qbasis' = (b_{i,j}')$. Similarly,
two $Y$-seeds $(\bfy,\qbasispr)$ and $(\bfy',\qbasispr')$ are \emph{equivalent} and denoted by $(\bfy,\qbasispr)\sim(\bfy',\qbasispr')$ if there exists a permutation $\sigma$ on $[n]$ such that 
\[
y_i' = y_{\sigma(i)}\quad\text{and}\quad
b_{i,j}' = b_{\sigma(i),\sigma(j)}\quad\text{for }1\le i,j \le n.
\]

To define cluster algebras, we introduce mutations on exchange matrices, and quivers, and seeds as follows. 
\begin{enumerate}
\item (Mutation on exchange matrices)
For an exchange matrix $\qbasis$ and $1 \le k \le n$, the mutation $\mu_k(\qbasis) = (b_{i,j}')$ is defined as follows.
\[
b_{i,j}' = \begin{cases}
-b_{i,j} & \text{ if } i = k \text{ or } j = k, \\
\displaystyle b_{i,j} + \frac{|b_{i,k}| b_{k,j} + b_{i,k} | b_{k,j}|} {2} & \text{ otherwise}.
\end{cases}
\]
We say that \emph{$\qbasis' =(b_{i,j}')$ is the mutation of $\qbasis$ at $k$}.
\item (Mutation on quivers)
We call a finite directed multigraph $\quiver$ a \emph{quiver} if it does not have
oriented cycles of length at most $2$. The adjacency matrix $\qbasis(\quiver)$ of a
quiver is always skew-symmetric. Moreover, $\mu_k(\qbasis(\quiver))$ is again 
the adjacency matrix of a quiver $\quiver'$. We define $\mu_k(\quiver)$ to be 
the quiver satisfying 
\[
 \qbasis(\mu_k(\quiver)) = \mu_k(\qbasis(\quiver)),
\]
and say that \emph{$\mu_k(\quiver)$ is the mutation of $\quiver$ at $k$}.
\item (Mutation on seeds)	For a seed $(\bfx, \qbasis)$ and an integer $1 \leq k \leq n$, the \emph{mutation} $\mutation_k(\bfx, \qbasis) = (\bfx', \mu_k(\qbasis))$ is defined as follows: 
\begin{equation*}
x_i' = \begin{cases}
x_i &\text{ if } i \neq k,\\
\displaystyle 
x_k^{-1}\left( \prod_{b_{j,k} > 0} x_j^{b_{j,k}} + \prod_{b_{j,k} < 0}x_j^{-b_{j,k}}
\right) & \text{ otherwise}.
\end{cases}
\end{equation*}
\item (Mutation on $Y$-seeds)
The \emph{$Y$-seed mutation} (or, \emph{cluster $\mathcal{X}$-mutation}, \emph{$\mathcal{X}$-cluster mutation}) on a $Y$-seed $(\bfy, \qbasispr)$ at $k\in[n]$ is a $Y$-seed $(\bfy'=(y_1',\dots, y_n'),\qbasispr'=\mu_k(\qbasispr))$, where for each $1 \le i\le n$,
\[
y_i' = \begin{cases}
    \displaystyle {y}_{i} {y}_{k}^{\max\{b_{i,k},0\}}(1+{y}_{k})^{-b_{i,k}} & \text{ if }i \neq k, \\
   {y}_{k}^{-1} &\text{ otherwise}.
\end{cases}
\]

\end{enumerate}

\begin{example}\label{example_mutation_skewsymmetrizable}
Let $n = m = 2$. Suppose that an initial seed is given by
\[
(\bfx_{t_0}, \qbasis_{t_0}) = \left(
(x_1,x_2), \begin{pmatrix}
0 & 1 \\ -3 & 0
\end{pmatrix}
\right).
\]
Considering mutations $\mu_1(\bfx_{t_0}, \qbasis_{t_0})$ and $\mu_2\mu_1(\bfx_{t_0}, \qbasis_{t_0})$, we obtain the following.
\begin{align*}
\mu_1(\bfx_{t_0}, \qbasis_{t_0}) &= \left(
\left(
\frac{1+x_2^3}{x_1},x_2
\right), \begin{pmatrix}
0 & -1 \\3 & 0
\end{pmatrix}
\right),\\
\mu_2\mu_1(\bfx_{t_0}, \qbasis_{t_0}) &= \left(
\left(
\frac{1+x_2^3}{x_1}, \frac{1+x_1+x_2^3}{x_1x_2}
\right),
\begin{pmatrix}
0 & 1  \\ -3 & 0
\end{pmatrix}
\right).
\end{align*}
\end{example}

\begin{remark}\label{rmk_mutation_on_quivers}
Let $k$ be a vertex in a quiver $\quiver$ on $[m]$.
The mutation $\mutation_k(\quiver)$ can also be described via a sequence of three steps:
\begin{enumerate}
\item For each directed two-arrow path $i \to k \to j$, add a new arrow $i \to j$.
\item Reverse the direction of all arrows incident to the vertex $k$.
\item Repeatedly remove directed $2$-cycles until unable to do so.
\end{enumerate}
\end{remark}
\begin{remark}\label{rmk_mutation_commutes}
Let $\qbasis = (b_{i,j})$ be an exchange matrix of size $m \times n$.
For $k,\ell \in [n]$, if $b_{k,\ell} = b_{\ell,k} = 0$, then the mutations at $k$ and $\ell$ commute with each other: $\mutation_{\ell}(\mutation_k(\qbasis)) = \mutation_k(\mutation_{\ell}(\qbasis))$. 
Similarly, for a quiver $\quiver$ on $[m]$, if there does not exist an arrow connecting mutable vertices $k$ and $\ell$, then we have $\mutation_{\ell}(\mutation_k(\quiver)) = \mutation_k(\mutation_{\ell}(\quiver))$.
\end{remark}

We say a quiver $\quiver'$ is \emph{mutation equivalent} to another quiver $\quiver$ if 
there exists a sequence  of mutations $\mutation_{j_1},\dots,\mutation_{j_{\ell}}$ 
which connects $\quiver'$ and $\quiver$, that is,
\[
\quiver' = (\mutation_{j_{\ell}} \cdots \mutation_{j_1})(\quiver).
\]
Similarly, we say an exchange matrix $\qbasis'$ is \emph{mutation equivalent} to another matrix $\qbasis$ if $\qbasis'$ is obtained by applying a sequence of mutations to $\qbasis$.

An immediate check shows that $\mutation_k(\bfx,\qbasis)$ is again a seed, $\mutation_k(\bfy,\qbasispr)$ is a $Y$-seed, and a mutation is an involution, that is, its square is the identity.
Also, note that the mutation on seeds does not change frozen variables $x_{n+1},\dots,x_m$.
Let $\mathbb{T}_n$ denote the $n$-regular tree whose edges are labeled by $1,\dots,n$. Except for $n = 1$, there are infinitely many vertices on the tree $\mathbb{T}_n$. For example, we present regular trees $\mathbb{T}_2$ and $\mathbb{T}_3$ in Figure~\ref{figure_regular_trees_2_and_3}.
\begin{figure}
	\begin{tabular}{cc}
		\begin{tikzpicture}
		\tikzset{every node/.style={scale=0.8}}
		\tikzset{cnode/.style = {circle, fill,inner sep=0pt, minimum size= 1.5mm}}
		\node[cnode] (1) {};
		\node[cnode, right of=1 ] (2) {};
		\node[cnode, right of=2 ] (3) {};
		\node[cnode, right of=3 ] (4) {};	
		\node[cnode, right of=4 ] (5) {};	
		\node[cnode, right of=5 ] (6) {};	
		
		\node[left of=1] {$\cdots$};
		\node[right of=6] {$\cdots$};
		
		\draw (1)--(2) node[above, midway] {$1$};
		\draw (2)--(3) node[above, midway] {$2$};
		\draw (3)--(4) node[above, midway] {$1$};
		\draw (4)--(5) node[above, midway] {$2$};
		\draw (5)--(6) node[above, midway] {$1$};
		\end{tikzpicture} &
		\begin{tikzpicture}
			\tikzset{every node/.style={scale=0.8}}
		\tikzset{cnode/.style = {circle, fill,inner sep=0pt, minimum size= 1.5mm}}
		\node[cnode] (1) {};
		\node[cnode, below right of =1] (2) {};
		\node[cnode, below of =2] (3) {};
		\node[cnode, above right of=2] (4){};
		\node[cnode, above of =4] (5) {};
		\node[cnode, below right of = 4] (6) {};
		\node[cnode, below of= 6] (7) {};
		\node[cnode, above right of = 6] (8) {};
		\node[cnode, above of = 8] (9) {};
		\node[cnode, below right of = 8] (10) {};
		\node[cnode, below of = 10] (11) {};
		\node[cnode, above right of = 10] (12) {};
		
		\node[left of = 1] {$\cdots$};
		\node[right of = 12] {$\cdots$};
		
		\draw (1)--(2) node[above, midway, sloped] {$1$};
		\draw (4)--(6) node[above, midway, sloped] {$1$};
		\draw (8)--(10) node[above, midway, sloped] {$1$};

		\foreach \x [evaluate ={ \x as \y using int(\x +2)} ] in {2, 6, 10}{
			\draw (\x)--(\y)  node[below, midway, sloped] {$2$}; 
		}
		\foreach \x [evaluate = {\x as \y using int(\x +1)}] in {2, 4, 6, 8, 10}{
			\draw (\x)--(\y) node[above, midway, sloped] {$3$};
		}
		\end{tikzpicture}\\[2ex]
		$\mathbb{T}_2$ & $\mathbb{T}_3$
	\end{tabular}
	\caption{The $n$-regular trees for $n=2$ and $n = 3$.}
	\label{figure_regular_trees_2_and_3}	
\end{figure}
A \emph{cluster pattern} (or \emph{seed pattern}) is an assignment
\[
\mathbb{T}_n \to \{\text{seeds in } \field\}, \quad t \mapsto (\bfx_t, \qbasis_t)
\]
such that if $\begin{tikzcd} t \arrow[r,dash, "k"] & t' \end{tikzcd}$ in $\mathbb{T}_n$, then $\mutation_k(\bfx_t, \qbasis_t) = (\bfx_{t'}, \qbasis_{t'})$.
Let $\{ (\bfx_t, \qbasis_t)\}_{t \in \mathbb{T}_n}$ be a cluster pattern with $\bfx_t = (x_{1;t},\dots,x_{m;t})$. Since the mutation does not change frozen variables, we may let $x_{n+1} = x_{n+1;t},\dots,x_m = x_{m;t}$.
Similarly, we define a \emph{cluster $Y$-pattern} (or a \emph{$Y$-pattern}) $\{(\bfy_t, \qbasispr_t)\}_{t \in \mathbb{T}_n}$ by an assignment from $\mathbb{T}_n$ to the set of $Y$-seeds regarding the mutation relations. 

\begin{definition}[{cf. \cite{FZ2_2003}}]
	Let $\{ (\bfx_t, \qbasis_t)\}_{t \in \mathbb{T}_n}$ be a cluster pattern with $\bfx_t = (x_{1;t},\dots,x_{m;t})$.
	The \emph{cluster algebra} $\cA(\{(\bfx_t, \qbasis_t)\}_{t \in \mathbb{T}_n})$ is defined to be the $\bbC[x_{n+1},\dots,x_m]$-subalgebra 
	of $\field$ generated by all the cluster variables 
	 $\bigcup_{t \in \mathbb{T}_n} \{x_{1;t},\dots,x_{n;t}\}$.
\end{definition}
If we fix a vertex $t_0 \in \mathbb{T}_n$, then a cluster pattern $\{ (\bfx_{t}, \qbasis_{t}) \}_{t \in \mathbb{T}_n}$ is 
constructed from the seed~$(\bfx_{t_0}, \qbasis_{t_0})$. In this case, we call $(\bfx_{t_0}, \qbasis_{t_0})$ an \emph{initial seed}. 
Because of this reason, we simply denote by $\cA(\bfx_{t_0}, \qbasis_{t_0})$ the cluster algebra given by the cluster 
pattern constructed from the initial seed~$(\bfx_{t_0}, \qbasis_{t_0})$.
\begin{example}\label{example_A2_example}
	Let $n = m = 2$. Suppose that an initial seed is given by 
	\[
		(\bfx_{t_0}, \qbasis_{t_0}) = \left(
			(x_1,x_2), \begin{pmatrix}
				0 & 1  \\ -1 & 0
			\end{pmatrix}
		\right).
	\]
	We present the cluster pattern obtained by the initial seed $(\bfx_{t_0}, \qbasis_{t_0})$.
	\begin{center}
	\begin{tikzcd}
		\left( (x_2,x_1), 
		\begin{pmatrix}
			0 & -1 \\ 1 & 0
		\end{pmatrix}
		\right)
		\arrow[r, color=white, "\textcolor{black}{\sim}" description]
		& (\bfx_{t_0}, \qbasis_{t_0})
			= \left(
			(x_1,x_2), 
			\begin{pmatrix}
				0 & 1 \\ -1 & 0
			\end{pmatrix}
		\right) \arrow[d,<->, "\mutation_1"]
		\\
		\left(
			(\frac{1+x_1}{x_2}, x_1), \begin{pmatrix}
				0 & 1 \\ -1 & 0
			\end{pmatrix}
		\right) \arrow[u,<->, "\mutation_1"]
		& 
		\left(
			\left(\frac{1+x_2}{x_1}, x_2\right), \begin{pmatrix}
				0 & -1 \\ 1 & 0
			\end{pmatrix}
		\right) \arrow[d, <->,"\mutation_2"]\\
		\left(
			\left(\frac{1+x_1}{x_2}, \frac{1+x_1+x_2}{x_1x_2}\right),
			\begin{pmatrix}
				0 & -1 \\ 1 & 0
			\end{pmatrix}
		\right) \arrow[u,<->, "\mutation_2"]
		&
		\left(
			\left(\frac{1+x_2}{x_1}, \frac{1+x_1+x_2}{x_1x_2}\right),
			\begin{pmatrix}
				0 & 1  \\ -1 & 0
			\end{pmatrix}
		\right) \arrow[l,<->, "\mutation_1"]
    \end{tikzcd}
\end{center}
Accordingly, we have
\[
\cA(\initialseed) = \cA(\{\seed_t\}_{t \in \mathbb{T}_n}) =  \bbC\left[x_1,x_2,\frac{1+x_2}{x_1}, \frac{1+x_1+x_2}{x_1x_2}, \frac{1+x_1}{x_2}\right].
\]
We notice that there are only five seeds in this case. Indeed, it is a cluster pattern of type~$\dynA_2$ (see Example~\ref{example_root_and_A2}).
\end{example}

\begin{remark}\label{rmk_ensembel_and_x_cluster_mutation}
One can associate a $Y$-pattern to a given cluster pattern in the following way.
Let $\{ (\bfx_t, \qbasis_t)\}_{t \in \mathbb{T}_n}$ be a cluster
pattern with $\bfx_t = (x_{1;t},\dots,x_{m;t})$. For $t \in \mathbb{T}_n$ and
$i \in [n]$, define $\hat{\mathbf y}_t= (\hat{y}_{1;t},\dots,\hat{y}_{n;t})$ to be 
\[
    \hat{y}_{i;t} = \prod_{j\in[m] } x_{j;t}^{b^{(t)}_{i,j}}.
\]
Here, $\qbasis_t = (b^{(t)}_{i,j})$ and $\qbasispr_t$ is the principal part of $\qbasis_t$.
Then the assignment $t \mapsto (\hat{\mathbf y}_t, \qbasispr)$ provides a 
$Y$-pattern and commutes with the mutation maps. Indeed, denoting by $p^{\ast}$ the assignment $(\hat{\mathbf y}_t, \qbasispr) \mapsto (\bfx_t, \qbasis_t)$, 
we obtain $\mutation_k(p^{\ast}(\hat{\mathbf y}_t, \qbasispr)) = p^{\ast}(\mutation_k(\hat{\mathbf y}_t, \qbasispr))$. 
\end{remark}

\subsubsection{Seed tori and cluster varieties}
For each  seed $\bfx_t = (x_{1;t},\dots,x_{m;t})$, define a torus $T_{\cA;t} \colonequals \Spec \mathbb{C}[x_{1;t}^{\pm 1}, \dots, x_{n;t}^{\pm 1}]$, called a \emph{seed tori}. Similarly, we define a \emph{$Y$-seed torus} $T_{\cX;t} \colonequals \Spec \mathbb{C} [{y}_{1;t}^{\pm 1},\dots,{y}_{n;t}^{\pm 1}]$.
For every edge  $\begin{tikzcd} t \arrow[r,dash, "k"] & t' \end{tikzcd}$ in $\mathbb{T}_n$, the associated seed tori are related by the mutation maps
\[
\begin{tikzcd}
	T_{\cA;t} \arrow[r, dashed, <->, "\mu_k"] & T_{\cA;t'} 
\end{tikzcd} \quad 
\begin{tikzcd}
	T_{\cX;t} \arrow[r, dashed, <->, "\mu_k"] & T_{\cX;t'}.
\end{tikzcd} 
\]

The \emph{$\cA$-cluster variety $\cA$} (also called a cluster $\cA$-variety or a cluster $K_2$ variety) is $\Spec(\cA(\{(\bfx_t, \qbasis_t)\}_{t \in \mathbb{T}_n}))$. The \emph{$\cX$-cluster variety $\cX$} (also called a cluster $\cX$-variety or a cluster Poisson variety) is given by gluing the seed tori $T_{\cX; t}$.
We call $T_{\cA;t}$ (respectively, $T_{\cX;t}$) a \emph{cluster chart} when we consider it as an \emph{embedded} torus in $\cA$ (respectively, in $\cX$). 

\begin{remark}
Let $p^{\ast}$ be the assignment given in Remark~\ref{rmk_ensembel_and_x_cluster_mutation}. 
Then $p^{\ast}$ provides a map between two seed tori given by seeds $\mathbf x_t$ and $\hat{\mathbf y}_t$:
\[
p \colon \Spec \mathbb{C} [x_{1;t}^{\pm 1},\dots,x_{n;t}^{\pm 1}] \to 
\Spec \mathbb{C} [\hat{y}_{1;t}^{\pm 1},\dots,\hat{y}_{n;t}^{\pm 1}].
\]
Indeed, we have a map $p \colon \cA \to \cX$, which is called the \emph{ensemble map}.
\end{remark}

Note that the mutation preserves the ranks of the exchange matrices. Indeed, it preserves the determinants of the principal parts of the exchange matrices up to sign 
as proved in~\cite[Lemma~3.2]{BFZ3_2005}. 
Accordingly, 
under the situation in Remark~\ref{rmk_ensembel_and_x_cluster_mutation}, 
if the exchange matrix $\qbasis_t$ has full rank, then the variables $\hat{y}_{1;t},\dots,\hat{y}_{n;t}$ are algebraically independent. 
Furthermore, if the exchange matrix $\qbasis_t$ is square having determinant $\pm 1$, then $p^{\ast}$ provides an isomorphism between the coordinate rings $ \mathbb{C} [x_{1,t}^{\pm},\dots,x_{n,t}^{\pm}]$ and $\mathbb{C} [\hat{y}_{1,t}^{\pm},\dots,\hat{y}_{n,t}^{\pm}]$ of seed tori. 

We summarize useful results for later use. 
\begin{proposition}[{\cite[Theorem~A.12]{GSW2020a}}]\label{prop_bijection_between_A_seed_and_tori}
	Let $\quiver$ be a quiver of full rank and let $\cA$ is the corresponding cluster $\cA$-variety defined over $\mathbb{C}$. The cluster charts of distinct cluster seeds of $\cA$ do not coincide. Indeed, there is a bijective correspondence between the set of cluster charts and that of cluster seeds. 
\end{proposition}

The full rank condition in the above proposition does not necessarily hold in general. However, if one could find an appropriate \emph{extension} of a given quiver by adding some vertices and edges, one could distinguish seed tori. 
We note that for $n \leq m$, one can naturally embed the $n$-regular tree $\mathbb{T}_n$ to the $m$-regular tree $\mathbb{T}_m$. 
\begin{corollary}\label{cor_Y-seeds_and_Y-charts}
Let $\quiver$ be a quiver on $[m]$ of full rank whose exchange matrix is square having determinant $\pm 1$ and let $\{(\bfy_t, \qbasispr_t)\}_{t \in \mathbb{T}_m}$ be the corresponding $Y$-pattern. 
Let $\{T_{\cX;t}\}_{t \in \mathbb{T}_m}$ be the set of $Y$-cluster charts corresponding to $\quiver$. For $n < m$, there is a bijective correspondence between the subset $\{T_{\cX;t}\}_{t \in \mathbb{T}_n}$ of $Y$-cluster charts and the subset $\{(\bfy_t, \qbasispr_t)\}_{t \in \mathbb{T}_n}$ of $Y$-cluster seeds. 
\end{corollary}
\begin{proof}
	Since $\quiver$ is a quiver of full rank, by Proposition~\ref{prop_bijection_between_A_seed_and_tori}, there is a bijective correspondence between the set of seed tori $\{T_{\cA;t}\}_{t \in \mathbb{T}_m}$ and that of cluster seeds $\{(\bfx_t, \qbasis_t)\}_{t \in \mathbb{T}_m}$. On the other hand, since the exchange matrix of $\quiver$ is square having determinant $\pm 1$, the ensemble map provide an isomorphism between a seed torus $T_{\cA;t}$ and a $Y$-seed torus $T_{\cX;t}$. Accordingly, there is a bijective correspondence between the set of $Y$-cluster charts $\{T_{\cX;t}\}_{t \in \mathbb{T}_m}$ and cluster charts $\{T_{\cA;t}\}_{t \in \mathbb{T}_m}$. 
	\[
	\begin{tikzcd}
		\{(\bfy_t, \qbasispr_t)\}_{t \in \mathbb{T}_n} \arrow[d] 
		\arrow[r, hook]
		& \{(\bfy_t, \qbasispr_t)\}_{t \in \mathbb{T}_m}  
		\arrow[r, <->, "\text{\tiny{bijective}}"]
		\arrow[d, <->, "\text{\tiny{bijective}}"]
		& \{(\bfx_t, \qbasis_t)\}_{t \in \mathbb{T}_m}  
		\arrow[<->, d, "\text{\tiny{bijective}}"]\\
		\{T_{\cX;t}\}_{t \in \mathbb{T}_n} \arrow[r, hook]
		& \{T_{\cX;t}\}_{t \in \mathbb{T}_m} 
		\arrow[<->, r, "\text{\tiny{bijective}}"]
		& \{T_{\cA;t}\}_{t \in \mathbb{T}_m}
	\end{tikzcd}
	\]
	Therefore, the correspondence between the set of $Y$-cluster charts $\{T_{\cX;t}\}_{t \in \mathbb{T}_m}$ and that of $Y$-cluster seeds $\{(\bfy_t, \qbasispr_t)\}_{t \in \mathbb{T}_m}$ is bijective. 
\end{proof}

\subsection{Cluster algebras of Dynkin type}
The number of cluster variables in Example~\ref{example_A2_example} is finite
even though the number of vertices in the graph $\mathbb{T}_2$ is infinite.
We call such cluster algebras \emph{of finite type}. More precisely, we
recall the following definition.
\begin{definition}[{\cite{FZ2_2003}}]
	A cluster algebra is said to be \emph{of finite type} if it has finitely
	many cluster variables.
\end{definition}
It has been realized that classifying finite type cluster algebras is related
to studying exchange matrices. The \emph{Cartan counterpart}
$C(\qbasispr) = (c_{i,j})$ of the principal part
$\qbasispr$ of an exchange matrix $\qbasis$ is defined by
\[
c_{i,j} = \begin{cases}
2 & \text{ if } i = j, \\
-|b_{i,j}| & \text{ otherwise}.
\end{cases}
\]
Since $\qbasispr$ is skew-symmetrizable, its Cartan
counterpart $C(\qbasispr)$ is symmetrizable. 

We say that a quiver $\quiver$ is \emph{acyclic} if it does not have directed cycles.
Similarly, for a skew-symmetrizable matrix $\qbasispr = (b_{i,j})$, we say that it is \emph{acyclic} if there are no sequences $j_1,j_2,\dots,j_{\ell}$ with $\ell \ge 3$ such that
\[
b_{j_1,j_2}, b_{j_2,j_3},\dots,b_{j_{\ell-1},j_{\ell}},b_{j_{\ell},j_1} > 0.
\]
We say a seed $\seed = (\mathbf x, \qbasis)$ is \emph{acyclic} if so is its principal part $\qbasispr$.

\begin{definition}\label{def_quiver_of_type_X}
For a finite or affine Dynkin type $\dynX$, we define a quiver $\quiver$, a matrix $\qbasis$, a cluster pattern $\{(\bfx_t, \qbasis_t)\}_{t \in \mathbb{T}_n}$, a $Y$-pattern $\{(\bfy_t, \qbasispr_t)\}_{t \in \mathbb{T}_n}$, or a cluster algebra $\cA(\bfx_{t_0}, \qbasis_{t_0})$ \emph{of type~$\dynX$} as follows.
\begin{enumerate}
\item A quiver is \textit{of type~$\dynX$} 
if it is mutation equivalent to an \emph{acyclic} quiver whose underlying graph is isomorphic to the Dynkin 
diagram of type $\dynX$.
\item A skew-symmetrizable matrix $\qbasispr$ is \textit{of 
type $\dynX$} if it is mutation equivalent to an acyclic skew-symmetrizable matrix whose 
Cartan counterpart $C(\qbasispr)$ is isomorphic to the Cartan matrix of type~$\dynX$. 
\item A cluster pattern $\{(\bfx_t, \qbasis_t)\}_{t \in \mathbb{T}_n}$ or a $Y$-pattern $\{(\bfy_t, \qbasispr_t)\}_{t \in \mathbb{T}_n}$ is \textit{of type $\dynX$} if for some $t \in \mathbb{T}_n$, the Cartan counterpart  $C(\qbasispr_t)$ is of type $\dynX$.
\item A cluster algebra $\cA(\bfx_{t_0}, \qbasis_{t_0})$ is \textit{of type $\dynX$} if its cluster pattern is of type $\dynX$.
\end{enumerate}
\end{definition}

Here, we say that two matrices $C_1$ and $C_2$ are \emph{isomorphic} if they are conjugate to each other via a permutation matrix, that is, $C_2 = P^{-1} C_1 P$ for some permutation matrix~$P$. 
One may wonder whether there exist exchange matrices in the same seed pattern having different Dynkin type. 
However, it is proved in~\cite[Corollary~4]{CalderoKeller06} that if two acyclic skew-symmetrizable matrices are mutation equivalent, then there exists a sequence of mutations from one to other such that intermediate skew-symmetrizable matrices are all acyclic. Indeed, if two acyclic skew-symmetrizable matrices are mutation equivalent, then their Cartan counterparts are isomorphic. 

\begin{proposition}[{cf. \cite[Corollary~4]{CalderoKeller06}}]\label{prop_quiver_of_same_type_are_mutation_equivalent}
Let $\qbasispr$ and $\qbasispr'$ be acyclic skew-symmetrizable matrices. Then the following are equivalent:
\begin{enumerate}
\item the Cartan matrices $C(\qbasispr)$ and $C(\qbasispr')$ are isomorphic;
\item $\qbasispr$ and $\qbasispr'$ are mutation equivalent.
\end{enumerate}
\end{proposition}
Accordingly, a quiver, a matrix, a cluster pattern, or a cluster algebra of type $\dynX$ is well-defined. 
The following
theorem presents a classification of cluster algebras of finite type.
\begin{theorem}[{\cite{FZ2_2003}}] \label{thm_FZ_finite_type}
	Let $\{ (\bfx_t, \qbasis_t)\}_{t \in \mathbb{T}_n}$ be a
	cluster pattern with an initial seed $(\bfx_{t_0},
	\qbasis_{t_0})$. Let $\mathcal{A}(\bfx_{t_0}, \qbasis_{t_0})$ be the corresponding
	cluster algebra. Then, the cluster algebra $\mathcal{A}(\bfx_{t_0}, \qbasis_{t_0})$ is of finite type if and only if $\mathcal{A}(\bfx_{t_0}, \qbasis_{t_0})$ is of finite Dynkin type.
\end{theorem}

We provide a list of all of the irreducible finite type root systems and their Dynkin diagram in Table~\ref{table_finite} (cf.~\cite{Humphreys}).
In Tables~\ref{table_standard_affine} and~\ref{table_twisted_affine}, we present lists of standard affine root 
systems and twisted affine root systems, respectively. They are the same as 
presented in Tables Aff 1, Aff 2, and Aff 3 of~\cite[Chapter~4]{Kac83}, and we 
denote by $\exdynX = \dynX^{(1)}$. 
We notice that the number of vertices of the standard affine Dynkin diagram of type $\exdynX_{n-1}$ is $n$ while we do not specify the vertex numbering. 

We note that all Dynkin diagram of finite or affine type but $\exdynA_{n-1}$ do not have (undirected) cycles. Accordingly, we may omit the acyclicity condition in Definition~\ref{def_quiver_of_type_X} except $\exdynA_{n-1}$-type. 
On the other hand, if a quiver is a directed $n$-cycle, then the corresponding Cartan counterpart is of type $\exdynA_{n-1}$ while it is mutation equivalent to a quiver of type $\dynD_n$ (see~Type IV in \cite{Vatne10}). 

The mutation equivalence classes of acyclic quivers of type $\exdynA_{n-1}$ are described in~\cite[Lemma~6.8]{FST08}. Let 
$\quiver$ and $\quiver'$ are two $n$-cycles for $n \geq 3$. Suppose that in 
$\quiver$, there are $p$ edges of one direction and $q = n - p$ edges of the 
opposite direction. Also, in $\quiver'$, there are $p'$ edges of one direction 
and $q' = n - p'$ edges of the opposite direction. Then two quivers $\quiver$ 
and $\quiver'$ are mutation equivalent if and only if the unordered pairs 
$\{p,q\}$ and $\{p',q'\}$ coincide. We say that a quiver 
$\quiver$ is of type $\exdynA_{p,q}$ if it has $p$ edges of one direction and 
$q$ edges of the opposite direction. We depict some examples for quivers of 
type $\exdynA_{p,q}$ in Figure~\ref{fig_example_Apq}.

\begin{figure}[ht]
\begin{tabular}{cccc}
\begin{tikzpicture}[scale = 0.5]			
\tikzset{every node/.style={scale=0.7}}
\node[Dnode] (3) {};
\node[Dnode] (1) [below left = 0.6cm and 0.6cm of 3]{};
\node[Dnode] (2) [below right = 0.6cm and 0.6cm  of 3] {};

\draw[->] (3)--(1);
\draw[->] (3)--(2);
\draw[->] (1)--(2);
\end{tikzpicture} 
&	\begin{tikzpicture}[scale = 0.5]			
\tikzset{every node/.style={scale=0.7}}
\node[Dnode] (1) {};
\node[Dnode] (2) [below = of 1] {};
\node[Dnode] (3) [right = of 2] {};
\node[Dnode] (4) [above = of 3] {};

\draw[->] (4) -- (1);
\draw[->]	(4) -- (3);
\draw[->]	(3) -- (2);
\draw[->]	(2) --(1);
\end{tikzpicture} 
&	\begin{tikzpicture}[scale = 0.5]			
\tikzset{every node/.style={scale=0.7}}
\node[Dnode] (1) {};
\node[Dnode] (2) [below = of 1] {};
\node[Dnode] (3) [right = of 2] {};
\node[Dnode] (4) [above = of 3] {};

\draw[->] (4) -- (1);
\draw[->]	(4) -- (3);
\draw[->]	(3) -- (2);
\draw[->]	(1)--(2);
\end{tikzpicture} 
&
\raisebox{4em}{	\begin{tikzpicture}[baseline=-.5ex,scale=0.6]
\tikzstyle{state}=[draw, circle, inner sep = 0.07cm]
\tikzset{every node/.style={scale=0.7}}
\foreach \x in {1,...,7, 9,10,11}{
\node[Dnode] (\x) at (\x*30:3) {};
}
\node(12) at (12*30:3) {$\vdots$};
\node[rotate=-30] (8) at (8*30:3) {$\cdots$};

\foreach \x [evaluate={\y=int(\x+1);}] in {3,...,6,9}{
\draw[->] (\x)--(\y);
}
\foreach \x [evaluate={\y=int(\x-1);}] in {3,2,11}{
\draw[->] (\x)--(\y);
} 
\curlybrace[]{100}{290}{3.5};
\draw (190:4) node[rotate=0] {$p$};
\curlybrace[]{-50}{80}{3.5};
\draw (15:4) node[rotate=0] {$q$};
\end{tikzpicture}}
\\	
$\exdynA_{1,2}$ & $\exdynA_{1,3}$ & $\exdynA_{2,2}$ & $\exdynA_{p,q}$
\end{tabular}
\caption{Quivers of type $\exdynA_{p,q}$.}\label{fig_example_Apq}
\end{figure}

\begin{table}[t]
\begin{center}
\begin{tabular}{c|l  }
\toprule
$\Roots$ & Dynkin diagram \\
\midrule
$\dynA_n$ $(n \geq 1)$ &
\begin{tikzpicture}[scale=.5, baseline=-.5ex]
	\tikzset{every node/.style={scale=0.7}}
	\node[Dnode] (1) {};
	\node[Dnode] (2) [right = of 1] {};
	\node[Dnode] (3) [right = of 2] {};
	\node[Dnode] (4) [right =of 3] {};
	\node[Dnode] (5) [right =of 4] {};			

	\draw (1)--(2)--(3)
		(4)--(5);
	\draw[dotted] (3)--(4);
\end{tikzpicture}  \\ 			
$\dynB_n$ $(n \geq 2)$ &
\begin{tikzpicture}[scale=.5, baseline=-.5ex]
	\tikzset{every node/.style={scale=0.7}}
			
	\node[Dnode] (1) {};
	\node[Dnode] (2) [right = of 1] {};
	\node[Dnode] (3) [right = of 2] {};
	\node[Dnode] (4) [right =of 3] {};
	\node[Dnode] (5) [right =of 4] {};

	\draw (1)--(2)
		(3)--(4);
	\draw [dotted] (2)--(3);
	\draw[double line] (4)--(5);
\end{tikzpicture}  \\ 			
$\dynC_n$ $(n \geq 3)$ & 
\begin{tikzpicture}[scale=.5, baseline=-.5ex]
	\tikzset{every node/.style={scale=0.7}}
			
	\node[Dnode] (1) {};
	\node[Dnode] (2) [right = of 1] {};
	\node[Dnode] (3) [right = of 2] {};
	\node[Dnode] (4) [right =of 3] {};
	\node[Dnode] (5) [right =of 4] {};

	\draw (1)--(2)
		(3)--(4);
	\draw [dotted] (2)--(3);
	\draw[double line] (5)--(4);
\end{tikzpicture}  \\ 
			
$\dynD_n$ $(n \geq 4)$ & 
\begin{tikzpicture}[scale=.5, baseline=-.5ex]
	\tikzset{every node/.style={scale=0.7}}
			
	\node[Dnode] (1) {};
	\node[Dnode] (2) [right = of 1] {};
	\node[Dnode] (3) [right = of 2] {};
	\node[Dnode] (4) [right =of 3] {};			
			
	\node[Dnode] (5) [above right= 0.3cm and 1cm of 4] {};
	\node[Dnode] (6) [below right= 0.3cm and 1cm of 4] {};

	\draw(1)--(2)
		(3)--(4)--(5)
		(4)--(6);
	\draw[dotted] (2)--(3);
\end{tikzpicture}
\\
\bottomrule
\end{tabular}
\end{center}
\caption{Dynkin diagrams of finite type}\label{table_finite}
\end{table}

\begin{table}[ht]
\begin{center}
\begin{tabular}{l|l}
\toprule
$\Roots$ & Dynkin diagram \\
\midrule
$\exdynA_1$  &
\begin{tikzpicture}[scale=.5, baseline=-.5ex, decoration={
    markings,
    mark=at position 0.4 with {\arrow[line width = 0.5pt,scale=1]{angle 90}}}]
\tikzset{every node/.style={scale=0.7}}
\node[Dnode ] (1) {};
\node[Dnode ] (2) [right = of 1] {};

\draw[double distance = 1.5pt, postaction={decorate}] (2)--(1); 
\draw[postaction={decorate}, draw=none] (1)--(2);
\end{tikzpicture}  \\ 
$\exdynA_{n-1}$ $(n \geq 3)$  &
\begin{tikzpicture}[scale=.5, baseline=-.5ex]
\tikzset{every node/.style={scale=0.7}}
\node[Dnode] (1) {};
\node[Dnode] (2) [right = of 1] {};
\node[Dnode] (3) [right = of 2] {};
\node[Dnode] (4) [right =of 3] {};
\node[Dnode] (5) [right =of 4] {};			
\node[Dnode] (6) [above =of 3] {};			

\draw (1)--(2)--(3)
(4)--(5)
(1)--(6)--(5);
\draw[dotted] (3)--(4);

\end{tikzpicture}  \\ 

$\exdynB_{n-1}$ $(n \geq 4)$  &
\begin{tikzpicture}[scale=.5, baseline=-.5ex]
\tikzset{every node/.style={scale=0.7}}

\node[Dnode] (1) {};
\node[Dnode] (2) [right = of 1] {};
\node[Dnode] (3) [right = of 2] {};
\node[Dnode] (4) [right =of 3] {};
\node[Dnode] (5) [right =of 4] {};
\node[Dnode] (6) [below left = 0.6cm and 0.6cm of 1] {};
\node[Dnode] (7) [above left = 0.6cm and 0.6cm of 1] {};			

\draw (1)--(2)
(3)--(4)
(6)--(1)--(7);
\draw [dotted] (2)--(3);
\draw[double line] (4)--(5);
\end{tikzpicture}  \\ 			
$\exdynC_{n-1}$ $(n \geq 3)$ & 
\begin{tikzpicture}[scale=.5, baseline=-.5ex]
\tikzset{every node/.style={scale=0.7}}

\node[Dnode] (1) {};
\node[Dnode] (2) [right = of 1] {};
\node[Dnode] (3) [right = of 2] {};
\node[Dnode] (4) [right =of 3] {};
\node[Dnode] (5) [right =of 4] {};			
\node[Dnode] (6) [left =of 1] {};
\draw (1)--(2)
(3)--(4);
\draw [dotted] (2)--(3);
\draw[double line] (5)--(4); 
\draw[double line] (6)--(1);
\end{tikzpicture}  \\
\bottomrule
\end{tabular}
\end{center}
\caption{Dynkin diagrams of standard affine root systems}\label{table_standard_affine}
\end{table}

\begin{table}[ht]
\begin{tabular}{l|l}
\toprule
$\Roots$ & Dynkin diagram \\
\midrule
$\dynA_2^{(2)}$ & 
\begin{tikzpicture}[scale=.5, baseline=-.5ex, decoration={
    markings,
    mark=at position 0.6 with {\arrow[line width = 0.5pt,scale=1]{angle 90}}}]
\tikzset{every node/.style={scale=0.7}}

\node[Dnode ] (1) {};
\node[Dnode ] (2) [right = of 1] {};

\draw[double distance = 2.7pt] (1)--(2);
\draw[double distance = 0.9pt, postaction={decorate}] (1)--(2);
\end{tikzpicture} 
\\
$\dynA_{2(n-1)}^{(2)}$ ($n \ge 3$)
&
\begin{tikzpicture}[scale=.5, baseline=-.5ex]
\tikzset{every node/.style={scale=0.7}}

\node[Dnode] (1) {};
\node[Dnode] (2) [right = of 1] {};
\node[Dnode] (3) [right = of 2] {};
\node[Dnode] (4) [right =of 3] {};
\node[Dnode] (5) [right =of 4] {};			
\node[Dnode] (6) [left =of 1] {};
\draw (1)--(2)
(3)--(4);
\draw [dotted] (2)--(3);
\draw[double line] (4)--(5);
\draw[double line] (6)--(1);
\end{tikzpicture}  \\ 
$\dynA_{2(n-1)-1}^{(2)}$ ($n \ge 4$) & 
\begin{tikzpicture}[scale=.5, baseline=-.5ex]
\tikzset{every node/.style={scale=0.7}}

\node[Dnode] (1) {};
\node[Dnode] (2) [right = of 1] {};
\node[Dnode] (3) [right = of 2] {};
\node[Dnode] (4) [right =of 3] {};
\node[Dnode] (5) [right =of 4] {};
\node[Dnode] (6) [below left = 0.6cm and 0.6cm of 1] {};
\node[Dnode] (7) [above left = 0.6cm and 0.6cm of 1] {};			

\draw (1)--(2)
(3)--(4)
(6)--(1)--(7);
\draw [dotted] (2)--(3);
\draw[double line] (5)--(4);
\end{tikzpicture}  \\ 		
$\dynD_{n}^{(2)}$ ($n \ge 3$)  & 
\begin{tikzpicture}[scale=.5, baseline=-.5ex]
\tikzset{every node/.style={scale=0.7}}

\node[Dnode] (1) {};
\node[Dnode] (2) [right = of 1] {};
\node[Dnode] (3) [right = of 2] {};
\node[Dnode] (4) [right =of 3] {};
\node[Dnode] (5) [right =of 4] {};			
\node[Dnode] (6) [left =of 1] {};
\draw (1)--(2)
(3)--(4);
\draw [dotted] (2)--(3);
\draw[double line] (4)--(5);
\draw[double line] (1)--(6);
\end{tikzpicture}  \\
\bottomrule
\end{tabular}
\caption{Dynkin diagrams of twisted affine root systems}\label{table_twisted_affine}
\end{table}
For a Dynkin type $\dynX$, we say that $\dynX$ is \emph{simply-laced} if its Dynkin diagram has only single edges, otherwise, $\dynX$ is \emph{non-simply-laced}.
Recall that the Cartan matrix associated to a Dynkin diagram~$\dynX$ can be read directly from the diagram~$\dynX$ as follows:
\begin{center}
\setlength{\tabcolsep}{20pt}
\begin{tabular}{ccccc}
\begin{tikzpicture}[scale =.5, baseline=-.5ex]
\tikzset{every node/.style={scale=0.7}}

\node[Dnode, label=below:{$i$}] (2) {};
\node[Dnode, label=below:{$j$}] (3) [right=of 2] {};

\draw (2)-- (3);

\end{tikzpicture}&
\begin{tikzpicture}[scale =.5, baseline=-.5ex]
\tikzset{every node/.style={scale=0.7}}

\node[Dnode, label=below:{$i$}] (2) {};
\node[Dnode, label=below:{$j$}] (3) [right=of 2] {};

\draw[double line] (2)--(3);

\end{tikzpicture}&
\begin{tikzpicture}[scale =.5, baseline=-.5ex]
\tikzset{every node/.style={scale=0.7}}

\node[Dnode,  label=below:{$i$}] (2) {};
\node[Dnode, label=below:{$j$}] (3) [right=of 2] {};

\draw[triple line] (2)--(3);
\draw (2)--(3);

\end{tikzpicture}
&
\begin{tikzpicture}[scale=.5, baseline=-.5ex, decoration={
    markings,
    mark=at position 0.6 with {\arrow[line width = 0.5pt,scale=1]{angle 90}}}]
\tikzset{every node/.style={scale=0.7}}

\node[Dnode, label=below:{$i$}] (1) {};
\node[Dnode, label=below:{$j$}] (2) [right = of 1] {};

\draw[double distance = 2.7pt] (1)--(2);
\draw[double distance = 0.9pt, postaction={decorate}] (1)--(2);
\end{tikzpicture} 
&
\begin{tikzpicture}[scale=.5, baseline=-.5ex, decoration={
    markings,
    mark=at position 0.4 with {\arrow[line width = 0.5pt,scale=1]{angle 90}}}]
\tikzset{every node/.style={scale=0.7}}
\node[Dnode, label=below:{$i$}] (1) {};
\node[Dnode, label=below:{$j$}] (2) [right = of 1] {};

\draw[double distance = 1.5pt, postaction={decorate}] (2)--(1); 
\draw[postaction={decorate}, draw=none] (1)--(2);

\end{tikzpicture}\\
$c_{i,j} = -1$
&$c_{i,j} = -2$
& $c_{i,j} = -3$
& $c_{i,j} = -4$
& $c_{i,j} = -2$ \\
$c_{j,i} = -1$
&$c_{j,i} = -1$
& $c_{j,i} = -1$
& $c_{j,i} = -1$
& $c_{j,i} = -2$
\end{tabular}
\end{center}
For example, the Cartan matrix $(c_{i,j})$
of the diagram %
.
\end{equation}
Therefore, for each non-simply-laced Dynkin diagram $\dynX$, any exchange matrix $\qbasispr$ of type $\dynX$ is \emph{not} skew-symmetric but skew-symmetrizable. Hence it never comes from any quiver.

\begin{assumption}\label{assumption_finite}
Throughout this paper, we assume that for any cluster algebra, the principal part $\qbasispr_{t_0}$ of the initial exchange matrix is acyclic of \textit{finite or affine Dynkin} type unless mentioned otherwise. 
\end{assumption}

In Table~\ref{table_seeds_and_cluster_variables}, we provide enumeration on
the number of cluster variables and clusters in each cluster algebra of
finite (irreducible) type (cf.~\cite[Figure~5.17]{FWZ_chapter45}). 
\begin{table}[htb]
\setlength{\tabcolsep}{4pt}
	\begin{tabular}{c|ccccccccc}
		\toprule
		$\Roots$ & $\dynA_n$ & $\dynB_n$ & $\dynC_n$ & $\dynD_n$ & $\dynE_6$ & $\dynE_7$ & $\dynE_8$ & $\dynF_4$ & $\dynG_2$ \\
        \midrule
        $\#$seeds &  $\displaystyle \frac{1}{n+2}{\binom{2n+2}{n+1}}$ & $\displaystyle  \binom{2n}{n}$
        & $\displaystyle  \binom{2n}{n}$ & $\displaystyle  \frac{3n-2}{n} \binom{2n-2}{n-1}$ & $833$ & $4160$ & $25080$ & $105$ & $8$ \\[1.5em]
        $\#$clvar & $\displaystyle  \frac{n(n+3)}{2}$ & $n(n+1)$ & $n(n+1)$ & $n^2$ & $42$ & $70$ & $128$ & $28$ & $8$ \\
		\bottomrule 
	\end{tabular}
\caption{Enumeration of seeds and cluster variables}\label{table_seeds_and_cluster_variables}
\end{table}
\begin{example}\label{example_root_and_A2}
	Continuing Example~\ref{example_A2_example}, the Cartan
	counterpart of the principal part $\qbasispr_{t_0}$ is given by
	\[
		C(\qbasispr_{t_0}) = \begin{pmatrix}
			2 & -1  \\ -1 & 2
		\end{pmatrix},
	\]
	which is the Cartan matrix of type $\dynA_2$. 
	Accordingly, by Theorem~\ref{thm_FZ_finite_type}, the cluster algebra
	$\cA(\bfx_{t_0}, \qbasis_{t_0})$ is of finite type. Indeed, there are only five seeds in the seed pattern.
\end{example}

\subsection{Folding}\label{sec:folding}
Under certain conditions, one can \textit{fold} cluster patterns to produce new ones. This procedure is used to study cluster algebras of non-simply-laced type from those of simply-laced type (see Figure~\ref{fig_folding} and Table~\ref{figure:all possible foldings}). In this section, we recall \textit{folding} of cluster algebras from~\cite{FWZ_chapter45}. We also refer the reader to~\cite{Dupont08}.

Let $\quiver$ be a quiver on $[m]$.
Let $G$ be a finite group acting on the set $[m]$. 
The notation $i \sim i'$ will mean that $i$ and $i'$ lie in the same
$G$-orbit. To study folding of cluster algebras, we prepare some
terminologies.

We denote by $\qbasis = \qbasis(\quiver)$ the submatrix $(b_{i,j})_{1\le i\le m, 1 \le j \le n}$ of the adjacency matrix $(b_{i,j})_{1 \le i,j\le m}$ of the quiver~$\quiver$. 
Also, we denote by $\qbasispr = \qbasispr(\quiver)$ the principal part of $\qbasis(\quiver)$.
For each $g \in G$, let $\quiver' = g \cdot \quiver$ be the quiver such that $\qbasis(\quiver') = (b_{i,j}')$ is given by
\[
b_{i,j}' = b_{g(i),b(j)}.
\]
\begin{definition}[{cf.~\cite[\S4.4]{FWZ_chapter45}  and~\cite[\S 3]{Dupont08}}]\label{definition:admissible quiver}
    Let $\quiver$ be a quiver on $[m]$ and $G$ a finite
    group acting on the set~$[m]$.
\begin{enumerate} 
\item A quiver $\quiver$ is \emph{$G$-invariant} if $g \cdot \quiver = \quiver$ for any $g \in G$.
\item A $G$-invariant quiver $\quiver$ is \emph{$G$-admissible} if \label{admissible}
\begin{enumerate}
\item for any $i \sim i'$, index $i$ is mutable if and only if so is $i'$; \label{mutable}
\item for mutable indices $i \sim i'$, we have $b_{i,i'} = 0$; \label{bii'=0}
\item for any $i \sim i'$, and any mutable $j$, we have $b_{i,j} b_{i',j} \geq 0$.\label{nonnegativity_of_bijbi'j}
\end{enumerate}
\item For a $G$-admissible quiver $\quiver$, we call a $G$-orbit \emph{mutable} (respectively, \emph{frozen}) if it consists of mutable (respectively, frozen) vertices. 
\end{enumerate}         
\end{definition}
For a $G$-admissible quiver $\quiver$, we define the matrix $\qbasis^G =
\qbasis(\quiver)^G = (b_{I,J}^G)$ whose rows (respectively, columns) are
labeled by the $G$-orbits (respectively, mutable $G$-orbits) by
\[
    b_{I,J}^G = \sum_{i \in I} b_{i,j}
\]
where $j$ is an arbitrary index in $J$. We then say $\qbasis^G$ is obtained
from $\qbasis$ (or from the quiver $\quiver$) by \textit{folding} with
respect to the given $G$-action.
\begin{remark}
We note that the $G$-admissibility and the folding can also be defined for exchange matrices. 
\end{remark}
\begin{example}\label{example_D4_to_G2}
    Let $\quiver$ be a quiver of type $\dynD_4$ given as follows.
\[
    \begin{tikzpicture}[node distance=0.7cm]
        \tikzstyle{state}=[draw, circle, inner sep = 0.07cm]
        \tikzset{every node/.style={scale=0.7}}    
        \tikzstyle{double line} = [
            double distance = 1.5pt, 
            double=\pgfkeysvalueof{/tikz/commutative diagrams/background color}
        ]
        \tikzstyle{triple line} = [
            double distance = 2pt, 
            double=\pgfkeysvalueof{/tikz/commutative diagrams/background color}
        ]
        \node[state, label=left:{$1$}] (1) {};
        \node[state, label =right:{$2$}] (2) [above right = 0.4cm and 0.7cm of 1] {};
        \node[state, label=right:{$3$}] (3) [right = 0.7cm of 1] {};
        \node[state, label= right:{$4$}] (4) [below right = 0.4cm and 0.7cm of 1] {};
        
        \draw (3)--(1)--(4)
        (1)--(2);
    
        \node[label={below:\normalsize{$\rightsquigarrow$}}] [above right = 0.1cm and 1.5cm of 3] {};

        \node[ynode] at (1) {};
        \node[gnode] at (2) {};
        \node[gnode] at (3) {};
        \node[gnode] at (4) {};

    \draw[<-] (1)--(2);
    \draw[<-] (1)--(3);
    \draw[<-] (1)--(4);
    \end{tikzpicture}
    \qquad
    \text{\raisebox{1.5em}{$\qbasis(\quiver) = \begin{pmatrix}
        0 & -1 & -1 &-1 \\
        1 & 0 & 0 & 0  \\
        1 & 0 & 0 & 0 \\
        1 & 0 & 0 & 0
    \end{pmatrix}$}}
\]
The finite group $G = \Z / 3 \Z$ acts on $[4]$ by sending $2 \mapsto 3
\mapsto 4 \mapsto 2$ and $1 \mapsto 1$. 
Here, we decorate vertices of the quiver $\quiver$ with \colorbox{cyclecolor2!50!}{\cyclecolornamesecond} and \colorbox{cyclecolor1!50!}{\cyclecolornamefirst} colors
for presenting sources and sinks, respectively. 
One may check that the quiver
$\quiver$ is $G$-admissible. By setting $I_1 = \{1\}$ and $I_2 = \{ 2,3,4\}$,
we obtain
\[
    \begin{split}
    b_{I_1,I_2}^G &= \sum_{i \in I_1} b_{i,2} = b_{1,2} = -1, \\
    b_{I_2,I_1}^G &= \sum_{i \in I_2} b_{i,1} = b_{2,1} + b_{3,1} + b_{4,1} = 3.
    \end{split}
\]
Accordingly, we obtain the matrix $\qbasis^G = \begin{pmatrix} 0 & -1 \\ 3 &
0 \end{pmatrix}$ whose Cartan counterpart is the Cartan matrix of type
$\dynG_2$ (cf.~\eqref{eq_Cartan_G2}).
\end{example}

For a $G$-admissible quiver $\quiver$ and a mutable $G$-orbit $I$, we
consider a composition of mutations given by
\[
    \mutation_I = \prod_{i \in I} \mutation_i
\]
which is well-defined because of the definition of admissible quivers (cf. Remark~\ref{rmk_mutation_commutes}). Moreover, $\mu_I(\quiver)$ is $G$-invariant by~\cite[Lemma~5.12]{Dupont08}. 
If $\mutation_I(\quiver)$ is $G$-admissible, then we have 
\begin{equation*}
    (\mutation_I(\qbasis))^G = \mutation_I(\qbasis^G).
\end{equation*}
We notice that the quiver $\mutation_I(\quiver)$ is \textit{not}
$G$-admissible in general. Therefore, we present the following definition.
\begin{definition}
    Let $G$ be a group acting on the vertex set of a quiver $\quiver$.
    We say that $\quiver$ is \emph{globally foldable} with respect to $G$ if
    $\quiver$ is $G$-admissible and moreover for any sequence of mutable
    $G$-orbits $I_1,\dots,I_\ell$, the quiver $(\mutation_{I_\ell}  \dots
     \mutation_{I_1})(\quiver)$ is $G$-admissible.
\end{definition}
For a globally foldable quiver, we can fold all the seeds in the
corresponding seed pattern. Let
$\field^G$ be the field of rational functions in $\# ([m]/G)$ independent variables.
Let $\psi \colon \field \to \field^G$ be a surjective 
homomorphism.
A seed $(\mathbf{x}, \qbasis)$ or a $Y$-seed $(\bfy, \qbasispr)$ is called \emph{$(G, \psi)$-invariant} or \emph{admissible} if 
\begin{itemize}
    \item $\quiver$ is a $G$-invariant or admissible quiver, respectively;
    \item we have
\begin{equation}\label{condition_on_psi_map}
\psi(x_i) = \psi(x_{i'}) \text{ or }\psi(y_i) = \psi(y_{i'}) \quad \text{ for any } i \sim i'.
\end{equation}
\end{itemize} 
In this situation, we define new ``folded'' seed $(\bfx,\qbasis)^G = (\bfx^G,
\qbasis^G)$ and $Y$-seed $(\bfy,\qbasispr)^G=(\bfy^G, \qbasispr^G)$ in $\field^G$ whose exchange matrix is given as before
and cluster variables $\bfx^G = (x_I)$ and $\bfy^G=(y_I)$ are indexed by the $G$-orbits and
given by $x_I = \psi(x_i)$ and $y_I=\psi(y_i)$ for a $G$-orbit $I$ and $i \in I$.

We notice that for a $(G,\psi)$-admissible seed $(\bfx, \qbasis)$ or a $(G,\psi)$-admissible $Y$-seed $(\bfy, \qbasispr)$, the folding process is equivariant under the orbit-wise mutation, that is, for any mutable $G$-orbit~$I$, we have
\[
(\mutation_I(\bfx,\qbasis))^G = \mutation_{I}((\bfx,\qbasis)^G) \quad 
\text{ and } \quad
(\mutation_I(\bfy,\qbasispr))^G = \mutation_{I}((\bfy,\qbasispr)^G).
\]

\begin{proposition}[{cf.~\cite[Corollary~4.4.11]{FWZ_chapter45}}]\label{proposition:folded cluster pattern}
    Let $\quiver$ be a quiver which is globally foldable with respect to a
    group $G$ acting on the set of its vertices. Let $(\mathbf{x},
    \qbasis)$ and $(\bfy, \qbasispr)$ be a seed and a $Y$-seed in the field $\field$ of rational functions
    freely generated by $\mathbf{x} = (x_1,\dots,x_m)$. Then we have the following.
\begin{enumerate}
\item   Let $\psi \colon \field  \to \field^G$ be the homomorphism satisfying~\eqref{condition_on_psi_map}.
Then, for any mutable
    $G$-orbits $I_1,\dots,I_\ell$, the seed $(\mutation_{I_\ell} \cdots 
    \mutation_{I_1})(\bfx,\qbasis)$ is $(G, \psi)$-admissible, and moreover the
    folded seeds $((\mutation_{I_\ell}  \dots 
    \mutation_{I_1})(\bfx,\qbasis))^G$ form a seed pattern in $\field^G$ 
    with the initial seed $(\bfx,\qbasis)^G=(\bfx^G, \qbasis^G)$.
\item Let $\psi \colon \field  \to \field^G$ be the homomorphism satisfying~\eqref{condition_on_psi_map}. Then, for any mutable
    $G$-orbits $I_1,\dots,I_\ell$, the $Y$-seed $(\mutation_{I_\ell} \dots 
    \mutation_{I_1})(\bfy,\qbasispr)$ is $(G, \psi)$-admissible, and moreover the
    folded $Y$-seeds $((\mutation_{I_\ell}  \cdots 
    \mutation_{I_1})(\bfy,\qbasispr))^G$ form a $Y$-pattern in $\field^G$ 
    with the initial seed $(\bfy,\qbasispr)^G=(\bfy^G, \qbasispr^G)$.
\end{enumerate}
\end{proposition}

\begin{example}\label{example_folding_ADE}
The quiver in Example~\ref{example_D4_to_G2} is globally foldable, and
moreover the corresponding seed pattern is of type $\dynG_2$. In fact, 
seed patterns of type~$\dynBCFG$  are obtained by 
folding quivers of type~$\dynADE$; seed 
patterns of type~$\exdynB\exdynC\exdynF\exdynG$ are obtained by folding quivers 
of type~$\exdynD\exdynE$ 
(cf.~\cite{FeliksonShapiroTumarkin12_unfoldings}).
In Figures~\ref{fig_folding} and~\ref{figure:G-actions}, we present the corresponding quivers of
type~$\dynADE$ and type $\exdynE$. We decorate vertices of quivers with \colorbox{cyclecolor2!50!}{\cyclecolornamesecond} and \colorbox{cyclecolor1!50!}{\cyclecolornamefirst} colors for presenting source and sink, respectively. 
As one may see, we put arrows on the Dynkin diagram alternatingly.
The alternating colorings on quivers of type $\dynADE$ provide 
that on quivers of type $\dynBCFG$ as displayed in the right column of Figure~\ref{fig_folding}.
Foldings between simply-laced and non-simply-laced finete and affine Dynkin diagrams are given in 
Table~\ref{figure:all possible foldings}.
\end{example}
\begin{figure}

    \caption{Foldings in Dynkin diagrams of finite type (for seed patterns)}\label{fig_folding}
\end{figure}

\begin{figure}
%
\caption{Foldings in Dynkin diagrams of affine type (for seed patterns)}
\label{figure:G-actions}
\end{figure}

\newcolumntype{?}{!{\vrule width 0.6pt}}

\begin{table}[ht]
{
\setlength{\tabcolsep}{4pt}
\renewcommand{\arraystretch}{1.5}		
\begin{tabular}{c||c?c?c?c}
\toprule
$\dynX$ & $\dynA_{2n-1}$ & $\dynD_{n+1}$ & $\dynE_6$ & $\dynD_4$ \\
\hline 
$G$ & $\Z/2\Z$& $\Z/2\Z$ & $\Z/2\Z$ & $\Z/3\Z$ 
\\
\hline 
$\dynY$ & $\dynB_n$ & $\dynC_n$ & $\dynF_4$ & $\dynG_2$ 
\\
\bottomrule
\end{tabular}}
\medskip

{
\setlength{\tabcolsep}{4pt}
\renewcommand{\arraystretch}{1.5}		
\begin{tabular}{c||c?c?c?c?c?c?c?c?c?c?c}
\toprule
$\dynX$ & $\exdynA_{2,2}$ & $\exdynA_{n,n}$ 
& \multicolumn{2}{c?}{$\exdynD_4$ }
&  \multicolumn{2}{c?}{$\exdynD_n$ } 
&  \multicolumn{2}{c?}{$\exdynD_{2n}$ } 
&  \multicolumn{2}{c?}{$\exdynE_6$ }
& $\exdynE_7$ \\
\hline 
$G$ & $\Z/2\Z$& $\Z/2\Z$  
& $(\Z/2\Z)^2$ & $\Z/3\Z$ 
& $\Z/ 2\Z$ & $\Z/2\Z$
& $\Z/2\Z$  & $(\Z/2\Z)^2$
& $\Z/3\Z$ & $\Z/2\Z$ 
& $\Z/2\Z$ \\
\hline 
$\dynY$ & $\exdynA_1$ & $\dynD_{n+1}^{(2)}$ 
& $\dynA_2^{(2)}$ & $\dynD_4^{(3)}$ 
& $\exdynC_{n-2}$ & $\dynA_{2(n-1)-1}^{(2)}$
& $\exdynB_n$ & $\dynA_{2n-2}^{(2)}$
& $\exdynG_2$ & $\dynE_6^{(2)}$
& $\exdynF_4$ \\
\bottomrule
\end{tabular}}
\caption{Foldings appearing in finite and affine Dynkin diagrams}
\label{figure:all possible foldings}
\end{table}

For any quiver of type $\dynADE$, one can prove that the $G$-invariance is equivalent to the $G$-admissible as follows:
\begin{theorem}\label{theorem:G-invariance and G-admissibility}
Let $\quiver$ be a quiver of type $\dynADE$, which is invariant under the $G$-action given by Figure~\ref{fig_folding}.
Then the quiver $\quiver$ is $G$-admissible. 
\end{theorem}
We notice that the quiver considering in Theorem~\ref{theorem:G-invariance and G-admissibility} can have any orientations so long as they are $G$-invariant.
The proof of Theorem~\ref{theorem:G-invariance and G-admissibility} is given in Appendix~\ref{section:invariance and admissibility}. As a direct corollary of Theorem~\ref{theorem:G-invariance and G-admissibility}, we have the following.
\begin{corollary}\label{corollary_G-invariance_and_globally_foldable}
Let $\quiver$ be a quiver of type $\dynADE$, which is invariant under the $G$-action given by Figure~\ref{fig_folding}. Then the quiver $\quiver$ is globally foldable.
\end{corollary}
\begin{proof}
Let $I$ be a mutable $G$-orbit. The quiver $\mutation_I(\quiver)$ is again $G$-invariant (see~\cite[Lemma~5.12]{Dupont08}) so it is $G$-admissible according to Theorem~\ref{theorem:G-invariance and G-admissibility}. Therefore, $\quiver$ is globally foldable. 
\end{proof}

As we saw in Definition~\ref{definition:admissible quiver}, if a seed $\seed = 
(\mathbf{x}, \quiver)$ is $(G,\psi)$-admissible, then $\seed$ is 
$(G,\psi)$-invariant. 
The converse holds when we consider the foldings presented in Table~\ref{figure:all possible foldings}, and moreover they form the folded cluster pattern.
\begin{theorem}[{\cite{AL2021}}]\label{thm_invariant_seeds_form_folded_pattern}
Let $(\dynX, G, \dynY)$ be a triple given by a column of Table~\ref{figure:all 
possible foldings}.
Let $\initialseed = (\mathbf{x}_{t_0},\quiver_{t_0})$ be a seed in the field $\field$. Suppose that $\quiver_{t_0}$ is of type $\dynX$ and $G$-admissible. Let $\psi 
\colon \field  \to \field^G$ be the homomorphism satisfying~\eqref{condition_on_psi_map}.
Then, for any seed $\seed = (\mathbf{x}, \quiver)$ in the cluster pattern, 
if the quiver~$\quiver$ is $G$-invariant, then it is $G$-admissible. Indeed, $\quiver$ is globally foldable.
Moreover, any $(G,\psi)$-invariant seed $\seed = (\mathbf{x}, \quiver)$ can be 
reached with a sequence of orbit mutations from the 
initial seed. Indeed,  the set of such seeds forms the cluster 
pattern of the `folded' cluster algebra $\cA(\initialseed^G)$ of type $\dynY$.
\end{theorem}

\subsection{Combinatorics of exchange graphs}
\label{sec_comb_of_exchange_graphs}
The \emph{exchange graph} of a cluster pattern or a $Y$-pattern is the $n$-regular (finite or
infinite) connected graph whose vertices are the seeds of the cluster pattern
and whose edges connect the seeds related by a single mutation. 
In this section, we recall the combinatorics of exchange
graphs which will be used later. For more details, we refer the reader
to~\cite{FZ2_2003, FZ_Ysystem03, FZ4_2007}.

\begin{definition}[Exchange graphs]
Exchange graphs for seed patterns or $Y$-patterns are defined as follows.
\begin{enumerate}
\item The \emph{exchange graph} $\exchange(\{(\bfx_t, \qbasis_t)\}_{t \in \mathbb T_n})$ of the cluster pattern $\{(\bfx_t, \qbasis_t)\}_{t \in \mathbb T_n}$ is a quotient of the tree $\mathbb{T}_n$ modulo the equivalence relation on vertices defined by setting $t \sim t'$ if and only if $(\bfx_t, \qbasis_t) \sim (\bfx_{t'},\qbasis_{t'})$. 
\item The \emph{exchange graph} $\exchange(\{(\bfy_t, \qbasispr_t)\}_{t \in \mathbb T_n})$ of the $Y$-pattern $\{(\bfy_t, \qbasispr_t)\}_{t \in \mathbb T_n}$ is a quotient of the tree $\mathbb{T}_n$ modulo the equivalence relation on vertices defined by setting $t \sim t'$ if and only if $(\bfy_t, \qbasispr_t) \sim (\bfy_{t'},\qbasispr_{t'})$. 
\end{enumerate}
\end{definition} 

For example, the exchange graph in Example~\ref{example_A2_example} is a cycle graph with~$5$~vertices. 
As we already have seen in Theorem~\ref{thm_FZ_finite_type}, cluster algebras
of finite type are classified by Cartan matrices of finite type. Moreover,
for a cluster algebra of finite or affine type, the exchange graph depends only on the
exchange matrix (see Theorem~\ref{thm_exchange_graph_Dynkin}). To explain this observation, we need some terminologies.

For $\initialseed = (\mathbf x_{t_0}, \qbasis_{t_0})$,
the cluster algebra $\cA(\initialseed)$ is said to have \emph{principal coefficients} if the exchange matrix $\qbasis_{t_0}$ is a $(2n \times n)$-matrix of the form $\begin{pmatrix}
\qbpr_{t_0} \\ \clusterfont{I}_n
\end{pmatrix}$, and have \emph{trivial coefficients} if $\qbasis_{t_0}=\qbpr_{t_0}$.
Here, $\clusterfont{I}_n$ is the identity matrix of size~$n \times n$. 
We recall the following result on the combinatorics of exchange graphs.
\begin{theorem}[{\cite[Theorem~4.6]{FZ4_2007}}]\label{thm_exchange_graph_covering}
The exchange graph of an arbitrary cluster pattern $\{(\bfx_t, \qbasis_t)\}_{t \in \mathbb T_n}$ is covered by\footnote{We say that a graph $\widetilde{G}$ is a \emph{covering graph} of another graph $G$, or say $G$ is \emph{covered} by $\widetilde{G}$, if there is a covering map $f$ from the vertex set $V(\widetilde{G})$ of $\widetilde{G}$ to the vertex set $V(G)$ of $G$. Here, a \emph{covering map} $f$ is a surjection such that the neighbourhood of a vertex $v$ in $\widetilde{G}$ is mapped bijectively onto the neighbourhood of the vertex $f(v)$ in $G$.} the exchange graph of the cluster pattern 
$\{(\bfx_t, \qbasis_t')\}_{t \in \mathbb T_n}$ having principal coefficients such that the principal parts of $\qbasis_t$ and $\qbasis_t'$ are the same. 
\end{theorem}
Moreover, the exchange graph of the cluster pattern 
$\{(\bfx_t, \qbasis_t)\}_{t \in \mathbb T_n}$ having trivial coefficients is covered by the exchange graph of the cluster pattern whose initial exchange matrix has the principal part $\qbasis_{t_0}$. For a fixed principal part of the exchange matrix, the cluster pattern having principal coefficients has the largest exchange graph while that having trivial coefficients has the smallest one (see~\cite[Section~4]{FZ4_2007}).

However, it is unknown whether the largest exchange graph is strictly larger than the smallest one or not. Indeed, it is conjectured in \cite[Conjecture~4.3]{FZ4_2007} that the exchange graph of a cluster pattern is determined by the initial principal part $\qbpr_{t_0}$ only.
The conjecture is confirmed for finite cases~\cite{FZ2_2003} or exchange matrices coming from quivers~\cite{IKLP13}.
We furthermore extend this result to cluster algebras whose initial exchange matrices are of affine type. 

\begin{theorem}[{cf. \cite[Theorem~1.13]{FZ2_2003} and \cite[Theorem~4.6]{IKLP13}}]\label{thm_exchange_graph_Dynkin}
Let $\initialseed = (\mathbf x_{t_0}, \qbasis_{t_0})$ be an initial seed.
If the principal part~$\qbasispr_{t_0}$ of $\qbasis_{t_0}$ is of \emph{finite or affine type},  then the exchange graph of the cluster pattern $\{(\bfx_t, \qbasis_t)\}_{t \in \mathbb T_n}$ only depends on $\qbasispr_{t_0}$.
\end{theorem}
\begin{proof}
We first notice that the statement holds if the principal part $\qbasispr_{t_0}$ is of finite type~\cite[Theorem~1.13]{FZ2_2003} or exchange matrices are obtained from quivers~\cite[Theorem~4.6]{IKLP13}. 
It is enough to consider the case when the principal part is of \emph{non-simply-laced affine type}. 
Let $(\dynX, G, \dynY)$ be a column in Table~\ref{figure:all 
possible foldings}.
Let $\quiver(\dynX)$ be the quiver of type $\dynX$ and let $\qbasispr(\dynX)=\qbasispr(\quiver(\dynX))$ be the adjacency matrix of $\quiver(\dynX)$, which is a square matrix of size $n$.
Let $\qbasis(\dynX) = \begin{pmatrix}
\qbasispr(\dynX)\\ \clusterfont{I}_n
\end{pmatrix}$ be the $(2n\times n)$ matrix having principal coefficients whose principal part is $\qbasispr(\dynX)$. 
On the other hand, we consider a quiver~$\overline{\quiver}(\dynX)$ by adding $n^G \colonequals \#([n]/G)$ frozen vertices and $n$ arrows. Here, each frozen vertex is indexed by a $G$-orbit and we draw an arrow from the frozen vertex to each mutable vertex in the corresponding $G$-orbit. 
For algebraic independent elements $\bfx=(x_1,\dots, x_n)$, $\overline{\bfx} = (x_1,\dots,x_n,x_{n+1},\dots,x_{n+n^G})$, and $\tilde\bfx=(x_1,\dots, x_n, x_{n+1},\dots, x_{2n})$ in $\field$, we obtain seeds
\[
\tilde{\seed}_{t_0} = (\tilde\bfx, \qbasis(\dynX)),
\quad 
\overline{\seed}_{t_0} = (\overline\bfx, \qbasispr(\overline\quiver(\dynX))), 
\quad\text{ and }\quad
\seed_{t_0} = (\bfx,\qbasispr(\dynX)).
\]
Since the exchange matrices come from quivers, the exchange graphs given by seeds $\tilde{\seed}_{t_0}, \overline{\seed}_{t_0}, \seed_{t_0}$ are isomorphic: $\exchange(\{ \tilde{\seed}_{t}\}_{t \in \mathbb T_n})
\cong \exchange(\{ \overline{\seed}_{t}\}_{t \in \mathbb T_n})
\cong \exchange(\{ {\seed}_{t}\}_{t \in \mathbb T_n})$. Indeed, we have
\begin{equation}\label{equation_exchange_graphs_are_the_same}
\{ \tilde{\seed}_{t}\}_{t \in \mathbb T_n}/\sim \;\; = 
\{ \overline{\seed}_{t}\}_{t \in \mathbb T_n}/\sim \;\;  =   
\{ {\seed}_{t}\}_{t \in \mathbb T_n}/\sim.
\end{equation}

Extending the action of $G$ on $\quiver$ of type $\dynX$ to  $\overline\quiver(\dynX)$ such that $G$ acts trivially on frozen vertices, the quiver $\overline\quiver(\dynX)$ becomes a globally foldable quiver with respect to $G$ (see~\cite[Lemma~5.5.3]{FWZ_chapter45}).
Moreover, via $\psi \colon \field \to \field^G$, the folded seed $\overline{\seed}_{t_0}^G = (\overline\bfx, \overline\quiver(\dynX))^G$ produces the principal coefficient cluster algebra of type $\dynY$.
This produces the following diagram.
\[
\begin{tikzcd}
\{ \tilde{\seed}_{t}\}_{t \in \mathbb T_n}/\sim
	\arrow[r,dash,shift left=.1em] \arrow[r,dash,shift right=.1em]
& \{ \overline{\seed}_{t}\}_{t \in \mathbb T_n}/\sim
	\arrow[r,dash,shift left=.1em] \arrow[r,dash,shift right=.1em]
& \{ {\seed}_{t}\}_{t \in \mathbb T_n}/\sim \\
& \{ \text{$(G,\psi)$-admissible seeds $\overline{\seed}_{t}$}\}/\sim
	\arrow[u, hookrightarrow]  
	\arrow[r,rightarrowtail]
	\arrow[d,dash,shift left=.1em] \arrow[d,dash,shift right=.1em]
&\{ \text{$(G,\psi)$-admissible seeds ${\seed}_{t}$}\}/\sim
	\arrow[u, hookrightarrow]
	\arrow[d,dash,shift left=.1em] \arrow[d,dash,shift right=.1em]
\\
& \{ \overline{\seed}_t^G\}_{t \in \mathbb T_n}/\sim 
\arrow[r,twoheadrightarrow]
& \{{\seed}_t^G\}_{t \in \mathbb T_n}/\sim
\end{tikzcd}
\] 
The equalities on the top row are obtained by~\eqref{equation_exchange_graphs_are_the_same}.
The surjectivity in the bottom row is induced by the maximality of the exchange graph of a cluster algebra having principal coefficients in Theorem~\ref{thm_exchange_graph_covering}. 
Moreover, the equalities connecting the second and third rows are given by Theorem~\ref{thm_invariant_seeds_form_folded_pattern}. 
This proves that there is a bijective correspondence between the set of vertices of $\exchange(\{ \overline{\seed}_t^G\}_{t \in \mathbb T_n})$ and that of $\exchange(\{{\seed}_t^G\}_{t \in \mathbb T_n})$. On the other hand, the graph  $\exchange(\{{\seed}_t^G\}_{t \in \mathbb T_n})$ is covered by $\exchange(\{ \overline{\seed}_t^G\}_{t \in \mathbb T_n})$ by Theorem~\ref{thm_exchange_graph_covering}. Accordingly, two graphs are the same and this proves the theorem.
\end{proof}

We recall from~\cite{CaoHuangLi20} 
the relation between the cluster pattern and $Y$-pattern having the \emph{same} initial exchange matrix.
\begin{proposition}[{\cite[Theorem~2.5]{CaoHuangLi20}}]\label{prop_Y-pattern_exchange_graph}
Let $(\bfy_{t_0}, \qbasispr_{t_0})$ be a $Y$-seed and let $\{(\bfy_t, \qbasispr_t)\}_{t \in \mathbb{T}_n}$ be the $Y$-pattern.
Let $(\bfx_{t_0},\qbasis_{t_0})$ be a cluster seed such that the principal part of the exchange matrix $\qbasis_{t_0}$ is $\qbasispr_{t_0}$ and let $\{(\bfx_t, \qbasis_t)\}_{t \in \mathbb{T}_n}$ be the cluster pattern.
Suppose that the initial variables $y_{1;t_0},\dots,y_{n;t_0}$ are algebraically independent. 
Then, we have
\[
\exchange(\{(\bfx_t, \qbasis_t)\}_{t \in \mathbb{T}_n}) = 
\exchange(\{(\bfy_t, \qbasispr_t)\}_{t \in \mathbb{T}_n}).
\]
\end{proposition}

Because of Assumption~\ref{assumption_finite}, Theorem~\ref{thm_exchange_graph_Dynkin}, and Proposition~\ref{prop_Y-pattern_exchange_graph}, 
when the initial variables $y_{1;t_0},\dots,y_{n;t_0}$ are algebraically independent, all the following exchange graphs are the same.
\[
\exchange(\{(\tilde\bfx_t, \qbasis_t)\}_{t \in \mathbb T_n})
 = \exchange(\{(\bfx_t, \qbasispr_t)\}_{t \in \mathbb T_n})
= \exchange(\{(\bfy_t, \qbasispr_t)\}_{t \in \mathbb T_n}).
\]
We simply denote the above exchange graphs with the associated root system $\Roots$ by
\begin{equation}\label{eq_exchange_graphs_are_the_same}
\exchange(\Roots) = \exchange(\{(\tilde\bfx_t, \qbasis_t)\}_{t \in \mathbb T_n})
 = \exchange(\{(\bfx_t, \qbasispr_t)\}_{t \in \mathbb T_n})
= \exchange(\{(\bfy_t, \qbasispr_t)\}_{t \in \mathbb T_n}).
\end{equation}

Since the exchange graph of a cluster pattern and that of a $Y$-pattern having the same type are the same, we will mainly treat exchange graphs of cluster patterns of finite or affine type from now on.

Let $\Roots$ be the root system defined by the Cartan counterpart of $\qbpr$.
It is proved in~\cite{FZ2_2003} and~\cite{ReadingStella20} that there is a bijective correspondence between a subset $\alposRoots \subset \Roots$, called \emph{almost positive roots}, and the set of cluster variables.
\begin{equation}\label{equation_bijective_vars_alpostRoots_facets}
\alposRoots \stackrel{1:1}{\longleftrightarrow}\{\text{cluster variables in $\cA$ of type $\dynX$}\} 
\end{equation}
More precisely, one may associate the set $-\SRoots$ of negative simple roots with the set of cluster variables $x_{1;t_0},\dots,x_{n;t_0}$ in the
initial seed $(\bfx_{t_0},\qbasis_{t_0})$;  
a positive root $\sum_{i=1}^n d_i \alpha_i$ is associated to a (non-initial) cluster variable of the
form
\[
	\frac{f(\bfx_{t_0})}{x_{1;t_0}^{d_1} \cdots x_{n;t_0}^{d_n}},\qquad
	f(\bfx_{t_0})\in\bbC[x_{1;t_0},\dots, x_{m;t_0}].
\]
Accordingly, each vertex of the exchange graph $\exchange(\Roots)$ corresponds to an $n$-subset of $\alposRoots$. 
We notice that when $\Roots$ is of finite type, the set $\alposRoots$ is given by  $\alposRoots\colonequals \Roots^+ \cup -\SRoots$. 
Here, $\Roots^+$ is the set of positive roots and $\SRoots=\{\alpha_1,\dots,\alpha_n\}$ is the set of simple roots. 

To study the combinatorics of exchange graphs, we prepare some terminologies.
Let $\Roots$ be a rank $n$ root system.
For every subset $J \subset [n]$, let $\Roots(J)$ denote the root subsystem
of $\Roots$ spanned by the set of simple roots $\{ \alpha_i \mid i \in J \}$.
Indeed, the Dynkin diagram of $\Roots(J)$ is the full subdiagram on the vertices in $J$. 
Note that $\Phi(J)$ may not be irreducible even if $\Phi$ is.

A \emph{Coxeter element} is a product of the simple
reflections. 
The order $h$ of a Coxeter element in $W$ is called the \emph{Coxeter number}
of $\Roots$. We present the known formula of Coxeter numbers $h$ in
Table~\ref{table_Coxeter_number} (see~\cite[Appendix]{Bourbaki02}).
\begin{table}[b]
	\begin{tabular}{c|ccccccccc}
		\toprule
		$\Roots$ & $\dynA_n$ & $\dynB_n$ & $\dynC_n$ & $\dynD_n$ & $\dynE_6$ & $\dynE_7$ & $\dynE_8$ & $\dynF_4$ & $\dynG_2$ \\
		\midrule
		$h$ & $n+1$ & $2n$ & $2n$ & $2n-2$ & $12$ & $18$ & $30$ & $12$ & $6$ 		\\
		\bottomrule 
	\end{tabular}
\caption{Coxeter numbers}\label{table_Coxeter_number}
\end{table}

The Dynkin diagrams of finite or affine root systems do not have cycles except 
of type $\exdynA_{n-1}$ for $n \geq 3$.
We consider \emph{bipartite coloring} on Dynkin diagrams except of type $\exdynA$, that is, we have a function $\varepsilon \colon [n] \to \{+,-\}$, called a \emph{coloring}, such that any two vertices $i$ and $j$ connected by an edge have different colors. 
Since we are considering tree-shaped diagrams, they admit bipartite colorings. 
We notice that a bipartite coloring on a Dynkin diagram decides a \emph{bipartite} skew-symmetrizable matrix $\qbasispr = (b_{i,j})$ of the same type by setting
\begin{equation}\label{eq_bipartite_matrix}
b_{i,j} > 0 \iff \varepsilon(i) = + \text{ and } \varepsilon(j) = -.
\end{equation}
Here, a skew-symmetrizable matrix is called \emph{bipartite} if there exists a coloring $\varepsilon$ satisfying~\eqref{eq_bipartite_matrix}.
Moreover, for a simply-laced Dynkin diagram, a bipartite coloring defines a \emph{bipartite quiver}, that is, each vertex of the quiver is either source or sink. More precisely, we let $i$ be a source if $\varepsilon(i) = +$; otherwise, a sink.

\begin{example}\label{example_F4_coloring}
Consider the coloring on the Dynkin diagram of $\dynF_{4}$. 
\[

\]
Here, \colorbox{cyclecolor2!50!}{\cyclecolornamesecond} nodes have color $+$; \colorbox{cyclecolor1!50!}{\cyclecolornamefirst} nodes have color $-$.
This coloring gives a skew-symmetrizable matrix $\qbasispr$ whose Cartan counterpart $C(\qbasispr)$ is of type $\dynF_4$.
\[
\qbasispr = %
.
\]
The coloring on the Dynkin diagram of $\dynE_6$ as shown on the left provides the bipartite quiver like the one on the right.
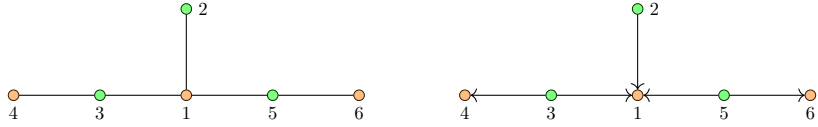
\[
%
		
\]
\end{example}

Let $I_+$ and $I_-$ be two parts of the set of vertices of the Dynkin diagram given by a bipartite coloring; they are determined
uniquely up to renaming.
Consider the composition
$\qcoxeter = \mutation_+ \mutation_-$ of a sequence of mutations where
\[
\mutation_{\varepsilon} = \prod_{i \in I_{\varepsilon}} \mutation_i \qquad \text{ for } \varepsilon \in \{ +, -\},
\]
which is well-defined (cf. Remark~\ref{rmk_mutation_commutes}).
We call $\qcoxeter$ a \emph{Coxeter mutation}.
Because of the definition, for a bipartite skew-symmetrizable matrix $\qbasispr$ or a bipartite quiver $\quiver$, we obtain
\[
\qcoxeter(\qbasispr) = \qbasispr,\qquad \qcoxeter(\quiver) = \quiver.
\]
The initial seed $\seed_{t_0} = \seed_0 = (\bfx_0, \qbasis_0)$ is included in the \textit{bipartite belt} consisting of the seeds $\seed_r = (\bfx_r, \qbasis_0)$ for $r \in \Z$ defined by 
\[
\seed_r = (\bfx_r, \qbasis_0) = \begin{cases}
\qcoxeter^r(\seed_0) & \text{ if } r > 0, \\
(\mutation_- \mutation_+)^{-r}(\seed_0) & \text{ if } r < 0.
\end{cases}
\]
We write 
\[
\bfx_r = (x_{1;r},\dots,x_{n;r}) \quad \text{ for }r \in \Z.
\]

It is known from~\cite{FZ_Ysystem03} and~\cite{ReadingStella20} that both $\mutation_+$ and
$\mutation_-$ act on the set $\alposRoots$ of almost positive roots and on the set $V(\exchange(\Roots))$ of vertices via the bijective correspondence~\eqref{equation_bijective_vars_alpostRoots_facets}. 
We summarize the properties of the action of Coxeter mutation as follows.

\begin{proposition}[{cf.~\cite[Propositions~2.5, 3.5, and~3.6]{FZ_Ysystem03} for finite type;~\cite[Propositions~5.4 and~5.14]{ReadingStella20} for affine type}]
    \label{prop_FZ_finite_type_Coxeter_element}
Let $\Roots$ be a finite or affine root system of type~$\dynX$. 
Let $\{(\mathbf x_t, \qbasis_t)\}_{t\in \mathbb T}$ be a cluster pattern of type $\dynX$ and $\exchange(\Roots)$ its exchange graph.
Then the following holds.
\begin{enumerate} 
\item For $\ell \in [n]$ and $r \in \Z$, we denote by $\exchangesub{\Roots}{x_{\ell;r}}$ the induced subgraph  of $\exchange(\Roots)$ consisting of seeds having the cluster variable $x_{\ell;r}$. Then, we have
\[
\exchangesub{\Roots}{x_{\ell;r}} \cong \exchange(\Roots([n] \setminus \{\ell\})).
\]
\item Both $\mutation_+$ and $\mutation_-$ act on the exchange graph $\exchange(\Roots)$.
\item For any seed $(\bfx, \qbasis) \in \exchange(\Roots)$, there exists $r\in \Z$ such that 
\[
|\{ x_{1;r},\dots,x_{n;r}\} \cap \{ x_{1},\dots,x_{n}  \}| \geq 2.
\]
Furthermore, if $\Roots$ is of finite type having even Coxeter number $h = 2e$, then $r \in \{0,1,\dots,e\}$.
\end{enumerate}
\end{proposition}

As a direct consequence of Proposition~\ref{prop_FZ_finite_type_Coxeter_element}, we 
have the following lemma which will be used later. 
\begin{lemma}\label{lemma:normal form}
Let $(\bfy_{t_0}, \qbasispr_{t_0})$ be a $Y$-seed such that the Cartan counterpart $C(\qbasispr_{t_0})$ is of finite or affine type.
For a $Y$-seed $(\bfy, \qbasispr)$ in the seed pattern, there exist $r \in\Z$, $\ell \in [n]$, and $j_1,\dots,j_{L} \in [n] \setminus \{\ell\}$ 
such that a sequence $\mutation_{j_1},\dots,\mutation_{j_L}$ of mutations 
connecting $\qcoxeter^r(\bfy_{t_0}, \qbasispr_{t_0})$ and $(\bfy, \qbasispr)$, that is, 
\[
(\bfy, \qbasispr) = (\mutation_{j_L} \cdots \mutation_{j_1})(\qcoxeter^r(\bfy_{t_0}, \qbasispr_{t_0})).
\]
Furthermore, if $\Roots$ is of finite type and has even Coxeter number $h = 2e$, then $r \in \{0,1,\dots,e\}$.
\end{lemma}
\begin{proof}
Since the exchange graph $\exchange(\{(\bfy_t, \qbasispr_t)\}_{t\in \mathbb{T}_n})$ is the graph $\exchange(\Roots)$ by Proposition~\ref{prop_Y-pattern_exchange_graph}, it is enough to prove the claim in terms of seeds. Let $(\bfx, \qbasis) \in \exchange(\Roots)$ be a seed. 
By Proposition~\ref{prop_FZ_finite_type_Coxeter_element}(3), there exist $\ell \in [n]$ and $r \in \Z$ such that 
$x_{\ell;r} \in \{x_1,\dots,x_n\}$. Accordingly, both seeds~$\qcoxeter^r(\bfx_{t_0}, \qbasis_{t_0})$ and $(\bfx, \qbasis)$ are contained in the induced subgraph $\exchangesub{\Roots}{x_{\ell;r}}$.
Since the subgraph $\exchangesub{\Roots}{x_{\ell;r}}$ itself is the exchange  graph of the root subsystem $\Roots([n] \setminus \{\ell\})$ by Proposition~\ref{prop_Y-pattern_exchange_graph}(1), it is connected. 
Therefore, two seeds~$\qcoxeter^r(\bfx_{t_0}, \qbasis_{t_0})$ and  $(\bfx, \qbasis)$ are connected without applying mutations at the vertex $\ell$, that is, there exists a sequence $j_1,\dots,j_L \in [n] \setminus \{\ell\}$ such that $(\bfx, \qbasis) = (\mutation_{j_L} \cdots \mutation_{j_1})(\qcoxeter^r(\bfx_{t_0}, \qbasis_{t_0}))$ as desired. Furthermore, if $\Roots$ is of finite type and has even Coxeter number $h = 2e$, then $r \in \{0,1,\dots,e\}$ because of Proposition~\ref{prop_FZ_finite_type_Coxeter_element}(3).
\end{proof}

For a finite root system $\Roots$, the exchange graph $\exchange(\Roots)$ becomes the one-skeleton of an $n$-dimensional polytope $P(\Roots)$, called the \emph{generalized associahedron}.  
Moreover, there is a bijective correspondence between the set $\facet(P(\Roots))$ of codimension-one faces, called \emph{facets}, of $P(\Roots)$ and the set of almost positive roots $\alposRoots$. 
We denote by $F_{\beta}$ the facet of the polytope $P(\Roots)$ corresponding to a root $\beta \in \alposRoots$. 
We demonstrate Proposition~\ref{prop_FZ_finite_type_Coxeter_element} for root systems of type $\dynA_3$ and $\dynD_4$. 
\begin{example}
	Consider the root system $\Roots$ of type $\dynA_3$. In this case, the
	Coxeter number is $4$, which is even (cf.
	Table~\ref{table_Coxeter_number}). In Table~\ref{table_A3_tau_action}, we
	present how $\qcoxeter$ acts on the set of almost positive roots. Here, we use the convention that
	$I_+ = \{1,3\}$ and $I_- = \{2\}$.
	\begin{table}[b]
		\begin{tabular}{c|ccc}
			\toprule
			$r$ & $\qcoxeter^r({-\alpha_1})$ 
			& $ \qcoxeter^r({-\alpha_2})$ 
			& $ \qcoxeter^r({-\alpha_3})$\\ 
			\midrule
			$0$ & ${-\alpha_1}$ & ${-\alpha_2}$ & ${-\alpha_3}$ \\
			$1$ & ${\alpha_1 + \alpha_2}$ & ${\alpha_2}$ &${\alpha_2 + \alpha_3}$ \\
			$2$ & ${\alpha_3}$ & ${\alpha_1 + \alpha_2 + \alpha_3}$ & ${\alpha_1}$\\
			\bottomrule
		\end{tabular}
		\caption{Computation $\qcoxeter^r({-\alpha_i})$ for type $\dynA_3$}\label{table_A3_tau_action}
	\end{table}
	The generalized associahedron of type $\dynA_3$ is presented in
	Figure~\ref{fig_asso_A3}. We label each codimension-one face the corresponding almost
	positive root. The back-side facets are associated with the set of
	negative simple roots. As one may see that the face posets of $\qcoxeter^r(F_{-\alpha_i})$ 
	are the same as that of the
	generalized associahedron $P(\Roots ([n]\setminus \{i\}))$. Indeed, the facets
	$\qcoxeter^r(F_{-\alpha_1})$ and
	$\qcoxeter^r(F_{-\alpha_3})$ are pentagons, and the facets 
	$\qcoxeter^r(F_{-\alpha_2})$ are squares.
For $(\bfx, \qbasis) = F_{-\alpha_1} \cap F_{-\alpha_2} \cap F_{-\alpha_3}$, 
we decorate the vertices $\{ \qcoxeter^r(\bfx,\qbasis) \mid r = 0,1,2 \}$ with green. As one can see, the orbits of $F_{-\alpha_1}, F_{-\alpha_2}, F_{-\alpha_3}$ exhaust all vertices as claimed in Proposition~\ref{prop_FZ_finite_type_Coxeter_element}(3). 
\end{example}

\begin{figure}
\tdplotsetmaincoords{110}{-30}

\caption{The type $\dynA_3$ generalized associahedron}\label{fig_asso_A3}
\end{figure}

\begin{example}
	We consider the generalized associahedron of type $\dynD_4$ and present
	four facets corresponding to the negative simple roots in
	Figure~\ref{fig_asso_D4}. The facet corresponding to $-\alpha_2$ is
    combinatorially equivalent to $P(\Roots(\{1\})) \times P(\Roots(\{3\})) 
    \times P(\Roots(\{4\}))$, which is a $3$-cube presented in the boundary. The
	intersection of these four facets is a vertex sits in the bottom colored
	in green. The Coxeter mutation $\qcoxeter$ acts on the face poset
	of the permutohedron, and four green vertices are in the same
	orbit.
\end{example}

\begin{remark}\label{remark:folding and Coxeter mutation}
As saw in Example~\ref{example_folding_ADE}, bipartite coloring 
on quivers of type~$\dynADE$ induce that on quivers of type~$\dynBCFG$.
Accordingly, if a seed pattern of simply-laced type $\dynX$ 
gives a seed pattern of type $\dynY$ via the folding procedure, then
the Coxeter mutation of type~$\dynY$ is the same as
that of type~$\dynX$. 
More precisely, for a globally foldable $Y$-seed $(\bfy,\qbasispr)$ with respect to $G$ 
of type $\dynX$ and its Coxeter mutation $\qcoxeter^{\dynX}$, we have
\[
 \qcoxeter^{\dynY}((\bfy,\qbasispr)^G) = (\qcoxeter^{\dynX}(\bfy,\qbasispr))^G.
\]
Here,  $\qcoxeter^{\dynY}$ is the Coxeter mutation on 
the seed pattern determined by $(\bfy,\qbasispr)^G$.

Moreover, Coxeter numbers of $\dynX$
and $\dynY$ are the same. Indeed, 
\[
\begin{split}
& h(\dynA_{2n-1}) = h (\dynB_n) = 2n, \\
& h(\dynD_{n+1}) = h(\dynC_n) = 2n, \\
& h(\dynE_6) = h(\dynF_4) = 12, \\
& h(\dynD_4) = h(\dynG_2) = 6.
\end{split}
\]
\end{remark}
\begin{figure}
	\subfigure[The generalized associahedron of type $\dynD_4$.]{
        \centering
\tdplotsetmaincoords{110}{260}
  
}
\caption{The generalized associahedron of type $\dynD_4$ and facets corresponding to some negative simple roots $-\alpha_1$, $-\alpha_3$, and $-\alpha_4$.}\label{fig_asso_D4}
\end{figure}

In the remaining part of this section, we recall~\cite{FZ4_2007} which
considers the combinatorics on mutations. Let
$\quiver$ be a bipartite quiver and $I_+$ and $I_-$ be the bipartite
decomposition of the vertex set of $\quiver$. Consider the composition
$\qcoxeter = \mutation_+ \mutation_-$ of a sequence of mutations where
\[
    \mutation_{\varepsilon} = \prod_{i \in I_{\varepsilon}} \mutation_i \qquad \text{ for } \varepsilon \in \{ +, -\}.
\]
We call $\qcoxeter$ a \emph{Coxeter mutation} as before.
We enclose this section by recalling the result~\cite[Theorem~8.8]{FZ4_2007}  on the order of Coxeter mutation on the cluster pattern. 
Recall from Proposition~\ref{prop_Y-pattern_exchange_graph} that for an exchange matrix $\qbasis_{t_0}$, if $\qbasispr_{t_0}$ is skew-symmetric, then the exchange graph of a seed pattern $\{(\bfx_t,\qbasis_t)\}_{t\in \mathbb{T}_n}$ and that of a $Y$-pattern $\{(\bfy_t, \qbasispr_{t})\}_{t\in\mathbb{T}_n}$ having algebraically independent variables $y_{1;t_0},\dots,y_{n;t_0}$ are the same. Accordingly, we obtain the following from~\cite[Theorem~8.8]{FZ4_2007}.
\begin{lemma}[{cf. \cite[Theorem~8.8]{FZ4_2007}}]\label{lemma:order of coxeter mutation}
Let $(\bfy_{t_0}, \qbasispr_{t_0})$ be an initial $Y$-seed. Suppose
that $\qbasispr_{t_0} = \qbasispr(\quiver)$ for a
bipartite quiver $\quiver$
and $y_{1;t_0},\dots,y_{n;t_0}$ are algebraically independent.
Then the set $\{ \qcoxeter^r (\bfy_{t_0}, \qbasispr_{t_0}) \}_{r \in \Z_{\geq 0}}$ 
of $Y$-seeds is finite if and only if $\qbasispr_{t_0}$ is of
finite type.

Moreover, for such a quiver $\quiver$, the order the $\qcoxeter$-action is given by $(h+2)/2$ if $h$ is even, or $h+2$ otherwise, where $h$ is the corresponding Coxeter number.
\end{lemma}

\section{Legendrians and \texorpdfstring{$N$-graphs}{N-graphs}}\label{sec:N-graph}

We recall from \cite{CZ2020} the notion of $N$-graphs and their combinatorial moves which encode the Legendrian isotopy data of corresponding Legendrian surfaces. As an application, we review how $N$-graphs can be use to find and to distinguish Lagrangian fillings for Legendrian links.

\subsection{\texorpdfstring{$N$}{N}-graphs and Legendrian weaves}

\begin{definition}\cite[Definition~2.2]{CZ2020}\label{definition:N-graph}
An  $N$-graph $\ngraph$ on a smooth surface $S$ is an $(N-1)$-tuple of graphs $(\ngraph_1,\dots, \ngraph_{N-1})$ satisfying the following conditions:
\begin{enumerate}
\item Each graph $\ngraph_i$ is embedded, trivalent, possibly empty and non necessarily connected.
\item Any consecutive pair of graphs $(\ngraph_i,\ngraph_{i+1})$, $1\leq i \leq N-2$, intersects only at hexagonal points depicted as in Figure~\ref{fig:hexagonal_point}.
\item Any pair of graphs $(\ngraph_i, \ngraph_j)$ with $1\leq i,j\leq N-1$ and $|i-j|>1$  intersects transversely at edges.
\end{enumerate}
\end{definition}

\begin{figure}[ht]
\begin{tikzpicture}
\begin{scope}
\draw[dashed] (0,0) circle (1cm);
\draw[red, thick] (60:1)--(0,0) (180:1)--(0,0) (-60:1)--(0,0);
\draw[blue, thick] (0:1)--(0,0) (120:1)--(0,0) (240:1)--(0,0);
\draw[thick,black,fill=white] (0,0) circle (0.05);
\end{scope}
\end{tikzpicture}
\caption{A hexagonal point}
\label{fig:hexagonal_point}
\end{figure}

Let $\pi_F:J^1S \cong T^*S\times\R\to S\times \R$ be the front projection, and we call the image $\pi_F(\Legendrian)$ of a Legendrian~$\Legendrian\subset J^1S$ a \emph{wavefront}.
Since $J^1S$ is equipped with the contact form $dz-p_x dx-p_y dy$, the coordinates~$(p_x,p_y)$ of the Legendrian $\Legendrian$ are recovered from $(x,y)$-slope of the tangent plane $T_{(x,y,z)}\pi_F(\Legendrian)$:
\begin{align*}
p_x&=\partial_x z(x,y),&  p_y&=\partial_y z(x,y).
\end{align*}
For any $N$-graph $\ngraph$ on a surface $S$, we associate a Legendrian surface $\Legendrian(\ngraph)\subset J^1S$. 
Basically, we construct the Legendrian surface by weaving the wavefronts in $S \times \R$ constructed from a local chart of $S$. 

Let $\ngraph\subset S$ be an $N$-graph.
A finite cover $\{U_i\}_{i\in I}$ of $S$ is called {\em $\ngraph$-compatible} if
\begin{enumerate}
\item each $U_i$ is diffeomorphic to the open disk $\mathring{\disk}^2$,
\item $U_i \cap \ngraph$ is connected, and
\item $U_i \cap \ngraph$ contains at most one vertex.
\end{enumerate}

For each $U_i$, we associate a wavefront $\wavefront(U_i)\subset U_i\times \R \subset S\times \R$.
Note that there are only five types of nondegenerate local charts for any $N$-graph $\ngraph$ as follows:
\begin{enumerate}[Type 1]
\item A chart without any graph component whose corresponding wavefront becomes
\[
\bigcup_{i=1,\dots,N}\mathring{\disk}^2\times\{i\}\subset \mathring{\disk}^2\times \R.
\]

\item A chart with single edge. The corresponding wavefront is the union of the $\dynA_1^2$-germ along the two sheets $\mathring\disk^2\times \{i\}$ and $\mathring\disk^2\times\{i+1\}$, and trivial disks $\disk^2\times\{j\}$, $j\in \{1,\dots,N\}\setminus\{i,i+1\}$.
The local model of $\dynA_1^2$ comes from the origin of the singular surface
\[
\wavefront(\dynA_1^2)=\{(x,y,z)\in \R^3 \mid x^2-z^2=0\}
\]
See Figure~\ref{fig:A_1^2 germ}.

\item A chart with two transversely intersecting edges. The wavefront consists of two $\dynA_1^2$-germs of $\mathring\disk^2\times\{i,i+1\}$ and $\mathring\disk^2\times\{j,j+1\}$ with $|i-j|>1$, and trivial disks $\disk^2\times\{k\}$, $k\in \{1,\dots,N\}\setminus\{i,i+1,j,j+1\}$.

\item A chart with a monochromatic trivalent vertex whose wavefront is the union of the $\dynD_4^-$-germ, see \cite[\S2.4]{Arn1990}, and trivial disks $\disk^2\times\{j\}$, $j\in \{1,\dots,N\}\setminus\{i,i+1\}$.
The local model for Legendrian singularity of type $\dynD_4^-$ is given by the image at the origin of 
\begin{align*}
\delta_4^-:\R^2\to \R^3:(x,y)\mapsto \left( x^2-y^2, 2xy, \frac{2}{3}(x^3-3xy^2) \right).
\end{align*}
See Figure~\ref{fig:D_4^- germ}.

\item A chart with a bichromatic hexagonal point. The induced wavefront is the union of the $\dynA_1^3$-germ along the three sheets $\mathring\disk^2\times \{*\}$, $*=i,i+1,i+2$, and the trivial disks $\disk^2\times\{j\}$, $j\in \{1,\dots,N\}\setminus\{i,i+1,i+2\}$. The local model of $\dynA_1^3$ is given by the origin of the singular surface
\[
\{(x,y,z)\in \R^3 \mid (x^2-z^2)(y-z)=0\}.
\]
See Figure~\ref{fig:A_1^3 germ}.
\end{enumerate}

\begin{figure}[ht]
\subfigure[The germ of $A_1^2$\label{fig:A_1^2 germ}]{\makebox[0.3\textwidth]{$

$}}
\subfigure[The germ of $A_1^3$\label{fig:A_1^3 germ}]{\makebox[0.3\textwidth]{$
%
$}}
\subfigure[The germ of $D_4^-$\label{fig:D_4^- germ}]{\makebox[0.3\textwidth]{$
%
$}}

\caption{Three-types of wavefronts of Legendrian singularities.}
\label{fig:legendrian_singularities}
\end{figure}

\begin{figure}[ht]
%
\caption{Local charts for $N$-graphs of Type 2,3,4, and 5.}
\label{fig:local_chart_3-graphs}
\end{figure}

\begin{definition}\cite[Definition~2.7]{CZ2020}
Let $\ngraph$ be an $N$-graph on a surface $S$. 
The {\em Legendrian weave}~$\Legendrian(\ngraph)\subset J^1 S$ is an embedded Legendrian surface whose wavefront $\wavefront(\ngraph)\subset S\times \R$ is constructed by weaving the wavefronts $\{\wavefront(U_i)\}_{i\in I}$ from a $\ngraph$-compatible cover $\{U_i\}_{i\in I}$ with respect to the gluing data given by $\ngraph$.
\end{definition}

\begin{remark}
When an $N$-graph $\ngraph$ is fixed, the space of possible Legendrian weaves $\Legendrian(\ngraph)$ is contractible via Legendrian isotopy.
So $\Legendrian(\ngraph)$ is well-defined up to Legendrian isotopy.
\end{remark}

We also list certain degenerate local models of $N$-graph as follows:
\begin{enumerate}[Type~D1]
\item A chart with double edges whose wavefront consists of two $\dynA_1^2$-germs of $\mathring\disk^2\times\{i,i+1\}$ and $\mathring\disk^2\times\{j,j+1\}$ for $|i-j|>1$, and trivial disks $\disk^2\times\{k\}$, $k\in \{1,\dots,N\}\setminus\{i,i+1,j,j+1\}$.
See the left-hand side of Figure~\ref{fig:degenerate type1}.

\item A chart with double trivalent vertices whose wavefront consists of two $\dynD_4^-$-germs at the level of $i, i+1$, and $j,j+1$ with $|i-j|>1$.
The other levels are trivial disks.
See the right-hand side of Figure~\ref{fig:degenerate type1}.

\item A chart with trichromatic graph of $(\ngraph_{i-1},\ngraph_i,\ngraph_{i+1})$ satisfying \label{degenerate_type2}
\begin{itemize}
\item each has a unique vertex of four valent,
\item $\ngraph_{i-1}$ and $\ngraph_{i+1}$ are identical, and
\item $\ngraph_i$ and $\ngraph_{i+1}$ are intersecting at the vertex of eight valent in an alternating way, see the middle one in Figure~\ref{fig:degenerate type3}.
\end{itemize}
\end{enumerate}

For $i=2$, the wavefront corresponding to a chart of \ref{degenerate_type2} inside $\disk^2\times \R$ consists of four disks $(\disk_1,\dots,\disk_4)$, which is the cone $C(\legendrian)=\legendrian\times[0,1]/\legendrian\times\{0\}$ of the following Legendrian front $\legendrian$ in $\sphere^1\times\R$
\[
(\sigma_{1,3}\sigma_2)^4=\vcenter{\hbox{

}},
\] 
where $\sigma_{1,3}$ is a $4$-braid isotopic to $\sigma_1\sigma_3$ (or equivalently, $\sigma_3\sigma_1$) such that two crossings $\sigma_1$ and $\sigma_3$ occur simultaneously. 

\begin{figure}[ht]
\subfigure[Type D1 and D2\label{fig:degenerate type1}]{
$%
$}
\caption{Local models for degenerate $N$-graphs and their perturbations}
\label{figure:perturbation of degenerated Ngraphs}
\end{figure}

\begin{remark}
The cone point neighborhood of the wavefront for the degenerate $N$-graph of~\ref{degenerate_type2}, is diffeomorphic to the union of four planes, $\{z=x\}, \{z=-x\}, \{z=y\},$ and $\{z=-y\}$ in $\R^3$, see Figure~\ref{fig:degenerated N-graph}.
\end{remark}

\begin{figure}[ht]
%
\caption{A wavefront for the degenerate $N$-graph.}
\label{fig:degenerated N-graph}
\end{figure}

We obtain (regular) $N$-graphs from degenerate $N$-graphs via (generic) perturbation of the wavefront as depicted in Figure~\ref{figure:perturbation of degenerated Ngraphs}.

The idea of $N$-graph is useful in the study of Legendrian surface, because the Legendrian isotopy of the Legendrian weave $\Legendrian(\ngraph)$ can be encoded in combinatorial moves of $N$-graphs.

\begin{theorem}\cite[Theorem~1.1]{CZ2020}\label{thm:N-graph moves and legendrian isotopy}
Let $\ngraph$ be a non-degenerate local $N$-graph. The combinatorial moves $\Move{I}\sim \Move{VI'}$ in Figure~\ref{fig:move1-6} are Legendrian isotopies for $\Legendrian(\ngraph)$.
\end{theorem}

We denote the equivalence class of an $N$-graph $\ngraph$ up to the moves $\Move{I}\sim \Move{VI'}$ in Figure~\ref{fig:move1-6} by~$[\ngraph]$.
Let us also list the combinatorial moves \Move{DI} and \Move{DII} for Legendrian isotopies involving degenerate $N$-graphs as depicted in Figure~\ref{fig:move1-6}.

\begin{corollary}\label{cor:degenerate N-graph moves and legendrian isotopy}

Let $\ngraph$ be a local degenerate $N$-graph. The combinatorial moves \Move{DI} and \Move{DII} in Figure~\ref{fig:move1-6} are Legendrian isotopies for $\Legendrian(\ngraph)$.
\end{corollary}

\begin{proof}
It is direct to check that the moves (DI) and (DII) for degenerate $N$-graphs can be obtained by composing the perturbations in Figure~\ref{figure:perturbation of degenerated Ngraphs} and moves in Figure~\ref{fig:move1-6}. See Appendix~\ref{appendix:DI and DII}.
\end{proof}

\begin{figure}[ht]

\caption{Combinatorial moves for Legendrian isotopies of surface $\Legendrian(\ngraph)$.
Here the pairs ({\color{blue} blue}, {\color{red} red}) and ({\color{red} red}, {\color{green} green}) are consecutive. Other pairs are not.}
\label{fig:move1-6}
\end{figure}

\begin{definition}\label{definiton:freeness}
An $N$-graph $\ngraph$ on $S$ is called {\em free} if the induced Legendrian weave $\Legendrian(\ngraph)\subset J^1S$ can be woven without interior Reeb chord. 
\end{definition}

\begin{example}\cite[Example 7.3]{CZ2020}\label{ex:free N-graph}
Let $\ngraph\subset \disk^2$ be a $2$-graph such that $\disk^2\setminus \ngraph$ is simply connected relative to the boundary $\boundary\disk^2\cap(\disk^2\setminus \ngraph)$. Then $\ngraph$ is free if and only if $\ngraph$ has no faces contained in~$\mathring\disk^2$.
Note that each of such faces admits at least one Reeb chord, see Figure~\ref{fig:N-graphs with Reeb chords}.
\end{example}

\begin{figure}[ht]
\begin{tikzcd}
\begin{tikzpicture}
\draw [dashed] (0,0) circle [radius=1.5];
\draw [thick,blue] (-1.5,0)-- (-1,0) to[out=90,in=90] (1,0)--(1.5,0) (-1,0) to[out=-90,in=-90] (1,0);
\draw [thick, blue, fill] (-1,0) circle (1.5pt)  (1,0) circle (1.5pt);
\end{tikzpicture}
&
\begin{tikzpicture}
\draw [dashed] (0,0) circle [radius=1.5];
\draw [thick,blue] (90:1.5) -- (90:1) (210:1.5) -- (210:1) (-30:1.5) -- (-30:1)
(90:1) -- (210:1) -- (-30:1)-- (90:1);
\draw [thick, blue, fill] (90:1) circle (1.5pt) (210:1) circle (1.5pt) (-30:1) circle (1.5pt);
\end{tikzpicture}
\end{tikzcd}
\caption{$N$-graphs with Reeb chords}
\label{fig:N-graphs with Reeb chords}
\end{figure}

In particular, we have the following lemma whose proof is omitted.
\begin{lemma}\label{lemma:tree Ngraphs are free}
Let $\ngraph=(\ngraph_1,\dots,\ngraph_{N-1})$ be an $N$-graph on $\disk^2$. Suppose that each $\ngraph_i$ is a tree or empty. Then $\ngraph$ is free.
\end{lemma}

Let us consider the Lagrangian projection $\pi_L:J^1S\cong T^*S\times \R \to T^*S$.
Then the image 
\[
L(\ngraph)\colonequals\pi_L(\Legendrian(\ngraph))
\] 
of the Legendrian weave gives us an exact, possibly immersed Lagrangian surface in $T^*S$.
The following lemma is a direct consequence of Theorem~\ref{thm:N-graph moves and legendrian isotopy} and Definition~\ref{definiton:freeness}.
\begin{lemma}
Let $\ngraph$ and $\ngraph'$ be two $N$-graphs on $S$. Then the following statements hold:
\begin{enumerate}
\item If $\ngraph$ is free, then the Lagrangian surface $L(\ngraph)=\pi_L(\Legendrian(\ngraph))$ is exact and embedded.
\item If $[\ngraph]=[\ngraph']$, then two Lagrangian surfaces
\[
L(\ngraph)=\pi_L(\Legendrian(\ngraph))\quad\text{ and }\quad
L(\ngraph')=\pi_L(\Legendrian(\ngraph'))
\]
in $T^*S$ are exact Lagrangian isotopic relative to boundary.
\end{enumerate}
\end{lemma}

\subsection{\texorpdfstring{$N$-graphs on $\disk^2$ and $\annulus$}{N-graphs on disk and annulus}}
In this section, we consider Legendrian links in $\R^3$ or $\sphere^3$, Lagrangian fillings in $\R^4$ and how to describe them in terms of $N$-graphs.

\subsubsection{Geometric setup}\label{sec:geometric setup}
Let us fix basic notions from $3$- and $5$-dimensional contact geometry.
Let $(\theta, p_\theta, z)$ be the coordinates of $J^1\sphere^1$ with the contact form $\alpha_{J^1\sphere^1}=dz-p_\theta d\theta$.
The Legendrian unknot $\legendrian_{\rm unknot}$ in $J^1\sphere^1$ is given by
\[
\legendrian_{\rm unknot}=\{ (\theta,0,0) \mid \theta \in \sphere^1\}\subset J^1\sphere^1.
\]
The symplectization of $J^1\sphere^1$ is 
\[(J^1\sphere^1\times \R_s,d(e^s(dz-p_\theta d\theta))),\] 
and its contactization becomes 
\[
(J^1\sphere^1\times \R_s \times \R_t,dt+e^s(dz-p_\theta d\theta))\] 
which is contactomorphic to $(J^1(\sphere^1\times \R_{r>0}), dw-p_\vartheta d\vartheta-p_r dr)$ under the strict contactomorphism~$\phi$ given by
\[
(\theta,p_\theta,z,s,t)\mapsto(\vartheta, r, p_\vartheta, p_r, w)=(\theta, e^s, e^s p_\theta, z, t+ e^s z).
\]
For each symplectization level $s=s_0$, the map $\phi$ induces a contact embedding $J^1\sphere^1 \hookrightarrow J^1(\sphere\times \R_{r>0})$ especially into $J^1(\sphere^1 \times \R_{r>0})\cap \{r=e^{s_0}\}$.

Furthermore, there is a strict contactomorphism
\[
\psi \colon (J^1(\sphere^1\times \R_{r>0}), dw-p_\vartheta d\vartheta-p_r dr) \to (J^1(\R^2\setminus\{{\bf 0}\}),dw-y_1dx_1-y_2dx_2)
\]
defined by
\[
(x_1,x_2,y_1,y_2,w)=\left(r \cos\vartheta,r\sin\vartheta, p_r \cos\vartheta -\frac{\sin\vartheta}{r}p_\vartheta, p_r \sin\vartheta + \frac{\cos\vartheta}{r}p_\vartheta,w\right).
\]
By compactifying the origin ${\bf 0}\in \R^2$, we have the following diagram:
\[
\begin{tikzcd}
J^1\sphere^1\times\mathbb{R}_s\times\mathbb{R}_t \arrow[r, "\cong", "\phi"'] \arrow[d,"\pi_t"] & 
J^1(\sphere^1\times\mathbb{R}_{r>0})\arrow[r,"\cong","\psi"'] &
J^1(\mathbb{R}^2\setminus\{\mathbf{0}\})\arrow[r, hookrightarrow] &
J^1\mathbb{R}^2=T^*\R^2\times \R_w \arrow[d,"\pi_L"]\\
J^1\sphere^1\times \R_s \arrow[rrr,hookrightarrow,"\Phi"] & & & T^*\R^2
\end{tikzcd}
\]
Here, the symplectic embedding $\Phi\colon J^1\sphere^1\times \R_s \hookrightarrow T^*\R^2$ is defined by
\begin{align*}
(\theta, p_\theta, z, s)\mapsto (x_1, x_2, y_1, y_2)=(e^s\cos\theta, e^s\sin\theta, z\cos\theta-p_\theta\sin\theta, z\sin\theta+p_\theta\cos\theta).
\end{align*}

On the other hand, we have another symplectomorphism 
\begin{align*}
\varphi \colon (\sphere^3\times \R_u,d(e^u\alpha_{\sphere^3}))&\to (T^*\R^2\setminus\{({\bf 0},{\bf 0})\},dx_1\wedge dy_1+dx_2\wedge dy_2);\\
(z_1,z_2,u)&\mapsto e^{u/2}(r_1\cos\theta_1,r_2\cos\theta_2, r_1\sin\theta_1,r_2\sin\theta_2)
\end{align*}
where $\sphere^3$ is the unit sphere in $\bbC^2_{z_1,z_2}$, $z_1=r_1 e^{i\theta_1}, z_2=r_2 e^{i\theta_2}$, and with the contact form
\begin{align*}
\alpha_{\sphere^3}=\frac{1}{2}r_1^2 d\theta_1 +\frac{1}{2} r_2^2d\theta_2,\quad r_1^2+r_2^2=1.
\end{align*}
So far, we have the following diagram of symplectic embeddings
\[
\begin{tikzcd}
J^1\sphere^1\times \R_s \arrow[rr, hookrightarrow, "\Phi"] \arrow[rd, hookrightarrow, "\Psi"'] & & T^*\R^2\setminus\{(\mathbf{0},\mathbf{0})\}\\
& \sphere^3\times\R_u \arrow[ru, "\cong"']
\end{tikzcd}
\]
where the map $\Psi(\theta, p_\theta, z, s) = (z_1, z_2, u)$ is defined by
\begin{align*}
z_1 &= \frac{e^s\cos\theta+ i (z\cos\theta-p_\theta\sin\theta)}{\sqrt{e^{2s}+z^2+p_\theta^2}},\\
z_2 &= \frac{e^s\sin\theta+ i (z\sin\theta+p_\theta\cos\theta)}{\sqrt{e^{2s}+z^2+p_\theta^2}},\\
e^u &= e^{2s}+z^2+p_\theta^2.
\end{align*}

Let us define $\iota:J^1\sphere^1\to \sphere^3$ as the composition of the inclusions
$J^1\sphere^1\cong J^1\sphere^1\times\{s=0\}\to J^1\sphere^1\times\R_s$, $\Psi:J^1\sphere^1\times\R_s\to \sphere^3\times\R^u$ and the projection $\sphere^3\times\R_u\to \sphere^3$ so that
\[
\iota(\theta, p_\theta, z)\colonequals
\left(
\frac{\cos\theta+i(z\cos\theta-p_\theta\sin\theta)}{\sqrt{1+z^2+p_\theta^2}},
\frac{\sin\theta+i(z\sin\theta+p_\theta\cos\theta)}{\sqrt{1+z^2+p_\theta^2}}
\right).
\]
Then the image of the Legendrian unknot $\legendrian_{\rm unknot}\subset J^1\sphere^1$ becomes
\[
\{(z_1,z_2) \mid z_1=\cos\theta, z_2=\sin\theta, \theta\in\sphere^1\}\subset \sphere^3\subset \bbC^2.
\]

Recall the stereographic projection of $\sphere^3$ with respect to $(0,-i)\in\bbC^2$, and see the corresponding image of $\legendrian_{\rm unknot}$:
\begin{align*}
(\sphere^3\setminus \{(0,-i)\},\alpha_{\sphere^3}) &\to (\R^3,dz'+x'dy'-y'dx')\cong \bbC\times \R;\\
(z_1,z_2)& \mapsto \left(\frac{i z_1}{i+z_2},\frac{-{\rm Re}(z_2)}{|i+z_2|^2}\right);\\
(\cos\theta,\sin\theta)&\mapsto \left(\frac{\cos\theta}{1+\sin^2\theta},\frac{\cos\theta\sin\theta}{1+\sin^2\theta},\frac{-\sin\theta}{1+\sin^2\theta}\right).
\end{align*}
Under the strict contactomorphism
\begin{align*}
(\R^3,dz'+x'dy'-y'dx') &\to (J^1\R,dz-ydx);\\
(x',y',z') &\mapsto (x,y,z)=(x',2y',z'+x'y'),
\end{align*}
the image of $\legendrian_{\rm unknot}$ becomes
\[
\left( \frac{\cos\theta}{1+\sin^2\theta},\frac{2\cos\theta\sin\theta}{1+\sin^2\theta},\frac{-2\sin^3\theta}{(1+\sin^2\theta)^2} \right)
\]
whose front projection looks like as follows:
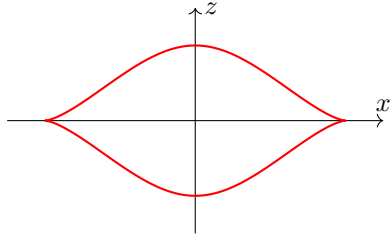
\[
\begin{tikzpicture}[baseline=-.5ex, scale=2]
\draw[->] (-1.25,0) -- (1.25,0) node[above] {$x$};
\draw[->] (0,-0.75) -- (0,0.75) node[right] {$z$};
\draw[thick, red] plot[domain=0:2*pi, samples=200] ({cos(\x r)/ (1+ sin(\x r)^2)}, {-2*sin(\x r)^3/(1+sin(\x r)^2)^2});
\end{tikzpicture}
\]

Let $\legendrian\subset J^1\sphere^1$ be a Legendrian link.
Then the image $\iota(\legendrian)$ can be isotoped into a neighborhood of the Legendrian unknot in $\R^3$.
We consider a Legendrian surface $\widehat\Legendrian \subset J^1(\sphere^1\times \R_{r>0})$ having cylindrical ends so that for some $S_1>S_2$,
\begin{align*}
\widehat\Legendrian \cap J^1(\sphere^1\times\R_{r\ge e^{S_1}}) &\cong \legendrian_1\times\R_{s\ge S_1},&
\widehat\Legendrian \cap J^1(\sphere^1\times\R_{r\le e^{S_2}}) &\cong \legendrian_2\times\R_{s\le S_2}.
\end{align*}
Then the projection $L_{\widehat\Legendrian}=\pi_L(\psi(\widehat\Legendrian))$ of the surface $\widehat\Legendrian$ inside $\sphere^3\times\R_u$ becomes an exact Lagrangian cobordism from $\iota(\legendrian_1)$ to $\iota(\legendrian_2)$.

Similarly, let $\widehat\Legendrian\subset J^1\R^2$ be a Legendrian surface having a cylindrical end. That is, for some $S\in \R$, 
\[
\widehat\Legendrian\cap J^1\R^2_{r\ge e^S}\cong \legendrian \times \R_{s\ge S}.
\]
Then the projection $L_{\widehat\Legendrian}=\pi_L(\widehat\Legendrian)$ in 
$T^*\R^2\cong(\bbC^2, \omega_{\mathsf{st}})$ becomes an exact Lagrangian filling of $\iota(\legendrian)$.
Note that the Lagrangian $\pi_L(\widehat\Legendrian)$ is embedded if and only if the Legendrian surface $\widehat\Legendrian$ has no Reeb chords.

\begin{lemma}\label{lem:legendrian and lagrangian}
Let $\widehat\Legendrian$ and $\widehat\Legendrian'$ be two Legendrian surfaces in $J^1\R^2$ without Reeb chords having the identical cylindrical ends
\[
\widehat\Legendrian\cap J^1\R^2_{r\ge e^S} \cong \legendrian \times \R_{s\ge S}
\cong
\widehat\Legendrian'\cap J^1\R^2_{r\ge e^S}
\]
for some $S\in\R$.
If the exact embedded Lagrangian fillings $L_{\widehat\Legendrian}=\pi_L(\widehat\Legendrian)$ and $L_{\widehat\Legendrian'}=\pi_L(\widehat\Legendrian')$ of $\iota(\legendrian)$ are exact Lagrangian isotopic, then $\widehat\Legendrian, \widehat\Legendrian'$ are Legendrian isotopic.
\end{lemma}

On the other hand, any compact Legendrian surface $\Legendrian\subset J^1\disk^2$ can be extended to $\widehat\Legendrian\subset J^1\R^2$ by attaching the cylindrical end $\boundary\Legendrian\times[1,\infty)$ in a smooth way.
For two compact Legendrian surfaces $\Legendrian, \Legendrian'\subset J^1\disk^2$, if $\widehat\Legendrian$ and $\widehat\Legendrian'$ are Legendrian isotopic if and only if $\Legendrian$ and $\Legendrian'$ are Legendrian isotopic relative to boundary.

\begin{corollary}
Let $\legendrian\subset J^1\sphere^1$ be a Legendrian link and $\Legendrian, \Legendrian'\subset J^1\disk^2$ be two Legendrian surfaces without Reeb chords whose boundaries are $\legendrian$.
Then two exact embedded Lagrangian fillings $\pi_{L}(\Legendrian)$ and $\pi_{L}(\Legendrian')$ are exact Lagrangian isotopic relative to boundary if and only if $\Legendrian$ and $\Legendrian'$ are Legendrian isotopic relative to boundary without making Reeb chords during the isotopy.
\end{corollary}

\begin{remark}
We are interested in exact Lagrangian fillings of Legendrian links up to \emph{exact Lagrangian isotopy} relative to boundary, 
an isotopy through exact Lagrangian fillings which fixes the Legendrian boundary.
This is equivalent to exact Lagrangian fillings up to \emph{Hamiltonian isotopy}, which is an isotopy through Hamiltonian diffeomorphism fixing the boundary. The similar holds for Lagrangian cobordisms.
\end{remark}

We end this section by investigating certain actions on the symplectic manifold $\sphere^3\times \R_u$ and induced actions on $J^1\sphere^1$.
Especially, we are interested in actions on $\sphere^3\times \R_u$ preserving the $\R_u$-coordinate, the symplectization coordinate. So actions on $\sphere^3$ determine the actions on the symplectic manifold $\sphere^3\times \R_u$.

Recall that $\sphere^3$ is the unit sphere in $\bbC^2$, i.e., coordinates $z_1=r_1 e^{i\theta_1}, z_2=r_2 e^{i\theta_2}$ with $r_1^2 + r_2^2=1$.

\paragraph{Rotation} A symplectomorphism $R_{\theta_0} \colon \sphere^3\times \R_u \to \sphere^3\times \R_u$, called {\em rotation}, is defined by
\[
R_{\theta_0}(z_1, z_2,u)=(z_1\cos\theta_0 -z_2\sin\theta_0, z_1\sin\theta_0+z_2\cos\theta_0,u).
\]
Note that the restriction $R_{\theta_0}|_{\sphere^3}$ fixes the contact form $\alpha_{\sphere^3}$.
Under the symplectic embedding $\Psi:J^1\sphere^1\times \R_s \hookrightarrow \sphere^3\times \R_u$, we have the following induced symplectomorphism
\[
J^1\sphere^1\times \R_s \to J^1\sphere^1\times \R_s,\quad (\theta,p_{\theta},z,s)\mapsto (\theta+\theta_0,p_{\theta},z,s).
\]
By restricting $R_{\theta_0}$ on $J^1\sphere^1$, we obtain
\[
J^1\sphere^1 \to J^1\sphere^1, \quad (\theta,p_{\theta},z)\mapsto (\theta+\theta_0,p_{\theta},z).
\]
We are especially interested in $\theta_0=\pi,2\pi/3$. They produce $\Z/2\Z$- and $\Z/3\Z$-action on the symplectic manifold $\sphere^3\times \R_u$ and Lagrangian fillings of satellite links of the Legendrian unknot, respectively.

\paragraph{Conjugation} An anti-symplectic involution $\tau \colon \sphere^3\times\R_u \to \sphere^3\times\R_u$, which we call {\em conjugation},  is defined by
\[
(z_1,z_2,u)\mapsto (\bar z_1,\bar z_2 ,u).
\]
It is direct to check that $\tau$ reverses the sign of symplectic form $\frac{i}{2}(dz_1\wedge d\bar z_1+dz_2\wedge d\bar z_2)$, and its restriction on $\sphere^3$ also reverse the sign of $\alpha_{\sphere^3}$. 
Again by the symplectic embedding~$\Psi$, the conjugation induces an action on $J^1\sphere^1\times \R_s$
\[
(\theta,p_\theta,z,s)\mapsto (\theta,-p_\theta,-z,s)
\]
whose restriction on $J^1\sphere^1$ becomes
\[
(\theta,p_\theta,z)\mapsto (\theta,-p_\theta,-z).
\]
This anti-symplectic involution naturally produce $\Z/2\Z$-action on the symplectic manifold and Lagrangian fillings as in the actions from the rotations.

\begin{figure}[ht]
\subfigure[Rotation\label{figure:rotation unknot}]{
\begin{tikzpicture}[baseline=-.5ex, scale=2]
\draw[->] (-1.25,0) -- (1.25,0) node[above] {$x$};
\draw[->] (0,-0.75) -- (0,0.75) node[right] {$z$};
\draw[thick, red] plot[domain=0:2*pi, samples=200] ({cos(\x r)/ (1+ sin(\x r)^2)}, {-2*sin(\x r)^3/(1+sin(\x r)^2)^2});
\draw[->,blue] (1,-0.1) to[out=210,in=30] (0.75,-0.25);
\draw[->,blue] (0.2,-0.6) to[out=190,in=-10] (-0.2,-0.6);
\draw[->,blue] (-0.75,-0.25) to[out=150,in=-30] (-1,-0.1);
\begin{scope}[rotate=180]
\draw[->,blue] (1,-0.1) to[out=210,in=30] (0.75,-0.25);
\draw[->,blue] (0.2,-0.6) to[out=190,in=-10] (-0.2,-0.6);
\draw[->,blue] (-0.75,-0.25) to[out=150,in=-30] (-1,-0.1);
\end{scope}
\end{tikzpicture}
}
\subfigure[Conjugation\label{figure:conjugation unknot}]{
\begin{tikzpicture}[baseline=-.5ex, scale=2]
\draw[->] (-1.25,0) -- (1.25,0) node[above] {$x$};
\draw[->] (0,-0.75) -- (0,0.75) node[right] {$z$};
\draw[thick, red] plot[domain=0:2*pi, samples=200] ({cos(\x r)/ (1+ sin(\x r)^2)}, {-2*sin(\x r)^3/(1+sin(\x r)^2)^2});
\draw[<->,blue] (0.65,-0.1) -- (0.75,-0.25);
\draw[<->,blue] (0,-0.4) -- (0,-0.6);
\draw[<->,blue] (-0.65,-0.1) -- (-0.75,-0.25);
\begin{scope}[rotate=180]
\draw[<->,blue] (0.65,-0.1) -- (0.75,-0.25);
\draw[<->,blue] (0,-0.4) -- (0,-0.6);
\draw[<->,blue] (-0.65,-0.1) -- (-0.75,-0.25);
\end{scope}
\end{tikzpicture}
}
\caption{Rotations and conjugation near the Legendrian unknot.}
\label{fig:actions on the unknot}
\end{figure}

\begin{lemma}\label{lem:rotation and conjugation}
Let $R_{\theta_0}$ and $\eta$ be rotation and conjugation defined on $\sphere^3\times \R$ as above, respectively.  Then the induced maps on the front projection $\pi_{F}:J^1\sphere^1\to\sphere^1\times\R$ becomes as follows:
\begin{align*}
R_{\theta_0}|_{\sphere^1\times \R}&:(\theta,z)\mapsto (\theta+\theta_0,z);\\
\eta|_{\sphere^1\times \R}&:(\theta,z)\mapsto (\theta,-z).
\end{align*}
\end{lemma}

\subsubsection{Positive $N$-braids}\label{section_stabilizations}
A \emph{positive $N$-braid} is a braid of $N$-strands represented by a finite word of positive generators $\sigma_1,\dots, \sigma_{N-1}$.
One may regard a positive $N$-braid $\beta$ as a Legendrian in $J^1\R^1$ (with cylindrical ends) whose front projection is the same as the usual braid diagram of $\beta$.

Let us start with two ways of obtaining a Legendrian link from a positive braid. The \emph{rainbow closure} is to close up a positive braid via nested copies of the Legendrian unknot.
The other way is called the \emph{$(-1)$-closure}, and it closes up a positive braid by considering parallel copies of the Legendrian unknot with respect to the Reeb direction.
Notice that the rainbow closure of $\beta_0$ is the same as the $(-1)$-closure of $\beta=\Delta_N\beta_0\Delta_N$ as seen in Figure~\ref{figure:closures}.
We will use the \emph{closure} to indicate the $(-1)$-closure unless mentioned otherwise.

\begin{figure}[ht]
\[
\setlength\arraycolsep{2pt}
\begin{array}{rcccc}
\legendrian_\beta&=&
\begin{tikzpicture}[baseline=-.5ex,scale=0.8]
\draw (-1,-1.125) rectangle node[yshift=-.5ex] {$\beta_0$} (0.5,-0.375);
\draw (-3,0) to[out=0,in=180] (-1,1) -- (1,1) to[out=0,in=180] (3,0);
\draw (-3,0) to[out=0,in=180] (-1,-1) (0.5,-1) -- (1,-1) to[out=0,in=180] (3,0);
\draw (-2.5,0) to[out=0,in=180] (-1,0.75) -- (1,0.75) to[out=0,in=180] (2.5,0);
\draw (-2.5,0) to[out=0,in=180] (-1,-0.75) (0.5,-0.75) -- (1,-0.75) to[out=0,in=180] (2.5,0);
\draw (-2,0) to[out=0,in=180] (-1,0.5) -- (1,0.5) to[out=0,in=180] (2,0);
\draw (-2,0) to[out=0,in=180] (-1,-0.5) (0.5,-0.5) -- (1,-0.5) to[out=0,in=180] (2,0);
\end{tikzpicture}
&=&
\begin{tikzpicture}[baseline=-.5ex,scale=0.8]
\draw[fill=white] (-1,-1.125) rectangle node[yshift=-.5ex] {$\beta_0$} (0.5,-0.375);
\draw (-3.5,0.5) to[out=0,in=180] (-2.5, 1) -- (2.5,1) to[out=0,in=180] (3.5,0.5);
\draw (-3.5,0.5) to[out=0,in=180] (-2.5,-0.5) -- (-1.75,-0.5) to[out=0,in=180] (-1, -1)
(0.5,-1) to[out=0,in=180] (1.25, -0.5) to[out=0,in=180] (2,-0.5) -- (2.5,-0.5) to[out=0,in=180] (3.5,0.5);
\draw (-3.5,0.25) to[out=0,in=180] (-2.5, 0.75) -- (2.5,0.75) to[out=0,in=180] (3.5,0.25);
\draw (-3.5,0.25) to[out=0,in=180] (-2.5,-0.75) -- (-1.75,-0.75) to[out=0,in=180] (-1.375,-1) to[out=0,in=180] (-1, -0.75)
(0.5,-0.75) to[out=0,in=180] (0.875, -1) to[out=0,in=180] (1.25,-0.75) -- (2.5,-0.75) to[out=0,in=180] (3.5,0.25);
\draw (-3.5,0) to[out=0,in=180] (-2.5, 0.5) -- (2.5,0.5) to[out=0,in=180] (3.5,0);
\draw (-3.5,0) to[out=0,in=180] (-2.5,-1) -- (-1.75,-1) to[out=0,in=180] (-1,-0.5)
(0.5,-0.5) to[out=0,in=180] (1.25, -1) -- (2.5,-1) to[out=0,in=180] (3.5,0);
\draw[dashed] (-1.75, -1.25) rectangle node[below=2.5ex] {$\beta$} (1.25, -0.25);
\end{tikzpicture}
\end{array}
\]
\caption{Rainbow and $(-1)$-closures of positive braids $\beta_0$ and $\beta$}
\label{figure:closures}
\end{figure}

For the closure $\legendrian_\beta$ of an $N$-braid $\beta$, 
The front projection $\pi_F(\legendrian_\beta)\subset \sphere^1\times \R$ of $\legendrian_\beta$ consists of $N$-strands with double points corresponding to the braid word $\beta$.
Hence, the Legendrian $\legendrian_\beta$ gives us an $(N-1)$-tuple $(\legendrian_1, \legendrian_2,\dots, \legendrian_{N-1})$ of subsets of points in $\sphere^1$, each of which corresponds to the generator $\sigma_i$ in the braid word $\beta$.

Conversely, let $(\legendrian_1,\dots, \legendrian_{N-1})$ be an $(N-1)$-tuple of disjoint\footnote{This condition can be weakened as follows: $\legendrian_i\cap \legendrian_{i+1}=\varnothing$ for each $1\le i<N$.} finite subsets of $\sphere^1$.
Then, from this data $(\legendrian_1,\dots,\legendrian_{N-1})$, one can build the Legendrian link $\legendrian$, which is the branched $N$-fold covering space of $\sphere^1$ such that the $i$-th and $(i+1)$-st covers are branched along the set $\legendrian_i$.

For a positive $N$-braid $\beta_0$, a \emph{stabilization} $\tilde\beta$ is a positive $(N+1)$-braid which satisfies the following:
\begin{enumerate}
\item The rainbow closures of $\beta_0$ and $\tilde\beta_0$ are Legendrian isotopic in $\sphere^3$, and
\item the braid $\beta_0$ can be recovered by forgetting a strand from $\tilde\beta_0$.
\end{enumerate}

The most typical example of a stabilization is as follows:
for a positive $N$-braid $\beta_0$, we introduce  a notation $S_0(\beta_0)$ for a specific type of \emph{stabilization} which is a positive $(N+1)$-braid defined by
$S_0(\beta_0) = \beta_0\sigma_N$, where $\beta_0$ in $S_0(\beta_0)$ is regarded as an $(N+1)$-braid by adding a trivial $(N+1)$-st strand.
For $\beta=\Delta_N\beta_0\Delta_N$, we introduce another notation $S(\beta) = (\sigma_1\dots\sigma_N) \beta (\sigma_N\dots\sigma_1) \sigma_1$. Then we have the following:
\begin{align*}
S(\beta) &=\Delta_{N+1} S_0(\beta_0) \Delta_{N+1}\mathrel{\dot{=}} \beta (\sigma_N\dots\sigma_2\sigma_1^3\sigma_2\dots\sigma_N)
\end{align*}
where $\mathrel{\dot{=}}$ means the same up to cyclic permutation of braid words.
See Figure~\ref{figure:stabilization of Legendrian}.

\begin{figure}[ht]
\[
\setlength\arraycolsep{2pt}

\]
\caption{A stabilization $\legendrian_{S(\beta)}$ of a Legendrian link $\legendrian_\beta$}
\label{figure:stabilization of Legendrian}
\end{figure}

The Legendrian $\legendrian_{S(\beta)}$ does depend on the braid word $\beta_0$.
For example, for each pair of positive $N$-braids $\beta_0^{(1)}$ and $\beta_0^{(2)}$ with $\beta_0=\beta_0^{(1)}\beta_0^{(2)}$, let $\beta_0'=\beta_0^{(2)}\beta_0^{(1)}$ and $\beta' = \Delta_N\beta_0'\Delta_N$.
Then two Legendrian links $\legendrian_\beta$ and $\legendrian_{\beta'}$ are Legendrian isotopic but $\legendrian_{S(\beta)}$ and $\legendrian_{S(\beta')}$ are \emph{not} Legendrian isotopic in general.
Therefore a stabilization of a Legendrian link $\legendrian$ which is a closure of a positive braid may not be uniquely determined.

\begin{example}\label{example:stabilization of An}
Let $\beta_0(a,b,c)=\2 \1^a \2^{b-1} \1^c$ and $\beta_0(\dynA_n)=\1^{n+1}$, then we deduce 
\[
\beta(a,b,c)\mathrel{\dot=}\sigma_2\sigma_1^{a+1}\sigma_2\sigma_1^{b+1}\sigma_2\sigma_1^{c+1}\quad\text{ and }\quad
\beta(\dynA_n)=\sigma_1^{n+3},
\] 
see Section~\ref{sec:linear} for details.
For each $b,c\ge 1$ with $b+c-1=n$, since $\beta(\dynA_n)=\sigma_1^{n+3} = \sigma_1^{c}\sigma_1^{b+2}$, we have
\begin{align*}
S(\beta(\dynA_n))&=(\sigma_1\sigma_2)\sigma_1\sigma_1^{c-1}\sigma_1^{b+1}\sigma_1(\sigma_2\sigma_1)\sigma_1=\beta(1,b,c).
\end{align*}
Therefore $\beta(1,b,c)$ is a stabilization of $\beta(n)$ for each $b+c-1=n$.
\end{example}

\begin{example}\label{ex:stabilization of (a,b,c)}
Let $\beta_{\degen,0}(p,q,r)=\sigma_{1,3}\sigma_2^{p}\sigma_{1,3}^{q-1}\sigma_2^r$,
where $\sigma_{1,3}$ is a $4$-braid isotopic to $\sigma_1\sigma_3$ (or equivalently, $\sigma_3\sigma_1$) such that two crossings $\sigma_1$ and $\sigma_3$ occur simultaneously.
Now consider
\[
\beta_\degen(p,q,r)=\Delta_4\tilde\beta_0(p,q,r)\Delta_4\mathrel{\dot{=}}\sigma_2^{p+1}\sigma_{1,3}\sigma_2\sigma_{1,3}^q \sigma_2^{r+1}\sigma_{1,3}\sigma_2\sigma_{1,3}^2.
\]
See Lemma~\ref{lemma:tripod to degenerated Ngraph}.

Let $\beta_0'(a,b,b)=\sigma_2^{b-1}\sigma_1\sigma_2^a\sigma_1^{b-1}\sigma_2\mathrel{\dot=}\beta_0(a,b,b)$ and $\beta'(a,b,b)=\Delta_3\beta_0'(a,b,b)\Delta_3$.
Then $\beta(a,b,b)\mathrel{\dot=}\beta'(a,b,b)$ and so $\legendrian_{\beta(a,b,b)}=\legendrian_{\beta'(a,b,b)}$. Moreover,
\begin{align*}
S(\beta'(a,b,b))&
\mathrel{\dot=}\Delta_4\sigma_1\sigma_2^a\sigma_1^{b-1}\sigma_2\sigma_3\sigma_2^{b-1}\Delta_4
\mathrel{\dot=}\Delta_4\sigma_{1,3}\sigma_2^{a}\sigma_{1,3}^{b-1}\sigma_2\Delta_4
=\beta_\degen(a,b,1),
\end{align*}
and so we conclude that the Legendrian $\legendrian_{\beta_\degen(a,b,1)}$ is a stabilization of $\legendrian_{\beta(a,b,b)}$. 
\end{example}
Recall the conjugate action on $\mathbb{C}^2$, which turns links upside down so that in terms of braid words, it interchanges $\sigma_i$ and $\sigma_{N-i}$ for each $N$-braid.
Hence, for $4$-braids, it preserves $\sigma_{1,3}$. Therefore $\beta_\degen(p,q,r)$ is invariant under conjugation and so is $\legendrian_{\beta_\degen(p,q,r)}$.

\begin{corollary}\label{corollary:invariance under conjugation}
The Legendrian $\legendrian_{\beta_\degen(p,q,r)}$ is invariant under conjugation. 
\end{corollary}

On the other hand, a stabilization $\legendrian_{S(\beta)}$ of $\legendrian_\beta$ will be represented by $N$-colored dots in $S^1$ while $\legendrian_\beta$ uses only $(N-1)$ colors. That is,
\[
\legendrian_{S(\beta)} \longleftrightarrow 
\begin{tikzpicture}[baseline=-.5ex, scale=0.6]
\draw (0,0) circle (2);
\curlybrace[]{90}{270}{2.2};
\draw (180:2.2) node[left] {$\beta$};
\draw[fill, violet] (75:2) circle (2pt) (-75:2) circle (2pt);
\draw[fill, yellow] (45:2) circle (2pt) (-45:2) circle (2pt);
\draw[fill, red] (30:2) circle (2pt) (-30:2) circle (2pt);
\draw[fill, blue] (15:2) circle (2pt) (-15:2) circle (2pt) (0:2) circle (2pt);
\draw (60:2.2) node[rotate=-30] {$\dots$} (-60:2.2) node[rotate=30] {$\dots$};
\end{tikzpicture}\subset J^1\sphere^1
\]
Then one can transfer an $N$-graph $\ngraph$ for $\legendrian_\beta$ into an $(N+1)$-graph $S(\ngraph)$ for $\legendrian_{S(\beta)}$ as follows:
\begin{center}
\begin{tikzcd}
\ngraph=\begin{tikzpicture}[baseline=-.5ex]
\draw [thick] (0,0) circle [radius=1];
\node at (0,0) {$\ngraph_{(1,\dots,N)}$};
\end{tikzpicture}
\arrow[leftrightarrow,"\mathrm{(S)}",r]&
\begin{tikzpicture}[baseline=-.5ex]
\draw (0,0) circle (1.5);
\draw[double] (0,1.5) -- (0,-1.5) arc (-90:-270:1.5);
\draw (-0.75,0) node {$\ngraph_{(1,\dots,N)}$};
\draw[violet] (-75:1.5) to[out=105,in=-90] (0.25,0) to[out=90,in=-105] (75:1.5);
\draw[yellow] (-45:1.5) to[out=135,in=-90] (0.75,0) to[out=90,in=-135] (45:1.5);
\draw[red] (-30:1.5) to[out=150,in=-90] (1,0) to[out=90,in=-150] (30:1.5);
\draw (0.5,0) node {$\scriptstyle\cdots$};
\draw[blue] (-15:1.5) to[out=135,in=-90] (1.25,0) to[out=90,in=-150] (15:1.5) (1.25,0) -- (1.5,0);
\draw[fill,blue] (1.25,0) circle (1pt);
\end{tikzpicture}=S(\ngraph)
\end{tikzcd}
\end{center}

\subsubsection{$N$-graphs on $\disk^2$ and $\annulus$}\label{section:annular Ngraphs}

Let $\ngraph=(\ngraph_1,\dots,\ngraph_{N-1})$ be an $N$-graph on $\disk^2$.
The \emph{boundary} $\boundary\ngraph$ of $\ngraph$ is a Legendrian link defined by an $N$-graph on $\sphere^1=\boundary\disk^2$ as
\[
\boundary\ngraph=(\boundary\ngraph_1,\dots,\boundary\ngraph_{N-1}),\quad
\boundary\ngraph_i\colonequals \ngraph_i\cap \sphere^1\subset \sphere^1.
\]
We say that $\ngraph$ is \emph{of type} $\legendrian$ or $\legendrian$ \emph{admits} an $N$-graph $\ngraph$ if $\boundary\ngraph=\legendrian$.

Let $\annulus$ be the oriented annulus with two boundary components $\boundary_+\annulus$ and $\boundary_-\annulus$.
For an $N$-graph~$\ngraph$ on $\annulus$, let $\boundary_\pm\ngraph\colonequals\ngraph\cap\boundary_\pm\annulus$ be Legendrian links at two boundaries $\boundary_\pm\annulus$, respectively.
We say that $\ngraph$ is \emph{of type} $(\legendrian_+, \legendrian_-)$ if $\boundary_\pm\ngraph=\legendrian_\pm$, respectively.

A typical example of annular $N$-graphs comes from Lagrangian cobordism between Legendrian links, which are closures of positive braids.
In particular, for two closures $\legendrian_1$ and $\legendrian_2$ of positive braids $\beta_1$ and $\beta_2$, any sequence of Legendrian braid moves from $\legendrian_2$ to $\legendrian_1$ will give us a special annular $N$-graph $\ngraph_{\legendrian_2\legendrian_1}$.\footnote{One may call the $N$-graph $\ngraph_{\legendrian_2\legendrian_1}$ a \emph{strict concordance} since it is a union of cylinders.}
Hence, for an $N$-graphs $\ngraph$ with $\boundary\ngraph=\legendrian_1$, we have the $N$-graph $\ngraph_{\legendrian_2\legendrian_1}\ngraph$ with boundary
\[
\boundary(\ngraph_{\legendrian_2\legendrian_1}\ngraph)=\legendrian_2.
\]

\begin{remark}
We are dealing with both Legendrian links $\legendrian$ and surfaces $\Legendrian$.
In order to avoid the confusion, we use the terminologies ``$\boundary$-Legendrian isotopy'' and ``Legendrian isotopy'' for isotopies between Legendrian links and surfaces, respectively.
\end{remark}

Since a closure of a Legendrian positive braid in $J^1\sphere^1$ should not have any cusp, possible $\boundary$-Legendrian isotopies are either plane isotopies \Move{R0} or the third Reidemeister move \Move{RIII} as follows:
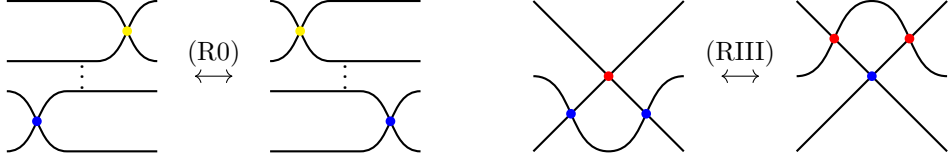
\begin{center}

\end{center}

Therefore, any annular $N$-graph corresponding to a sequence of Reidemeister moves between Legendrian links is a concatenation of \emph{elementary annular $N$-graphs}, which are  $\ngraph_{\Move{R0}}$ and $\ngraph_{\Move{RIII}}$ on the annulus $\annulus$ as depicted in Figure~\ref{fig:elementary annulus N-graph}.
We call an annular $N$-graph \emph{tame} if it is a concatenation of elementary annular $N$-graphs. 

\begin{figure}[ht]
\subfigure[Annular $N$-graph \Move{R0}]{\makebox[0.45\textwidth]{$
%
$
}}
\caption{$\boundary$-Legendrian isotopy and elementary annulus $N$-graphs}
\label{fig:elementary annulus N-graph}
\end{figure}

\begin{example}
A \emph{rotational annular $N$-graph}, which has no vertices and rotates a certain angle as depicted below is tame.
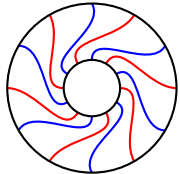
\[
%
\]

It is known that the rotational annular $N$-graph acts on the set of $N$-graphs for the Legendrian torus link $\legendrian(n,m)$ of maximal Thurston--Bennequin number.
This type of annular $N$-graphs play a crucial role in producing a sequence of distinct exact Lagrangian fillings of positive braid Legendrian links, see \cite{Kal2006, CG2020, GSW2020b}.
\end{example}

\begin{definition}\label{def:boundary Legendrian isotopic}
We say that two $N$-graphs $\ngraph$ and $\ngraph'$ with $\boundary\ngraph=\legendrian_1$ and $\boundary\ngraph=\legendrian_2$ are \emph{$\boundary$-Legendrian isotopic} if there exists a tame annular $N$-graph $\ngraph_{\legendrian_2\legendrian_1}$ such that $[\ngraph']= [\ngraph_{\legendrian_2\legendrian_1}\ngraph]$.
\end{definition}

\subsubsection{Annular $N$-graphs and Legendrian loops}
Let $\beta, \beta_+, \beta_-\subset J^1\R^1$ be Legendrian positive $N$-braids.
We denote by $\Ngraphs(\beta)$ and $\Ngraphs(\beta_+, \beta_-)$ the sets of equivalence classes of (degenerate) $N$-graphs on $\disk^2$ and $\annulus$ satisfying boundary conditions given by the closure $\legendrian_\beta$ or a pair of closures $(\legendrian_{\beta_+}, \legendrian_{\beta_-})$
up to local (degenerate) moves in Figure~\ref{fig:move1-6} relative to the boundary.
\begin{align*}
\Ngraphs(\beta)&\colonequals
\{[\ngraph]\mid \ngraph\text{ is an $N$-graph on $\disk^2$ of type $\legendrian_\beta$}\}\\
\Ngraphs(\beta_+, \beta_-)&\colonequals
\{[\ngraph]\mid \ngraph\text{ is an $N$-graph on $\annulus$ of type $(\legendrian_{\beta_+}, \legendrian_{\beta_-})$}\}.
\end{align*}
Here, we are assuming that we are aware of where each braid word starts.

By a direct consequence of Theorem~\ref{thm:N-graph moves and legendrian isotopy} and Corollary~\ref{cor:degenerate N-graph moves and legendrian isotopy}, if $[\ngraph]=[\ngraph']$ as in the elements of $\Ngraphs(\beta_+, \beta_-)$, then $\Legendrian(\ngraph)$ and $\Legendrian(\ngraph')$ are Legendrian isotopic relative to the boundary $(\legendrian_{\beta_+},\legendrian_{\beta_-})$.

Then it is direct to check that these sets are invariant under the cyclic rotation of the braid words up to bijection. More precisely, for $N$-braids $\beta^{(1)}, \beta^{(2)}$ and $\beta^{(1)}_\pm, \beta^{(2)}_\pm$, closures of $\beta^{(1)}\beta^{(2)}$ and $\beta^{(2)}\beta^{(1)}$ are identical in $J^1\sphere^1$ and there are one-to-one correspondences between sets of $N$-graphs
\begin{align}
\begin{split}
\Ngraphs\left(\beta^{(1)}\beta^{(2)}\right) &\cong \Ngraphs\left(\beta^{(2)}\beta^{(1)}\right),
\end{split}\\
\begin{split}
\Ngraphs\left(\beta_+, \beta^{(1)}_-\beta^{(2)}_-\right) &\cong \Ngraphs\left(\beta_+, \beta^{(2)}_-\beta^{(1)}_-\right),\\
\Ngraphs\left(\beta^{(1)}_+\beta^{(2)}_+, \beta_-\right) &\cong \Ngraphs\left(\beta^{(2)}_+\beta^{(1)}_+, \beta_-\right).
\end{split}
\label{equation:cyclic rotation of word}
\end{align}
Indeed there are infinitely many bijections in each case which are induced by rotating a boundary (counter)clockwise by appropriate angle and so indexed canonically by the set $\mathbb{Z}$. We omit the detail.

Suppose that $\ngraph_1\in\Ngraphs(\beta_2,\beta_1)$ and $\ngraph_2\in\Ngraphs(\beta_3,\beta_2)$.
Then two $N$-graphs can be merged or piled in a natural way to obtain the annular $N$-graph, denoted by~$\ngraph_2\ngraph_1\in\Ngraphs(\beta_3, \beta_1)$.
On the other hand, for $\ngraph\in\Ngraphs(\beta)$ and $\ngraph_1\in\Ngraphs(\beta', \beta)$, the concatenation $\ngraph_1\ngraph\in\Ngraphs(\beta')$ is well-defined by gluing along the boundary $\legendrian_\beta$.
Hence, we have two natural maps
\begin{align*}
\Ngraphs(\beta_3,\beta_2)\times\Ngraphs(\beta_2,\beta_1)&\to \Ngraphs(\beta_3,\beta_1),\\
\Ngraphs(\beta', \beta)\times\Ngraphs(\beta) &\to \Ngraphs(\beta').
\end{align*}

In particular, for each $\boundary$-Legendrian isotopy from $\legendrian'=\legendrian_{\beta'}$ and $\legendrian=\legendrian_\beta$, we have a tame annular $N$-graph $\ngraph_{\legendrian'\legendrian}\in\Ngraphs(\beta',\beta)$, where $\legendrian$ and $\legendrian'$ are closures of $\beta$ and $\beta'$, respectively.
Moreover, we also have a tame annular $N$-graph $\ngraph^{-1}_{\legendrian'\legendrian}$ obtained by flipping the annulus inside out corresponding to the inverse isotopy from $\legendrian$ to $\legendrian'$.
Hence, we have two maps inverses to each other
\[
\Ngraphs(\beta) \to \Ngraphs(\beta'),\quad\text{ and }\quad
\Ngraphs(\beta')\to \Ngraphs(\beta),
\]
defined by
\[
\ngraph\mapsto \ngraph_{\legendrian'\legendrian}\cdot\ngraph,\quad\text{ and }\quad
\ngraph'\mapsto \ngraph^{-1}_{\legendrian'\legendrian}\cdot\ngraph',
\]
respectively.

Let $\Ngraphs_0(\beta,\beta)$ be the subset of tame annular $N$-graphs of type $(\beta,\beta)$.
\begin{lemma}
Let $\beta$ be a Legendrian positive $N$-graph.
The set $\Ngraphs_0(\beta,\beta)$ becomes a group under the concatenation which acts on the set $\Ngraphs(\beta)$.
\end{lemma}
\begin{proof}
It is easy to see that the set $\Ngraphs_0(\beta,\beta)$ is closed under the concatenation, which is associative.
The trivial $\boundary$-Legendrian isotopy gives us the identity annular $N$-graph.

Finally, for each $\ngraph\in\Ngraphs_0(\beta,\beta)$, the $N$-graph $\ngraph^{-1}$ plays the role of the inverse of $\ngraph$ due to the Move \Move{I} and \Move{V} of $N$-graphs in Figure~\ref{fig:move1-6}.
Hence $\Ngraphs_0(\beta,\beta)$ becomes a group acting on the set $\Ngraphs(\beta)$ by concatenation, and so we are done.
\end{proof}

\begin{definition}[Legendrian loop]\label{definition:Legendrian loops}
Let $\legendrian \subset (\R ^3, \xi_{\rm st})$ be a Legendrian link and $\cL(\legendrian)$ be the space of Legendrian links isotopic to $\legendrian$. 
A {\em Legendrian loop} $\vartheta$ is a continuous map $\vartheta\colon(\sphere^1,{\rm pt})\to (\cL(\legendrian), \legendrian)$ and said to be \emph{tame} if the Legendrian $\vartheta(\theta)$ is a closure of a positive braid for each $\theta\in\sphere^1$.
\end{definition}

\begin{remark}
One can regard each Legendrian loop $\vartheta$ for $\legendrian$ as an element of the fundamental group $\pi_1(\cL(\legendrian), \legendrian)$.
\end{remark}

Let $\legendrian$ be the closure of a positive braid $\beta$.
Then each tame Legendrian loop for $\legendrian$ corresponds to a $\boundary$-Legendrian isotopy from $\legendrian$ to $\legendrian$ and can be regarded as an element $\ngraph_\vartheta$ in $\Ngraphs_0(\beta,\beta)$.
Conversely, any element $\ngraph$ in $\Ngraphs_0(\beta,\beta)$ defines a tame Legendrian loop $\vartheta_\ngraph$ obviously.

In summary, we have the following lemma.
\begin{lemma}\label{lemma:Legendrian loops and tame annular Ngraphs}
Let $\beta$ be a Legendrian positive $N$-braid. Then there is one-to-one correspondence between $\Ngraphs_0(\beta,\beta)$ and the subset of homotopy classes of tame Legendrian loops for $\legendrian=\legendrian_\beta$.
In particular, each tame Legendrian loop acts on $\Ngraphs(\beta)$.
\end{lemma}

\subsection{One-cycles in Legendrian weaves}\label{sec:1-cycles in Legendrian weaves}
Let us recall from \cite{CZ2020} how to construct a seed from an $N$-graph~$\ngraph$.
Each one-cycle in $\Legendrian(\ngraph)$ corresponds to a vertex of the quiver,
and a monodromy along that cycle gives a coordinate function at that vertex.
The quiver is obtained from the intersection data among one-cycles.
Moreover, there is an operation in $N$-graph, called \emph{Legendrian mutation}, which is a counterpart of the mutation in the cluster structure.
The Legendrian mutation is crucial in constructing and distinguishing $N$-graphs.
In turn, these will give as many Lagrangian fillings of a given Legendrian links as seeds in the associated cluster pattern to the link which will be discussed in Section~\ref{sec:N-graph of finite or affine type}.

Let $\ngraph\subset \disk^2$ be a free $N$-graph and let  $\Legendrian(\ngraph)$ be the induced Legendrian weave.
We express one-cycles of $\Legendrian(\ngraph)$ in terms of subgraphs of $\ngraph$.

\begin{definition}
A subgraph $\sfT$ of a nondegenerate $N$-graph $\ngraph$ is said to be \emph{admissible} if at each vertex, it looks locally one of pictures depicted in Figure~\ref{fig:T cycle}.
For a degenerate $N$-graph $\ngraph$, a subgraph $\sfT$ is \emph{admissible} if so is its perturbation as a subgraph of the perturbation of $\ngraph$. See Figure~\ref{figure:perturbation of admissible subgraphs}.

For an admissible subgraph $\sfT\subset\ngraph$, let $\ell(\sfT)\subset\disk^2$ be an oriented, immersed, labelled loop given by gluing paths whose local pictures look as depicted in Figure~\ref{fig:T cycle}.
\end{definition}

\begin{figure}[ht]
\subfigure[A trivalent vertex: case 1
\label{figure:loop near vertex1}]{\makebox[.3\textwidth]{

}}%
\subfigure[A trivalent vertex: case 2
\label{figure:loop near vertex2}]{\makebox[.3\textwidth]{
%
}}%
\subfigure[A trivalent vertex: case 3
\label{figure:loop near vertex3}]{\makebox[.3\textwidth]{
%
}}
\subfigure[A hexagonal vertex: case 1
\label{figure:loop near hexagon II1}]{\makebox[.3\textwidth]{
%
}}
\subfigure[A hexagonal vertex: case 2
\label{figure:loop near hexagon II2}]{\makebox[.3\textwidth]{
%
}}
\subfigure[A hexagonal vertex: case 3
\label{figure:loop near hexagon II3}]{\makebox[.3\textwidth]{
%
}}
\caption{Local configurations on cycles and corresponding arcs of $\ngraph\subset\disk^2$}
\label{fig:T cycle}
\end{figure}

\begin{figure}
\[
\begin{tikzcd}
%
\end{tikzcd}
\]
\caption{Local configurations on degenerate cycles and its perturbation.}
\label{figure:perturbation of admissible subgraphs}
\end{figure}

The loop $\ell(\sfT)$ defines a unique lift $\tilde\ell(\sfT)\subset\wavefront(\ngraph)$ via $\pi_{\disk^2}:\wavefront(\ngraph)\to\disk^2$ so that each $s_j$-labelled arc in $\ell(\sfT)$ is contained in the $s_j$-th sheet of $\wavefront(\ngraph)$.
Moreover, the immersed loop $\tilde\ell(\sfT)$ lifts uniquely to an embedded loop $\cycle(\sfT)$ in $\Legendrian(\ngraph)$ via the front projection $\pi_F:\Legendrian(\ngraph)\to\wavefront(\ngraph)$.

\begin{definition}\label{def:one-cycles}[$\sfT$-cycle]
For an admissible subgraph $\sfT\subset\ngraph$, if $\sfT \cap \boundary \disk^2=\emptyset$ we call the cycle $[\cycle(\sfT)]\in H_1(\Legendrian(\ngraph);\Z)$ a \emph{$\sfT$-cycle}. When $\sfT \cap \boundary \disk^2 \neq \emptyset$, we then call the relative cycle $[\cycle(\sfT)]\in H_1(\Legendrian(\ngraph),\legendrian(\boundary \ngraph);\Z)$ a \emph{relative $\sfT$-cycle}.
\end{definition}

\begin{example}[(Long) $\sfI$-cycles]
For an edge $e$ of $\ngraph$ connecting two trivalent vertices, let $\sfI(e)$ be the subgraph of $\ngraph$ consisting of a single edge $e$.
Then the cycle~$[\cycle(\sfI(e))]$ depicted in Figure~\ref{figure:I-cycle} is called an \emph{$\sfI$-cycle}.
Similarly, for an edge $e$ of $\ngraph$ connecting a point on the boundary $\boundary\disk^2$ and a trivalent vertex, we call the induced relative cycle $[\cycle(\sfI(e))]$ in $H_1(\Legendrian(\ngraph),\legendrian(\boundary \ngraph))$ a \emph{relative $\sfI$-cycle}.

In general, a linear chain of edges $(e_1,e_2,\dots, e_n)$ satisfying
\begin{itemize}
\item $e_i$ connects a trivalent vertex and a hexagonal point for $i=1,n$;
\item $e_i$ and $e_{i+1}$ meet at a hexagonal point in the opposite way, see Figure~\ref{figure:long I-cycle}, for $i=2,\dots, n-1$
\end{itemize}
forms an adissible subgraph $\sfI(e_1,\dots, e_n)$, and the cycle $[\cycle(\sfI(e_1,\dots, e_n))]$ is called a \emph{long $\sfI$-cycle}. See Figure~\ref{figure:long I-cycle}.
\end{example}

\begin{example}[$\sfY$-cycles]
Let $e_1,e_2,e_3$ be monochromatic edges joining a hexagonal point $h$ and trivalent vertices $v_i$ for $i=1,2,3$.
Then the subgraph $\sfY(e_1,e_2,e_3)$ consisting of three edges $e_1, e_2$ and $e_3$ is an admissible subgraph of $\ngraph$ and it defines a cycle $[\cycle(\sfY(e_1,e_2,e_3))]$ called an \emph{upper} or \emph{lower}~\emph{$\sfY$-cycle} according to the relative position of sheets that edges represent.
See Figures~\ref{figure:Y-cycle_1} and~\ref{figure:Y-cycle_2}.
\end{example}

\begin{figure}[ht]
\subfigure[An $\sfI$-cycle $\cycle(\sfI(e))$\label{figure:I-cycle}]{\makebox[.4\textwidth]{
}}
\caption{(Long) $\sfI$- and $\sfY$-cycles}
\label{fig:I and Y cycle}
\end{figure}

One of the benefit of cycles from admissible subgraphs is that one can keep track how cycles are changed under the $N$-graph moves described in Figure~\ref{fig:move1-6}, especially under Move~\Move{I} and Move~\Move{II}.
Note that Move~\Move{III} can be decomposed into a sequence of Move~\Move{I} and Move~\Move{II}.
Some of such changes are given in Figure~\ref{fig:cycles under moves}.

\begin{figure}[ht]
\[
\begin{tikzcd}[row sep=0pc]
%
 
\end{tikzcd}
\]
\caption{Cycles under Moves~\Move{I}, \Move{II}, \Move{DI} and \Move{DII}.}
\label{fig:cycles under moves}
\end{figure}

\begin{remark}
It is important to note that not every cycle can be represented by a subgraph. For example, the cycle on the left of the following picture can not be expressed by a subtree but it can be after Move~\Move{I}.
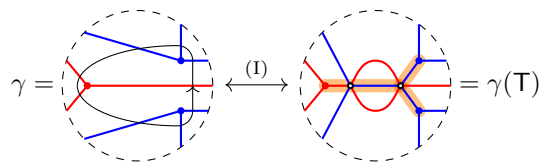
\[
\begin{tikzcd}
\cycle=%
=\cycle(\sfT)
\end{tikzcd}
\]

On the other hand, there might be a one-cycle having two different subgraph presentations as follows:
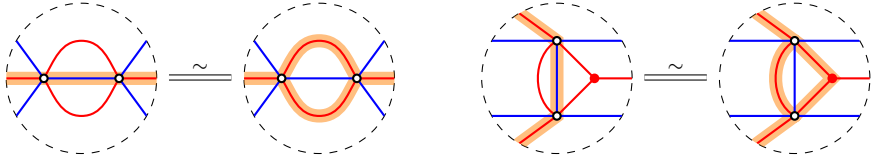
\[
\begin{tikzcd}
%
\end{tikzcd}
\]
Therefore, there is a bit subtle issue for picking up nice cycles in a consistent way.
\end{remark}

\begin{definition}\label{def:good cycle}
Let $\ngraph\subset \disk^2$ be an $N$-graph, and let $\Legendrian(\ngraph)$ be an induced Legendrian surface in~$J^1\disk^2$.
A (relative) cycle $[\cycle]$ in $H_1(\Legendrian(\ngraph))$ or $H_1(\Legendrian(\ngraph),\legendrian(\boundary \ngraph))$ is \emph{good} if $[\cycle]$ can be transformed to an (relative) $\sfI$-cycle in $H_1(\Legendrian(\ngraph'))$ or $H_1(\Legendrian(\ngraph'),\legendrian(\boundary \ngraph'))$ for some~$[\ngraph']=[\ngraph]$, respectively.
\end{definition}

\begin{example}\label{ex:I,Y,T-cycles}
The following cycles are good.
\begin{enumerate}
\item All (long) $\sfI$- and $\sfY$-cycles

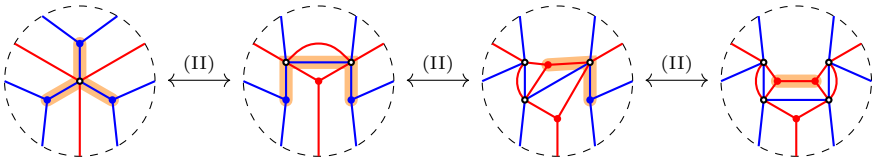
\[
\begin{tikzcd}[row sep=0pc]

\end{tikzcd}
\]

\item The cycle $\cycle(\sfT)$ for an admissible tree $\sfT$ without local configurations depicted in Figures~\ref{figure:loop near vertex2} and \ref{figure:loop near vertex3}
\end{enumerate}
\end{example}

\begin{definition}\label{def:equiv on N-graph and N-basis}
Let $(\ngraph, \nbasistilde)$ and $(\ngraph', \nbasistilde')$ be pairs of an $N$-graph and a set of good (relative) cycles.
We say that $(\ngraph, \nbasistilde)$ and $(\ngraph', \nbasistilde')$ are \emph{equivalent}
if $[\ngraph]=[\ngraph']$ and the induced isomorphism
\[
H_1(\Legendrian(\ngraph),\legendrian(\boundary\ngraph))\cong H_1(\Legendrian(\ngraph'),\legendrian(\boundary\ngraph'))
\]
identifies $\nbasistilde$ with $\nbasistilde'$.
We denote the equivalent class of $(\ngraph, \nbasistilde)$ by $[\ngraph, \nbasistilde]$.
\end{definition}

\subsection{Flag moduli spaces}\label{sec:flag moduli spaces}

We recall from \cite{STZ2017,CZ2020} central algebraic invariants $\cM(\legendrian)$ and $\mathcal{M}(\ngraph)$ of the Legendrians $\legendrian$ and $\Legendrian(\ngraph)$, respectively.
The main idea is to consider moduli spaces of constructible sheaves adapted to the Legendrians.

\subsubsection{Flag moduli spaces for Legendrian links}
Let $\legendrian=\legendrian_\beta$ be a Legendrian in $J^1\sphere^1$, then $\beta$ gives us an $(N-1)$-tuple of points $X=(X_1,\dots, X_{N-1})$ in $\sphere^1$.
Notice that we are not assuming that all $X_i$'s are disjoint.
\begin{remark}
By abuse of notation, we regard $\beta$ as a braid word on $\sphere^1$, where two generators can be occur simultaneously, such as $\sigma_{1,3}$.
In this case $X_1$ and $X_3$ is not disjoint.
\end{remark}
The position of the points in $X_i$ correspond to the position of $\sigma_i$ in $\beta$ for $i=1,\dots, N-1$. Let $\{f_k\}_{k\in K}$ be the set of closures of connected components of $\sphere^1\setminus X$.

In order to introduce a legible model for the constructible sheaves adapted to $\legendrian_\beta$, let us consider a full flag, i.e. a nested sequence of subspaces in $\bbC^N$:
\[
\cF^\bullet \in \{(\cF^i)_{i=0}^N \mid \dim \cF^i=i,\  \cF^j\subset \cF^{j+1}, 1\leq j\leq N-1,\  \cF^N=\bbC^N \}.
\]
Let $f_1,f_2$ be consecutive faces sharing a point $x$ in $X$. Then the corresponding flags $\cF^\bullet(f_1),\cF^\bullet(f_2)$ satisfy
\begin{align}\label{equation:flag conditions}
\cF^i(f_1)=\cF^i(f_2) \Longleftrightarrow x\not\in X_i.
\end{align}

The flags $\flags=\{\cF_\legendrian^\bullet(f_k)\}_{k\in K}$ in~$\bbC^N$ satisfying the conditions will be called simply by \emph{flags on~$\legendrian$}.
Let us denote the moduli space of such flags by $\widetilde\cM(\legendrian)$, then the general linear group $\operatorname{GL}_N(\bbC)$ acts on all flags at once.
The \emph{flag moduli space} for Legendrian link $\legendrian$ is defined by the quotient space
\[
\cM(\legendrian)\colonequals\widetilde{\cM}(\legendrian)/\operatorname{GL}_N(\bbC).
\]
It is well known that $\cM(\legendrian)$ is isomorphic to~$\Sh_\legendrian^1(\R^2)_0$ which is a Legendrian isotopy invariant, see \cite[Theorem 1.1]{STZ2017}.

\begin{remark}\label{rmk:non-trivial monodromy}
In general, we may consider moduli of flags $\{\cF_\legendrian^\bullet(f_k)\}_{k\in K}$ which possibly have non-trivial monodromy along the base $\sphere^1$.
In the current manuscript, however, we are only interested in flags which can be extended to $\disk^2$, which is adapted to the $N$-graphs. So it is enough to consider the current setup $\{\cF_\legendrian^N(f_j)=\bbC^N\}_{k\in K}$ of trivial monodromy along $\sphere^1$.
\end{remark}

\subsubsection{Flag moduli spaces for $N$-graphs}
Let $\ngraph=(\ngraph_1,\dots,\ngraph_{N-1})$ be an $N$-graph on $\disk^2$. Let $\{F_\ell \}_{\ell \in L}$ be a set of closures of connected components of $\disk^2\setminus \ngraph$, call each closure a \emph{face}. 
The \emph{framed flag moduli space} $\widetilde \cM(\ngraph)$ is a collection of \emph{flags} $\cF_{\Legendrian(\ngraph)}=\{\cF^\bullet(F_\ell )\}_{\ell \in L}$ in $\bbC^N$ satisfying the following: 
Let $F_1,F_2$ be a pair of faces sharing an edge $e$ in $\ngraph$. Then the corresponding flags $\cF^\bullet(F_1),\cF^\bullet(F_2)$ satisfy the condition
\[
\cF^i(F_1) = \cF^i(F_2)\Longleftrightarrow e\not\in \ngraph_i
\]
which is equivalent to the condition \eqref{equation:flag conditions}.

Then \emph{flag moduli space} for the $N$-graph $\ngraph$ is defined by
\[
\cM(\ngraph)\colonequals\widetilde{\cM}(\ngraph)/\operatorname{GL}_N(\bbC),
\] 
which is a stack in general.

Let $\Sh(\disk^2 \times \R)$ be the category of \emph{constructible sheaves} on $\disk^2\times \R$. Under the identification $J^1\disk^2\cong T^{\infty,-}(\disk^2\times \R)$, an $N$-graph $\ngraph\subset \disk^2$ gives a Legendrian 
\[
\Legendrian(\ngraph)\subset J^1 \disk^2 
\cong T^{\infty,-}(\disk^2\times \R) 
\subset T^\infty(\disk^2\times \R).
\]
This can be used to define a Legendrian isotopy invariant $\Sh_{\Legendrian(\ngraph)}^1(\disk^2 \times \R)_{0}$ of $\Sh(\disk^2 \times \R)$ consisting of constructible sheaves 
\begin{itemize}
\item whose singular support at infinity lies in $\Legendrian(\ngraph)
\subset T^\infty(\disk^2\times \R)$,
\item whose microlocal rank is one, and
\item which are zero near $\disk^2\times \{-\infty\}$.
\end{itemize} 
See \cite{CZ2020,GKS2012,STZ2017} for more details.

\begin{theorem}[{\cite[Theorem~5.3]{CZ2020}}]
The flag moduli space $\cM(\ngraph)$ is isomorphic to $\Sh_{\Legendrian(\ngraph)}^1(\disk^2\times\R)_0$. Hence $\cM(\ngraph)$ is a Legendrian isotopy invariant of $\Legendrian(\ngraph)$.\footnote{
Indeed, the actual theorem is about a connected surface, not only for $\disk^2$.}
\end{theorem}

We end this section by introducing a concept 
relating two moduli spaces $\cM(\ngraph)$ and $\cM(\boundary\ngraph)$ as follows.
For each $\ngraph$ on $\disk^2$, we have a canonical map $\cM(\ngraph)\to\cM(\boundary\ngraph)$ induced by the restriction map, which does not have to be injective or surjective.
However, since each internal edge in $\ngraph$ gives us an additional open condition, the image of $\cM(\ngraph)$ is open in $\cM(\boundary\ngraph)$ and therefore it is a birational equivalence if it is injective.

\begin{definition}[deterministic $N$-graphs]\label{def:good N-graph}
An $N$-graph $\ngraph$ on $\disk^2$ is said to be \emph{deterministic} if the induced map $\cM(\ngraph)\to\cM(\boundary\ngraph)$ between flag moduli spaces is a birational equivalence, or equivalently, the function field of $\cM(\ngraph)$ is canonically isomorphic to $\cM(\boundary\ngraph)$, i.e.,
\[
\bbC(\cM(\ngraph))\cong  \bbC(\cM(\boundary\ngraph)).
\]
\end{definition}

\begin{example}
Let us compare $\cM(\ngraph(a,b,c))$ and $\cM(\boundary\ngraph(a,b,c))$. Interior edges of $\ngraph(a,b,c)$ produce additional open conditions on flags adapted to $\boundary\ngraph(a,b,c)$. 
Moreover, each flag on $(\sphere^1,\boundary\ngraph(a,b,c))$ can be extended uniquely on $(\disk^2,\ngraph(a,b,c))$ if possible.
This implies that $\cM(\ngraph(a,b,c))$ can be seen as an open subset (indeed, an algebraic torus) of $\cM(\boundary\ngraph(a,b,c))$. So $\cM(\ngraph(a,b,c))$ and $\cM(\boundary\ngraph(a,b,c))$ are birationally equivalent.
A similar argument works for $\ngraph(\exdynD_n)$.
We have $\ngraph(a,b,c)$ and $\ngraph(\exdynD_n)$ are deterministic.
\end{example}

The following observations are obvious: let $\ngraph$ be an $N$-graph on $\disk^2$.
(i) if $\ngraph$ consists of trees, then it is deterministic; (ii) if an $N$-graph $\ngraph\subset \disk^2$ is deterministic and $[\ngraph]=[\ngraph']$, then so is $\ngraph'$.

\subsection{$Y$-seeds and Legendrian mutations}\label{sec:N-graphs and seeds}
Let $\ngraph\subset \disk^2$ be an $N$-graph, and let $\nbasis=\{[\cycle_1],\dots, [\cycle_n]\}$ be a set of good (absolute) cycles in $H_1(\Legendrian(\ngraph))$.
For two cycles $[\cycle_i]$ and $[\cycle_j]$, let $i([\cycle_i], [\cycle_j])$ be the algebraic intersection number.
In particular, if $\cycle_i$ is an $\sfI$-cycle $\cycle(\sfI(e))$ and $\cycle_j$ is a $\sfT$-cycle for some admissible subgraph $\sfT$, then 
\[
i([\cycle_i], [\cycle_j]) = \sum_{e'\in \sfT} i(e, e'),
\]
where $i(e,e')\in\{0,1,-1\}$ is defined as follows: 
\[
i(e,e') \colonequals
\begin{cases}
0 & \text{ if }e=e'\text{ or }e\cap e'=\varnothing;\\
1 & \text{ if }e' \text{ is lying on the left side of }e;\\
-1 & \text{ if }e' \text{ is lying on the right side of }e.
\end{cases}
\]
Geometrically, two representatives of $\cycle_i$ and $\cycle_j$ look locally as depicted in Figure~\ref{fig:I-cycle with orientation and intersections}. Their intersection $i([\cycle_i], [\cycle_j])$ is defined to be $+1$ by using the counterclockwise rotation convention of two tangent directions of cycles $\cycle_i$ and $\cycle_j$ at the intersection point as depicted in the third picture in Figure~\ref{fig:I-cycle with orientation and intersections}.
Note that our convention is opposite to the one in~\cite{CZ2020}.

\begin{figure}[ht]
\[
\def\arraycolsep{1pc}

\]
\caption{$\sfI$-cycles with intersections.}
\label{fig:I-cycle with orientation and intersections}
\end{figure}

\begin{definition} 
For each a pair $(\ngraph, \nbasis)$ of an $N$-graph and a set of good cycles, we define a quiver $\quiver=\quiver(\ngraph,\nbasis)$ as follows: let $n=\#(\nbasis)$.
\begin{enumerate}
\item the set of vertices is $[n]$, 
\item the $(i,j)$-entry $b_{i,j}$ for $\qbasispr(\quiver)=(b_{i,j})$ is the algebraic intersection number between $[\cycle_i]$ and~$[\cycle_j]$
\[
b_{i,j} = i([\cycle_i], [\cycle_j])\quad\text{for}\quad 1\le i,j\leq n.
\]
\end{enumerate}
\end{definition}

In order to assign a coefficient to each cycle, let us review  the microlocal monodromy functor from \cite{STZ2017}
\[
\mmon_\Legendrian:\Sh_\Legendrian^\bullet \to\Loc^\bullet(\Legendrian).
\]
In our case, this functor sends microlocal rank-one sheaves $\cF \in \cM(\ngraph)\cong \Sh_{\Legendrian(\ngraph)}^1(\disk^2\times \R)_0$, or equivalently, flags $\{\cF^\bullet(F_i)\}_{i \in I}\in\cM(\ngraph)$ to rank-one local systems $\mmon_{\Legendrian(\ngraph)}(\cF)$ on the Legendrian surface $\Legendrian(\ngraph)$. 
From now on, we regard flags $\cF$ as a formal parameter for the flag moduli space $\cM(\ngraph)$.
Then the coefficients in the coefficient tuple $\bfy$ for the pair $(\ngraph, \nbasis)$ are defined by
\begin{align}\label{eqn:y-varibles}
\bfy(\ngraph,\nbasis)=\left(
\mmon_{\Legendrian(\ngraph)}(-)([\cycle_1]),
\dots,
\mmon_{\Legendrian(\ngraph)}(-)([\cycle_n])\right),
\end{align}
where $\mmon_{\Legendrian(\ngraph)}(-)([\cycle_j]):\cM(\ngraph)\to \bbC$.  
Let us denote the above assignment by 
\[
\Psi(\ngraph, \nbasis)=(\bfy(\ngraph, \nbasis),\quiver(\ngraph,\nbasis)).
\]
By the Legendrian isotopy invariance of $\Sh_{\Legendrian(\ngraph)}^1(\disk^2\times \R)_0$ in \cite{GKS2012}, and the functorial property of the microlocal monodromy functor $\mmon$ \cite{STZ2017}, the assignment $\Psi$ is well-defined up to isotopy of~$\Legendrian(\ngraph)$. That is, if two pairs $(\Legendrian(\ngraph),\nbasis)$ and $(\Legendrian(\ngraph'),\nbasis')$ are Legendrian isotopic, or $[\ngraph, \nbasis]=[\ngraph',\nbasis']$ in particular, then they give us the same seed via $\Psi$.

\begin{theorem}\cite[\S7.2.1]{CZ2020}\label{thm:N-graph to seed}
Let $\ngraph\subset \disk^2$ be a $N$-graph with a tuple of cycles $\nbasis$ in $H_1(\Legendrian(\ngraph))$. Then the assignment $\Psi$ to a $Y$-seed in a cluster structure
\[
\Psi([\ngraph,\nbasis])= (\bfy(\ngraph,\nbasis),\quiver(\ngraph,\nbasis))
\]
is well-defined.
\end{theorem}

When an $N$-graph $\ngraph$ is deterministic, 
the coefficient tuple $\bfy$ originally defined on $\bbC[\cM(\ngraph)]$ can be restricted to the coordinate ring $\bbC[\cM(\legendrian)]$ of the moduli spaces of flags on~$\legendrian=\boundary\ngraph$, which is actually a $\cX$-cluster variety due to the result of Shen--Weng \cite[Theorem~1.1]{SW2019}.

The monodromy $\mmon_{\Legendrian(\ngraph)}(\cF)$ along a loop $[\gamma]\in H_1(\Legendrian(\ngraph))$ can be obtained by restricting the constructible sheaf $\cF$ to a tubular neighborhood of $\gamma$. 
Let us investigate how the monodromy can be computed explicitly in terms of flags $\{\cF^\bullet(F_i)\}_{i \in I}$.

Let us consider an $\sfI$-cycle $[\cycle]$ represented by a loop $\cycle(e)$ for some monochromatic edge $e$ as in Figure~\ref{figure:I-cycle with flags}.
Let us denote four flags corresponding to each region by $F_1,F_2,F_3,F_4$, respectively.
Suppose that $e \subset \ngraph_i$, then by the construction of flag moduli space $\cM(\ngraph)$, a two-dimensional vector space $V\colonequals\cF^{i+1}(F_*)/\cF^{i-1}(F_*)$ is independent of $*=1,2,3,4$. Moreover, $\cF^{i}(F_*)/\cF^{i-1}(F_*)$ defines a one-dimensional subspace $v_*\subset V$ for $*=1,2,3,4$, satisfying
\[
v_1\neq v_2 \neq v_3 \neq v_4 \neq v_1.
\]
Then $\mmon_{\Legendrian(\ngraph)}(\cF)$ along the one-cycle $[\gamma(e)]$ is defined by the cross ratio \[
\mmon_{\Legendrian(\ngraph)}(\cF)([\cycle])\colonequals\langle v_1,v_2,v_3,v_4 \rangle=\frac{v_1 \wedge v_2}{v_2 \wedge v_3}\cdot\frac{v_3 \wedge v_4}{v_4\wedge v_1}.
\]

Suppose that local flags $\{F_j\}_{j\in J}$ near the upper $\sfY$-cycle $[\cycle_U]$ look like in Figure~\ref{figure:Y-cycle with flags I}. 
Let $\ngraph_i$ and $\ngraph_{i+1}$ be the $N$-subgraphs in red and blue, respectively.
Then the $3$-dimensional vector space $V=\cF^{i+2}(F_*)/\cF^{i-1}(F_*)$ is independent of $*\in J$. Now regard $a,b,c$ and $A,B,C$ are subspaces of $V$ of dimension one and two, respectively. Then the microlocal monodromy along the $\sfY$-cycle~$[\cycle_U]$  becomes
\[
\mmon_{\Legendrian(\ngraph)}(\cF)([\cycle_U])\colonequals\frac{A(c)B(a)C(b)}{A(b)B(c)C(a)}.
\]
Here, $B(a)$ can be seen as a paring between a vector $v_a$ with $\langle  v_a \rangle=a$, and a covector $w_B$ with $\langle w_B \rangle=B^\perp$.

Now consider the lower $\sfY$-cycle $[\cycle_L]$ whose local flags given as in Figure~\ref{figure:Y-cycle with flags II}. We already have seen that the orientation convention of the loop in Figure~\ref{fig:I and Y cycle} for the upper and lower $\sfY$-cycle is different. Then microlocal monodromy along $[\cycle_L]$ follows the opposite orientation and becomes 
\[
\mmon_{\Legendrian(\ngraph)}(\cF)([\cycle_L])\colonequals\frac{A(b)B(c)C(a)}{A(c)B(a)C(b)}.
\]

\begin{figure}[ht]
\begin{tikzcd}
\subfigure[$\sfI$-cycle with flags.\label{figure:I-cycle with flags}]{
}
&
\subfigure[Upper $\sfY$-cycle with flags.\label{figure:Y-cycle with flags I}]{
%
}
&
\subfigure[Lower $\sfY$-cycle with flags.\label{figure:Y-cycle with flags II}]{
%
}
\end{tikzcd}
\caption{$\sfI$- and $\sfY$-cycles with flags. Here $ab$ means the span of $a$ and $b$, and $AB$ means the intersection of $A$ and $B$.}
\label{fig:I and Y cycle with flags}
\end{figure}

Let us define an operation called (\emph{Legendrian}) \emph{mutation} on $N$-graphs $\ngraph$ which corresponds to a geometric operation on the induced Legendrian surface $\Legendrian(\ngraph)$ that producing a smoothly isotopic but not necessarily Legendrian isotopic to $\Legendrian(\ngraph)$, see \cite[Definition~4.19]{CZ2020}.
Note that operation has an intimate relation with the wall-crossing phenomenon~\cite{Aur2007}, Lagrangian surgery \cite{Pol1991}, and quiver (or $Y$-seed) mutations \cite{FZ1_2002}.

\begin{definition}\cite{CZ2020}\label{def:legendrian mutation}
Let $\ngraph$ be a (local) $N$-graph and $e\in \ngraph_i\subset \ngraph$ be an edge between two trivalent vertices corresponding to an $\sfI$-cycle $[\cycle]=[\cycle(e)]$. The mutation $\mutation_\cycle(\ngraph)$ of $\ngraph$ along $\cycle$ is obtained by applying the local change depicted in Figure~\ref{figure:I-mutation}.
\end{definition}

\begin{figure}[ht]
\subfigure[Legendrian mutation along $\sfI$-cycle.\label{figure:I-mutation}]{
\makebox[0.45\textwidth]{

}}
\subfigure[Legendrian mutation along $\sfY$-cycle.\label{figure:Y-mutation}]{
\makebox[0.45\textwidth]{
%
}}
\subfigure[Legendrian local mutation along  $\sfI$-cycles\label{figure:local I-cycle}]{
\makebox[0.45\textwidth]{
%
}}
\subfigure[Legendrian mutation along degenerate $\sfI$-cycles\label{figure:degen I-cycle}]{
\makebox[0.45\textwidth]{
%
}}
\subfigure[Legendrian local mutation along degenerate $\sfI$-cycles\label{figure:local degen I-cycle}]{
\makebox[0.45\textwidth]{
%
}}
\subfigure[Legendrian local mutation along long $\sfI$-cycle\label{figure:local long I-cycle}]{
\makebox[0.45\textwidth]{
%
}}
\subfigure[Legendrian local mutation along degenerate long $\sfI$-cycles\label{figure:local degen long I-cycle}]{
\makebox[0.45\textwidth]{
%
}}
\subfigure[Example of the mutation on a long  $\sfI$-cycle\label{figure:mutation example long I-cycle}]{
\makebox[0.45\textwidth]{
%
}}
\caption{Legendrian (local) mutations at (degenerate, long) $\sfI$- and $\sfY$-cycles.}
\label{fig:Legendrian mutation on N-graphs}
\end{figure}

For the $\sfY$-cycle, the Legendrian mutation becomes as in the right of Figure~\ref{figure:Y-mutation}.
Note that the mutation at $\sfY$-cycle can be decomposed into a sequence of Move~\Move{I} and Move~\Move{II} together with a mutation at $\sfI$-cycle, see Example~\ref{ex:I,Y,T-cycles}.

One can easily verify Legendrian (local) mutations on degenerate $N$-graph shown in Figures~\ref{figure:degen I-cycle} and \ref{figure:local degen I-cycle} via perturbation.
For Figures~\ref{figure:local long I-cycle} and \ref{figure:local degen long I-cycle}, see Appendix~\ref{appendix:local mutations}.

\begin{remark}\label{rem:local mutation}
Note that Figures~\ref{figure:local degen I-cycle},\ref{figure:local long I-cycle}, and \ref{figure:local degen long I-cycle} depict the effect of the Legendrian mutation on a certain \emph{part} of $\sfT$-cycles. Since the boundaries of each side does not match, it seems not well-defined local operations at first glance. Indeed, they produce a well-defined operation when we apply each of the local operations to the (whole part of) $\sfT$-cycle. By combining Figure~\ref{figure:local I-cycle} and \ref{figure:local long I-cycle}, for example, we obtain Figure~\ref{figure:mutation example long I-cycle}.
\end{remark}

Let us remind our main purpose of finding exact embedded Lagrangian fillings for a Legendrian links.
The following lemma guarantees that Legendrian mutation preserves the embedding property of Lagrangian fillings.
\begin{proposition}\cite[Lemma~7.4]{CZ2020}
Let $\ngraph\subset \disk^2$ be a free $N$-graph. Then mutation $\mutation(\ngraph)$ at any $\sfI$- or $\sfY$-cycle is again free $N$-graph. 
\end{proposition}

An important observation is the Legendrian mutation on $(\ngraph,\nbasis)$ induces a $Y$-seed mutation on the induced seed $\Psi(\ngraph,\nbasis)$.

\begin{proposition}[{\cite[\S7.2]{CZ2020}}]\label{proposition:equivariance of mutations}
Let $\ngraph\subset \disk^2$ be a $N$-graph and $\nbasis$ be a set of good cycles in $H_1(\Legendrian(\ngraph))$.
Let $\mutation_{\cycle_i}(\ngraph,\nbasis)$ be a Legendrian mutation of $(\ngraph,\nbasis)$ along a one-cycle $\cycle_i$. Then
\[
\Psi(\mutation_{\cycle_i}(\ngraph,\nbasis))=\mutation_{i}(\Psi(\ngraph,\nbasis)).
\]
Here, $\mutation_{i}$ is the $Y$-seed mutation 
at the vertex $i$.
\end{proposition}

\begin{remark}\label{remark:boundary-Legendrian isotopy}
Let $\legendrian$ and $\legendrian'$ be two isotopic closures of positive $N$-braids.
By fixing an isotopy between them, we have an annular $N$-graph $\ngraph_{\legendrian \legendrian'}$ which induces a bijection between sets of $N$-graphs for $\legendrian$ and $\legendrian'$ by attaching $\ngraph_{\legendrian\legendrian'}$.
Then, indeed, this bijection is equivariant under the Legendrian mutation if it is defined,  that is, for $[\cycle] \in H_1(\Legendrian(\ngraph))$, 
\[
\mutation_{\cycle}(\ngraph_{\legendrian \legendrian'} \cdot \ngraph)
= \ngraph_{\legendrian \legendrian'} \cdot \mutation_{\cycle}(\ngraph).
\]
In other words, two $\boundary$-Legendrian isotopic $N$-graphs will generate equivariantly bijective sets of $N$-graphs under Legendrian mutations.
\end{remark}

\begin{remark}\label{remark:Stabilization}
Similarly, a stabilization $S(\ngraph)$ of $\ngraph$ will generate equivariantly bijective sets of $N$-graphs under Legendrian mutations as well since the stabilization part in $S(\ngraph)$ is away from chosen cycles and does not affect the Legendrian mutability.
\end{remark}

\begin{proposition}\label{prop_mutation_preserves_deterministic}
Let $\ngraph\subset \disk^2$ be a deterministic $N$-graph.
Then, for any $\sfI$- or $\sfY$-cycle~$\gamma$, the mutation~$\mutation_\gamma(\ngraph)$ is again a deterministic $N$-graph.
\end{proposition}

\begin{proof}
The proof is straightforward from the notion of the deterministic $N$-graph in Definition~\ref{def:good N-graph} and of the Legendrian mutation depicted in Figure~\ref{figure:I-mutation}.
Note that the Legendrian mutation~$\mutation_\gamma(\ngraph)$ at $\sfY$-cycle $\gamma$ is also deterministic, since $\mutation_\gamma(\ngraph)$ is a composition of Moves \Move{I} and \Move{II}, and a mutation at $\sfI$-cycle.
\end{proof}

\subsection{Relative cycles and a seed from a $Y$-seed}\label{section:relative cycles}

For a pair $(\ngraph,\nbasis)$, we have constructed a $Y$-seed $\Psi(\ngraph,\nbasis)$ in a $Y$-pattern. On the other hand, the pair corresponds to an exact Lagrangian filling $L=L(\ngraph)$ of $\legendrian=\boundary\ngraph$ which gives a toric chart $Loc(L)\cong\cM(\ngraph)$ in the corresponding $\cX$-cluster variety $\cM(\legendrian)$ by considering local systems on $L$. Unfortunately, distinct seeds in a $Y$-pattern may share the same cluster chart in the $\cX$-cluster variety, e.g. two $Y$-seeds in $\dynA_1$ cluster pattern. In the $\cA$-cluster structure, however, cluster charts in the $\cA$-cluster variety can be distinguished by seeds in the cluster pattern. This is suitable for our purpose of distinguishing Lagrangian fillings via cluster charts for distinct seeds.
Our strategy is to construct a seed in the cluster pattern (or $\cA$-cluster structure) from a $Y$-seed by considering additional relative cycles. 
See also \cite{CW2024} for the role of relative cycles and microlocal merodromy in the study of Lagrangian fillings of Legendrian links and cluster structure.

Let $\ngraph$ be an $N$-graph on $\disk^2$ and $\nbasistilde=\{[\gamma_1],\dots,[\gamma_m]\}$ be a set of good (relative) cycles in $H_1(\Legendrian(\ngraph),\legendrian(\boundary(\ngraph)))$.
We define the \emph{extended exchange matrix} $\qbasis=\qbasis(\ngraph,\nbasistilde)$ of size $m\times m$ by the algebraic intersection number among cycles in $\nbasistilde$ as before.

\begin{definition}\label{def:admissible relative cycle}
We call $\nbasistilde$ \emph{admissible} if 
\begin{enumerate}
\item $\nbasistilde$ is the union of $\nbasis=\{[\gamma_1],\dots,[\gamma_n]\}$ and $\nbasis_{rel}=\{[\gamma_{n+1}],\dots,[\gamma_m]\}$ consisting of absolute and relative cycles, respectively,
\item $\nbasis$ forms a basis of $H_1(\Legendrian(\ngraph))$,
\item $\nbasistilde$ forms a basis of $H_1(\Legendrian(\ngraph),\legendrian(\boundary(\ngraph)))$, and
\item $\qbasis$ has determinant $\pm1$.
\end{enumerate}
\end{definition}

Now let us introduce a new pair $(\tilde\ngraph,\tilde\nbasis)$.
Here $\tilde\ngraph$ is an $N$-graph on $\disk^2$ which is obtained from $\ngraph$ by padding an annular $N$-graph containing (trivalent) vertices for each point in $\gamma\cap\boundary\disk^2$ for all $[\gamma]\in \nbasis_{rel}$, see Figure~\ref{fig:N-graph with rel cycle and extended one}.
Note that the above annular $N$-graph gives a (fixed) Lagrangian cobordism from $\legendrian=\legendrian(\boundary \ngraph)$ to $\tilde\legendrian \colonequals \legendrian(\boundary\tilde\ngraph)$, let us denote it by $L_{\legendrian \tilde\legendrian}$.

\begin{figure}[ht]
\subfigure[An $N$-graph $\ngraph$ with relative cycle $\gamma_2$.\label{N-graph with relative cycle}]{\makebox[.4\textwidth]{
\begin{tikzpicture}
\draw [black, thick] (0,0) circle [radius=1.5];

\draw [color=cyclecolor1, line cap=round, line width=5, opacity=0.5] (-1/2,0) to (1/2,0);
\draw [color=black, line cap=round, line width=5, opacity=0.5] (1/2,0) to (30:1.5);
\draw [blue, thick] (150:1.5)--(-1/2,0);
\draw [blue, thick] (210:1.5)--(-1/2,0);
\draw [blue, thick] (30:1.5)--(1/2,0)  node[above left, midway] {$\gamma_2$};
\draw [blue, thick] (-30:1.5)--(1/2,0);
\draw [blue, thick] (-1/2,0)--(1/2,0) node[above, midway] {$\gamma_1$};

\clip (0,0) circle (1.5);

\draw[thick,blue,fill=blue] (-1/2,0) circle (0.05);
\draw[thick,blue,fill=blue] (1/2,0) circle (0.05);

\end{tikzpicture}
}}
\subfigure[The extended $N$-graph $\tilde \ngraph$ with corresponding cycles.\label{extended N-graph}]{\makebox[.4\textwidth]{
\begin{tikzpicture}
\draw [black, thick] (0,0) circle [radius=1.5];
\begin{scope}[yscale=0.7]
\draw [red, dashed, thick] (0,0) circle [radius=0.9];
\end{scope}
\draw [color=cyclecolor1, line cap=round, line width=5, opacity=0.5] (-1/2,0) to (1/2,0);
\draw [color=black, line cap=round, line width=5, opacity=0.5] (1/2,0) to (30:1);
\draw [blue, thick] (150:1.5)--(-1/2,0);
\draw [blue, thick] (210:1.5)--(-1/2,0);
\draw [blue, thick] (30:1)--(1/2,0)  node[above left, midway] {$\tilde{\gamma_2}$};
\draw [blue, thick] (30:1)--(20:1.5);
\draw [blue, thick] (30:1)--(40:1.5);
\draw [blue, thick] (-30:1.5)--(1/2,0);
\draw [blue, thick] (-1/2,0)--(1/2,0) node[above, midway] {$\tilde{\gamma_1}$};

\clip (0,0) circle (1.5);

\draw[thick,blue,fill=blue] (-1/2,0) circle (0.05);
\draw[thick,blue,fill=blue] (1/2,0) circle (0.05);
\draw[thick,blue,fill=blue] (30:1) circle (0.05);

\end{tikzpicture}
}}
\caption{An example $N$-graph with relative cycle and the induced extended $N$-graph with cycle. The inner side of dotted red line of $\tilde\ngraph$ can be identified with the $N$-graph $\ngraph$.}
\label{fig:N-graph with rel cycle and extended one}
\end{figure}

Note that there is a natural inclusion $i:\ngraph \to \tilde\ngraph$. For a (good) cycle $[\gamma]$ in $H_1(\Legendrian(\ngraph))$, we have a corresponding (good) cycle $[\tilde\gamma]=[i(\gamma)]$ in $H_1(\Legendrian(\tilde\ngraph))$.
When $[\gamma]$ is a good \emph{relative} cycle in $H_1(\Legendrian(\ngraph),\legendrian(\boundary\ngraph))$, we associate a good (absolute) cycle $[\tilde \gamma]$ in $H_1(\Legendrian(\tilde\ngraph))$ by using the trivalent vertex corresponding to $\gamma$ in the annular $N$-graph, see Figure~\ref{fig:N-graph with rel cycle and extended one}. 
By the same construction as in Section~\ref{sec:N-graphs and seeds}, we have the coefficient tuple $\tilde\bfy=(\tilde y_1,\dots,\tilde y_m)$ for the pair $(\tilde \ngraph, \tilde \nbasis)$ which are defined by measuring microlocal monodromies along the cycles
\[
\tilde y_i=\mmon_{\Legendrian(\tilde\ngraph)}(-)([\tilde\cycle_i]):\cM(\Legendrian(\tilde\ngraph))\to \bbC, \qquad i=1,\dots, m.
\]
In summary, we have constructed the pair $(\tilde\bfy,\qbasis)$ out of $(\tilde \ngraph,\tilde\nbasis)$. Denote this assignment by $\tilde\Psi$.

Let us produce new pairs from $(\tilde\ngraph,\tilde\nbasis)$ by applying Legendrian mutations \emph{only} along the cycles in $\nbasis$. We then obtain pairs $\{(\tilde\bfy_t,\qbasis_t)\}_{t\in \mathbb{T}_n}$ which are obviously corresponds to the $Y$-seeds $\{(\bfy_t,\qbasispr_t)\}_{t\in \mathbb{T}_n}$ of the cluster pattern.
Here, we have additional tuple $(\tilde y_{n+1},\dots, \tilde y_m)$, and regard it as an analogy the frozen variables (or formal variables) in the $Y$-seed.

Now we investigate the relation between flag moduli spaces $\cM(\ngraph)$, $\cM(\tilde\ngraph)$, $\cM(\legendrian)$, and $\cM(\tilde\legendrian)$. Note that the Legendrian link $\tilde\legendrian$ can be obtained from $\legendrian$ by adding $\#(\nbasis_{rel})$ positive crossings in suitable positions. Again by the work of Shen--Weng \cite[Theorem~1.1]{SW2019}, the moduli space $\cM(\tilde\legendrian)$ becomes an $\cX$-cluster variety.

Using the language of local system, see \cite{JT2017}, we have the following embeddings
\begin{align*}
i_{\ngraph}:\cM(\ngraph) \to \cM(\legendrian),&&
i_{\tilde\ngraph}:\cM(\tilde\ngraph) \to \cM(\tilde \legendrian).
\end{align*}
On the other hand, we have a (natural) restriction
\[
r_{\ngraph}:\cM(\tilde\ngraph) \to \cM(\ngraph).
\]
The variables $y_i\in \bbC[\cM(\ngraph)]$ and $\tilde y_i\in \bbC[\cM(\tilde\ngraph)]$ for $1\leq i \leq n$ can be identified via the restriction~$r_\ngraph$.

Now we collect $N$-graphs $\{\tilde\ngraph_t\}_{t\in \mathbb{T}_n}$ which are obtained from $\tilde\ngraph$ by a sequence of Legendrian mutations \emph{only} along the cycles in $\nbasis$.
Then consider an $\cX$-cluster \emph{sub}variety inside $\cM(\tilde\legendrian)$ which consists of cluster charts $\{\cM(\tilde\ngraph_t)\}_{t\in \mathbb{T}_n}$.
Let us denote the subvariety by $\cM(\tilde\legendrian)'$.
Note that the tuple $\tilde\bfy$ is originally in $\bbC[\cM(\tilde\ngraph)]$, and can be seen as in $\bbC[\cM(\tilde\legendrian)']$ the coordinate ring of the newly considered $\cX$-cluster variety.

For each relative cycle $[\gamma]$ in $\nbasis_{rel}$, we have the following local flag description in $\cM(\tilde\ngraph)$:
\begin{figure}[ht]
\begin{tikzpicture}
\draw[thick, black] (-1.5,1)--(1.5,1);
\draw [color=black, line cap=round, line width=5, opacity=0.5] (0,0.2)--(0,-1);
\draw[thick, blue] (-0.7,1)--(0,0.2) (0.7,1)--(0,0.2)--(0,-1);
\draw[thick, red, dashed] (-1.5,-0.2)--(1.5,-0.2);
\draw[dashed] (-1.5,1)--(-1.5,-1)--(1.5,-1)--(1.5,1);
\node at (-0.7,0.2) {$\cF_1$};
\node at (0.7,0.2) {$\cF_2$};
\node at (-0.7,-0.7) {$\cF_1$};
\node at (0.7,-0.7) {$\cF_2$};
\node at (0,0.7) {$\cF_3$};
\draw[thick,blue,fill=blue] (0,0.2) circle (0.05);
\end{tikzpicture}
\caption{Local flag configuration near a relative cycle}
\label{figure:doubling}
\end{figure}

\noindent Note that the regions below the dotted red line are the ones in the original $N$-graph $\ngraph$.
Even though $\cF_i$, $i=1,2,3$ are flags in $\bbC^N$ in general, we may assume that they are flags in $\bbC^2$ by modding out identical vector spaces in $\cF_i$, $i=1,2,3$.
Every cluster charts $\{\cM(\tilde\ngraph_t)\}_{t\in \mathbb{T}_n}$ should satisfy $\cF_1\neq \cF_2$ for each relative cycle.
Since $\cM(\tilde\legendrian)'$ is the gluing of such charts, the same hold for $\cM(\tilde\legendrian)'$. This condition allows us to define a restriction 
\begin{align*}
r_{\legendrian}:\cM(\tilde\legendrian)' \to \cM(\legendrian).
\end{align*}
Note that there is no natural restriction from $\cM(\tilde\legendrian)$ to $\cM(\legendrian)$.
One can easily check that the following diagram commutes:
\[
\begin{tikzcd}
\cM(\ngraph) \arrow[r, "i_{\ngraph}"] & \cM(\legendrian) \\
\cM(\tilde\ngraph)  \arrow[u, "r_{\ngraph}"'] \arrow[r,"i_{\tilde\ngraph}"]& \cM(\tilde\legendrian)' \arrow[u, "r_{\legendrian}"]
\end{tikzcd}
\]

Now we are ready to prove that there are at least as many Lagrangian fillings for Legendrian link as seeds in the cluster structure.
\begin{proposition}\label{prop:distinct seeds imples distinct fillings}
	Let $(\ngraph, \nbasis)$ and $(\ngraph',\nbasis')$ be pairs of free and deterministic $N$-graphs on $\disk^2$ and sets of good 1-cycles such that 
	\begin{enumerate}
		\item $\boundary\ngraph = \boundary \ngraph'$, 
		\item $\nbasis$ and $\nbasis'$ can be extended to admissible sets of good cycles, and
		\item there exists a sequence of Legendrian mutations from $(\ngraph,\nbasis)$ to $(\ngraph', \nbasis')$. 
	\end{enumerate}
If $(\ngraph, \nbasis)$ and $(\ngraph',\nbasis')$ define different $Y$-seeds via $\Psi$, then there is no exact Lagrangian isotopy between two Lagrangian fillings $L(\ngraph)$ and $L(\ngraph')$ of a Legendrian link $\legendrian(\boundary\ngraph)=\legendrian(\boundary\ngraph')$.
\end{proposition}

\begin{proof}
We first notice that $\Psi(\ngraph, \nbasis)\neq\Psi(\ngraph',\nbasis')$ implies that $\tilde\Psi(\tilde\ngraph,\nbasistilde)\neq \tilde\Psi(\tilde\ngraph',\nbasistilde')$. Since $\nbasistilde$ is of full rank, by Proposition~\ref{prop_bijection_between_A_seed_and_tori}, the cluster charts $\cM(\tilde\ngraph)$ and $\cM(\tilde\ngraph')$ are different in $\cM(\tilde\legendrian)$ and so is in $\cM(\tilde\legendrian)'$ by Corollary~\ref{cor_Y-seeds_and_Y-charts}.

Assume on the contrary that there is an exact Lagrangian isotopy between $L(\ngraph)$ and $L(\ngraph')$ then so is between $L(\tilde\ngraph)$ and $L(\tilde\ngraph')$.
Then, by \cite{JT2017}, the toric charts $\Loc^1(L(\tilde\ngraph))\cong\cM(\tilde\ngraph)$ and $\Loc^1(L(\tilde\ngraph'))\cong\cM(\tilde\ngraph')$ have the same images in $\cM(\tilde\legendrian)$ under $i_{\tilde\ngraph}$ and $i_{\tilde\ngraph'}$,
which yields a contradiction.
\end{proof}

\section{Lagrangian fillings for Legendrian links of finite or affine type}\label{sec:N-graph of finite or affine type}

Let $\legendrian\subset J^1\sphere^1$ be a Legendrian knot or link which is a closure of a positive braid and bounds a Legendrian surface~$\Legendrian(\ngraph)$ in $J^1\disk^2$ for some free $N$-graph $\ngraph$.
We fix a set $\nbasis$ of good cycles in the sense of Definition~\ref{def:good cycle}.
Then, by Theorem~\ref{thm:N-graph to seed}, we obtain a $Y$-seed $\Psi(\ngraph,\nbasis)$ which is a pair of a coefficient tuple~$\bfy(\Legendrian(\ngraph),\nbasis)$ 
and a quiver $\quiver(\Legendrian(\ngraph),\nbasis)$. 

We say that the pair $(\ngraph,\nbasis)$ is \emph{of finite type} or \emph{of
infinite type} if so is the cluster algebra defined by
$\quiver(\Legendrian(\ngraph),\nbasis)$.
Similarly, it is said to be \emph{of type $\dynX$} for some Dynkin diagram $\dynX$ if so is the associated cluster algebra.
In particular, it is said to be \emph{of type~$\dynADE$} or \emph{of affine type} if the quiver is of type $\dynADE$ or of affine type. See Definition~\ref{def_quiver_of_type_X}.

\subsection{\texorpdfstring{$N$-graphs}{N-graphs} of finite or affine types}

In \cite{GSW2020b}, Gao, Shen, and Weng describe a procedure as follows. Starting from a (positive) braid word, they associate a so-called brick diagram, which includes the data of a quiver. The reader is encouraged to see their paper for more details. We will not define these notions here, but will sketch their result for a number of examples which yield quivers of finite and affine type.  We will prefer to use mutation-equivalent models for our purposes, as they are more amenable to performing the desired Legendrian mutations, though we include the Gao--Shen--Weng procedure to indicate that finding models for a given Dynkin type $\dynX$ is essentially algorithmic using their work.

\begin{remark}
In this section, we will define braids $\beta$ and $\tilde\beta$ of each type, where $\tilde\beta$ is obtained from $\beta$ by doubling chosen generators so that the closure of $\tilde\beta$ has only one component. The chosen generator of $\beta$ will be decorated by the box. For example, if $\beta=\sigma_1^n \fbox{$\sigma_1$}$, then $\beta=\sigma_1^{n+1}$ and $\tilde\beta$ is either $\sigma_1^{n+1}$ or $\sigma_1^{n+2}$.
\end{remark}

\subsubsection{Linear and tripod \texorpdfstring{$N$-graphs}{N-graphs}}\label{sec:linear}
For $n\ge 1$ and a triple $(a,b,c)$ with $a,b,c\ge 1$, let us define positive braids $\beta_0(\dynA_n), \beta(\dynA_n), \beta_0(a,b,c)$ and $\beta(a,b,c)$ as follows:
\begin{align*}
\beta_0(\dynA_n)\colonequals& \sigma_1^{n}\fbox{$\sigma_1$},&
\beta_0(a,b,c)\colonequals& \sigma_2\fbox{$\sigma_1$}\sigma_1^{a-1}\sigma_2^{b-1}\sigma_1^{c-1}\fbox{$\sigma_1$}\\
\beta(\dynA_n)\colonequals& \Delta_2\beta_0(\dynA_n)\Delta_2 \mathrel{\dot=} \sigma_1^{n+1}\fbox{$\sigma_1$} \sigma_1&
\beta(a,b,c)\colonequals& \Delta_3\beta_0(\dynA_n)\Delta_3.
\end{align*}
Then the braids $\tilde\beta_0(\dynA_n)$ and $\tilde\beta_0(a,b,c)$ are exactly the same as 
$\tilde\beta_0(\dynA_n)\colonequals\beta_0(\dynA_{n+\epsilon})$ and $\tilde\beta_0(a,b,c)\colonequals\beta_0(a+\epsilon_1,b,c+\epsilon_2)$,
where $\epsilon$ and $\epsilon_i$ are either $0$ or $1$ such that $n+\epsilon$ is odd, and exactly two of $a+\epsilon_1,b$, and $c+\epsilon_2$ are odd.

We define $\legendrian(\dynA_n)$, $\tilde\legendrian(\dynA_n)$, $\legendrian(a,b,c)$, and $\tilde\legendrian(a,b,c)$ as the rainbow closures of $\beta_0(\dynA_n)$, $\tilde\beta_0(\dynA_n)$, $\beta_0(a,b,c)$, and $\tilde\beta_0(a,b,c)$, or equivalently, the $(-1)$-closures of $\beta(\dynA_n)$ and $\beta(a,b,c)$, respectively.
\begin{align*}
\legendrian(\dynA_n)&=
$ is equivalent to $%
$.

One can easily check the quivers $\quiver^\brick(\dynA_n)$ and $\quiver^\brick(a,b,c)$ from \emph{brick diagrams} of $\beta_0(\dynA_n)$ and $\beta_0(a,b,c)$ described in \cite{GSW2020b} look as follows:
\begin{align*}
\quiver^\brick(\dynA_n)&=%
\end{align*}

Then there are canonical $N$-graphs $(\ngraph^\brick(\dynA_n),\nbasistilde^\brick(\dynA_n))$ and 
$(\ngraph^\brick(a,b,c),\nbasistilde^\brick(a,b,c))$ on $\disk^2$ with (relative) cycles as shown in Figure~\ref{figure:brick linear and tripod N-graphs} such that 
\begin{align*}
\quiver^\brick(\dynA_n)&=
\quiver(\ngraph^\brick(\dynA_n), \nbasistilde^\brick(\dynA_n)),&
\quiver^\brick(a,b,c)&=
\quiver(\ngraph^\brick(a,b,c), \nbasistilde^\brick(a,b,c)).
\end{align*}

\begin{figure}[ht]
\begin{align*}
(\ngraph^\brick(\dynA_n),\nbasistilde^\brick(\dynA_n))&=

\end{align*}
\caption{Brick linear and tripod $N$-graphs with (relative) cycles}
\label{figure:brick linear and tripod N-graphs}
\end{figure}

The colors on cycles in Figure~\ref{figure:brick linear and tripod N-graphs} are nothing to do with the bipartite coloring, but we define $N$-graphs with bipartite coloring, which is equivalent to the original brick $N$-graphs and will play the roles of the initial seeds.
Throughout this section, relative cycles are indicated in gray.

\begin{definition}[Linear and tripod $N$-graphs]
For $n\ge 1$, the \emph{linear $N$-graph} $(\ngraph(\dynA_n), \nbasistilde(\dynA_n))$ is the $2$-graph on $\disk^2$ depicted in Figure~\ref{figure:linear N-graph}.

For $a,b,c\ge 1$, the \emph{tripod $N$-graph} $(\ngraph(a,b,c), \nbasistilde(a,b,c))$ is a free $3$-graph on $\disk^2$ depicted in Figure~\ref{figure:tripod N-graph}.
\end{definition}

\begin{figure}[ht]
\subfigure[$2$-graph $(\ngraph(\dynA_n),\nbasistilde(\dynA_n))$\label{figure:linear N-graph}]{\makebox[0.45\textwidth]{

}}
\subfigure[Bipartite linear quivers $\quiver(\dynA_n)$\label{figure:linear quiver}]{\makebox[0.45\textwidth]{
%
}}
\subfigure[Bipartite tripod quiver $\quiver(a,b,c)$\label{figure:tripod quiver}]{\makebox[0.45\textwidth]{
%
}}
\caption{Bipartite linear and tripod $N$-graphs with chosen cycles and their quivers}
\end{figure}

\begin{lemma}\label{lemma:full rank}
For each $n\ge 1$ and a triple $(a,b,c)$, both $N$-graphs $(\ngraph^\brick(\dynA_n),\nbasistilde^\brick(\dynA_n))$ and $(\ngraph^\brick(a,b,c),\nbasistilde^\brick(a,b,c))$ are free, deterministic and equivalent to $(\ngraph(\dynA_n),\nbasistilde(\dynA_n))$ and $(\ngraph(a,b,c),\nbasistilde(a,b,c))$ up to $\boundary$-Legendrian isotopy and mutations, respectively, and their quivers are the same as shown in Figures~\ref{figure:linear quiver} and~\ref{figure:tripod quiver}.
\begin{align*}
\quiver(\dynA_n)&=\quiver(\ngraph(\dynA_n), \nbasistilde(\dynA_n)),&
\quiver(a,b,c)&=\quiver(\ngraph(a,b,c), \nbasistilde(a,b,c)).
\end{align*}
\end{lemma}
\begin{proof}
For $\dynA_n$, this is trivial.

For a triple $(a,b,c)$, the freeness and deterministicity of $\ngraph(a,b,c)$ follows from Lemma~\ref{lemma:tree Ngraphs are free}.
Since $\legendrian(a,b,c)$ is the $(-1)$-closure of $\beta(a,b,c)$, we need to check that $\legendrian(a,b,c)$ and $\boundary\ngraph(a,b,c)$ are equivalent in $J^1\sphere^1$.
Indeed, 
\begin{align*}
\beta(a,b,c)&
\mathrel{\dot=}\sigma_2\fbox{$\sigma_1$}\sigma_1^{a-1}\Delta_3\sigma_1^{b-1}\Delta_3\sigma_1^{c-1}\fbox{$\sigma_1$},
\end{align*}
whose the $(-1)$-closure is the same as $\boundary\ngraph(a,b,c)$.
Here $\Delta_N$ is the half-twist braid of $N$-strands.

The $N$-graph equivalence for $\ngraph(a,b,c)$ will be given in Appendix~\ref{appendix:Ngraph of type abc} and it is easy to check that the quiver $\quiver^\brick(a,b,c)$ is mutation equivalent to $\quiver(a,b,c)$.
\end{proof}

Note that the linear and tripod $N$-graphs have certain symmetries as follows:
\begin{lemma}[Rotational symmetries]
By ignoring relative cycles,
\begin{enumerate}
\item the $N$-graph $(\ngraph(\dynA_n),\nbasis(\dynA_n))$ with cycles is invariant under $\pi$-rotation for odd $n\ge 1$, and 
\item the $N$-graph $(\ngraph(a,a,a),\nbasis(a,a,a))$ with cycles is invariant under $2\pi/3$-rotation for each $a\ge 1$.
\end{enumerate}
\end{lemma}

\begin{lemma}[Conjugation symmetries]
The $N$-graph $(\ngraph(\dynA_n),\nbasistilde(\dynA_n))$ with cycles is invariant under the conjugation.
\end{lemma}

For any triple $(a,b,c)$, the $N$-graph $\ngraph(a,b,c)$ is never invariant under conjugation, which acts on the Legendrian $\legendrian(a,b,c)$ as interchanging $\sigma_1$ and $\sigma_2$ so that $\overline{\legendrian(a,b,c)}$ is the rainbow closure of $\overline{\beta_0(a,b,c)}=\sigma_1\sigma_2^a\sigma_1^{b-1}\sigma_2^c$.
The $N$-graph $\overline{(\ngraph(a,b,c),\nbasistilde(a,b,c))}$ corresponding to $\overline{\legendrian(a,b,c)}$ is depicted below.
\[
\overline{(\ngraph(a,b,c),\nbasistilde(a,b,c))}\colonequals
\begin{tikzpicture}[baseline=-.5ex,xscale=0.6,yscale=0.6]
\useasboundingbox(-4,-3.5)rectangle(4,3.5);
\draw[thick] (0,0) circle (3cm);
\begin{scope}
\clip (0,0) circle (3);
\draw[color=gray, line cap=round, line width=5, opacity=5] (50:2) -- (60:3);
\draw[color=gray, line cap=round, line width=5, opacity=5] (290:2) -- (300:3);
\end{scope}
\draw[color=cyclecolor2, line cap=round, line width=5, opacity=0.5] (60:1) -- (50:1.5) (70:1.75) -- (50:2) (180:1) -- (170:1.5) (190:1.75) -- (170:2) (300:1) -- (290:1.5) (310:1.75) -- (290:2);
\draw[color=cyclecolor1, line cap=round, line width=5, opacity=0.5] (0,0) -- (60:1) (0,0) -- (180:1) (0,0) -- (300:1) (50:1.5) -- (70:1.75) (170:1.5) -- (190:1.75) (290:1.5) -- (310:1.75);
\draw[blue, thick] (0,0) -- (0:3) (0,0) -- (120:3) (0,0) -- (240:3);
\draw[red, thick, fill] (0,0) -- (60:1) circle (2pt) -- (100:3) (60:1) -- (50:1.5) circle (2pt) -- (20:3) (50:1.5) -- (70:1.75) circle (2pt) -- (80:3) (70:1.75) -- (50:2) circle (2pt) -- (40:3);
\draw[red, thick, dashed] (50:2) -- (60:3);
\draw[red, thick, fill] (0,0) -- (180:1) circle (2pt) -- (220:3) (180:1) -- (170:1.5) circle (2pt) -- (140:3) (170:1.5) -- (190:1.75) circle (2pt) -- (200:3) (190:1.75) -- (170:2) circle (2pt) -- (160:3);
\draw[red, thick, dashed] (170:2) -- (180:3);
\draw[red, thick, fill] (0,0) -- (300:1) circle (2pt) -- (340:3) (300:1) -- (290:1.5) circle (2pt) -- (260:3) (290:1.5) -- (310:1.75) circle (2pt) -- (320:3) (310:1.75) -- (290:2) circle (2pt) -- (280:3);
\draw[red, thick, dashed] (290:2) -- (300:3);
\draw[thick, fill=white] (0,0) circle (2pt);
\curlybrace[]{10}{110}{3.2};
\draw (-60:3.5) node[rotate=30] {$c+1$};
\curlybrace[]{130}{230}{3.2};
\draw (180:3.5) node[rotate=90] {$b+1$};
\curlybrace[]{250}{350}{3.2};
\draw (-300:3.5) node[rotate=-30] {$a+1$};
\end{tikzpicture}
\]

On the other hand, if one of $a,b,c$ is $1$, then the quiver $\quiver(a,b,c)$ is of type $\dynA_n$.  
As seen in Example~\ref{example:stabilization of An}, the Legendrian link $\legendrian(1,b,c)$ is a stabilization of $\legendrian(\dynA_n)$ for $n=b+c-1$. Indeed, the $N$-graph $\ngraph(1,b,c)$ is a stabilization of $\ngraph(\dynA_n)$.
See Appendix~\ref{appendix:tripod with a=1 is of type An} for the proof.
\begin{lemma}\label{lemma:stabilized An}
The $N$-graph $\ngraph(1,b,c)$ is a stabilization of $\ngraph(\dynA_n)$ for $n=b+c-1$.
\end{lemma}
One consequence of this lemma is that two $N$-graphs $\ngraph(\dynA_n)$ and $\ngraph(1,b,c)$ with $n=b+c-1$ will generate bijective sets of $N$-graphs under mutations as seen in Remarks~\ref{remark:boundary-Legendrian isotopy} and \ref{remark:Stabilization}, where the bijection preserves the mutation.

Notice that the quivers $\quiver(a,b,c)$ together with $\quiver(\dynA_n)$ cover all quivers of finite type and some quivers of affine type. Indeed, for $1\le a\le b\le c$ and $n=a+b+c-2$, the quivers $\quiver(1,b,c)$ and $\quiver(\dynA_n)$ are of type $\dynA_n$, and the quivers $\quiver(2,2,n-2)$ and $\quiver(2,3,m-3)$ are of type $\dynD_n$ and $\dynE_m$.
Moreover, $\quiver(3,3,3)$, $\quiver(2,4,4)$ and $\quiver(2,3,6)$ are of type $\exdynE_6, \exdynE_7$ and $\exdynE_8$, respectively.
Hence we denote Legendrians, quivers, $N$-graphs, and so on by using $\dynX$ for $\dynX=\dynD_n, \dynE_n$ or $\exdynE_n$ instead of the triple $(a,b,c)$ corresponding to $\dynX$ as seen in Table~\ref{table:short notations}.

\begin{table}[ht]
\[
\renewcommand\arraystretch{1.5}
\begin{array}{c||c|c|c|c|c|c|c|c}
\toprule
\dynX & \dynA_n & \dynD_n & \dynE_6 & \dynE_7 & \dynE_8 & \exdynE_6 & \exdynE_7 & \exdynE_8\\
\midrule
(a,b,c) & (1,b,c) & (2,2,n-2) & (2,3,3) & (2,3,4) & (2,3,5) & (3,3,3) & (2,4,4) & (2,3,6)\\
\bottomrule
\end{array}
\]
\caption{Triples $(a,b,c)$ of type $\dynA\dynD\dynE$ and $\exdynE$}
\label{table:short notations}
\end{table}

\subsubsection{Degenerate $N$-graphs}
For $p, q, r\ge 1$, we define the positive $4$-braids $\beta_{\degen,0}(p,q,r)$ and $\beta_\degen(p,q,r)$ as
\begin{align*}
\beta_{\degen,0}(p,q,r)&\colonequals \sigma_{1,3}\sigma_2^p\sigma_{1,3}^{q-2}\fbox{$\sigma_{1,3}$}\sigma_2^{r-1}\fbox{$\sigma_2$}&
\beta_\degen(p,q,r)&\colonequals\Delta_4\beta_{\degen,0}(p,q,r)\Delta_4.
\end{align*}
Then $\tilde\beta_{\degen,0}(p,q,r)$ is the positive braid $\beta_{\degen,0}(p,q',r')$ for some $q\le q'\le q+1$ and $r\le r'\le r+1$ such that both $q'$ and $p+r'$ are odd.

Let $\legendrian_\degen(p,q,r)$ be the rainbow closures of $\beta_{\degen,0}(p,q,r)$, or equivalently, the $(-1)$-closure of $\beta_{\degen}(p,q,r)$. Then we denote its brick quiver and canonical $N$-graph with cycles by $\quiver^\brick(\beta_{\degen,0}(p,q,r))$ and $(\ngraph^\brick_\degen(p,q,r), \tilde\nbasis^\brick_\degen(p,q,r))$ as before. See Figure~\ref{figure:degenerate brick}.

\begin{figure}[ht]
\subfigure[Legendrian link $\legendrian_\degen(p,q,r)$]{$
\begin{aligned}
\legendrian_\degen(p,q,r)&=

\end{aligned}
$}
\subfigure[Brick quiver for $\beta_{\degen,0}(p,q,r)$]{$
\begin{aligned}
\quiver^\brick(\beta_{\degen,0}(p,q,r)))&=%
\end{aligned}
$}
\subfigure[Brick $N$-graph and its perturbation]{$
\begin{aligned}
(\ngraph^\brick_\degen(p,q,r),\nbasistilde^\brick_\degen(p,q,r))&=
%
\end{aligned}
$}
\caption{Brick quiver and $N$-graph with cycles and its perturbation for the Legendrian link $\legendrian_{\degen}(p,q,r)$}
\label{figure:degenerate brick}
\end{figure}

\begin{definition}[Degenerate $4$-graphs for $\legendrian_\degen(p,q,r)$]
We define a degenerate $4$-graph $\ngraph_\degen(p,q,r)$ for $\legendrian_\degen(p,q,r)$ as depicted in the left of Figure~\ref{figure:degenerated 4-graph}.
\end{definition}

\begin{figure}[ht]
\[
%
\]
\caption{Degenerate $4$-graphs $(\ngraph_\degen(p,q,r),\nbasistilde_\degen(p,q,r))$ and cycles in the perturbation}
\label{figure:degenerated 4-graph}
\end{figure}

\begin{lemma}\label{lemma:tripod to degenerated Ngraph}
The pairs $(\ngraph_\degen(p,q,r),\nbasistilde_\degen(p,q,r))$ and $(\ngraph^\brick_\degen(p,q,r),\nbasistilde^\brick_\degen(p,q,r))$ are equivalent up to $\boundary$-Legendrian isotopy and Legendrian mutations.
\end{lemma}
\begin{proof}
We first show that $\legendrian_{\degen}(p,q,r)$ is the same as $\boundary \ngraph_\degen(p,q,r)$ as follows:
\begin{align*}
\beta_\degen(p,q,r)
&\mathrel{\dot{=}}\sigma_2^p(\sigma_2\sigma_{1,3}\sigma_2\sigma_{1,3})\sigma_{1,3}^{q-1}\sigma_2^r(\sigma_2\sigma_{1,3}\sigma_2\sigma_{1,3})\sigma_{1,3}
=\sigma_2^{p+1}\sigma_{1,3}\sigma_2\sigma_{1,3}^q\sigma_2^{r+1}\sigma_{1,3}\sigma_2\sigma_{1,3}^2,
\end{align*}
whose the $(-1)$-closure is the same as $\boundary\ngraph_\degen(p,q,r)$.

It is straightforward to check that we obtain the following degenerate $N$-graph from $\ngraph^\brick_\degen(p,q,r)$ by applying a sequence of Move \Move{DI} to the left part of the figure and Move \Move{DII} to the right part.
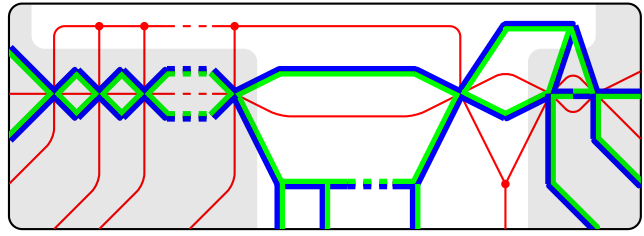
\[

\]
Let us ignore the shaded regions whose union is tame under perturbation, see \S~\ref{section:annular Ngraphs}, then it is obvious that the resulting $N$-graph together with a set of one cycles become $(\ngraph_\degen(p,q,r),\nbasistilde_\degen(p,q,r))$ in Figure~\ref{figure:degenerated 4-graph} after a sequence of Legendrian mutations.
\end{proof}

The following observation is obvious since all of $\beta_{\degen}(p,q,r)$ and $\legendrian_\degen(p,q,r)$ are invariant under conjugation, so is the pair $(\ngraph_\degen(p,q,r),\nbasistilde_\degen(p,q,r))$.
\begin{lemma}
The degenerate $N$-graph $(\ngraph_\degen(p,q,r),\nbasistilde_\degen(p,q,r))$ with cycles is invariant under conjugation.
\end{lemma}

Note that the $4$-graph $\ngraph_\degen(p,q,1)$ is indeed a stablization of the tripod $3$-graph $\ngraph(p,q,q)$ up to $\boundary$-Legendrian isotopy and Legendrian mutations.
In particular, when $(p,q,r) = (n-2,2,1), (2,3,1), (3,3,1)$ and $(2,4,1)$, we denote the braid, their closures and $N$-graphs by $\beta_{\degen}(\dynX)$, $\legendrian_\degen(\dynX)$ and $\ngraph_\degen(\dynX)$ for $\dynX=\dynD_n, \dynE_6, \exdynE_6$ and $\exdynE_7$, respectively. The degenerate $N$-graphs and the perturbed $N$-graphs with cycles listed above are depicted in Table~\ref{table:degenerated 4-graphs}.

\begin{table}[ht]
\renewcommand{\arraystretch}{1.5}
\begin{tabular}{c||c|c|c|c}
\toprule
$\dynX$ & $\dynD_n$ & $\dynE_6$ & $\exdynE_6$ & $\exdynE_7$\\
\midrule
\hline
$\beta_\degen(\dynX)$ & $\beta_\degen(n-2,2,1)$ & $\beta_\degen(2,3,1)$ & $\beta_\degen(3,3,1)$ & $\beta_\degen(2,4,1)$ \\
\hline
$\legendrian_\degen(\dynX)$ & $\legendrian_\degen(n-2,2,1)$ & $\legendrian_\degen(2,3,1)$ & $\legendrian_\degen(3,3,1)$ & $\legendrian_\degen(2,4,1)$ \\
\bottomrule
\end{tabular} 
\caption{Degenerate $4$-graphs $\ngraph_\degen(\dynX)$ and cycles in the perturbations for $\dynX=\dynD_n, \dynE_6, \exdynE_6$ and $\exdynE_7$}
\label{table:degenerated 4-graphs}
\end{table}

\begin{remark}\label{remark:degenerated Ngraph of type A}
As observed in Lemma~\ref{lemma:stabilized An}, one can think $\quiver(1,n,n)$ and $\ngraph(1,n,n)$ for $\dynA_{2n-1}$ instead of $\quiver(\dynA_{2n-1})$ and $\ngraph(\dynA_{2n-1})$. Therefore we may obtain a degenerate $N$-graph~$\ngraph_\degen(\dynA_{2n-1})=\ngraph_\degen(1,n,1)$, which is obviously invariant under the conjugation.
\end{remark}

We also consider the degenerate $4$-graph $(\ngraph_\degen(\exdynD_4), \nbasistilde_\degen(\exdynD_4))=(\ngraph_\degen(2,2,2), \nbasistilde_\degen(2,2,2))$ with cycles of type $\exdynD_4$ as follows:
\[
(\ngraph_\degen(\exdynD_4), \nbasistilde_\degen(\exdynD_4))=

\]
which defines a bipartite quiver of type $\exdynD_4$.

\subsubsection{\texorpdfstring{$N$-graphs}{N-graphs} of type \texorpdfstring{$\exdynD_n$}{affine Dn}}
Let us start by defining the Legendrian link $\legendrian(\exdynD_n)$ of type $\exdynD_n$
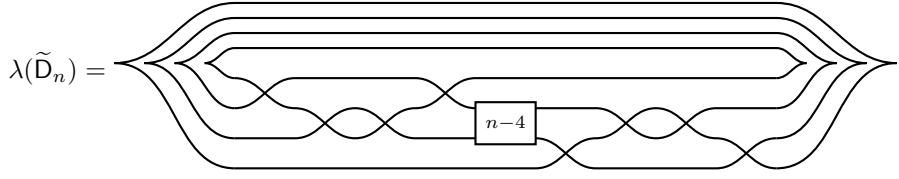
\begin{align*}
\legendrian({\exdynD}_n)&=
%
\end{align*}
Then $\legendrian(\exdynD_n)$ is the rainbow closure of the positive braid $\beta_0(\exdynD_n)$ 
\begin{align*}
\beta_0({\exdynD}_n)&=\sigma_3\fbox{$\sigma_2$}\sigma_2\sigma_3 \fbox{$\sigma_2$}\sigma_2^{n-5} \sigma_1 \fbox{$\sigma_2$} \sigma_2 \sigma_1,
\end{align*}
or the $(-1)$-closure of $\beta(\exdynD_n)=\Delta_4\beta_0(\exdynD_n)\Delta_4$.
Since $\legendrian(\exdynD_n)$ has three or four components, we have
\[
\tilde\beta_0(\exdynD_n) =\begin{cases}
\sigma_3\sigma_2^3\sigma_3 \sigma_2^{n-4} \sigma_1 \sigma_2^3 \sigma_1 & n\text{ is odd};\\
\sigma_3\sigma_2^3\sigma_3 \sigma_2^{n-3} \sigma_1 \sigma_2^3 \sigma_1 & n\text{ is even}.
\end{cases}
\]

The Legendrian link $\legendrian({\exdynD}_n)$ admits the brick quiver diagram $\quiver^\brick({\exdynD}_n)$.

\begin{align*}
\quiver^\brick({\exdynD}_n)&=

\end{align*}

\begin{definition}[$N$-graph of type $\exdynD_{n}$]
We define a free $4$-graph $\ngraph(\exdynD_{n})$ for $\legendrian(\exdynD_n)$ as depicted in Figure~\ref{figure:4-graph of type affine Dn}.
\end{definition}

\begin{figure}[ht]
\begin{align*}
(\ngraph(\exdynD_{2k+4}),\nbasistilde(\exdynD_{2k+4}))&=
%
\end{align*}
\caption{$N$-graphs of type $\exdynD_{2k+3}$ and $\exdynD_{2k+4}$ for $k\ge 0$}
\label{figure:4-graph of type affine Dn}
\end{figure}

\begin{lemma}\label{lemma:Ngraphs of affine Dn}
The pairs $(\ngraph(\exdynD_n),\nbasis(\exdynD_n))$ and $(\ngraph^\brick(\exdynD_n),\nbasis^\brick(\exdynD_n))$ are equivalent up to $\boundary$-Legendrian isotopy and Legendrian mutations.
\end{lemma}
\begin{proof}
We first introduce an auxiliary $N$-graph $(\ngraph(\exdynD_n)', \nbasistilde(\exdynD_n)')$
\[
(\ngraph(\exdynD_n)', \nbasistilde(\exdynD_n)') \coloneqq\begin{tikzpicture}[baseline=-.5ex,scale=0.4]
\draw[rounded corners=5, thick] (-6.5, -2.5) rectangle (6.5, 2.5);
\draw (0.5, -2.5) node[below] {$\underbrace{\hphantom{\hspace{2cm}}}_{n-4}$};
\clip[rounded corners=5] (-6.5, -2.5) rectangle (6.5, 2.5);
\draw[color=gray, line cap=round, line width=5, opacity=0.5] (-5.5,-1) -- (-5.5,-3);
\draw[color=gray, line cap=round, line width=5, opacity=0.5] (-1.5,0) -- (-1.5,-3);
\draw[color=gray, line cap=round, line width=5, opacity=0.5] (4.5,-1.75) -- (6.5,-1.75);
\draw[thick, blue, fill]
(-0.5, -2.5) -- (-0.5,0) circle (2pt)
(1.5, -2.5) -- (1.5,0) circle (2pt)
;
\begin{scope}[xscale=-1]
\draw[thick, blue, fill]
(-2.5, -2.5) -- (-2.5,0) circle (2pt)
(-0.5, -2.5) -- (-0.5,0) circle (2pt)
(1.5, -2.5) -- (1.5,0) circle (2pt)
;
\end{scope}
\draw[thick, green, rounded corners] (-2.5, 2.5) -- (-2.5, -2.5);
\draw[thick, red] 
(-3.5, -2.5) -- (-3.5, 2.5)
(-6.5, 0) -- (-3.5, 0)
;
\draw[thick, red] 
(3.5, 2.5) -- (3.5, -2.5)
(6.5, 0) -- (3.5, 0)
;
\draw[thick, blue, fill] 
(-3.5, 0) -- (3.5, 0)
(-3.5, 0) -- (-4.5, 1) circle (2pt) -- (-4.5, 2.5)
(-4.5, 1) -- (-6.5, 1)
(-5.5, 1) circle (2pt) -- (-5.5, 2.5)
(-3.5, 0) -- (-4.5, -1) circle (2pt) -- (-4.5, -2.5)
(-4.5, -1) -- (-6.5, -1)
(-5.5, -1) circle (2pt) -- (-5.5, -3)
;
\begin{scope}[xscale=-1]
\draw[thick, blue, fill] 
(-3.5, 0) -- (-4.5, 1) circle (2pt) -- (-4.5, 2.5)
(-4.5, 1) -- (-6.5, 1)
(-5.5, 1) circle (2pt) -- (-5.5, 2.5)
(-3.5, 0) -- (-4.5, -1) circle (2pt) -- (-4.5, -2.5)
(-4.5, -1) -- (-6.5, -1)
(-4.5, -1.75) circle (2pt) -- (-6.5, -1.75)
;
\end{scope}
\draw[thick, fill=white] (-3.5, 0) circle (2pt) (3.5, 0) circle (2pt);
\end{tikzpicture}
\]
Then $\boundary \ngraph(\exdynD_n)'=\sigma_2\sigma_1^3\sigma_2\sigma_1\fbox{$\sigma_1$}\sigma_1\sigma_2\sigma_3\fbox{$\sigma_1$}\sigma_1^{n-5}\sigma_2\sigma_1^3\sigma_2\sigma_1\fbox{$\sigma_1$}\sigma_1\sigma_2\sigma_3$ is equivalent to $\legendrian(\exdynD_n)$ as follows:
\begin{align*}
\beta(\exdynD_n)&=\Delta_4\sigma_3\fbox{$\sigma_2$}\sigma_2\sigma_3 \fbox{$\sigma_2$}\sigma_2^{n-5} \sigma_1 \fbox{$\sigma_2$} \sigma_2 \sigma_1\Delta_4\\
&\mathrel{\dot=}\sigma_2\sigma_1^2\sigma_2\sigma_1{\sigma_2\sigma_3\sigma_2}\sigma_1\fbox{$\sigma_2$}\sigma_2\sigma_3\fbox{$\sigma_2$}\sigma_2^{n-4}
\sigma_1\sigma_2\sigma_1\sigma_2\fbox{$\sigma_1$}\sigma_1\sigma_2\sigma_3\\
&=\boundary \ngraph(\exdynD_n)'.
\end{align*}

Moreover, as seen in Appendix~\ref{appendix:Ngraph of type affine Dn}, $(\ngraph^\brick(\exdynD_n),\nbasis^\brick(\exdynD_n))$ is equivalent to $(\ngraph(\exdynD_n)',\nbasistilde(\exdynD_n)')$, which is equivalent to $(\ngraph(\exdynD_n)',\nbasistilde(\exdynD_n)')$ up to $\boundary$-Legendrian isotopy and Legendrian mutations.
\end{proof}
As before, the freeness is obvious since $\ngraph(\exdynD_n)$ consists of trees.
Moreover, $\ngraph(\exdynD_{2k+4})$ has a $\pi$-rotation symmetry. That is, we obtain the following lemma.
\begin{lemma}
The pair $\ngraph(\exdynD_{2k+4})$ is invariant under $\pi$-rotation.
\end{lemma}

Finally, two canonical $N$-graphs $\ngraph(\exdynD_4)$ and $\ngraph_\degen(\exdynD_4)$ for $\exdynD_4$, which are indeed equivalent.
\begin{lemma}\label{lem:exdynD_4}
The pair $(\ngraph(\exdynD_4),\nbasistilde(\exdynD_4))$ is equivalent to the pair $(\ngraph_\degen(\exdynD_4),\nbasistilde_\degen(\exdynD_4))$ up to $\boundary$-Legendrian isotopy and Legendrian mutations. 
\end{lemma}
See Appendix~\ref{appendix:affine D4} for the proof.

\subsubsection{Exchange matrices and graphs}

Notice that the $N$-graphs $\ngraph(\dynX)$ and $\ngraph^\brick(\dynX)$ are deterministic for $\dynX = \dynA,\dynD,\dynE,\exdynD, \exdynE$.
Therefore, the coefficients in $\bfy(\ngraph(\dynX),\nbasis(\dynX))$ are defined on $\bbC[\cM(\legendrian(\dynX))]$. Here, $\cM(\legendrian)$ is the moduli spaces of flags on $\legendrian$ and is turned out to be a cluster Poisson variety as mentioned earlier.

On the other hand, one can show that the variables $\{X_a\}_{a\in I_\legendrian}$, Shen--Weng constructed in~\cite[\S 3.2]{SW2019}, coincide with the coefficients in the coefficient tuple $\bfy(\ngraph^\brick(\dynA_n),\nbasis^\brick(\dynA_n))$, $\bfy(\ngraph^\brick(a,b,c),\nbasis^\brick(a,b,c))$, $\bfy(\ngraph^\brick_\degen(p,q,r),\nbasis^\brick_\degen(p,q,r))$, or $\bfy(\ngraph^\brick(\exdynD_n),\nbasis^\brick(\exdynD_n))$.
Moreover, coefficients are algebraically independent.
In summary, we have the following corollary, which is a direct consequence of the above discussion, Proposition~\ref{prop_Y-pattern_exchange_graph}, and~\eqref{eq_exchange_graphs_are_the_same}.
\begin{corollary}\label{corollary:algebraic independence}
Let $(\ngraph_{t_0},\nbasis_{t_0})$ be either $(\ngraph(a,b,c), \nbasis(a,b,c))$ or $(\ngraph(n), \nbasis(n))$ of type $\dynX$, and let $(\bfy_{t_0},\qbasispr_{t_0})=\Psi(\ngraph_{t_0}, \nbasis_{t_0})$ and $\qbasispr_{t_0}=\qbasispr(\quiver(\ngraph_{t_0},\nbasis_{t_0}))$.
Then the exchange graph of the $Y$-pattern given by the initial $Y$-seed $(\bfy_{t_0},\qbasispr_{t_0})$ is the same as the exchange graph $\exchange(\Roots)$ of the root system~$\Roots$ of type $\dynX$.
\end{corollary}

For each $n\ge 1$ and triples $(a,b,c)$ and $(p,q,r)$, let $\beta_0$ be either $\beta_0(\dynA_n)$, $\beta_0(a,b,c)$, $\beta_0(p,q,r)$, or $\beta_0(\exdynD_n)$, and $\tilde\beta_0$ be the braid obtained by doubling chosen generators of $\beta_0$ if necessary so that the closure of $\tilde\beta_0$ has only one component as before.
Notice that the $N$-graph $\ngraph^\brick(\tilde\beta_0)$ is equivalent to the $N$-graph $\widetilde{\ngraph^\brick(\beta_0)}$ under the mutation on the cycle corresponding to each relative cycle in $\ngraph^\brick(\beta_0)$.

\begin{lemma}\label{lemma:admissible extension}
Let $\legendrian$ and $\tilde\legendrian$ be the rainbow closures of $\beta_0$ and $\tilde\beta_0$, respectively.
\begin{enumerate}
\item Both $\cM(\legendrian)$ and $\cM(\tilde\legendrian)$ admit the $\clusterfont{X}$-cluster structure.
\item There is a cluster subvariety $\cM(\tilde\legendrian)'$ in $\cM(\tilde\legendrian)$, whose cluster structure coincides with that of $\cM(\legendrian)$.
\item For each $N$-graph $(\ngraph,\nbasistilde)$ for $\legendrian$, the canonical extension $(\tilde\ngraph,\nbasistilde)$ yields the restriction map $\cM(\tilde\ngraph)\to\cM(\ngraph)$ between toric charts in $\cM(\tilde\legendrian)'$ and $\cM(\legendrian)$.
\item For $N$-graphs $(\ngraph(\dynA_n),\nbasistilde(\dynA_n))$, $(\ngraph(a,b,c),\nbasistilde(a,b,c))$, $(\ngraph(\exdynD_n),\nbasistilde(\exdynD_n))$, and  $(\ngraph_\degen(p,q,r),\nbasistilde_\degen(p,q,r))$, their exchange matrices are admissible in the sense of Definition~\ref{def:admissible relative cycle}.
\end{enumerate}
\end{lemma}
\begin{proof}
(1) This follows from \cite{SW2019}.

\noindent (2) and (3) These are observed in \S~\ref{section:relative cycles}.

\noindent (4) This follows easily from the direct computation.
\end{proof}

\subsection{Legendrian Coxeter mutations}

For a bipartite quiver $\quiver$, we have two sets of vertices $I_+$ and
$I_-$ so that all edges are oriented from $I_+$ to $I_-$.
Let $\mutation_+$ and $\mutation_-$ be sequences of mutations defined by 
compositions of mutations corresponding to each and every vertex in $I_+$ and 
$I_-$, respectively.
A Coxeter mutation~$\qcoxeter$ and its inverse $\qcoxeter^{-1}$ are the compositions
\begin{align*}
\mutation_\quiver&=\prod_{i\in I_+}\mutation_i \cdot \prod_{i\in I_-} \mutation_i,&
\mutation_\quiver^{-1}&=\prod_{i\in I_-}\mutation_i \cdot \prod_{i\in I_+} \mutation_i.
\end{align*}
Note that $\prod_{i\in I_+}\mutation_i$ does not depend on the order of composition of mutations $\mutation_i$ among $i\in I_+$, and the same holds for $I_-$.

\begin{remark}\label{rmk_mutation_convention}
For any sequence $\mutation$ of mutations, we will use the right-to-left convention. Namely, the rightmost mutation will be applied first on the quiver $\quiver$.
\end{remark}

Let us say that a pair $(\ngraph, \nbasis)$ is \emph{bipartite} if so is $\quiver= \quiver(\ngraph,\nbasis)$.
In this case, we decompose $\nbasis$ into $\nbasis_+$ and $\nbasis_-$ corresponding to sets $I_+$ and $I_-$ of vertices in $\quiver$.

Then similarly, we define the Legendrian Coxeter mutation, which will be denoted by $\ncoxeter$, on a bipartite $N$-graph $\ngraph$ as follows:
\begin{definition}[Legendrian Coxeter mutation]\label{def:Legendrian Coxeter mutation}
For a bipartite $N$-graph $\ngraph$ with decomposed sets of cycles $\nbasis=\nbasis_+\cup\nbasis_-$, we define the \emph{Legendrian Coxeter mutation} $\ncoxeter$ and its inverse $\ncoxeter^{-1}$ as the compositions of Legendrian mutations
\begin{align*}
\mu_\ngraph&=\prod_{\gamma\in \nbasis_+}\mutation_\gamma \cdot \prod_{\gamma\in \nbasis_-}\mutation_\gamma,&
\mu_\ngraph^{-1}&=\prod_{\gamma\in \nbasis_-}\mutation_\gamma \cdot \prod_{\gamma\in \nbasis_+}\mutation_\gamma.
\end{align*}
\end{definition}

It is worth mentioning that the Legendrian Coxeter mutations make sense only when each Legendrian mutation $\mutation_\cycle$ \emph{exists}. Also note that each $\mutation_\ngraph^{\pm1}$ does not depend on the order of mutations if cycles in each of $\nbasis_\pm$ are disjoint. 
This directly implies that $\mutation_\ngraph^{-1}$ is indeed the inverse of 
$\mutation_\ngraph$. Note that all cycles in each of $\nbasis_\pm(\dynX)$ for $\dynX =\dynA,\dynD,\dynE,\exdynD, \exdynE$ are disjoint as seen in Figures~\ref{figure:linear N-graph}, \ref{figure:tripod N-graph}, and \ref{figure:4-graph of type affine Dn}.

\subsubsection{Legendrian Coxeter mutation for linear $N$-graphs}

\begin{lemma}\label{lemma:Legendriam Coxeter mutation of type An}
The effect of the Legendrian Coxeter mutation on $(\ngraph(\dynA_n),\nbasis(\dynA_n))$ is the clockwise $\frac{2\pi}{n+3}$-rotation and therefore
\[
\ncoxeter(\ngraph(\dynA_n),\nbasis(\dynA_n))=
\coxeterpadding(\dynA_n)(\ngraph(\dynA_n),\nbasis(\dynA_n)),
\]
where $\coxeterpadding(\dynA_n)$ is an annular $N$-graph called the \emph{Coxeter padding} of type $\dynA_n$ as follows:
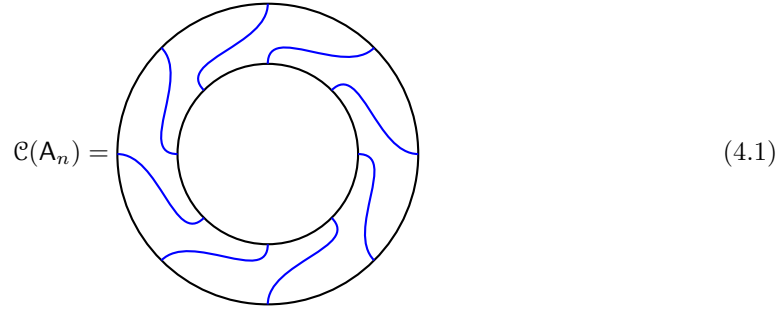
\begin{equation}\label{equation:Coxeter padding of type An}
\coxeterpadding(\dynA_n)=
\begin{tikzpicture}[baseline=-.5ex,scale=0.4]
\draw[thick] (0,0) circle (5) (0,0) circle (3);
\foreach \i in {45, 90, ..., 360} {
\draw[blue, thick] (\i:5) to[out=\i-180,in=\i+45] (\i+45:3);
}
\end{tikzpicture}
\end{equation}
\end{lemma}
\begin{proof}
We may assume that the Coxeter element $\ncoxeter$ can be represented by the sequence
\[
\ncoxeter=\mutation_+\mutation_-=(\mutation_{\cycle_2}\mutation_{\cycle_4}\mutation_{\cycle_6}\cdots)(\mutation_{\cycle_1}\mutation_{\cycle_3}\mutation_{\cycle_5}\dots).
\]
Then the action of $\ncoxeter$ on $\ngraph(\dynA_n)$ is as depicted in Figure~\ref{figure:Legendrian Coxeter mutation on An}, which is nothing but the clockwise $\frac{2\pi}{n+3}$-rotation of the original $N$-graph $(\ngraph(\dynA_n),\nbasis(\dynA_n))$ as claimed.

The last statement is obvious as seen in Figure~\ref{figure:coxeter padding of type An}.
\end{proof}

\begin{figure}[ht]
\subfigure[Legendrian Coxeter mutation for $\ngraph(\dynA_n)$\label{figure:Legendrian Coxeter mutation on An}]{
\begin{tikzcd}[ampersand replacement=\&]

\end{tikzcd}
}

\subfigure[Coxeter padding $\coxeterpadding(\dynA_n)$ of type $\dynA_n$\label{figure:coxeter padding of type An}]{$
%
$}
\caption{Legendrian Coxeter mutation $\ncoxeter$ on $(\ngraph(\dynA_n), \nbasis(\dynA_n))$}
\end{figure}

\begin{remark}\label{rmk_order_of_Coxeter_mutation}
The order of the Coxeter mutation is either $(n+3)/2$ if $n$ is odd or $n+3$ otherwise.
Since the Coxeter number $h=n+1$ for $\dynA_n$, this verifies Lemma~\ref{lemma:order of coxeter mutation} in this case.
\end{remark}

\subsubsection{Legendrian Coxeter mutation for tripod $N$-graphs}

Let us consider the Legendrian Coxeter mutation for tripod $N$-graphs.
By the mutation convention mentioned in Remark~\ref{rmk_mutation_convention}, for each tripod $\ngraph(a,b,c)$, we always take a mutation at the central $\sfY$-cycle $\cycle$ first.
After the Legendrian mutation on $(\ngraph(a,b,c),\nbasis(a,b,c))$ at $\cycle$, we have the $N$-graph on the left in Figure~\ref{figure:center mutation}.
Then there are three shaded regions that we can apply the generalized push-through moves, see Appendix~\ref{appendix:tripod with a=1 is of type An}, so that we obtain the $N$-graph on the right in Figure~\ref{figure:center mutation}.
\begin{figure}[ht]
\subfigure[\label{figure:center mutation}After the mutation at the central vertex]{
\begin{tikzcd}[ampersand replacement=\&]

\end{tikzcd}}
\subfigure[\label{figure:coxeter mutation}After Legendrian Coxeter mutation]{$
%
$}
\caption{Legendrian Coxeter mutation for $(\ngraph(a,b,c),\nbasis(a,b,c))$} 
\end{figure}
Notice that in each triangular shaded region, the $N$-subgraph looks like the $N$-graph of type $\dynA_{a-1}, \dynA_{b-1}$, or $\dynA_{c-1}$.
Moreover, the mutations corresponding to the rest sequence is just a composition 
of Legendrian Coxeter mutations of type $\dynA_{a-1},\dynA_{b-1}$, and $\dynA_{c-1}$, 
which are essentially the same as the clockwise rotations by Lemma~\ref{lemma:Legendriam Coxeter mutation of type An}.
Therefore, the result of the Legendrian Coxeter mutation will be given as depicted in 
Figure~\ref{figure:coxeter mutation}.

Then the resulting $N$-graph becomes very similar to the original $N$-graph $\ngraph(a,b,c)$.
Indeed, the inside is identical to $\ngraph(a,b,c)$ but the colors are switched, which is the conjugation $\overline{\ngraph(a,b,c)}$ by definition.
The complement of $\overline{\ngraph(a,b,c)}$ in $\qcoxeter(\ngraph(a,b,c),\nbasis(a,b,c))$ is an annular $N$-graph.

\begin{definition}[Coxeter padding of type $(a,b,c)$]
For each triple $a,b,c$, the annular $N$-graph depicted in Figure~\ref{figure:coxeter padding} is denoted by $\coxeterpadding(a,b,c)$ and called the \emph{Coxeter padding} of type $(a,b,c)$.
We also denote the Coxeter padding with color switched by $\overline{\coxeterpadding(a,b,c)}$, which is the conjugation of $\coxeterpadding(a,b,c)$.
\end{definition}

\begin{figure}[ht]
\subfigure[$\coxeterpadding(a,b,c)$]{\makebox[0.48\textwidth]{
$
$
}}
\caption{Coxeter paddings $\coxeterpadding(a,b,c)$, $\bar\coxeterpadding(a,b,c)$ and their inverses.}
\label{figure:coxeter padding}
\end{figure}

Notice that two Coxeter paddings $\coxeterpadding(a,b,c)$ and $\overline{\coxeterpadding(a,b,c)}$ can be glued without any ambiguity
and so we can also pile up Coxeter paddings $\coxeterpadding(a,b,c)$ and $\overline{\coxeterpadding(a,b,c)}$ alternatively as many times as we want.

We also define the concatenation of the Coxeter padding $\overline{\coxeterpadding(a,b,c)}$ on the pair $(\ngraph(a,b,c),\nbasis(a,b,c))$ as the pair $(\ngraph', \nbasis')$ such that
\begin{enumerate}
\item the $N$-graph $\ngraph'$ is obtained by gluing $\overline{\coxeterpadding(a,b,c)}$ on $\ngraph(a,b,c)$, and 
\item the set $\nbasis'$ of cycles is the set of $\sfI$- and $\sfY$-cycles identified with $\nbasis(a,b,c)$ in a canonical way.
\end{enumerate}

\begin{proposition}\label{proposition:effect of Legendrian Coxeter mutation}
Let $(\ngraph, \nbasis) = (\ngraph(a,b,c), \nbasis(a,b,c))$.
The Legendrian Coxeter mutation on $(\ngraph, \nbasis)$ or $\overline{(\ngraph,\nbasis)}$ is given as the concatenation
\begin{align*}
\ncoxeter(\ngraph, \nbasis) &= \coxeterpadding\overline{(\ngraph,\nbasis)},&
\ncoxeter^{-1}(\ngraph, \nbasis) &= \bar\coxeterpadding^{-1}\overline{(\ngraph,\nbasis)},&
\ncoxeter\overline{(\ngraph,\nbasis)} &= \bar \coxeterpadding (\ngraph, \nbasis),&
\ncoxeter^{-1}\overline{(\ngraph,\nbasis)} &= \coxeterpadding^{-1} (\ngraph, \nbasis),
\end{align*}
where $\coxeterpadding=\coxeterpadding(a,b,c)$, $\bar\coxeterpadding=\overline{\coxeterpadding(a,b,c)}$.

In general, for $r\ge 0$, we have
\begin{align*}
\ncoxeter^r(\ngraph,\nbasis) &= \begin{cases}
\coxeterpadding\bar\coxeterpadding\cdots \bar\coxeterpadding (\ngraph,\nbasis)& \text{ if }r\text{ is even},\\
\coxeterpadding\bar\coxeterpadding\cdots \coxeterpadding \overline{(\ngraph,\nbasis)}& \text{ if }r\text{ is odd}.
\end{cases}\\
\ncoxeter^{-r}(\ngraph,\nbasis) &= \begin{cases}
\bar\coxeterpadding^{-1}\coxeterpadding^{-1}\cdots \coxeterpadding^{-1} (\ngraph,\nbasis)& \text{ if }r\text{ is even},\\
\bar\coxeterpadding^{-1}\coxeterpadding^{-1}\cdots \bar\coxeterpadding^{-1} \overline{(\ngraph,\nbasis)}& \text{ if }r\text{ is odd}.
\end{cases}
\end{align*}
\end{proposition}
\begin{proof}
This follows directly from the above observation.
\end{proof}

It is important that this proposition holds only when we take the Legendrian Coxeter mutation on the very standard $N$-graph $\ngraph(a,b,c)$ with the cycles $\nbasis(a,b,c)$.
Otherwise, the Legendrian Coxeter mutation will not be expressed as simple as above.

Let $(\ngraph, \nbasis)$ be a pair of a deterministic $N$-graph, a set of good cycles.
Suppose that the quiver~$\quiver(\ngraph,\nbasis)$ is bipartite and the Legendrian Coxeter mutation $\ncoxeter(\ngraph,\nbasis)$ is realizable.
Then, by Proposition~\ref{proposition:equivariance of mutations}, we have
\[
\Psi(\ncoxeter(\ngraph,\nbasis)) = \qcoxeter(\Psi(\ngraph,\nbasis)).
\]
In particular, for quivers of type $\dynA_n$ or tripods, we have the following corollary.
\begin{corollary}\label{corollary:Coxeter mutations}
For each $n\ge 1$ and $a,b,c\ge 1$, the Legendrian Coxeter mutation $\ncoxeter$ on $(\ngraph(\dynA_n),\nbasis(\dynA_n))$ or $(\ngraph(a,b,c),\nbasis(a,b,c))$ corresponds to the Coxeter mutation $\qcoxeter$ on $\quiver(\dynA_n)$ or $\quiver(a,b,c)$, respectively.
In other words,
\begin{align*}
\Psi(\ncoxeter(\ngraph(\dynA_n),\nbasis(\dynA_n))) &= \qcoxeter(\Psi(\ngraph(\dynA_n),\nbasis(\dynA_n)));\\
\Psi(\ncoxeter(\ngraph(a,b,c),\nbasis(a,b,c))) &= \qcoxeter(\Psi(\ngraph(a,b,c),\nbasis(a,b,c))).
\end{align*}
\end{corollary}

\begin{theorem}\label{theorem:infinite fillings}
For $a,b,c\ge 1$ with $\frac 1a+\frac1b+\frac1c\le 1$,
The Legendrian knot or link $\legendrian(a,b,c)$ in $J^1\sphere^1$ admits infinitely many distinct exact embedded Lagrangian fillings.
\end{theorem}
\begin{proof}
By Proposition~\ref{proposition:effect of Legendrian Coxeter mutation}, the effect of the Legendrian Coxeter mutation on $(\ngraph(a,b,c), \nbasis(a,b,c))$ is just to attach the Coxeter padding on $(\bar\ngraph(a,b,c),\bar\nbasis(a,b,c))$.
In particular, as mentioned earlier, the iterated Legendrian Coxeter mutation
\[
\ncoxeter^r(\ngraph(a,b,c), \nbasis(a,b,c))
\]
is well-defined for each $r\in\mathbb{Z}$.
Each of these $N$-graphs defines a Legendrian weave $\Legendrian(\ncoxeter^r(\ngraph(a,b,c), \nbasis(a,b,c)))$, whose Lagrangian projection is a Lagrangian filling 
\[
L_r(a,b,c)\colonequals(\pi\circ\iota)(\Legendrian(\ncoxeter^r(\ngraph(a,b,c), \nbasis(a,b,c)))
\]
as desired. Therefore it suffices to prove that Lagrangians $L_r(a,b,c)$ for $r\ge 0$ are pairwise distinct up to exact Lagrangian isotopy when $\frac1a+\frac1b+\frac1c\le 1$.

Now suppose that $\frac1a+\frac1b+\frac1c\le1$, or equivalently, $\quiver(a,b,c)$ is of infinite type, that is, it is not of finite Dynkin type (cf. Definition~\ref{def_quiver_of_type_X}(1)).
Then the order of the Coxeter mutation is infinite by Lemma~\ref{lemma:order of coxeter mutation} and so is the order of the Legendrian Coxeter mutation by Corollary~\ref{corollary:Coxeter mutations}.
In particular, the set 
\[
\left\{\Psi(\ncoxeter^r(\ngraph(a,b,c), \nbasis(a,b,c)))\mid r\in\mathbb{Z}\right\}
\]
is a set of infinitely many pairwise distinct $Y$-seeds in the $Y$-pattern for $\quiver(a,b,c)$. 
Hence, by Lemma~\ref{lemma:admissible extension} and Proposition~\ref{prop:distinct seeds imples distinct fillings}, we have pairwise distinct Lagrangian fillings $L_r(a,b,c)$.
\end{proof}

\begin{remark}
For Legendrian links of non $\dynADE$-type, there are lots of examples having infinitely many distinct Lagrangian fillings given by a number of different researchers and groups. A non-exhaustive list includes \cite{CG2020, CN2021, CZ2020, GSW2020b}.
\end{remark}

\subsubsection{Legendrian Coxeter mutations for $N$-graphs of type $\exdynD_n$}\label{sec:ngraph for exdynDn}
We will perform the Legendrian Coxeter mutation $\mutation_\ngraph$ on $(\ngraph(\exdynD_n), \nbasis(\exdynD_n))$ in order to provide the pictorial proof of Proposition~\ref{proposition:coxeter realization D-type}.

Before we take mutations, we first introduce a useful operation on $N$-graphs 
described below, called the \emph{move} $\mathrm{(Z)}$.
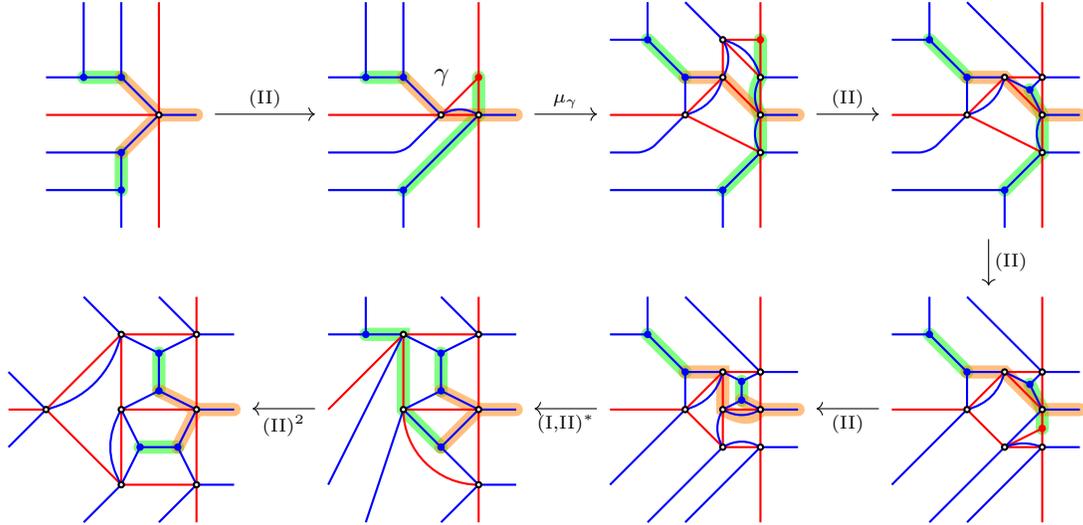
\[
\begin{tikzcd}
\arrow[l,"\mathrm{(II)}"]
\end{tikzcd}
\]

\begin{remark}\label{rmk:moveZ}
The reader should not confuse that even though we call this operation the 
\emph{move}, it does not induce any equivalence on $N$-graphs since it involves 
a mutation $\mutation_\cycle$.
\end{remark}

One important observation is that one can take the move $\mathrm{(Z)}$ instead of the Legendrian mutation~$\mutation_\cycle$ on the $\sfY$-like cycle\footnote{We use an ambiguous terminology `$\sfY$-like cycle' since the global shape of $\cycle$ is unknown.
	However, the meaning is obvious and we omit the detail.
}~$\cycle$, and after the move, the $\sfY$-like cycle becomes the $\sfY$-like cycle and $\sfI$-cycles become $\sfI$-cycles again.

For example, let us consider $(\ngraph(\exdynD_4), \nbasis(\exdynD_4))$. Then the Legendrian Coxeter mutation $\ncoxeter(\ngraph(\exdynD_4), \nbasis(\exdynD_4))$ is obtained by the composition $(\mutation_{\cycle_2}\mutation_{\cycle_3}\mutation_{\cycle_4}\mutation_{\cycle_5})$ followed by the mutation $\mutation_{\cycle_1}$. See Figure~\ref{figure:Legendrian Coxeter mutation for affine D4}.

\begin{figure}[ht]
\[
\begin{tikzcd}

\end{tikzcd}
\]
\caption{Legendrian Coxeter mutation for $\ngraph(\exdynD_4)$}
\label{figure:Legendrian Coxeter mutation for affine D4}
\end{figure}

Therefore, $\ncoxeter(\ngraph(\exdynD_4), \nbasis(\exdynD_4))$ is the same as the concatenation 
\[
\ncoxeter(\ngraph(\exdynD_4), \nbasis(\exdynD_4))=\coxeterpadding(\exdynD_4)(\ngraph(\exdynD_4),\nbasis(\exdynD_4,)),
\]
where the annular $N$-graph $\coxeterpadding(\exdynD_4)$ looks as follows:
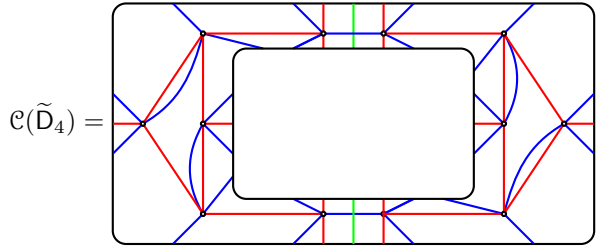
\[
\coxeterpadding(\exdynD_4)=
%
\]

In general, for the $N$-graph $(\ngraph(\exdynD_n), \nbasis(\exdynD_n))$, the Legendrian Coxeter mutation is the same as the concatenation of the Coxeter padding of type $\coxeterpadding^{\pm1}(\exdynD_n)$, which is an annular $N$-graph depicted in Figure~\ref{figure:coxeter paddings for affine D}.

\begin{figure}[ht]
\subfigure[$\coxeterpadding({\exdynD}_{n})$]{\makebox[.49\textwidth]{
%
}}
\caption{Coxeter paddings $\coxeterpadding(\exdynD_n)^{\pm1}$}
\label{figure:coxeter paddings for affine D}
\end{figure}

\begin{proposition}\label{proposition:coxeter realization D-type}
For any $r\in\Z$, the Legendrian Coxeter mutation $\mutation_\ngraph^r$ on the pair $(\ngraph(\exdynD_n),\nbasis(\exdynD_n))$ is given by piling the Coxeter paddings $\coxeterpadding(\exdynD_n)^{\pm1}$. That is,
\begin{align*}
\mutation_\ngraph^{r}(\ngraph(\exdynD_n),\nbasis(\exdynD_n))
=
\begin{cases}
\coxeterpadding(\exdynD_n)\coxeterpadding(\exdynD_n)\cdots\coxeterpadding(\exdynD_n)(\ngraph(\exdynD_n),\nbasis(\exdynD_n)) & r\ge 0;\\
\coxeterpadding(\exdynD_n)^{-1}\coxeterpadding(\exdynD_n)^{-1}\cdots \coxeterpadding(\exdynD_n)^{-1}(\ngraph(\exdynD_n),\nbasis(\exdynD_n)) & r<0.
\end{cases}
\end{align*}
\end{proposition}

\begin{corollary}\label{cor:coxeter realization D-type}
For any $r\in \Z$, 
the Legendrian Coxeter mutation $\mutation_\ngraph^r(\ngraph(\exdynD_n),\nbasis(\exdynD_n))$ is realizable by $N$-graphs and set of good cycles.
\end{corollary}

For the notational clarity, it is worth mentioning that $\coxeterpadding(\exdynD_n)$ and $\coxeterpadding(\exdynD_n)^{-1}$ are the inverse to each other with respect to the concatenation introduced in Section~\ref{section:annular Ngraphs}.

For example, one can present the Coxeter padding $\coxeterpadding(\exdynD_n)^{\pm1}$ as follows:
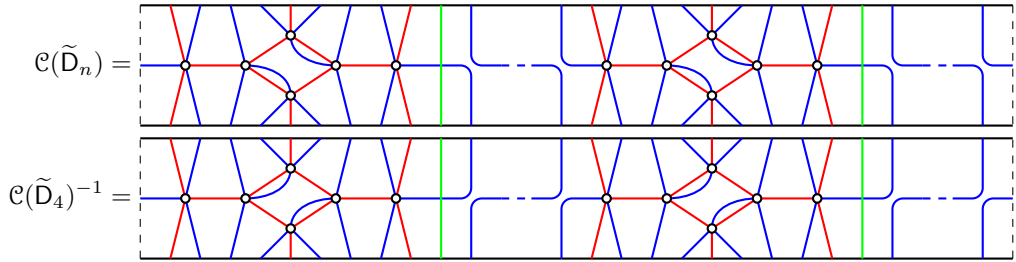
\begin{align*}
\coxeterpadding(\exdynD_n)&=

\end{align*}
Then it is direct to check that the concatenations $\coxeterpadding(\exdynD_n) \coxeterpadding(\exdynD_n)^{-1}$ and $\coxeterpadding(\exdynD_n)^{-1} \coxeterpadding(\exdynD_n)$ become trivial annulus $N$-graphs after a sequence of Move (I) for all $n\geq 4$.

\subsubsection{Legendrian Coxeter mutations for degenerate $N$-graphs}

For degenerate $N$-graphs $\ngraph_\degen(p,q,1)$ and $\ngraph_\degen(\exdynD_4)=\ngraph_\degen(2,2,2)$, the Legendrian Coxeter mutations are as depicted in Figure~\ref{figure:Legendrian Coxeter mutations for degenerate Ngraphs}.

\begin{figure}[ht]
\subfigure[$\mutation_\ngraph(\ngraph_\degen(p,q,1))$]{$
\begin{tikzcd}[ampersand replacement=\&, column sep=1pc]

\end{tikzcd}
$}
\caption{Legendrian Coxeter mutations for degenerate $N$-graphs}
\label{figure:Legendrian Coxeter mutations for degenerate Ngraphs}
\end{figure}

Then by using \Move{DI} and \Move{DII} several times, one can show easily that the Legendrian Coxeter mutations are equivalent to $N$-graphs depicted below:
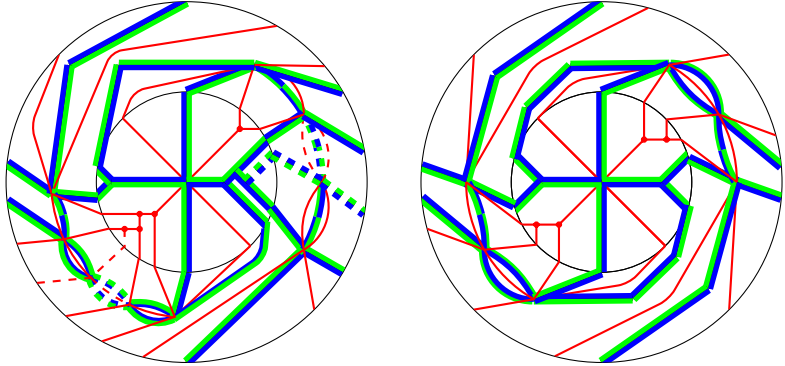
\[
%
\]
Therefore one can conclude that the effect of the Legendrian Coxeter mutation on each degenerate $N$-graph $\ngraph_\degen(p,q,1)$ or $\ngraph_\degen(\exdynD_4)$ is equivalent to attaching an annular $N$-graph which defines the Coxeter padding $\coxeterpadding_\degen(p,q,1)$ or $\coxeterpadding_\degen(\exdynD_4)$.

\begin{proposition}\label{proposition:coxeter realization denegerate type}
Let $(\ngraph_\degen, \nbasis_\degen)$ be either $(\ngraph_\degen(p,q,1), \nbasis_\degen(p,q,1))$ or $(\ngraph_\degen(\exdynD_4), \nbasis_\degen(\exdynD_4))$.
Then for each $r\in\Z$, the Legendrian Coxeter mutation $\mutation_\ngraph^r$ on the pair $(\ngraph_\degen, \nbasis_\degen)$ is given as
\[
\mutation_\ngraph^{r}(\ngraph_\degen,\nbasis_\degen)
=
\begin{cases}
\coxeterpadding_\degen\coxeterpadding_\degen\cdots\coxeterpadding_\degen(\ngraph_\degen,\nbasis_\degen)& r\ge 0;\\
\coxeterpadding_\degen^{-1}\coxeterpadding_\degen^{-1}\cdots\coxeterpadding_\degen^{-1}(\ngraph_\degen,\nbasis_\degen)& r< 0.
\end{cases}
\]
where $\coxeterpadding_\degen$ is either $\coxeterpadding_\degen(p,q,1)$ or $\coxeterpadding_\degen(\exdynD_n)$, which are degenerate annular $N$-graphs defined as follows:
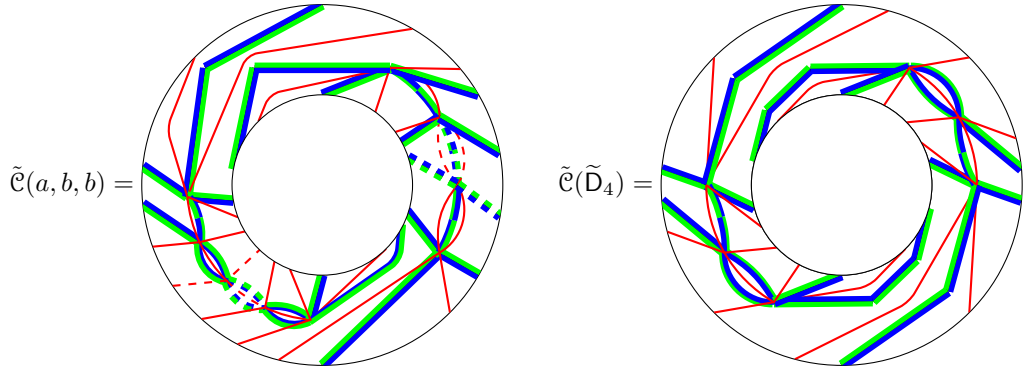
\begin{align*}
\coxeterpadding_\degen(p,q,1)&=

\end{align*}
\end{proposition}

\subsection{Legendrian loops}\label{sec:legendrian loop}
Recall Legendrian loops defined in Definition~\ref{definition:Legendrian loops}.
The goal of this section is to interpret the Legendrian Coxeter paddings with tame Legendrian loops.

Obviously, the Legendrian Coxeter paddings for $\dynA_n$ depicted in \eqref{equation:Coxeter padding of type An} is tame.
Moreover, it corresponds to the tame $\boundary$-Legendrian isotopy which moves the very first generator $\sigma_1$ to the rightmost position along the closure part of $\legendrian(\dynA_n)$ as follows:
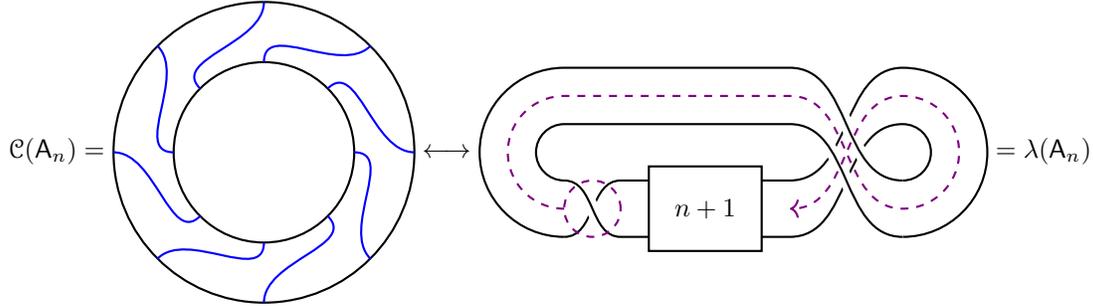
\[
\coxeterpadding(\dynA_n)=
\begin{tikzpicture}[baseline=-.5ex,scale=0.4]
\draw[thick] (0,0) circle (5) (0,0) circle (3);
\foreach \i in {45, 90, ..., 360} {
\draw[blue, thick] (\i:5) to[out=\i-180,in=\i+45] (\i+45:3);
}
\end{tikzpicture}
\longleftrightarrow
\begin{tikzpicture}[baseline=-.5ex, scale=1.5]
\draw[thick] (-3, -0.75) to[out=0,in=180] (-2.5, -0.25) -- (-2.25, -0.25);
\draw[white, line width=5] (-3, -0.25) to[out=0,in=180] (-2.5, -0.75);
\draw[thick] (-3, -0.25) to[out=0,in=180] (-2.5, -0.75) -- (-2.25, -0.75);
\draw[thick] (-2.25, -0.125) rectangle (-1.25, -0.875) (-1.75, -0.5) node {$n+1$};
\draw[thick] (-1.25, -0.25) -- (-1, -0.25);
\draw[thick] (-1.25, -0.75) -- (-1, -0.75);
\draw[thick] (-1, -0.25) to[out=0,in=180] (0,0.75) arc (90:-90:0.75);
\draw[thick] (-1, -0.75) to[out=0,in=180] (0,0.25) arc (90:-90:0.25);
\draw[white, line width=5] (-1,0.25) to[out=0,in=180] (0, -0.75);
\draw[white, line width=5] (-1,0.75) to[out=0,in=180] (0, -0.25);
\draw[thick] (-3, -0.25) arc (-90:-270:0.25) -- (-1,0.25) to[out=0,in=180] (0, -0.75);
\draw[thick] (-3, -0.75) arc (-90:-270:0.75) -- (-1,0.75) to[out=0,in=180] (0, -0.25);
\draw[thick, violet, dashed] (-2.75, -0.5) circle (0.25);
\draw[thick, violet, dashed, ->] (-3,-0.5) arc (-90:-270:0.5) -- (-1, 0.5) to[out=0,in=180] (0, -0.5) arc (-90:90:0.5) to[out=180,in=0] (-1, -0.5);
\end{tikzpicture}
=\legendrian(\dynA_n)
\]

\begin{lemma}
Legendrian Coxeter paddings of type $(a,b,c)$ and $\exdynD$ are tame.
\end{lemma}
\begin{proof}
We provide decompositions of the Coxeter paddings $\coxeterpadding(a,b,c)$ and $\coxeterpadding(\exdynD_4)$ into sequences of elementary annular $N$-graphs in Figures~\ref{fig:coxeter padding affine E is tame} and~\ref{fig:coxeter padding affine D4 is tame}, respectively. We omit other cases.
\end{proof}

\begin{figure}[ht]
\subfigure[$\coxeterpadding(a,b,c)$\label{fig:coxeter padding affine E is tame}]{$
\coxeterpadding(a,b,c)=

$}

\subfigure[$\coxeterpadding({\exdynD}_n)$\label{fig:coxeter padding affine D4 is tame}
]{$
\begin{aligned}
\phantom{^{-1}}\coxeterpadding(\exdynD_n)&=
%
\end{aligned}
$}

\caption{A sequence of elementary annulus $N$-graphs for Legendrian Coxeter paddings} 
\end{figure}

Then we may translate the sequence of Reidemeister moves induced by $\bar\coxeterpadding(a,b,c) \coxeterpadding(a,b,c)$ into the Legendrian loop $\vartheta(a,b,c)$ depicted as in Figure~\ref{fig:legendrian loop of E_intro}.
Note that the path of Legendrians from the bottom left to the top right Legendrian corresponds to $\coxeterpadding(a,b,c)$ while the path from the top right to  the bottom left one corresponds to $\bar\coxeterpadding(a,b,c)$.

\begin{figure}[ht]
\[
\begin{tikzcd}[ampersand replacement=\&]

\arrow[l]
\end{tikzcd}
\]
\caption{A Legendrian loop $\vartheta(a,b,c)$ induced from Legendrian Coxeter mutation $\mutation_\ngraph^{2}$ on $(\ngraph(a,b,c),\nbasis(a,b,c))$.}
\label{fig:legendrian loop of E_intro}
\end{figure}

In order to see the effect of Legendrian Coxeter mutation of type $\exdynD_n$ efficiently, let us present it by a sequence of braid moves together with keep tracking braid words shaded by violet color as follows:
\[
\begin{tikzcd}[column sep = -8pt, row sep = -5pt]
\beta({\exdynD}_n) 
&= &\2&\1&\1&\color{red}\1&\color{violet}\circled{\text{$\2$}}&\color{red}\1&\1&\1&\2&\color{violet}\circled{\text{$\1$}}&\1^{k-1}&\3&\2&\1&\1&\color{red}\1&\color{violet}\circled{\text{$\2$}}&\color{red}\1&\1&\1&\2&\color{violet}\circled{\text{$\1$}}&\1^{\ell-1}&\3
\\
&= &\2&\1&\color{red}\1&\color{red}\2&\color{violet}\circled{\text{$\1$}}&\2&\1&\color{red}\1&\color{red}\2&\color{violet}\circled{\text{$\1$}}&\1^{k-1}&\3&\2&\1&\color{red}\1&\color{red}\2&\color{violet}\circled{\text{$\1$}}&\2&\1&\color{red}\1&\color{red}\2&\color{violet}\circled{\text{$\1$}}&\1^{\ell-1}&\3
\\
&= &\color{red}\2&\color{red}\1&\color{violet}\circled{\text{$\2$}}&\1&\2&\color{red}\2&\color{red}\1&\color{violet}\circled{\text{$\2$}}&\1&\2&\1^{k-1}&\3&\color{red}\2&\color{red}\1&\color{violet}\circled{\text{$\2$}}&\1&\2&\color{red}\2&\color{red}\1&\color{violet}\circled{\text{$\2$}}&\1&\2&\1^{\ell-1}&\3
\\
&= &\color{violet}\circled{\text{$\1$}}&\2&\1&\1&\2&\color{violet}\circled{\text{$\1$}}&\2&\1&\1&\2&\1^{k-1}&\3&\color{violet}\circled{\text{$\1$}}&\2&\1&\1&\2&\color{violet}\circled{\text{$\1$}}&\2&\1&\1&\2&\1^{\ell-1}&\3
\\
&\mathrel{\dot{=}} &\2&\1&\1&\2&\color{violet}\circled{\text{$\1$}}&\2&\1&\1&\2&\1^{k-1}&\color{red}\3&\color{violet}\circled{\text{$\1$}}&\2&\1&\1&\2&\color{violet}\circled{\text{$\1$}}&\2&\1&\1&\2&\1^{\ell-1}&\color{red}\3&\color{violet}\circled{\text{$\1$}}
\\
&= &\2&\1&\1&\color{red}\2&\color{violet}\circled{\text{$\1$}}&\color{red}\2&\1&\1&\2&\1^{k-1}&\color{violet}\circled{\text{$\1$}}&\3&\2&\1&\1&\color{red}\2&\color{violet}\circled{\text{$\1$}}&\color{red}\2&\1&\1&\2&\1^{\ell-1}&\color{violet}\circled{\text{$\1$}}&\3
\\
&= &\2&\1&\1&\1&\color{violet}\circled{\text{$\2$}}&\1&\1&\1&\2&\1^{k-1}&\color{violet}\circled{\text{$\1$}}&\3&\2&\1&\1&\1&\color{violet}\circled{\text{$\2$}}&\1&\1&\1&\2&\1^{\ell-1}&\color{violet}\circled{\text{$\1$}}&\3
&=\beta(\exdynD_n)
\end{tikzcd}
\]
The corresponding annular $N$-graph is depicted in Figure~\ref{fig:coxeter padding affine D4 is tame}.
Finally, the effect of Coxeter padding $\coxeterpadding({\exdynD}_n)$ onto $\beta({{\exdynD}_n})$ can be presented as a Legendrian loop $\vartheta(\exdynD_n)$, which is a composition 
\[
\vartheta(\exdynD_n) = \varphi \vartheta_0(\exdynD_n) \varphi^{-1}
\]
as depicted in Figure~\ref{fig:legendrian loop of D_intro},
where $\varphi$ is a Legendrian Reidemeister move (III).

\begin{figure}[ht]
\[
\begin{tikzcd}[row sep=2pc]

\end{tikzcd}
\]

\caption{A Legendrian loop $\vartheta(\exdynD)=\varphi\vartheta_0(\exdynD_n)\varphi^{-1}$ induced from Legendrian Coxeter mutation $\mutation_\ngraph$ on $(\ngraph(\exdynD_n),\nbasis(\exdynD_n))$.}
\label{fig:legendrian loop of D_intro}
\end{figure}

\begin{theorem}\label{thm:legendrian loop}
The square $\mutation_\ngraph^{\pm 2}$ of the Legendrian Coxeter mutation 
 on $(\ngraph(a,b,c),\nbasis(a,b,c))$ 
and the Legendrian Coxeter mutation $\mutation_\ngraph^{\pm1}$ on 
$(\ngraph(\exdynD),\nbasis(\exdynD))$ induce 
tame Legendrian loops $\vartheta(a,b,c)$ and $\vartheta(\exdynD)$ in Figures~\ref{fig:legendrian loop of E_intro} and \ref{fig:legendrian loop of D_intro}, respectively. 
\end{theorem}

\subsection{Lagrangian fillings}
In this section, we will prove one of our main theorem on `as many exact embedded Lagrangian fillings as seeds' as follows:
\begin{theorem}\label{theorem:seed many fillings}
Let $\legendrian$ be a Legendrian knot or link of type~$\dynADE$ or type $\exdynD\exdynE$.
Then it admits at least as many distinct exact embedded Lagrangian fillings up to exact Lagrangian isotopy (rel boundary) as the number of seeds in the seed pattern of the same type.
\end{theorem}

Indeed, this theorem follows from considering the following general question.
\begin{question}\label{question_CZ}
For a given $N$-graph $\ngraph$ with a chosen set $\nbasis$ of cycles, can we take a Legendrian mutation as many times as we want? Or equivalently, after applying a mutation $\mutation_k$ on $(\ngraph, \nbasis)$, is the set $\mutation_k(\nbasis)$ still good in $\mutation_k(\ngraph)$?
\end{question}

This question has been raised previously in \cite[Remark~7.13]{CZ2020}.
One of the main reason making the question nontrivial is that 
the potential difference of geometric and algebraic intersections between two cycles.

Instead of attacking Question~\ref{question_CZ} directly, we will prove the following:
\begin{proposition}\label{proposition:realizability}
For $\dynX = \dynA,\dynD,\dynE,\exdynD, \exdynE$, 
let $(\ngraph_{t_0},\nbasis_{t_0})=(\ngraph(\dynX), \nbasis(\dynX))$ as depicted in Figures~\ref{figure:linear N-graph}, \ref{figure:tripod N-graph}, and \ref{figure:4-graph of type affine Dn}.
Suppose that $(\bfy, \qbasispr)$ is a $Y$-seed in the $Y$-pattern given by the initial $Y$-seed $(\bfy_{t_0},\qbasispr_{t_0}) = \Psi(\ngraph_{t_0},\nbasis_{t_0})$. Then $\legendrian(\dynX)$ admits an $N$-graph $(\ngraph, \nbasis)$ on $\disk^2$ with $\boundary \ngraph = \legendrian(\dynX)$ such that 
\[
\Psi(\ngraph, \nbasis) = (\bfy, \qbasispr).
\]
\end{proposition}

Under the aid of this proposition, one can prove Theorem~\ref{theorem:seed many fillings}.
\begin{proof}[Proof of Theorem~\ref{theorem:seed many fillings}]
Let $\legendrian$ be given as above.
Then, by Proposition~\ref{proposition:realizability}, we have the set of pairs of $N$-graphs and set of good cycles which has a one-to-one correspondence via $\Psi$ with the set of $Y$-seeds in the $Y$-pattern of type $\dynX$.
Hence any pair of the Lagrangian fillings coming from these $N$-graphs is never exact Lagrangian isotopic by Lemma~\ref{lemma:admissible extension} and Proposition~\ref{prop:distinct seeds imples distinct fillings}.
Finally, by Corollary~\ref{corollary:algebraic independence}, there is a one-to-one correspondence between the set of $Y$-seeds and that of seeds, which completes the proof.
\end{proof}

\subsubsection{Proof of Proposition~\ref{proposition:realizability}}

We use an induction argument on the rank $n$ of the root system~$\Roots(\dynX)$.
The initial step is either
\[
(\ngraph_{t_0},\nbasis_{t_0})=(\ngraph(1,1,1),\nbasis(1,1,1))\quad\text{ or }\quad (\ngraph_{t_0},\nbasis_{t_0})=(\ngraph(\dynA_1),\nbasis(\dynA_1)).
\]
Since there are no obstructions for mutations on these $N$-graphs, we are done for the initial step of the induction.

Now suppose that $n\ge2$.
By the induction hypothesis, we assume that the assertion holds for each type $\dynX' =\dynA, \dynD, \dynE, \exdynD, \exdynE$ having rank strictly small than $n$.

Let $(\bfy, \qbasispr)$ be an $Y$-seed of type $\dynX$. By Lemma~\ref{lemma:normal form}, there exist $r\in\Z$ and a sequence $\mutation_{j_1},\dots,\mutation_{j_L}$ of mutations such that
\[
(\bfy, \qbasispr)=
\mutation'((\mutation_\quiver)^r(\bfy_{t_0},\qbasispr_{t_0})),\quad
\mutation'=\mutation_{j_L}\dots\mutation_{j_1},
\]
where indices $j_1,\dots,j_L$ miss at least one index $i$.
It suffices to prove that the $N$-graph
\[
(\ngraph, \nbasis)
=\mutation'((\ncoxeter)^r(\ngraph_{t_0},\nbasis_{t_0}))
\]
is well-defined.

Notice that by Lemma~\ref{lemma:Legendriam Coxeter mutation of type An}, Propositions~\ref{proposition:effect of Legendrian Coxeter mutation} and \ref{proposition:coxeter realization D-type}, the Legendrian Coxeter mutation $\ncoxeter^r(\ngraph_{t_0}, \nbasis_{t_0})$ is realizable so that
\[
\Psi(\ncoxeter^r(\ngraph_{t_0}, \nbasis_{t_0}))
=(\mutation_\quiver)^r(\bfy_{t_0},\qbasispr_{t_0}).
\]

Since $\ncoxeter^r(\ngraph_{t_0}, \nbasis_{t_0})$ is the concatenation of Coxeter paddings on the initial $N$-graph $(\ngraph_{t_0},\nbasis_{t_0})$, it suffices to prove that the Legendrian mutation
$\mutation'(\ncoxeter^r(\ngraph_{t_0}, \nbasis_{t_0}))$
is realizable, which is equivalent to the realizability of
$\mutation'(\ngraph_{t_0}, \nbasis_{t_0})$.

By assumption, the indices $j_1,\dots, j_L$ misses the index $i$ and therefore the sequence of mutations $\mutation_{j_1},\dots,\mutation_{j_L}$ can be performed inside the subgraph of the exchange graph $\exchange(\Roots(\dynX))$, which is isomorphic to $\exchange(\Roots(\dynX \setminus \{i\}))$.
Here, with abuse of notation, we denote by $\dynX$ the Dynkin diagram of type $\dynX$. Moreover, we denote by $\Roots(\dynX \setminus \{i\})$ the root system corresponding to the Dynkin diagram~$\dynX \setminus \{i\}$. 
Then the root system $\Roots(\dynX\setminus \{i\})$ is not necessarily irreducible and may be decomposed into $\Roots(\dynX^{(1)}), \dots, \Roots(\dynX^{(\ell)})$ for $\dynX\setminus\{i\} = \dynX^{(1)}\cup\cdots\cup\dynX^{(\ell)}$ so that
\begin{align*}
\Roots(\dynX\setminus\{i\}) &\cong
\Roots(\dynX^{(1)})\times\cdots\times\Roots(\dynX^{(\ell)}),\\
\exchange(\dynX\setminus\{i\}) &\cong
\exchange(\dynX^{(1)})\times\cdots\times\exchange(\dynX^{(\ell)}),\\
\quiver_{t_0}\setminus\{i\}&\cong \quiver^{(1)}\amalg\cdots\amalg\quiver^{(\ell)},
\end{align*}
where the subquiver $\quiver^{(k)}$ is of type $\dynX^{(k)}$.
Moreover, the composition $\mutation'$ of mutations can be decomposed into sequences $\mutation^{(1)},\dots, \mutation^{(\ell)}$ of mutations on $\quiver^{(1)},\dots,\quiver^{(\ell)}$.

Similarly, we may decompose the $N$-graph $(\ngraph_{t_0}, \nbasis_{t_0})$ into $N$-subgraphs 
\[
(\ngraph^{(1)}, \nbasis^{(1)}),\dots,(\ngraph^{(\ell)}, \nbasis^{(\ell)})
\]
along $\cycle_i\in\nbasis_{t_0}$, which are the restrictions of $(\ngraph_{t_0}, \nbasis_{t_0})$ onto $\disk^{(i)}\subset \disk^2$ as follows: 
\begin{enumerate}
\item For $\legendrian=\legendrian(\dynA_n)$, we have the following two cases (Figure~\ref{figure:decomposition of linear Ngraphs}):
\begin{enumerate}
\item If $\cycle_i$ corresponds to a leaf, then we have the $2$-subgraph $(\ngraph(\dynA_{n-1}), \nbasis(\dynA_{n-1}))$.
\item If $\cycle_i$ corresponds to a bivalent vertex, then for some $1\le r,s$ with $r+s+1=n$, we have two $2$-subgraphs $(\ngraph(\dynA_r),\nbasis(\dynA_r))$ and $(\ngraph(\dynA_s),\nbasis(\dynA_s))$.
\end{enumerate}
\begin{figure}[ht]
\subfigure[At a leaf]{\makebox[0.4\textwidth]{

}}
\caption{Decompositions of $\ngraph(\dynA_n)$}
\label{figure:decomposition of linear Ngraphs}
\end{figure}

\item For $\legendrian=\legendrian(a,b,c)$, we have the following three cases (Figure~\ref{figure:decomposition of tripod Ngraphs}):
\begin{enumerate}
\item If $\cycle_i$ corresponds to the central vertex, then we have three $3$-subgraphs $(\ngraph_{(3)}(\dynA_{a-1}), \nbasis_{(3)}(\dynA_{a-1}))$, $(\ngraph_{(3)}(\dynA_{b-1}),\nbasis_{(3)}(\dynA_{b-1}))$, and $(\ngraph_{(3)}(\dynA_{c-1}), \nbasis_{(3)}(\dynA_{c-1}))$.
\item  If $\cycle_i$ corresponds to a bivalent vertex, then for some $1\le r,s$ with $r+s+1=a$, up to permuting indices $a,b,c$, we have two $3$-subgraphs $(\ngraph_{(3)}(\dynA_s),\nbasis_{(3)}(\dynA_s))$ and $(\ngraph(r,b,c),\nbasis(r,b,c))$.
\item Otherwise,  if $\cycle_i$ corresponds to a leaf, then up to permuting indices $a,b,c$, we have the $3$-subgraph $(\ngraph(a-1,b,c), \nbasis(a-1,b,c))$.
\end{enumerate}

\begin{figure}[ht]
\subfigure[At the central vertex]{$
\begin{aligned}

\end{aligned}
$}
\caption{Decomposition of $\ngraph(a,b,c)$}
\label{figure:decomposition of tripod Ngraphs}
\end{figure}
\item For $\legendrian=\legendrian(\exdynD_n)$, we have the following four cases (Figure~\ref{figure:decomposition of Ngraph of type affine Dn}):
\begin{enumerate}
\item If $n=4$ and $\gamma_i$ corresponds to the central vertex, then we have four $4$-graphs of type $\dynA_1$.
\item If $n\ge 5$ and $\gamma_i$ corresponds to a trivalent vertex, then we have three $4$-graphs of type $\dynA_1, \dynA_1$ and $(2,2, n-4)$.
\item If $n\ge 6$, $\gamma_i$ corresponds to a bivalent vertex, then for some $r+s=n-3$, we have two $4$-graphs of type $(2,2,r)$ and $(2,2,s)$.
\item If $\gamma_i$ corresponds to a leaf, then we have the $4$-graph $(\ngraph'(\dynD_n),\nbasis'(\dynD_n))$.
\end{enumerate}

\begin{figure}[ht]
\subfigure[At the central vertex of $\exdynD_4$]{\makebox[0.47\textwidth]{

}}
\caption{Decompositions of $\ngraph(\exdynD_n)$}
\label{figure:decomposition of Ngraph of type affine Dn}
\end{figure}
\end{enumerate}

Here, $\ngraph_{(3)}(\dynA_{n'}), \ngraph_{(4)}(\dynA_{n'})$ and $\ngraph_{(4)}(a',b',c')$ are the $3$- and $4$-graphs obtained from $\ngraph(\dynA_{n'})$ and $\ngraph(a',b',c')$ by adding trivial planes at the top.
Hence, the realizabilities of Legendrian mutations on $\ngraph_{(3)}(\dynA_{n'}), \ngraph_{(4)}(\dynA_{n'})$ and $\ngraph_{(4)}(a',b',c')$ are the same as those on $\ngraph(\dynA_{n'})$ and $\ngraph(a',b',c')$.

Except for the very last case (3d), all the other cases are reduced to either linear and tripod $N$-graphs with strictly lower rank.
Hence, by the induction hypothesis, any composition $\mutation^{(k)}$ of mutations on $\quiver^{(k)}$ for $1\leq k\leq \ell$ can be realized as a composition of Legendrian mutations on $(\ngraph^{(k)},\nbasis^{(k)})$.
This guarantees the realizability of $\mutation'(\ngraph_{t_0}, \nbasis_{t_0})$.

For the case (3d), one can apply a sequence of Move $\Move{II}$ on $(\ngraph'(\dynD_n), \nbasis'(\dynD_n))$ as follows:
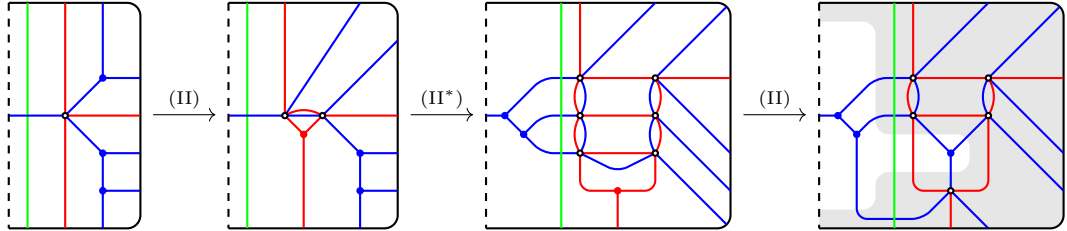
\[
\begin{tikzcd}

\end{tikzcd}
\]
Then in the last picture, the shaded part corresponds to a tame annular $N$-graph and so the $N$-graph $(\ngraph'(\dynD_n), \nbasis'(\dynD_n))$ is $\boundary$-Legendrian isotopic to the following $N$-graph
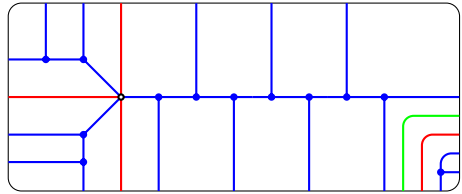
\[
%
\]
which is a stabilization of $(\ngraph(\dynD_n), \nbasis(\dynD_n))=(\ngraph(2,2,n-2), \ngraph(2,2, n-2))$.
Therefore the induction hypothesis completes the proof.

\begin{remark}
It is not claimed above that two mutations $\mutation'$ and $\ncoxeter$ commute.
Indeed, if we first mutate $(\ngraph_{t_0},\nbasis_{t_0})$ via $\mutation'$, then the result may not look like either $(\ngraph_{t_0},\nbasis_{t_0})$ or $\overline{(\ngraph_{t_0},\nbasis_{t_0})}$ and hence $\ncoxeter$ will not work as expected. Besides it is not even clear whether $\ncoxeter \mutation'(\ngraph_{t_0},\nbasis_{t_0})$ is realizable.
\end{remark}

\section{Foldings}\label{section:folding}

In this section, we will consider cluster structures of type $\dynBCFG$ and all standard affine type on $N$-graphs with certain symmetry.

Recall that if a quiver of type $\dynX$ is globally foldable with respect to the $G$-action, then the folded cluster pattern is of type $\dynY$.
We consider the triples $(\dynX, G, \dynY)$ shown in Table~\ref{table:foldings}.

\begin{table}[ht]
\[
\renewcommand\arraystretch{1.5}
\begin{array}{c|cccc}
\hline
\multirow{2}{*}{\text{rotation}}
&(\dynA_{2n-1},\Z/2\Z, \dynB_n)
&(\dynD_4,\Z/3\Z, \dynG_2)
&(\exdynE_6,\Z/3\Z, \exdynG_2)\\
&(\exdynD_{2n\ge6},\Z/2\Z, \exdynB_n)
&(\exdynD_4,\Z/2\Z, \exdynC_2)
\\
\hline
\multirow{2}{*}{\text{conjugation}}
&(\dynD_{n+1},\Z/2\Z, \dynC_n)
&(\dynE_{6},\Z/2\Z, \dynF_4)\\
&(\exdynE_6,\Z/2\Z, \dynE_6^{(2)})
&(\exdynE_7,\Z/2\Z, \exdynF_4)
&(\exdynD_4,\Z/2\Z, \dynA_5^{(2)})\\
\hline
\end{array}
\]
\caption{Folding by rotation and conjugation}
\label{table:foldings}
\end{table}

\subsection{Group actions on \texorpdfstring{$N$}{N}-graphs}

For each triple $(\dynX, G, \dynY)$, we first consider the $G$-action on each $N$-graph of type $\dynX$.

\subsubsection{Rotation action}
Let $(\dynX, G, \dynY)$ be one of five cases in the first row of Table~\ref{table:foldings}.
We will denote the generator of $G=\Z/2\Z$ or $\Z/3\Z$ by $\tau$, which acts on $N$-graphs $\ngraph$ by $\pi$-rotation or $2\pi/3$-rotation, respectively.

Notice that for each $\dynX=\dynA_{2n-1}, \dynD_4, \exdynE_6, \exdynD_{2n\ge6},$ or $\exdynD_4$, we may assume that the Legendrian $\legendrian(\dynX)$ in $J^1\sphere^1$ is invariant under the $\pi$-rotation or $2\pi/3$-rotation since the braid $\beta(\dynX)$ representing $\legendrian(\dynX)$ has the rotation symmetry as follows:
\begin{align*}
\beta(\dynA_{2n-1})&=\left(\sigma_1^{n+1}\right)^2,&
\beta(\dynD_4) &=\left(\sigma_2\sigma_1^3\right)^3,&
\beta(\exdynE_6)&=\left(\sigma_2\sigma_1^4\right)^3,\\
\beta(\exdynD_{2n})&=\left(\sigma_2\sigma_1^3\sigma_2\sigma_1^3\sigma_2\sigma_3\sigma_1^{n-2}\right)^2,n\ge 3,&
\beta(\exdynD_4)&=\left(\sigma_2\sigma_1^3\sigma_2\sigma_1^3\sigma_2\sigma_3\right)^2.
\end{align*}

Now the generator $\tau$ acts on the set $\Ngraphs(\legendrian(\dynX))$ of equivalent classes of $N$-graphs with cycles whose boundary is precisely $\legendrian(\dynX)$.
Indeed, for each $(\ngraph, \nbasis)$ in $\Ngraphs(\legendrian(\dynX))$, we have
\[
\tau\cdot(\ngraph, \nbasis)=
\begin{cases}
R_\pi (\ngraph, \nbasis) & \text{ if }\tau\in \Z/2\Z;\\
R_{2\pi/3} (\ngraph, \nbasis) & \text{ if }\tau\in \Z/3\Z,
\end{cases}
\]
where $R_\theta$ is the induced action on $N$-graphs with cycles from the $\theta$-rotation on $\disk^2$. See Figure~\ref{figure:action on Ngraph of type A}.

\begin{figure}[ht]
\subfigure[$\Z/2\Z$-action on $(\ngraph, \nbasis)$\label{figure:action on Ngraph of type A}]{
\begin{tikzcd}[ampersand replacement=\&]

\arrow[ll,"\tau", bend left=35]
\end{tikzcd}
}
\caption{Rotation actions on $N$-graphs}
\label{figure:rotation action}
\end{figure}

\subsubsection{Conjugation action}
Assume that $(\dynX, G, \dynY)$ is one of five cases in the second row of Table~\ref{table:foldings}.
We denote the generator for $G=\Z/2\Z$ by $\eta$.
Then, as before, the Legendrian $\tilde\legendrian(\dynX)$ is represented by the braid $\tilde\beta(\dynX)$ which is invariant under the conjugation as follows:
\begin{align*}
\tilde\beta(\dynD_{n+1})&=\sigma_2^n \sigma_{1,3}\sigma_2\sigma_{1,3}^3 \sigma_2\sigma_{1,3}\sigma_2^2\sigma_{1,3},&
\tilde\beta(\dynE_6)&=\sigma_2^3 \sigma_{1,3}\sigma_2\sigma_{1,3}^4 \sigma_2\sigma_{1,3}\sigma_2^2\sigma_{1,3},\\
\tilde\beta(\exdynE_6)&=\sigma_2^4 \sigma_{1,3}\sigma_2\sigma_{1,3}^4 \sigma_2\sigma_{1,3}\sigma_2^2\sigma_{1,3},&
\tilde\beta(\exdynE_7)&=\sigma_2^3 \sigma_{1,3}\sigma_2\sigma_{1,3}^5 \sigma_2\sigma_{1,3}\sigma_2^2\sigma_{1,3},&
\tilde\beta(\exdynD_4)&=(\sigma_2\sigma_{1,3}\sigma_2\sigma_{1,3}^2)^2.
\end{align*}

Therefore, the generator $\eta$ acts on the set $\Ngraphs(\tilde\legendrian(\dynX))$ by conjugation.
That is, for each $(\ngraph, \nbasis)\in\Ngraphs(\tilde\legendrian(\dynX))$, we have
\[
\eta\cdot (\ngraph, \nbasis) = \overline{(\ngraph, \nbasis)}.
\]

\begin{remark}
One may consider the conjugation invariant degenerate $N$-graph $\tilde\ngraph(\dynA_{2n-1})$ instead of the rotation invariant $N$-graph $\ngraph(\dynA_{2n-1})$ as seen earlier in Remark~\ref{remark:degenerated Ngraph of type A}. 
Then it can be checked that these two actions are identical.
\end{remark}

\begin{remark}
The denegerated $N$-graph $\tilde\ngraph(\exdynD_4)$ admits the $\pi$-rotation action as well, which is essentially equivalent to the conjugation action on $\tilde\ngraph(\exdynD_4)$. We omit the detail.
\end{remark}

\subsection{Invariant \texorpdfstring{$N$}{N}-graphs and Lagrangian fillings}
Throughout this section, we assume that $(\dynX, G, \dynY)$ is one of the triples in Table~\ref{table:foldings}.
For an $N$-graph $\ngraph$ in $\Ngraphs(\legendrian(\dynX))$ or  $\Ngraphs(\tilde\legendrian(\dynX))$, we say that $(\ngraph, \nbasis)$ is \emph{$G$-invariant} if for each $g\in G$,
\[
g\cdot(\ngraph, \nbasis) = (\ngraph, \nbasis).
\]
Namely,
\begin{enumerate}
\item the $N$-graph $\ngraph$ is invariant under the action of $g$,
\item the sets of cycles $\nbasis$ and $g(\nbasis)$ are identical up to relabeling $\cycle \leftrightarrow g(\cycle)$ for $\cycle\in\nbasis$.
\end{enumerate}

The following statements are obvious but important observations.
\begin{lemma}\label{lemma:Lagrangian fillings with symmetry}
For a free $N$-graph $\ngraph$, let
\[
L(\ngraph)\colonequals(\pi\circ\iota)(\Legendrian(\ngraph))
\]
be the Langrangian surface defined by $\ngraph$ in $\mathbb{C}^2$.
\begin{enumerate}
\item If $\ngraph$ is invariant under the $\theta$-rotation, then $L(\ngraph)$ is invariant under the $\theta$-rotation in $\mathbb{C}^2$
\[
(z_1,z_2)\mapsto (z_1\cos(\theta) +z_2\sin(\theta), -z_1\sin(\theta)+z_2\cos(\theta)).
\]
\item If $\ngraph$ is invariant under the conjugation, then $L(\ngraph)$ is invariant under the antisymplectic involution in $\mathbb{C}^2$
\[
(z_1,z_2)\mapsto (\bar z_1, \bar z_2).
\]
\end{enumerate}
\end{lemma}

\begin{lemma}\label{lemma:initial Ngraphs are G-invariant}
The $N$-graphs with cycles $(\ngraph(\dynX), \nbasis(\dynX))$ for $\dynX=\dynA_{2n-1}, \dynD_4, \exdynE_6, \exdynD_{2n\ge 6}, \exdynD_4$ and the degenerate $N$-graphs with cycles $(\tilde\ngraph(\dynX), \nbasis(\dynX))$ for $\dynX=\dynD_{n+1}, \dynE_6, \exdynE_6, \exdynE_7, \exdynD_4$ are all invariant under the $G$-action.
\end{lemma}

\begin{lemma}\label{lemma:mutation preserves invariance}
Suppose that $g\in G$ acts on $(\ngraph, \nbasis)$.
If the Legendrian mutation $\mutation_{\cycle}(\ngraph, \nbasis)$ is realizable, then
\[
\mutation_{g(\cycle)}\left(g\cdot(\ngraph,\nbasis)\right) = g\cdot\left(\mutation_{\cycle}(\ngraph,\nbasis)\right).
\]
In particular, for a $G$-orbit $I\subset\nbasis$ consists of pairwise disjoint cycles, if $(\ngraph,\nbasis)$ is $G$-invariant and the Legendrian orbit mutation $\mutation_{I}(\ngraph,\nbasis)$ is realizable, then $\mutation_I(\ngraph,\nbasis)$ is $G$-invariant as well.
\end{lemma}

On the other hand, if we have a $G$-invariant $N$-graph $(\ngraph, \nbasis)$ with cycles, it gives us a $G$-admissible quiver $\quiver(\ngraph, \nbasis)$.
\begin{lemma}\label{lemma:G-invariant Ngraphs imply G-admissible quivers}
Let $(\ngraph, \nbasis)$ be a $G$-invariant $N$-graph with cycles.
Then the quiver $\quiver(\ngraph, \nbasis)$ is globally foldable.
\end{lemma}
\begin{proof}
By definition of $G$-invariance of $(\ngraph, \nbasis)$, it is obvious that the quiver $\quiver=\quiver(\ngraph, \nbasis)$ is $G$-invariant.
On the other hand, since $\dynX$ is either a finite or an affine Dynkin diagram, the $G$-invariance of the quiver $\quiver$ implies the globally foldability of $\quiver$ by Corollary~\ref{corollary_G-invariance_and_globally_foldable}. Hence the result follows.
\end{proof}

\begin{proposition}\label{proposition:G-invariant Ngraphs}
For each $Y$-seed $(\bfy',\qbasispr')$ of type $\dynY$, there exists a $G$-invariant $N$-graph with cycles $(\ngraph, \nbasis)$ of type $\dynX$ such that
\[
\Psi(\ngraph, \nbasis)^G = (\bfy',\qbasispr').
\]
\end{proposition}
\begin{proof}
We use a similar argument as in the proof of Proposition~\ref{proposition:realizability}.
For each $\dynX$, let $(\ngraph_{t_0},\nbasis_{t_0})$ be the $N$-graph with cycles defined as follows:
\begin{align*}
(\ngraph_{t_0},\nbasis_{t_0}) \colonequals
\begin{cases}
(\ngraph(\dynX),\nbasis(\dynX)) & \text{ if } \dynX=\dynA_{2n-1}, \dynD_4, \exdynE_6, \exdynD_{2n\ge 6}, \exdynD_4, G\text{ acts as rotation};\\
(\tilde\ngraph(\dynX),\nbasis(\dynX)) & \text{ if } \dynX=\dynD_{n+1}, \dynE_6, \exdynE_6, \exdynE_7, \exdynD_4, G\text{ acts as conjugation}.
\end{cases}
\end{align*}

We regard the $Y$-seed defined by $(\ngraph_{t_0},\nbasis_{t_0})$ as the intial seed $(\bfy_{t_0}, \qbasispr_{t_0})$
\[
(\bfy_{t_0}, \qbasispr_{t_0})=\Psi(\ngraph_{t_0},\nbasis_{t_0}).
\]
As seen in Lemma~\ref{lemma:initial Ngraphs are G-invariant}, the tuple~$(\ngraph_{t_0},\nbasis_{t_0})$ of the $N$-graph with cycles is $G$-invariant. Therefore, the quiver $\quiver(\ngraph_{t_0},\nbasis_{t_0})$ is globally foldable by Lemma~\ref{lemma:G-invariant Ngraphs imply G-admissible quivers}.
Therefore, we have the folded seed $(\bfy_{t_0}, \qbasispr_{t_0})^G$ which plays the role of the initial seed of the $Y$-pattern of type~$\dynY$.

Let $(\bfy', \qbasispr')$ be an $Y$-seed of the $Y$-pattern of type $\dynY$.
By Lemma~\ref{lemma:normal form}, there exist $r\in \Z$ and a sequence of mutations $\mutation_{j_1}^\dynY, \dots, \mutation_{j_L}^\dynY$ such that 
\[
(\bfy', \qbasispr') = (\mutation_{j_L}^\dynY\cdots\mutation_{j_1}^\dynY)
((\mutation_\quiver^\dynY)^r((\bfy_{t_0}, \qbasispr_{t_0})^G)).
\]
Moreover, the indices $j_1,\dots, j_L$ misses at least one index, say $i$.

Then Theorem~\ref{thm_invariant_seeds_form_folded_pattern} implies the existence of the $G$-admissible $Y$-seed $(\bfy, \qbasispr)$ of type $\dynX$ such that $(\bfy,\qbasispr)^G=(\bfy',\qbasispr')$ and
\[
(\bfy,\qbasispr) = (\mutation_{I_L}^\dynX\cdots\mutation_{I_1}^\dynX)
((\mutation_\quiver^\dynX)^r(\bfy_{t_0},\qbasispr_{t_0})),
\]
where $I_k$ is $G$-orbit corresponding to $j_k$ for each $1\le k\le L$.
It suffices to prove that the $N$-graph
\[
(\ngraph, \nbasis)=(\mutation_{I_L}\cdots\mutation_{I_1})((\ncoxeter)^r(\ngraph_{t_0},\nbasis_{t_0}))
\]
is well-defined and $G$-invariant so that $(\bfy, \qbasispr)=\Psi(\ngraph,\nbasis)$ is $G$-admissible by Proposition~\ref{proposition:equivariance of mutations} as desired.

By Lemma~\ref{lemma:Legendriam Coxeter mutation of type An}, Propositions~\ref{proposition:effect of Legendrian Coxeter mutation}, \ref{proposition:coxeter realization D-type} and \ref{proposition:coxeter realization denegerate type}, the Legendrian Coxeter mutation $\ncoxeter^r(\ngraph_{t_0}, \nbasis_{t_0})$ is realizable so that
\[
\Psi(\ncoxeter^r(\ngraph_{t_0}, \nbasis_{t_0}))
=(\mutation_\quiver)^r(\bfy_{t_0},\qbasispr_{t_0}).
\]

Since $\ncoxeter^r(\ngraph_{t_0}, \nbasis_{t_0})$ is the concatenation of Coxeter paddings on the initial $N$-graph $(\ngraph_{t_0},\nbasis_{t_0})$,
the realizability of $(\mutation_{I_L}\cdots\mutation_{I_1})(\ngraph_{t_0}, \nbasis_{t_0})$ by Legendrian mutations implies that our desired mutation $(\mutation_{I_L}\cdots\mutation_{I_1})(\ncoxeter^r(\ngraph_{t_0}, \nbasis_{t_0}))$ is also realizable by Legendrian mutations because the corresponding mutations of the former do not interact with the Coxeter padding.  

On the other hand, since the indices $j_1,\dots, j_L$ misses the index $i$, the orbits $I_1,\dots, I_L$ misses one orbit $I$ corresponding to $i$.
In other words, the sequence of mutations $\mutation_{I_1},\dots,\mutation_{I_L}$ can be performed inside the subgraph of the exchange graph $\exchange(\Roots(\dynX))$, which is isomorphic to $\exchange(\Roots(\dynX \setminus I))$.
Then the root system $\Roots(\dynX\setminus I)$ is decomposed into $\Roots(\dynX^{(1)}), \dots, \Roots(\dynX^{(\ell)})$, where $\dynX\setminus I = \dynX^{(1)}\cup\cdots\cup\dynX^{(\ell)}$.
Moreover, the sequence of mutations $\mutation_{I_1},\dots,\mutation_{I_L}$ can be decomposed into sequences $\mutation^{(1)},\dots, \mutation^{(\ell)}$ of mutations on $\dynX^{(1)},\dots,\dynX^{(\ell)}$.

Similarly, we may decompose the $N$-graph $(\ngraph_{t_0}, \nbasis_{t_0})$ into $N$-subgraphs 
\[
(\ngraph^{(1)}, \nbasis^{(1)}),\dots,(\ngraph^{(\ell)}, \nbasis^{(\ell)})
\]
along cycles in $I\subset\nbasis_{t_0}$ as done in the previous section.
Then the Legendrian mutation $(\mutation_{I_L}\cdots\mutation_{I_1})(\ngraph_{t_0},\nbasis_{t_0})$ is realizable if and only if so is $\mutation^{(j)}(\ngraph^{(j)},\nbasis^{(j)})$ for each $1\le j\le \ell$.
This can be done by induction on rank of the root system and so the $N$-graph $(\ngraph, \nbasis)$ with $\Psi(\ngraph,\nbasis)=(\bfy,\qbasispr)$ is well-defined.

Finally, the $G$-invariance of $(\ngraph, \nbasis)$ follows from Lemma~\ref{lemma:mutation preserves invariance}.
\end{proof}

\begin{theorem}[Folding of $N$-graphs]\label{thm:folding of N-graphs}
The following holds:
\begin{enumerate}
\item The Legendrian $\lambda(\dynA_{2n-1})$ has at least $\binom{2n}{n}$ distinct Lagrangian fillings up to exact Lagrangian isotopy \textup{(}rel boundary\textup{)} which are invariant under the $\pi$-rotation and  admit the $Y$-pattern of type $\dynB_n$.
item The Legendrian $\lambda(\dynD_{4})$ has at least $8$ distinct Lagrangian fillings up to exact Lagrangian isotopy \textup{(}rel boundary\textup{)} which are invariant under the $2\pi/3$-rotation and admit the $Y$-pattern of type $\dynG_2$.
\item The Legendrian $\lambda(\exdynE_{6})$ has  infinitely many distinct Lagrangian fillings up to exact Lagrangian isotopy \textup{(}rel boundary\textup{)} which are invariant under the $2\pi/3$-rotation and admit the $Y$-pattern of type $\exdynG_2$.
\item The Legendrian $\lambda(\exdynD_{2n})$ with $n\ge 3$ has infinitely many distinct Lagrangian fillings up to exact Lagrangian isotopy \textup{(}rel boundary\textup{)} which are invariant under the $\pi$-rotation and admit the $Y$-pattern of type $\exdynB_n$.
\item The Legendrian $\lambda(\exdynD_4)$ has infinitely many distinct Lagrangian fillings up to exact Lagrangian isotopy \textup{(}rel boundary\textup{)} which are invariant under the $\pi$-rotation and admit the $Y$-pattern of type $\exdynC_2$.
\item The Legendrian $\tilde\lambda(\dynE_{6})$ has at least $105$ distinct Lagrangian fillings up to exact Lagrangian isotopy \textup{(}rel boundary\textup{)} which are invariant under the antisymplectic involution and admit the $Y$-pattern of type $\dynF_4$.
\item The Legendrian $\tilde\lambda(\dynD_{n+1})$ has  at least $\binom{2n}{n}$ Lagrangian fillings up to exact Lagrangian isotopy \textup{(}rel boundary\textup{)} which are invariant under the antisymplectic involution and admit the $Y$-pattern of type $\dynC_n$.
\item The Legendrian $\tilde\lambda(\exdynE_{6})$ has infinitely many distinct Lagrangian fillings up to exact Lagrangian isotopy \textup{(}rel boundary\textup{)} which are invariant under the antisymplectic involution and admit the $Y$-pattern of type $\dynE_6^{(2)}$.
\item The Legendrian $\tilde\lambda(\exdynE_{7})$ has infinitely many distinct Lagrangian fillings up to exact Lagrangian isotopy \textup{(}rel boundary\textup{)} which are invariant under the antisymplectic involution and admit the $Y$-pattern of type $\exdynF_4$.
\item The Legendrian $\tilde\lambda(\exdynD_4)$ has infinitely many distinct Lagrangian fillings up to exact Lagrangian isotopy \textup{(}rel boundary\textup{)} which are invariant under the antisymplectic involution and admit the $Y$-pattern of type $\dynA_5^{(2)}$.
\end{enumerate}
\end{theorem}
\begin{proof}
Let $(\dynX, G,\dynY)$ be one of the triples in Table~\ref{table:foldings}.
By Proposition~\ref{proposition:G-invariant Ngraphs}, each $Y$-seed of the $Y$-pattern of type $\dynY$ is realizable by a $G$-invariant $N$-graph, which gives us a Lagrangian filling with a certain symmetry by Lemma~\ref{lemma:Lagrangian fillings with symmetry}.
This completes the proof.
\end{proof}

\addtocontents{toc}{\protect\setcounter{tocdepth}{1}}
\appendix
\section{\texorpdfstring{$G$}{G}-invariance and \texorpdfstring{$G$}{G}-admissiblity of finite type}
\label{section:invariance and admissibility}

In this section, we will provide a proof of Theorem~\ref{theorem:G-invariance and G-admissibility}. 
Recall from Definition~\ref{def_quiver_of_type_X} that for a finite or affine Dynkin type $\dynX$, 
a quiver $\quiver$ is \emph{of type $\dynX$} if it is mutation equivalent to an acyclic quiver whose underlying graph is isomorphic to the Dynkin diagram of type $\dynX$. 
\begin{lemma}\label{lemma:no weird cycles in An}
Let $\quiver$ be a quiver of type $\dynA_{2n-1}$.
Suppose that $\quiver$ is invariant under an action of $G = \Z/2\Z$. 
Let $\tau \in \Z/2\Z$ be the generator of $G$.
Then there is no oriented cycle of the form
\[
j\to i\to\tau(j)\to\tau(i)\to j
\]
for any vertices $i$ and $j$ of $\quiver$ which are not invariant under $\tau$. 
Here, we are allowing any labelling of the vertices of $\quiver$.
\end{lemma}
\begin{proof}
It is well known that a quiver $\quiver$ of type $\dynA$ corresponds to a triangulation of a polygon, where diagonals and triangles define mutable vertices and arrows (cf.~\cite[Definition~2.2.1]{FWZ_chapter123}). More precisely, for a triangulation $T$ of an $(n+2)$-gon, the quiver $\quiver(T)$ is defined as follows. The frozen vertices of $\quiver(T)$ are labeled by the sides of the $(n+2)$-gon; and the mutable vertices of $\quiver(T)$ are labeled by the diagonals of $T$. If two diagonals, or a diagonal and a boundary segment, belong to the same triangle, we connect the corresponding vertices in $\quiver(T)$ by an arrow whose orientation is determined by the clockwise orientation of the boundary of the triangle. 
Therefore, any minimal cycle in $\quiver$ if exists is of length $3$, which is also proved in~\cite[\S2]{BuanVatne08}.
Hence if an oriented cycle $j\to i\to\tau(j)\to\tau(i)\to j$ of length $4$ exists, then there must be an edge connecting $i$ and $\tau(i)$; or an edge connecting $j$ and $\tau(j)$ in $\quiver$. Hence $b_{i,\tau(i)}\neq0$ or $b_{j,\tau(j)}\neq0$ for $\qbasispr=(b_{k,\ell})=\qbasispr(\quiver)$.

This is impossible because $\quiver$ is $\Z/2\Z$-invariant and so
\begin{equation}\label{equation_b_tauii_is_zero}
b_{i,\tau(i)} = b_{\tau(i),\tau(\tau(i))} = b_{\tau(i), i} = -b_{i,\tau(i)}\quad\Longrightarrow\quad
b_{i,\tau(i)}=0.
\end{equation}
Therefore we are done.
\end{proof}

\begin{proposition}\label{proposition:admissibility for An}
Let $\quiver$ be a quiver of type $\dynA_{2n-1}$, which is $\Z/2\Z$-invariant as above. Then $\quiver$ is $\Z/2\Z$-admissible.
\end{proposition}
\begin{proof}
We will check the conditions~\eqref{mutable},~\eqref{bii'=0}, and~\eqref{nonnegativity_of_bijbi'j} for the admissibility according to Definition~\ref{definition:admissible quiver}.
Let $\tau$ be the generator of $\Z/2 \Z$ and $\qbasispr = (b_{i,j}) = \qbasispr(\quiver)$. 

\noindent \eqref{mutable} Since all vertices in $\quiver$ are mutable, the condition~\eqref{mutable} is obviously satisfied.

\noindent \eqref{bii'=0} The $\Z/2\Z$-invariance of $\quiver$ implies $b_{i,\tau(i)} = 0$ by~\eqref{equation_b_tauii_is_zero}. 

\noindent \eqref{nonnegativity_of_bijbi'j} Finally, we need to prove that for each $i, j$,
\[
b_{i,j}b_{\tau(i),j}\ge 0.
\]

If $j$ is invariant under the action of $\tau$, i.e., $\tau(j) = j$, then we have
\[
b_{i,j}b_{\tau(i),j}=b_{i,j}b_{\tau(i),\tau(j)} = b_{i,j}b_{i,j}\ge 0.
\]
Similarly, if $i$ is invariant under the action of $\tau$, i.e., $\tau(i) = i$, then 
\[
b_{i, j}b_{\tau(i),j} = b_{i,j}b_{i,j}\ge 0.
\]

Suppose that for some $i, j$, which are not invariant under $\tau$, we have
\[
b_{i,j}b_{\tau(i),j}<0.
\]
By changing the roles of $i$ and $\tau(i)$ if necessary, we may assume that $b_{i,j}<0<b_{\tau(i),j}$.
Then we also have
\[
b_{\tau(i),\tau(j)}<0<b_{i,\tau(j)},
\]
which implies that there is an oriented cycle in $\quiver$
\[
j\to i \to \tau(j) \to \tau(i) \to j.
\]
However, this contradicts to Lemma~\ref{lemma:no weird cycles in An} and therefore $\quiver$ satisfies all conditions in Definition~\ref{definition:admissible quiver}.
\end{proof}

\begin{proposition}\label{proposition:admissibility for D4}
Let $\quiver$ be a quiver of type $\dynD_{4}$, which is invariant under the $\Z/3\Z$-action given by
\begin{align*}
1&\stackrel{\tau}{\longleftrightarrow} 1,& 2&\stackrel{\tau}{\longrightarrow} 3\stackrel{\tau}{\longrightarrow} 4\stackrel{\tau}{\longrightarrow} 2.
\end{align*}
Here, we denote by $\tau$ the generator of $\Z/3\Z$ 
and we are allowing any labelling of the vertices of~$\quiver$.
Then the quiver $\quiver$ is $\Z/3\Z$-admissible.
\end{proposition}
\begin{proof}
\noindent \eqref{mutable} This is obvious as before.

\noindent \eqref{bii'=0} Let $\qbasispr=(b_{i,j})=\qbasispr(\quiver)$. Suppose that $b_{2,3}\neq0$. Since the quiver is $\Z/3\Z$-invariant,
\[
b_{2,3}=b_{3,4}=b_{4,2}\neq0
\]
and so $\quiver$ has a directed cycle either
\[
2\to3\to4\to 2\quad\text{or}\quad 2\to4\to3\to 2.
\]
Then according to the value $b_{1,2}$, the underlying graph of the quiver $\quiver$ is either the complete graph $K_4$ or a disconnected graph.
However, both are impossible as shown in~\cite[Figure~1]{BuanTorkildsen09}. 
Therefore, we obtain
\[
b_{2,3}=b_{3,4}=b_{4,2}=0.
\]

\noindent \eqref{nonnegativity_of_bijbi'j} The only entries we need to check are $b_{1,j}$'s, which are all equal by the $\Z/3\Z$-invariance of $\quiver$. Therefore
\[
b_{1,j}b_{1,j'}\ge 0.
\]
This completes the proof.
\end{proof}

\begin{lemma}\label{lemma:no weird cycles in E6}
Let $\quiver$ be a quiver on $[6]$ of type $\dynE_6$, which is invariant under the $\Z/2\Z$-action defined by
\begin{align*}
i&\stackrel{\eta}{\longleftrightarrow} i, i\le 2,&
3&\stackrel{\eta}{\longleftrightarrow} 5,&
4&\stackrel{\eta}{\longleftrightarrow} 6.
\end{align*}
Here, we denote by $\eta$ the generator of $\Z/2\Z$
and we are allowing any labelling of the vertices of~$\quiver$.
Then there is no oriented cycle, which is either
\begin{equation}\label{eq_weird_cycles_in_E6}
3\to4\to5\to6\to3\quad\text{or}\quad
3\to6\to5\to4\to3.
\end{equation}
\end{lemma}
\begin{proof}
We first recall from~\cite[Theorem~1.8]{FZ2_2003} that 
\begin{equation}\label{equation_bij_le_1}
|b_{i,j}|\le 1 \qquad \text{ for all }i,j\in[6].
\end{equation} 
Otherwise, $\quiver$ produces a cluster pattern of infinite type.
Hence, $\quiver$ is a simple directed graph.

Suppose that $\quiver$ contains an oriented cycle in~\eqref{eq_weird_cycles_in_E6}.
By relabeling if necessary, we may assume that the $\quiver$ contains an oriented cycle $3\to4\to5\to6\to 3$.
Then, since $\quiver$ is connected, at least one of vertices $1$ and $2$ is joined with one of vertices $3,4,5$ and $6$ by an edge.
Without loss of generality, we may assume that such a vertex is $1$.

Let $\quiver'$ be the quiver on $\{1,3,4,5,6\}$ obtained by forgetting the vertex $2$ in $\quiver$.
Then, by the invariance of $\quiver$ under $\Z/2\Z$-action, 
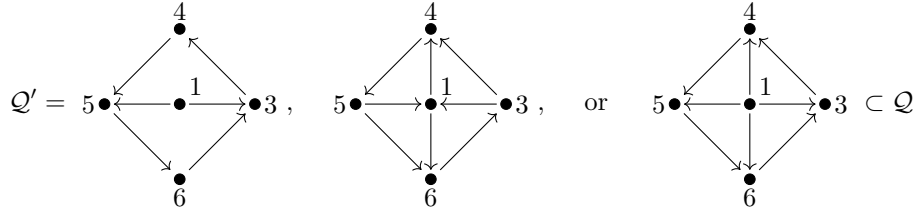
\begin{align*}\label{equation:squares}
\quiver' = 

\subset \quiver
\end{align*}
up to relabeling and the mutation $\mutation_1$.
Then, via further mutations, each of these quivers can be transformed to a quiver producing a cluster pattern of infinite type because of the condition~\eqref{equation_bij_le_1} as follows:
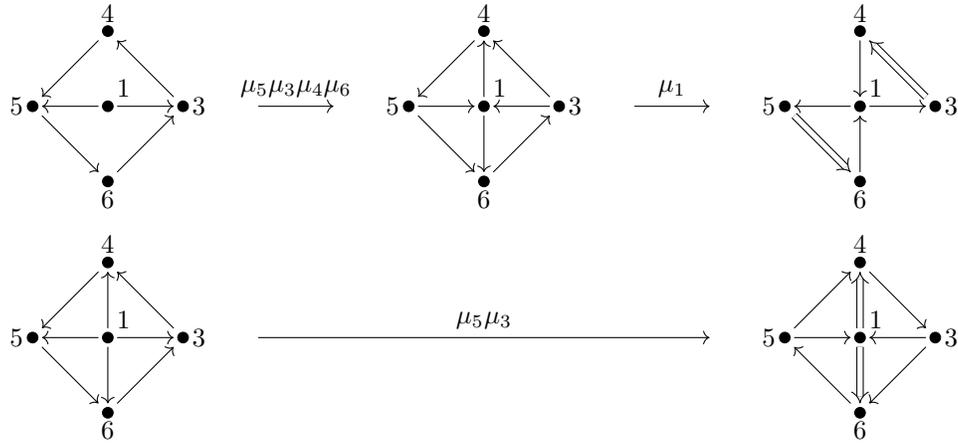
\begin{align*}
&%
\end{align*}
Here, the label $\mutation_5\mutation_3\mutation_4\mutation_6$ by mean that we are applying the mutation $\mutation_6$ first, then $\mutation_4$, and so on. 
Since any subquiver of a quiver mutation equivalent to $\quiver$ is of finite type, we get a contradiction which completes the proof.
\end{proof}
\begin{remark}\label{rmk_proof_of_E6}
Since there are only finitely many quivers of type $\dynE_6$, the above lemma can be verified by a computer but we gave here a combinatorial proof.
\end{remark}

\begin{proposition}\label{proposition:admissibility for Dn and E6}
Let $\quiver$ be a quiver of type $\dynX=\dynD_{n+1}$ or $\dynE_6$, which is invariant under $\Z/2\Z$-action defined by
\begin{align*}
i&\stackrel{\eta}{\longleftrightarrow} i, i < n,&
n&\stackrel{\eta}{\longleftrightarrow} n+1
\end{align*}
for $\dynX=\dynD_{n+1}$, or
\begin{align*}
i&\stackrel{\eta}{\longleftrightarrow} i, i\le 2,&
3&\stackrel{\eta}{\longleftrightarrow} 5,&
4&\stackrel{\eta}{\longleftrightarrow} 6,
\end{align*}
for $\dynX=\dynE_6$.
Here, $\eta$ is the generator of $\Z/2\Z$ and we are allowing any labelling of the vertices of~$\quiver$.
Then the quiver $\quiver$ is $\Z/2\Z$-admissible.
\end{proposition}
\begin{proof}
\noindent \eqref{mutable} This is obvious as before.

\noindent \eqref{bii'=0} Let $\qbasispr=(b_{i,j})=\qbasispr(\quiver)$. Then, by the $\Z/2\Z$-invariance of $\quiver$, 
\[
b_{i,\eta(i)}=b_{\eta(i), \eta(\eta(i))} = b_{\eta(i), i} = -b_{i, \eta(i)}\quad
\Longrightarrow\quad
b_{i,\eta(i)}=0.
\]

\noindent \eqref{nonnegativity_of_bijbi'j} If $\dynX=\dynD_{n+1}$, then we only need to show
\[
b_{i,n}b_{i,n+1}\ge 0
\]
for $i<n$. This is obvious since 
\[
b_{i,n+1} = b_{\eta(i), \eta(n+1)} = b_{i,n}.
\]

If $\dynX=\dynE_6$, then all we need to show inequalities
\begin{align*}
b_{i,j}b_{i,j+2}&\ge 0,&
b_{3,4}b_{3,6}&\ge 0
\end{align*}
hold for $i=1,2$ and $j=3,4$.

The first inequality is obvious since
\begin{align*}
b_{i,j+2}&=b_{\eta(i),\eta(j+2)} = b_{i,j}.
\end{align*}
Suppose that $b_{3,4}b_{3,6}<0$. Then, since $b_{3,4}=b_{5,6}$ and $b_{3,6}=b_{5,4}$, the $\quiver$ has a loop either
\[
3\to 4\to 5\to 6 \to 3\quad\text{or}\quad
3\to 6\to 5\to 4 \to 3
\]
which yields a contradiction by Lemma~\ref{lemma:no weird cycles in E6}. This completes the proof.
\end{proof}

\begin{proof}[Proof of Theorem~\ref{theorem:G-invariance and G-admissibility}] 
Let $\quiver$ be a $G$-invariant quiver. Then it is $G$-admissible because of Propositions~\ref{proposition:admissibility for An}, \ref{proposition:admissibility for D4}, and \ref{proposition:admissibility for Dn and E6}. 
\end{proof}

\section{Supplementary pictorial proofs}\label{sec:supplementary pictorial proofs}
\subsection{Justifications of moves \Move{DI} and \Move{DII} for denegenerate $N$-graphs}
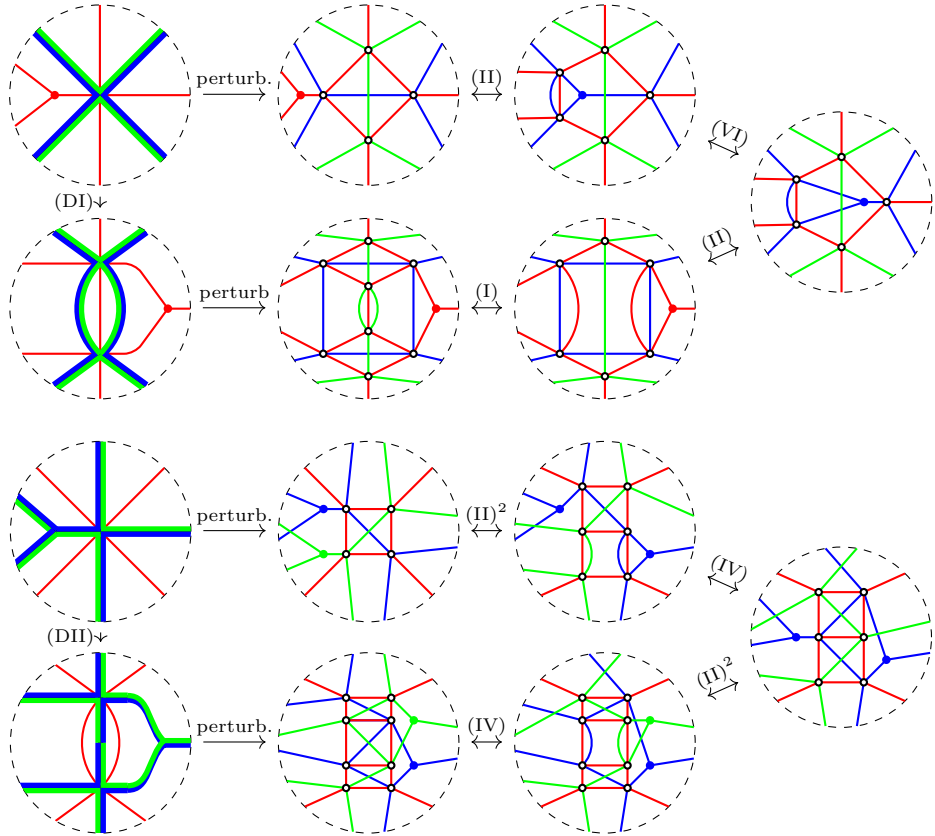
\label{appendix:DI and DII}
\[
\begin{tikzcd}[row sep=-3pc, column sep=1pc]
 
\ar[from=l,leftrightarrow, "\Move{I}"] 
\end{tikzcd}
\]

\[
\begin{tikzcd}[row sep=-3pc, column sep=1pc]
%
 
\end{tikzcd}
\]

\subsection{Justifications of Legendrian local mutations in degenerate $N$-graphs}
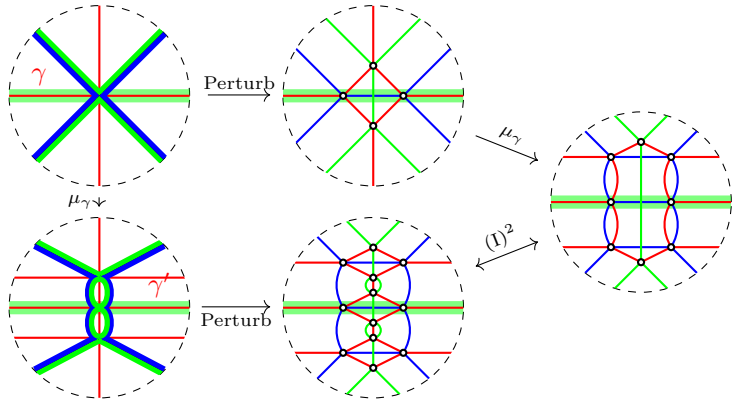
\label{appendix:local mutations}
\[
\begin{tikzcd}[row sep=-3pc]
%
\arrow[dl, leftrightarrow, sloped, "\Move{I}^2"]\\
%
\end{tikzcd}
\]

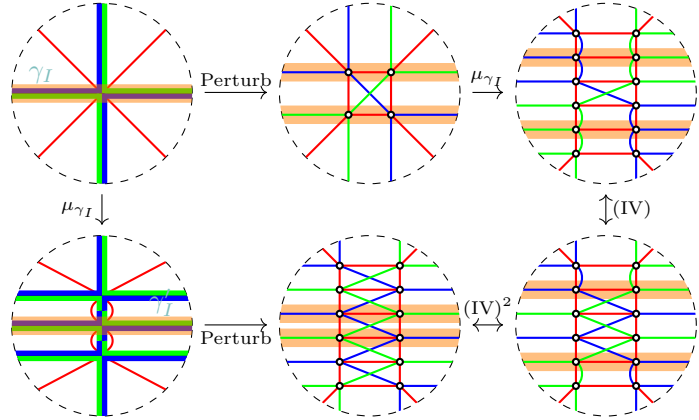
\[
\begin{tikzcd}[row sep=1pc, column sep=1pc]
%
\end{tikzcd}
\]
See Remark~\ref{rem:local mutation} for the incoherency of the boundaries of the initial $N$-graphs and the terminal ones.

\subsection{Equivalence between $\ngraph(1,b,c)$ and a stabilization of $\ngraph(\dynA_n)$}
\label{appendix:tripod with a=1 is of type An}

A stabilization of $\ngraph(\dynA_n)$ is a $3$-graph which is $\boundary$-Legendrian isotopic to a $3$-graph $S(\ngraph(\dynA_n))$ given as follows:
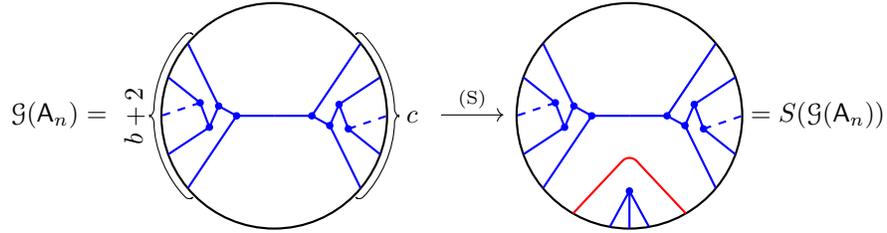
\[
\begin{tikzcd}
\ngraph(\dynA_n)=%
=S(\ngraph(\dynA_n))
\end{tikzcd}
\]

Now we attach the annular $3$-graph corresponding to Legendrian isotopy from $S(\beta(\dynA_n))$ to~$\beta(1,b,c)$ given above.
Then we have the following $3$-graph which is $\boundary$-Legendrian isotopic to~$S(\ngraph(\dynA_n))$.
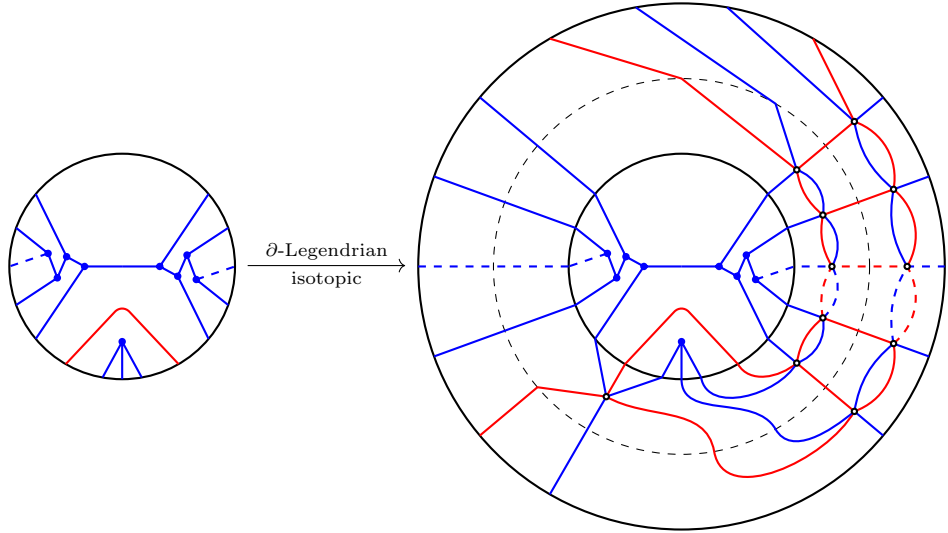
\[
\begin{tikzcd}[column sep=5pc]
%
\end{tikzcd}
\]
By applying the following generalized push-through move twice, once for each attached annulus region,
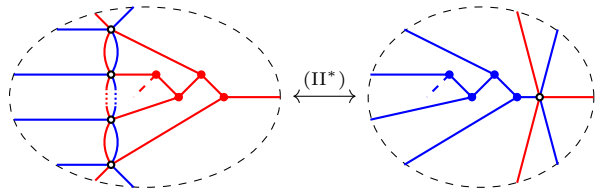
\[
\begin{tikzcd}
%
\end{tikzcd}
\]
we obtain the $3$-graph in the left of the following three equivalent $3$-graphs
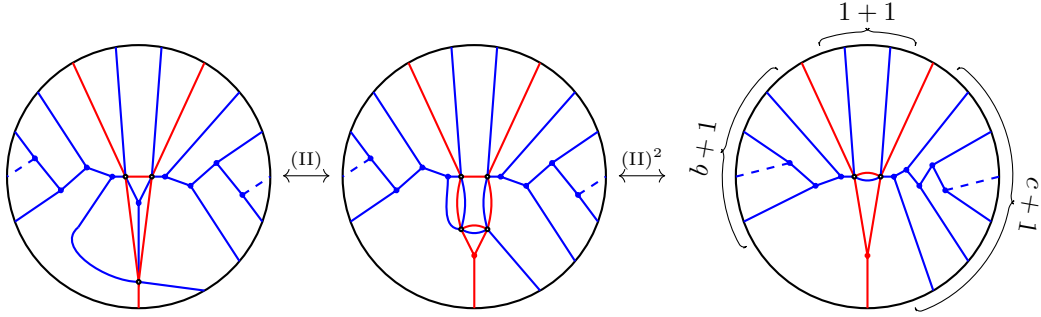
\[
\begin{tikzcd}[column sep=1.5pc]
%
\end{tikzcd}
\]
where the right one is equivalent to the $3$-graph $\ngraph(1,b,c)$ via the Move~\Move{II} as claimed.

\subsection{Proof of Lemma~\ref{lemma:full rank}: Equivalence between $(\mathscr{G}^{\mathsf{brick}}(a,b,c),\nbasis^{\mathsf{brick}}(a,b,c))$ and $(\mathscr{G}(a,b,c),\nbasis(a,b,c))$}
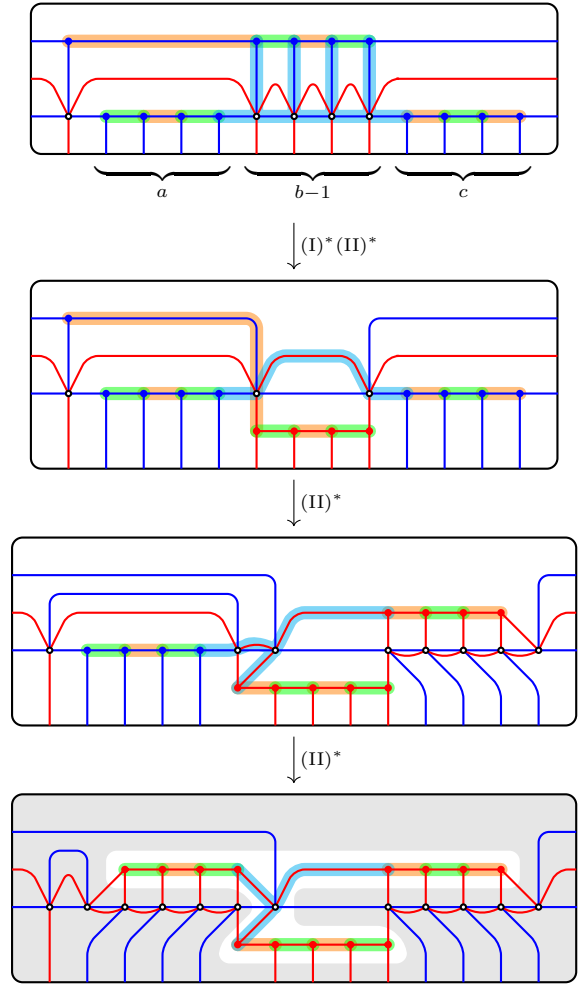
\label{appendix:Ngraph of type abc}

\[
\begin{tikzcd}
%
\end{tikzcd}
\]

The innermost $N$-graph is the same as $\overline{\ngraph(a,b,c)}$ up to Legendrian mutations, which is $\boundary$-Legendrian isotopic to the Legendrian Coxeter mutation of $\ngraph(a,b,c)$ by Proposition~\ref{proposition:effect of Legendrian Coxeter mutation}.

\subsection{The proof of Lemma~\ref{lemma:Ngraphs of affine Dn}: Equivalence between $(\mathscr{G}^{\mathsf{brick}}(\exdynD_n),\nbasis^{\mathsf{brick}}(\exdynD_n))$ and $(\mathscr{G}(\exdynD_n),\nbasis(\exdynD_n))$}\label{appendix:Ngraph of type affine Dn}

We first show that $\ngraph^{\mathsf{brick}}(\exdynD_n)$ is Legendrian mutation equivalent to the following $N$-graph up to $\boundary$-Legendrian isotopy. Even though we omit the data of cycles for the pictorial simplicity, one can keep track $\nbasis(\exdynD_n)$ through the following ($\boundary$-)Legendrian isotopies:
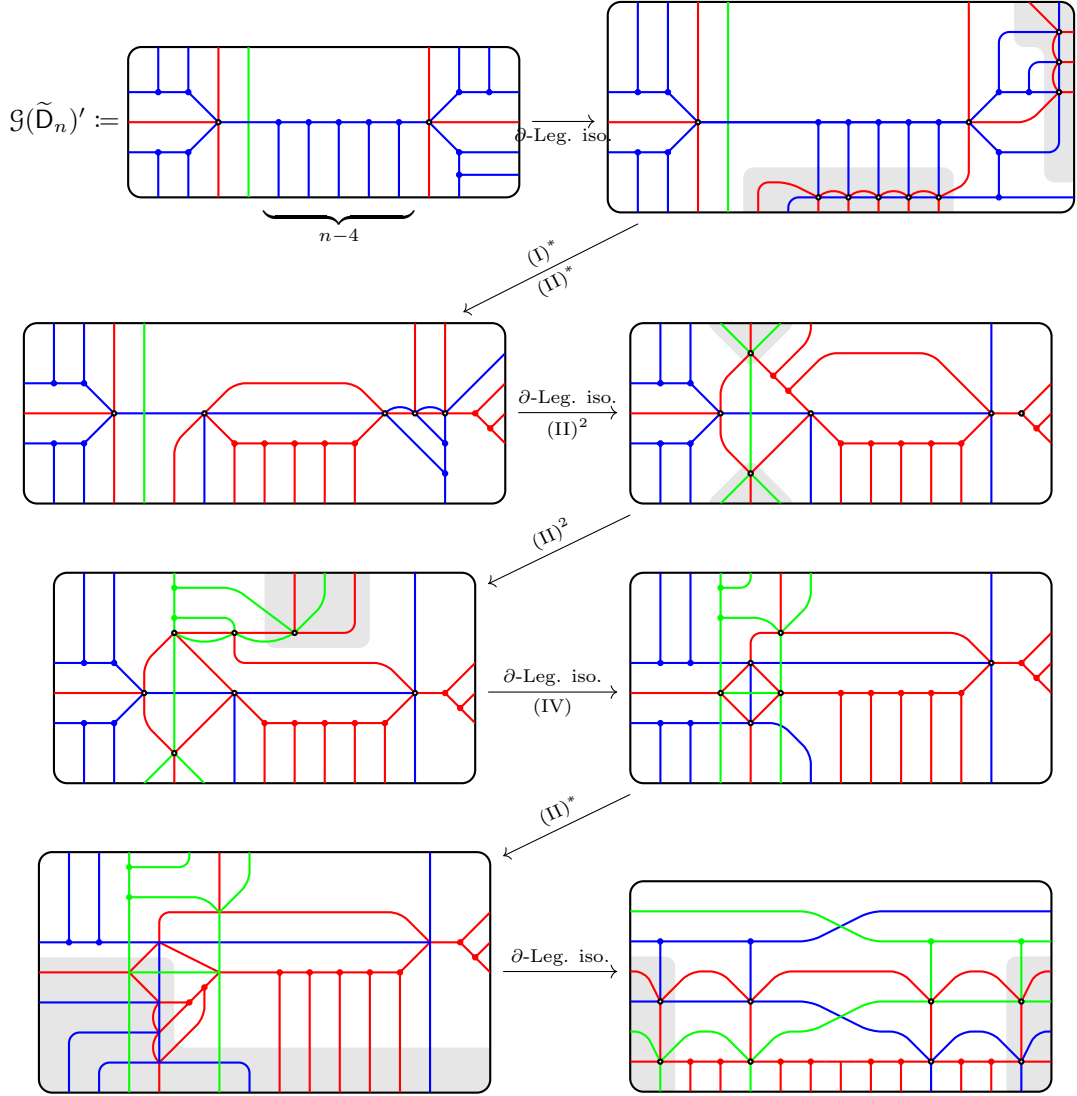
\[
\begin{tikzcd}
\ngraph(\exdynD_n)'\coloneqq

\end{tikzcd}
\]

On the other hand, we have the following move (Z) from Section~\ref{sec:ngraph for exdynDn}
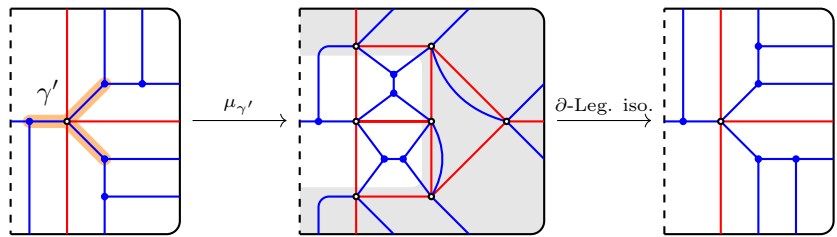
\[
\begin{tikzcd}[column sep=3pc]
%
\end{tikzcd}
\]
which flips up the downward leg.
Finally, the downward and upward legs can be interchanged via Legendrian mutations and therefore the $N$-graph $\ngraph(\exdynD_n)'$ is Legendrian mutation equivalent to $\ngraph(\exdynD_n)$ up to $\boundary$-Legendrian isotopy.

\subsection{A proof of Lemma~\ref{lem:exdynD_4}: Equivalence between $(\tilde{\mathscr{G}}(\exdynD_4),\tilde{\nbasis}(\exdynD_4))$ and $(\mathscr{G}(\exdynD_4),\nbasis(\exdynD_4))$}\label{appendix:affine D4}
It is enough to show the equivalence between $(\tilde\ngraph(\exdynD_4),\tilde\nbasis(\exdynD_4))$ and $(\ngraph^{\mathsf{brick}}(\exdynD_4),\nbasis^{\mathsf{brick}}(\exdynD_4))$ by Lemma~\ref{lemma:Ngraphs of affine Dn}.

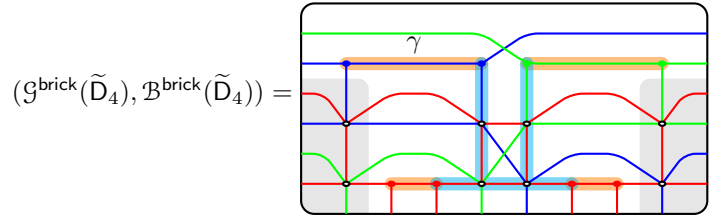
\[
\begin{tikzcd}
(\ngraph^{\mathsf{brick}}(\exdynD_4),\nbasis^{\mathsf{brick}}(\exdynD_4))=

\end{tikzcd}
\]
By cutting out the shaded region and taking a Legendrian mutation on $\cycle$, the $\sfI$-cycle between the two blue vertices, we have a degenerate $N$-graph below, (here we omit cycle data,) which is $\tilde\ngraph(\exdynD_4)$ up to $\boundary$-Legendrian isotopy and Legendrian mutations.
\[
\begin{tikzcd}
%
\end{tikzcd}
\]

\bibliographystyle{plain}
\bibliography{references}

\end{document}